\documentclass[12pt,PhD]{muthesisDateChange}

\usepackage{amsmath,amssymb,amsthm,wasysym,enumerate,multind1, tikz,multicol, multirow}
\usepackage[hyphens]{url}
\theoremstyle{definition}
\newtheorem{thm}[equation]{Theorem}
\newtheorem*{thm*}{Theorem}
\newtheorem{cor}[equation]{Corollary}
\newtheorem*{cor*}{Corollary}
\newtheorem{lemma}[equation]{Lemma}
\newtheorem*{lemma*}{Lemma}
\newtheorem{defn}[equation]{Definition}
\newtheorem*{defn*}{Definition}
\newtheorem{prop}[equation]{Proposition}
\newtheorem*{prop*}{Proposition}
\newtheorem{rem}[equation]{Remark}
\newtheorem*{rem*}{Remark}
\newtheorem{rems}[equation]{Remarks}
\newtheorem*{rems*}{Remarks}

\newtheorem*{claim*}{Claim}

\newtheorem*{proofofclaim*}{Proof of Claim}
\newtheorem{que}[equation]{Question}
\newtheorem*{que*}{Question}
\newtheorem{ques}[equation]{Questions}
\newtheorem*{ques*}{Questions}
\newtheorem{conj}[equation]{Conjecture}
\newtheorem*{conj*}{Conjectuve}

\newtheorem*{ass*}{Assumption}
\newtheorem{hyp}[equation]{Hypotheses}
\newtheorem*{hyp*}{Hypotheses}
\newtheorem{hypsing}[equation]{Hypothesis}
\newtheorem{example}[equation]{Example}
\newtheorem*{example*}{Example}

\newtheorem*{obs*}{Observation}

\newtheorem*{fact*}{Fact}
\newtheorem{conditions}[equation]{Conditions}
\newtheorem*{conditions*}{Conditions}
\newtheorem{code}[equation]{Code}
\newtheorem*{code*}{Code}

\newtheorem{notation}[equation]{Notation}

\numberwithin{equation}{section} 
\makeatletter
\let\c@equation\c@figure
\makeatother
\numberwithin{figure}{section}
\allowdisplaybreaks

\makeatletter
\newenvironment{proofof}[1]{\par
  \pushQED{\qed}%
  \normalfont \topsep6\p@\@plus6\p@\relax
  \trivlist
  \item[\hskip\labelsep
        \bfseries
    Proof of #1\@addpunct{.}]\ignorespaces
}{%
  \popQED\endtrivlist\@endpefalse
}
\makeatother

\newcommand{\C}{\ensuremath{\mathbb{C}}}

\newcommand{\N}{\ensuremath{\mathbb{N}}} 
 
\newcommand{\Z}{\ensuremath{\mathbb{Z}}}

\newcommand{\calO}{\ensuremath{\mathcal{O}}}
\newcommand{\calL}{\ensuremath{\mathcal{L}}}
\newcommand{\calF}{\ensuremath{\mathcal{F}}}

\newcommand{\proj}[2]{\ensuremath{\mathbb{P}_{#1}^{#2}}}

\newcommand{\twomat}[4]{\begin{pmatrix} #1 & #2 \\ #3 & #4 \end{pmatrix}}

\newcommand{\sesge}[3]{\ensuremath{1 \rightarrow #1 \rightarrow #2 \rightarrow #3 \rightarrow 0}}
\newcommand{\les}[3]{\ensuremath{\xymatrix{
    0 \ar[r] & H^0(#1) \ar[r] & H^0(#2) \ar[r] & H^0(#3)
                \ar@{->} `r/8pt[d] `/10pt[l] `^dl[ll] `^r/3pt[dll] [dll] & \\
             & H^1(#1) \ar[r] & H^1(#2) \ar[r] & H^1(#3) 
                \ar@{->} `r/8pt[d] `/10pt[l] `^dl[ll] `^r/3pt[dll] [dll] & \\
             & H^2(#1) \ar[r] & H^2(#2) \ar[r] & H^2(#3) \ar[r] & 0 
    }}}

\newcommand{\exclaim}{!}
\makeindex{notation}
\makeindex{term}

\makeatletter
\def\printindex#1#2{\@restonecoltrue\if@twocolumn\@restonecolfalse\fi
  \columnseprule \z@ \columnsep 35pt
  \csname phantomsection\endcsname
  \chapter*{#2}
  \markright{\uppercase{#2}}
  \addcontentsline{toc}{chapter}{#2}
  \begin{multicols}{2}
  \@input{#1.ind}%
  \end{multicols}%
}
\makeatother

\begin{document}
\title{Cocycle twists of algebras}
\author{Andrew Phillip Davies}
\school{Mathematics}
\faculty{Engineering and Physical Sciences}
\def\wordcount{42491}

% Uncomment the line below to suppress the `List of Tables' page (optional)
\tablespagefalse

% Uncomment the line below to suppress the `List of Figures' page (optional)
\figurespagefalse

% Uncomment the line below to use a customised Declaration statement
%\def\declaration{All the work in this thesis has been sourced from Google.}

\beforeabstract
Let $A$ be a $k$-algebra where $k$ is an algebraically closed field and $G$ be a finite abelian group whose order is not divisible by the characteristic of $k$. If $G$ acts on $A$ by $k$-algebra automorphisms then the action induces a $G$-grading on $A$ which, in conjunction with a normalised 2-cocycle of the group, can be used to twist the multiplication of the algebra. Such twists can be formulated as Zhang twists \cite{zhang1998twisted} as well as in the language of Hopf algebras, as done in \cite[\S 7.5]{montgomery1993hopf}. 

We investigate such cocycle twists with an emphasis on the situation where $A$ also possesses a connected graded structure and the action of $G$ respects this grading. The twisting operation uses the induced $G$-grading of $A$ rather than its \N-grading, as is often the case with Zhang twists. As a result, geometric data encoded by special modules in the category of graded right modules, $\text{grmod}(A)$, may not be preserved. Nevertheless, we show using an alternative construction of the twist and faithful flatness arguments that many properties are preserved; the strongly noetherian property, finite global dimension and Artin-Shelter regularity for example. 

The above concepts are then illustrated by applying cocycle twists to the 4-dimen\-sional Sklyanin algebras, $A:=A(\alpha,\beta,\gamma)$, first studied ring-theoretically in \cite{smith1992regularity}. Such algebras are examples of noncommutative projective surfaces and display many interesting geometric properties. Much of this geometry is controlled by a factor ring $B:=B(E,\calL,\sigma)$, which is a twisted homogeneous coordinate ring.

We define an action of the Klein four-group $G=(C_2)^2$ on $A$ such that the action restricts to the factor ring $B$. The cocycle twists under the induced $G$-grading, namely $A^{G,\mu}$ and $B^{G,\mu}$ respectively, have very different geometric properties to their untwisted counterparts. While $A$ has point modules parameterised by an elliptic curve $E$ and four extra points, $A^{G,\mu}$ has only 20 point modules when the automorphism $\sigma$ has infinite order. The point modules over $A$ can be used to construct fat point modules of multiplicity 2 over $A^{G,\mu}$, and there are isomorphisms among such objects corresponding to orbits of a natural action of $G$ on $E$. This is just one example of the interplay between the 1-critical modules over $A$ and $A^{G,\mu}$.

The ring $B^{G,\mu}$ falls under the purview of Artin and Stafford's classification of noncommutative curves \cite{artin2000semiprime}, and we describe it in geometric terms. It can be expressed as a more general twisted homogeneous coordinate ring, where the key object is a sheaf of orders over $\calO_{E^{G}}$. In fact, the fat point modules over $A^{G,\mu}$ --- which are also $B^{G,\mu}$-modules --- can be explained by an equivalence of categories consequent to this work.

Other examples of twists are also studied, one of which relates to Rogalski and Zhang's classification of AS-regular algebras of dimension 4 with three degree 1 generators and an additional $\Z^{2}$-grading \cite{rogalski2012regular}. Their work classified such algebras into 8 families up to isomorphism; we demonstrate that several of these families are related by cocycle twists.

\afterabstract

\prefacesection{Acknowledgements}
I would first like to thank EPSRC for their generous grant enabling me to study for a PhD, since undoubtedly I could not have afforded to do so otherwise. 

My supervisor, Professor Toby Stafford, has been a superb guide and mentor throughout the last few years; I have certainly come to understand why his former students hold him in such high esteem. Whenever I may have been feeling frustrated with maths, a meeting with Toby would invariably result in me leaving full of his contagious enthusiasm. 

I would also like to acknowledge the fellow PhD students I have known at Manchester, particularly those with whom I have shared an office. Their willingness to discuss maths, other `enlightening' conversations and numerous distractions have all made my time as a PhD student thoroughly enjoyable. 

Finally, I would like to thank my family and my girlfriend Kirsti -- jag \"{a}lskar dig. I'm very lucky to have you, and I preemptively forgive you if you give up reading after a couple of pages/paragraphs/words\ldots

% The next line is NOT optional and MUST appear
\afterpreface

% Indices with Chapter headings (and chapter status in the table of contents)
\printindex{notation}{Index of notation}
\printindex{term}{Index of terminology}

\chapter{Introduction}\label{chap: introduction}
\chaptermark{Introduction}

\section{Overview}\label{sec: introoverview}
\sectionmark{Overview}
In this thesis we study a twisting operation on algebras. This operation is closely related to the classical notion of
using a 2-cocycle to deform the multiplication in a group-graded algebra. Our primary concern is to establish whether
certain properties are preserved under the twisting operation, the first results in this direction being proved in
\cite{zhang1998twisted} and \cite{montgomery2005algebra}. In the former the twists we study are phrased as Zhang twists,
while Montgomery's results in the latter hold for generalised twists in the setting of Hopf algebras.
 
We use another construction of these \emph{cocycle twists} that was independently described in \cite{bazlov2012cocycle}.
This construction gives one much greater traction when trying to prove that properties are preserved under twisting,
since it enables the use of faithful flatness arguments. Our main results are stated later in this introduction, namely
in Proposition \ref{prop: twoconstrsaresame} and Theorems \ref{thm: thebigone}, \ref{thm: sklyanin} and \ref{thm:
geometricthcrdesc}.

The motivation for our work was an example of Odesskii, appearing in the survey paper \cite{odesskii2002elliptic} on
\emph{elliptic algebras}\index{term}{elliptic algebras}. Although the example as we state it does not use a specific
type of algebra, the original example in \cite{odesskii2002elliptic} used a \emph{4-dimensional Sklyanin algebra}. In
Chapter \ref{chap: sklyanin} we will study cocycle twists of such Sklyanin algebras, although they will also be
discussed briefly in  \S\ref{subsec: examples} (see Definition \ref{def: sklyanin relations} for a presentation by
generators and relations).

In order to state Odesskii's example, we need to make the following assumptions. Let $k= \C$ and consider an associative
$k$-algebra $A$ that is generated as a $k$-algebra by $x_0,x_1,x_2$ and $x_3$. We will assume that the Klein four-group
$G=(C_2)^2= \langle g_1, g_2 \rangle$ acts on $A$ by algebra automorphisms, i.e. $G \rightarrow
\text{Aut}_{\text{alg}}(A)$. Denoting the action of $g \in G$ by the superscript $-^g$, the action is defined on
generators by
\begin{equation}\label{eq: odesskiiaction}
x_i^{g_{1}} = x_{i+2},\;\;\; x_i^{g_{2}} = (-1)^i x_{i},
\end{equation}
where $i \in \{0,1,2,3\}$ and indices are taken modulo 4. 

There is an action of $G$ by $k$-algebra automorphisms on $M_2(k)$, the ring of $2 \times 2$ matrices over $k$. This
action is defined by 
\begin{equation}\label{eq: matrixaction}
M^{g_{1}} =\begin{pmatrix}
-1 & 0 \\
0 & 1 
\end{pmatrix}M\begin{pmatrix}
-1 & 0 \\
0 & 1 
\end{pmatrix},\; M^{g_{2}}=\begin{pmatrix}
0 & 1 \\
1 & 0 
\end{pmatrix}M\begin{pmatrix}
0 & 1 \\
1 & 0 
\end{pmatrix},
\end{equation}
for a matrix $M \in M_2(k)$ and the group generators $g_1$ and $g_2$. A 2-cocycle arises here via an action of $G$ on
$k^2$ -- such objects are important for the constructions we will soon describe.

Now consider the tensor product of $k$-algebras $A \otimes_{k} M_2(k)$. Given the action of $G$ in \eqref{eq:
matrixaction} and the embedding $G \rightarrow \text{Aut}_{\text{alg}}(A)$, there is a natural diagonal action of $G$ on
the tensor product, where
\begin{equation}\label{eq: diagaction}
(a \otimes_k M)^g = a^g \otimes_k M^g,
\end{equation}
for all $a \in A$ and $M \in M_2(k)$. 

Odesskii's construction concludes by taking the invariant ring of $A \otimes_{k} M_2(k)$ under this diagonal action. We
highlight this algebra for future reference.
\begin{example}[{\cite[Introduction]{odesskii2002elliptic}}]\label{ex: odesskii}
Odesskii's example is given by the invariant ring $(A \otimes_{k} M_2(k))^G$.
\end{example}

It is natural to ask if the properties of $A$ are shared by $(A \otimes_{k} M_2(k))^G$. Moreover, can the construction
can be generalised to any algebra or group? Our attempts to answer these questions form the basis of the work in this
thesis. 

\subsection{Twisting theory}\label{subsec: theoryoftwist}
Before stating some of our main results we must define the constructions involved in our work. Let us make the following
assumptions that will remain in place until the end of \S\ref{subsec: theoryoftwist}. Our base field $k$ will be assumed
to be algebraically closed and $A$ will denote an associative $k$-algebra with identity. The finite abelian group $G$
will satisfy $\text{char}(k) \nmid |G|$ and will act on $A$ by $k$-algebra automorphisms. We will also fix a duality
isomorphism between $G$ and its group of characters $G^{\vee}$, where the character to which an element $g \in G$ is
mapped is denoted by $\chi_g$.

\subsubsection{Construction 1}\label{subsubsec: constr1}
Let us define the first of two twisting constructions that we will use. There are two key components of Construction 1;
first, the inducement of a $G$-grading on $A$ via the action of $G$, and second, the notion of using a normalised
2-cocycle to twist the $G$-graded multiplication in $A$.  

The definition of a \emph{$G$-graded algebra}\index{term}{g@$G$-grading} is given in Definition \ref{defn:
ggradedalgebra}. For now it suffices to state that for all $g\in G$ the corresponding component of the induced
$G$-grading on $A$ is defined by
\begin{equation}\label{eq: inducedGgrading}
A_g := \{a \in A:\; a^h = \chi_{g^{-1}}(h)a,\; \forall \; h \in G \}. 
\end{equation}

Given such a grading, one can use a normalised 2-cocycle to deform the multiplicative structure of $A$.  
\begin{defn}\label{def: normalised2cocycle}
Consider a function $\mu: G \times G \rightarrow k^{\times}$ satisfying the following relations for all $g,h,l \in G$
and the group identity element $e$:
\begin{equation}\label{eq: cocycleidinto}
\mu(g,h)\mu(gh,l)=\mu(g,hl)\mu(h,l), \;\;\; \mu(e,g)=\mu(g,e)=1.
\end{equation} 
We say that $\mu$ is a \emph{normalised 2-cocycle}\index{term}{2-cocycle!normalised} of $G$ with values in $k^{\times}$.
\end{defn}

Although normalised 2-cocycles can be interpreted in terms of cohomology --- as we do in \S\ref{subsec: cocycletwists}
--- for now we will use them solely to define a new multiplication on $A$. We will write $(A,\cdot)$ to denote the
$G$-graded algebra $A$ and the standard multiplication $\cdot : A \times A \rightarrow A$, usually written as
juxtaposition of elements.

One can define a new multiplication on the underlying $G$-graded vector space structure of $A$ as follows: for
homogeneous elements $a \in A_g$ and $b \in A_h$ define $a \ast_{\mu} b:=\mu(g,h)ab$. This gives rise to the algebra
$(A, \ast_{\mu})$. By virtue of $\mu$ satisfying \eqref{eq: cocycleidinto} this new algebra is associative and has the
same algebra identity element as $(A,\cdot)$ (proved in Proposition \ref{prop: preserveassociative}). We will use the
notation $A^{G,\mu}:=(A, \ast_{\mu})$\index{notation}{a@$A^{G,\mu}$} and call such an object a cocycle twist of the
induced $G$-grading on $A$ using the normalised 2-cocycle $\mu$. For brevity we may also use the term \emph{cocycle
twist}\index{term}{cocycle twist} of $A$, since the notation makes clear both the group acting and the normalised
2-cocycle.

Perhaps the simplest nontrivial example of a cocycle twist is that of a twisted group algebra.
\begin{example}\label{ex: twistedgroupalgebra}
Consider the group algebra $AG := A \otimes_k kG$\index{notation}{a@$AG$}, with $A$ and $G$ subject to the assumptions
at the beginning of \S\ref{subsec: theoryoftwist}. The multiplication in this algebra is given by 
\begin{equation*}\label{eq: groupalgebramult}
(a \otimes_k g)(b \otimes_k h) = ab \otimes_k gh, 
\end{equation*}
for all $a,b \in A$ and $g,h \in G$. 

The group algebra has an obvious $G$-graded structure given by $A_g=\{a \otimes_k g : \; a \in A\}$. Twisting this
grading using a normalised 2-cocycle $\mu$ produces a \emph{twisted group algebra}\index{term}{twisted group algebra}
$AG_{\mu}$\index{notation}{a@$AG_{\mu}$}. The multiplication in this algebra is defined by
\begin{equation*}
(a \otimes_k g) \ast_{\mu} (b \otimes_k h) =\mu(g,h)(ab \otimes_k gh),
\end{equation*}
for all $a, b \in A$ and $g, h \in G$.
\end{example}

Cocycle twists have been studied in many other places and generalised vastly -- they can be described as Zhang twists
for example (see Theorem \ref{thm: cocycleaszhang}), which were studied in detail by Zhang in \cite{zhang1998twisted}.
One can also recover them from the \emph{twisted $H$-comodule algebra} construction of \cite[\S 7.5]{montgomery1993hopf}
by using the Hopf algebra $H=kG$, the group algebra of a finite abelian group.

Let us exhibit some of the results already in the literature in relation to the preservation of properties under cocycle
twisting. Although some of the results hold more generally, we state them in the setting of Construction 1.
\begin{prop}\label{prop: propspreservedalready}
Assume that $A$ and $G$ are subject to the assumptions at the beginning of \S\ref{subsec: theoryoftwist} and let $\mu$
be a normalised 2-cocycle. Then if $A$ has one of the following properties, so does $A^{G,\mu}$:
\begin{itemize}
 \item[(i)] it is finitely generated as a $k$-algebra \cite[cf. Proposition 3.1(1)]{montgomery2005algebra}; 
 \item[(ii)] it satisfies the polynomial identity (PI) property (cf. \cite[Proposition 5.6]{zhang1998twisted},
\cite[Proposition 3.1(4)]{montgomery2005algebra}); 
 \item[(iii)] it is noetherian (cf. \cite[Proposition 5.1]{zhang1998twisted}, \cite[Proposition
3.1(3)]{montgomery2005algebra}). 
\end{itemize} 
\end{prop}
\begin{rem}
One can replicate this construction without using an identification of $G$ and $G^{\vee}$ by instead using a 2-cocycle
over $G^{\vee}$.
\end{rem}

\subsubsection{Construction 2}\label{subsubsec: constr2}
We now move on to the second construction. One first defines an action of $G$ on the twisted group algebra $kG_{\mu}$ by
$k$-algebra automorphisms. The algebra $kG_{\mu}$ can be obtained by taking $A=k$ in Example \ref{ex:
twistedgroupalgebra}. 

The action of $G$ on $kG_{\mu}$ is defined as follows: for all $g,h \in G$ and $\alpha \in k$, let $(\alpha \otimes_k
g)^h:=\chi_g(h) \alpha\otimes_k g$. By assumption one has an action of $G$ on $A$ by $k$-algebra automorphisms, thus one
can define a diagonal action of $G$ on the tensor product $A \otimes_k kG_{\mu}$ by
\begin{equation*}\label{eq: diagactionconstr2}
(a \otimes_k g)^h := \chi_g(h)(a^h \otimes_k g),
\end{equation*}
for all $a \in A$ and $g, h \in G$. 

The construction is then completed by taking the invariant ring under this action, $(A \otimes_k kG_{\mu})^G$, which is
isomorphic to $(AG_{\mu})^G$\index{notation}{a@$(AG_{\mu})^G$} as a $k$-algebra.
 
Note that for the Klein-four group $G$ and a particular normalised 2-cocycle $\mu$ one has $kG_{\mu} \cong M_2(k)$ as
$k$-algebras by Lemma \ref{lem: kgmuiso}. This provides some indication that Example \ref{ex: odesskii} can be described
using Construction 2. The demonstration of this is delayed until \S\ref{subsec: odesskiiegtwist}.

\subsubsection{Consequences}\label{subsubsec: consequences}
We now gather together a collection of our main results regarding Constructions 1 and 2, the first of which shows that
they are in fact the same. 
\begin{prop}[{Proposition \ref{prop: twoconstrequal}}]\label{prop: twoconstrsaresame}
Assume that $A$ and $G$ satisfy the base assumptions of \S\ref{subsec: theoryoftwist}. Then for any normalised 2-cocycle
$\mu$ one has an isomorphism of $k$-algebras $A^{G,\mu} \cong (AG_{\mu})^G$.
\end{prop}

The equivalence of the two twisting constructions allows us to adopt the notation $A^{G,\mu}$ for a cocycle twist formed
using \textit{either} construction. Our main tool for proving the preservation of properties under twisting is provided
by the following lemma; the twisted bimodule structures involved are defined prior to Proposition \ref{prop: fflat}.
\begin{prop}[{Proposition \ref{prop: fflat}}]\label{lem: bimoddecomp}
Under the base assumptions of \S\ref{subsec: theoryoftwist}, the twisted group algebra $AG_{\mu}$ has the following
decomposition as an $(A^{G,\mu},A^{G,\mu})$-bimodule:
\begin{equation*}%\label{eq: decompbimod}
AG_{\mu} \cong \bigoplus_{g \in G} {^{\text{id}}(A^{G,\mu})^{\phi_{g}}}.
\end{equation*}
Here $\phi_{g}$ is some $k$-algebra automorphism of $A^{G,\mu}$, with $\phi_e=\text{id}$. Each summand is free as a left
and right $A^{G,\mu}$-module. Consequently, $AG_{\mu}$ is a faithfully flat extension of $A^{G,\mu}$ on both the left
and the right. Furthermore, $AG_{\mu}$ is a faithfully flat extension of $A$ on both the left and the right too.
\end{prop}

Figure \ref{fig: twists} illustrates how we can use Proposition \ref{lem: bimoddecomp} to adopt a new strategy when
trying to show that properties are preserved under cocycle twists. Instead of the direct route from $A$ to $A^{G,\mu}$,
we can attempt to push and pull properties along the faithfully flat extensions via the twisted group algebra
$AG_{\mu}$.
\begin{figure}[ht]
\centering
\begin{tikzpicture}
\draw [->] (1.7,1.7) -- (4,2.9);
\draw [->] (7.2,1.8) -- (5,2.9);
\draw [->,dashed] (1.7,1.3) .. controls (3.45,0.8) and (5.45,0.8) .. (7.2,1.3); 

\node at (4.5,3) {$AG_{\mu}$};
\node at (1.5,1.5) {$A$};
\node at (7.5,1.5) {$A^{G,\mu}\cong(AG_{\mu})^G$};

\node[left] at (2.5,2.5) {\tiny{Faithfully flat}};
\node[right] at (6.5,2.5) {\tiny{Faithfully flat}};
\node at (4.5,0.7) {\tiny{Cocycle twist}};

\end{tikzpicture}
\caption{New technique for proving properties are preserved under twisting}
\label{fig: twists}
\end{figure}
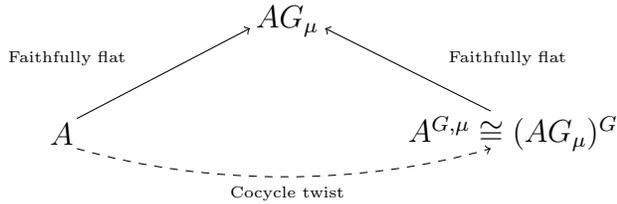

Using the new technique highlighted by Figure \ref{fig: twists} allows us to prove the following main theorem concerning
the preservation of properties under cocycle twists. It brings together Lemma \ref{lem: hilbseries}, Corollaries
\ref{cor: uninoeth} and \ref{cor: asreg}, and Propositions \ref{prop: koszul} and \ref{prop: cohenmac}. Connected graded
algebras and Hilbert series are defined in Definition \ref{def: conngraded} and Definition \ref{def: hilbseries}
respectively. Definitions of the remaining properties involved are deferred until their appearance in \S\ref{sec:
preservation}.
\begin{thm}\label{thm: thebigone}
Further to the base assumptions of \S\ref{subsec: theoryoftwist}, assume that $A$ is connected graded and the action of
$G$ on $A$ by $k$-algebra automorphisms preserves this grading. Then for any normalised 2-cocycle $\mu$, $A$ has one of
the following properties if and only if $A^{G,\mu}$ does:
\begin{itemize}
\item[(i)] it is strongly noetherian;
\item[(ii)] it is AS-regular of global dimension $n$ for some $n \in \N$;
\item[(iii)] it is Koszul;
\item[(iv)] it is Auslander regular;
\item[(v)] it is Cohen-Macaulay.
\end{itemize}
Furthermore, $A$ and $A^{G,\mu}$ have the same Hilbert series.
\end{thm}

\subsection{Examples of twists}\label{subsec: examples}
Having described some of our results regarding properties preserved under cocycle twists, let us now turn to examples.
We will be primarily concerned with twisting AS-regular algebras of global dimension 4, whose classification is an
ongoing and central project in noncommutative algebraic geometry. This term will refer throughout to the school that was
created in the wake of the seminal papers of Artin, Tate and Van den Bergh in the early 1990's
\cite{artin1990some,artin1991modules}.

The algebras that will form the main focus of this thesis are the \emph{4-dimensional Sklyanin
algebras}\index{term}{Sklyanin algebra, 4-dimensional}, whose twists are studied in Chapter \ref{chap: sklyanin}. Such
algebras were studied in \cite{smith1992regularity,smith1993irreducible,levasseur1993modules}, not to mention
\cite{tate1996homological}, where they are viewed as a special case of a more general construction. 

In the following definition we use the notation $[x,y] := xy-yx$ and $[x,y]_{+} :=
xy+yx$\index{notation}{$[-,-],[-,-]_{+}$} to denote certain types of commutator in an algebra.
\begin{defn}[{\cite{smith1992regularity}}]\label{def: sklyanin relations}
Let $k$ be an algebraically closed field  with $\text{char}(k)\neq 2$ and $\alpha, \beta, \gamma \in k$ be scalars
satisfying
\begin{equation}\label{eq: 4sklyanincoeffcondintro}
\alpha + \beta + \gamma + \alpha \beta \gamma = 0 \;\text{ and }\; \{\alpha,\beta, \gamma\} \cap \{0,\pm 1\}=\emptyset.
\end{equation}
We define the \emph{4-dimensional Sklyanin algebra} $A(\alpha,\beta,\gamma)$\index{notation}{a@$A(\alpha,\beta,\gamma)$}
associated to such parameters to be the quotient of the free $k$-algebra $k\{x_0,x_1,x_2,x_3\}$ by the ideal whose
generators are the six quadratic relations in \cite[Equation 0.2.2]{smith1992regularity}:
\begin{equation}\label{eq: 4sklyaninrelnsintro}
\begin{array}{llll}
[x_0,x_1]-\alpha[x_2,x_3]_+, &&& [x_0,x_1]_+ -[x_2,x_3], \\ \relax
[x_0,x_2]-\beta[x_3,x_1]_+, &&& [x_0,x_2]_+ -[x_3,x_1], \\ \relax
[x_0,x_3]-\gamma[x_1,x_2]_+, &&& [x_0,x_3]_+ -[x_1,x_2]. 
\end{array}
\end{equation}
\end{defn}

The latter condition of \eqref{eq: 4sklyanincoeffcondintro} allows us to avoid certain degenerate cases (see \cite[\S
1]{smith1992regularity}). We will omit the parameters of $A(\alpha,\beta,\gamma)$ for the remainder of \S\ref{subsec:
examples}.

The construction of $A$ can also be phrased in terms of a smooth elliptic curve $E$ and the translation automorphism
$\sigma$ associated to a point on the curve (as discussed in detail in \cite[\S 2.9-2.14]{smith1992regularity}). By
generalising this construction one can define $n$-dimensional Sklyanin algebras for all integers $n \geq 3$. Indeed, it
is this family which was studied in \cite{tate1996homological}. The 4-dimensional Sklyanin algebra has many good
properties, including being AS-regular of global dimension 4 and having Hilbert series $1/(1-t)^4$ by \cite[Theorem
0.3]{smith1992regularity}. 

There are many actions of the Klein-four group --- which we denote by $G$ for the remainder of \S\ref{subsec: examples}
--- on $A$ by \N-graded algebra automorphisms. In Chapter \ref{chap: sklyanin} we study a cocycle twist $A^{G,\mu}$ for
a particular action of $G$ and a normalised 2-cocycle $\mu$. Some of the more surprising properties of $A^{G,\mu}$ are
grouped together in parts (ii) and (iii) of the following result; it summarises Proposition \ref{claim: fatpoints} and
Theorems \ref{thm: 4sklytwistprops} and \ref{prop: finitepointscheme}. 
\begin{thm}\label{thm: sklyanin}
Let $E$ and $\sigma$ be the elliptic curve and automorphism associated to a 4-dimensional Sklyanin algebra $A$. There is
an action of the Klein-four group $G$ on $A$ by \N-graded algebra automorphisms which, together with a normalised
2-cocycle $\mu$, produces a cocycle twist $A^{G,\mu}$ with the following properties:
\begin{itemize}
 \item[(i)] it is strongly noetherian AS-regular domain of global dimension 4, which has Hilbert series $1/(1-t)^4$;
 \item[(ii)] when $|\sigma| = \infty$ its point scheme consists of only 20 points (see Definition \ref{def:
pointscheme});
 \item[(iii)] it has a family of fat point modules of multiplicity 2 parameterised by $E^G$, an elliptic curve whose
underlying structure comes from the orbit space of $E$ under a natural action of $G$.
\end{itemize}
\end{thm}

The behaviour exhibited in Theorem \ref{thm: sklyanin} (ii) is particularly surprising given that the point scheme of
$A$ consists of the associated elliptic curve $E$ and four extra points \cite[Propositions 2.4 and
2.5]{smith1992regularity}. 

The presence of the modules occurring in Theorem \ref{thm: sklyanin} (iii) can be related to a factor ring of $A$. There
exist two central elements $\Omega_1,\Omega_2 \in A_2$\index{notation}{o@$\Omega_1,\Omega_2$} by \cite[Corollary
3.9]{smith1992regularity}. The factor ring $A/(\Omega_1,\Omega_2)$ is isomorphic to the twisted homogeneous coordinate
ring $B:=B(E,\calL,\sigma)$, with such rings being defined in Definition \ref{def: thcr}. One can use the Noncommutative
Serre's theorem (Theorem \ref{thm: noncomserrethm}) to see that the point modules over $A$ that are parameterised by $E$
correspond geometrically to skyscraper sheaves over the elliptic curve.

Since $\Omega_1$ and $\Omega_2$ are fixed by the action of $G$ on $A$ used in Theorem \ref{thm: sklyanin}, the action of
$G$ is compatible with factoring out the ideal $(\Omega_1,\Omega_2)$. We prove the following result, which summarises
Theorem \ref{thm: geomdescthcrtwist} and Proposition \ref{prop: fatpointsincohE}. Any undefined terms are defined in
\S\ref{subsec: geomdata} or \S\ref{subsec: geomdesc}.
\begin{thm}\label{thm: geometricthcrdesc}
Let $A^{G,\mu}$ be as in Theorem \ref{thm: sklyanin}. There is a factor ring of $A^{G,\mu}$, labelled $B^{G,\mu}$, which
is a cocycle twist of a factor ring of $A$. The ring $B^{G,\mu}$ satisfies the following properties:
\begin{itemize}
 \item[(i)] it can be described as a twisted ring over $\mathcal{E}$, where $\mathcal{E}$ is a sheaf of
$\calO_{E^G}$-orders;
 \item[(ii)] it has a family of fat point modules of multiplicity 2 parameterised by $E^G$, arising in a natural way
from certain coherent sheaves of $\mathcal{E}$-modules.
\end{itemize}
\end{thm}

The importance of this result will be stressed in \S\ref{sec: noncommgeom} in relation to canonical maps to twisted
rings. Briefly, it illustrates that twisted rings can capture geometric properties of noncommutative algebras, in this
case $A^{G,\mu}$.

Other interesting examples can be found by studying the algebras classified by Rogalski and Zhang in
\cite{rogalski2012regular}. Such algebras are AS-regular of dimension 4 with three degree 1 generators and an additional
$\Z^{\times 2}$ grading. We disregard algebras with these properties that are normal extensions of AS-regular algebras
of dimension 3. The classification up to isomorphism of the remaining algebras consists of 8 families
$\mathcal{A}-\mathcal{H}$. Theorem \ref{thm: rogzhang} shows that one can relate several of these families using cocycle
twists.
\begin{thm}[{cf. Theorem \ref{thm: rogzhangmymain}}]\label{thm: rogzhang}
Consider Rogalski and Zhang's classification in \cite{rogalski2012regular}. For the Klein-four group $G$ and a certain
normalised 2-cocycle $\mu$, one has the following isomorphisms of $k$-algebras:
\begin{equation*}%\label{eq: rogzhangmymainisos}
\mathcal{A}(1,-1)^{G,\mu}\cong \mathcal{D}(1,1),\;\; \mathcal{B}(1)^{G,\mu} \cong \mathcal{C}(1), \;\;
\mathcal{E}(1,\gamma)^{G,\mu}\cong \mathcal{E}(1,-\gamma), \;\; \mathcal{G}(1,\gamma)^{G,\mu} \cong
\mathcal{G}(1,\overline{\gamma}).
\end{equation*}
\end{thm} 

\subsection{Contents}\label{subsec: contents}
We now give a brief description of the contents of this thesis. The remainder of the current chapter consists of
\S\ref{sec: notation}, where our basic conventions and notation are given.

In Chapter \ref{chap: background} we describe some of the background material that is needed. We begin with \S\ref{sec:
noncommgeom}, which consists of an introduction to noncommutative algebraic geometry emphasising the aspects that are
most pertinent to cocycle twists. This is followed by \S\ref{sec: twistsofalgebras}, in which we study several related
concepts involving the twisting of algebras; cocycle twists, Zhang twists and crossed products are all defined. The
chapter ends with the short section \S\ref{sec: goldietheory}, which describes the Goldie theory that we will use.

The next chapter is the first containing original work. In \S\ref{sec: construction} two twisting constructions are
defined and then subsequently shown to be the same in Proposition \ref{prop: twoconstrequal}. The preservation of
properties under such twists is then treated in \S\ref{sec: preservation}. A result of particular note is Corollary
\ref{cor: asreg}, where the property of being AS-regular is shown to be preserved under cocycle twists. In \S\ref{sec:
modules} we study the relationship between point modules over an algebra and fat point modules over a cocycle twist of
it. The machinery we construct is used in several of the examples later in the thesis.

Chapter \ref{chap: sklyanin} deals with the main examples of cocycle twists studied in the thesis, namely twists of
4-dimensional Sklyanin algebras, which we denote by $A^{G,\mu}$. In \S\ref{subsec: 4sklyanintwistandptscheme} such
examples are shown to satisfy many good properties (in particular Theorem \ref{thm: 4sklytwistprops}), while in Theorem
\ref{prop: finitepointscheme} their point schemes are determined. We also show that the full version of Odesskii's
example from \cite[Introduction]{odesskii2002elliptic} (which was simplified in Example \ref{ex: odesskii}), can be
described as a cocycle twist of a Sklyanin algebra (Proposition \ref{prop: odesskiiegdone}). The existence of fat point
modules over $A^{G,\mu}$ is proved in \S\ref{subsec: fatpoints}; such modules arise as direct sums of point modules over
Sklyanin algebras by Proposition \ref{claim: fatpoints}. This is just one indication of the interplay between 1-critical
modules over $A$ and $A^{G,\mu}$, and we prove several more results in this direction.

In Chapter \ref{chap: thcrtwist} we study a twist of the factor ring $B=A/(\Omega_1,\Omega_2)$. Under the induced
$G$-grading on $A$, $B$ is a $G$-graded factor ring. Thus one can regard the twist $B^{G,\mu}$ as a factor of
$A^{G,\mu}$, which was studied in the Chapter \ref{chap: sklyanin}. The algebra $B^{G,\mu}$ is shown --- under some
conditions --- to have trivial centre in Proposition \ref{prop: centrethcrtwist}. From this proposition we derive
Corollary \ref{cor: centretwistskly}, which  describes the centre of $A^{G,\mu}$. In \S\ref{subsubsec:
geomdescrthcrtwist} we describe $B^{G,\mu}$ in terms of the geometric classification in \cite{artin2000semiprime},
concerning semiprime, noetherian rings of quadratic growth. This description is given in Theorem \ref{thm:
geomdescthcrtwist}, after which we show that the twist is a prime ring (Corollary \ref{cor: actuallyprime}). The fat
point modules over $A^{G,\mu}$ and the fact that $B^{G,\mu}$ possesses no point modules can both be related to this
geometric realisation (see Propositions \ref{prop: onlyfatpoints} and \ref{prop: fatpointsincohE}).  

Chapter \ref{chap: othertwists} considers cocycle twists of other algebras, beginning in \S\ref{subsec: staffordalgs}
with cocycle twists of algebras discovered by Stafford in \cite{stafford1994regularity}. The next examples are the
algebras classified in \cite{rogalski2012regular} by Rogalski and Zhang. These are shown in Theorem \ref{thm:
rogzhangmymain} to have some hidden relations via cocycle twists. Of particular note amongst the remaining examples in
the chapter is a twist of an algebra from \cite{vancliff1994quadratic} in \S\ref{sec: vancliffql}. We contrast the
geometry relating to a factor ring of such a twist and compare it with the twist studied in Chapter \ref{chap:
thcrtwist}. In \S\ref{subsec: gradedskewclifford} we twist a graded skew Clifford algebra, while the main portion of the
thesis ends by briefly addressing the ungraded situation -- we study twists of a universal enveloping algebra and its
homogenisation in \S\ref{subsec: homenvelopalg}.

Appendix \ref{app: calc} contains any calculations that have been omitted from the main thesis, while Appendix \ref{app:
comp} contains any computer code used to prove earlier results.

\section{Basic conventions and notation}\label{sec: notation}
\sectionmark{Notation}
In this short section we state our basic conventions with regard to cocycle twists and also some more general notation.

\subsection{Basic conventions}\label{subsec: conventions}
The conventions that follow will hold throughout this thesis apart from certain occurrences where deviation from them
will be stated explicitly. 

We assume that $0 \in \N$ for grading purposes and work over a base field $k$ which is assumed to be algebraically
closed. Any algebra $A$ will be an associative $k$-algebra with identity that is finitely generated (f.g.)\ as an
algebra. By $G$ we will denote a finite \emph{abelian} group with identity element $e$. Furthermore, we will assume that
the characteristic of the base field does not divide the order of $G$, that is $\text{char}(k) \nmid |G|$. A normalised
2-cocycle will be denoted by $\mu$. Unadorned tensor products $\otimes$ will mean tensor products over $k$,
$\otimes_k$. 

\subsection{Notation}\label{subsec: notation}

In this section we will describe our basic notation that will be in place throughout the rest of the thesis, unless
otherwise stated. More specialised notation may subsequently be defined as and when it is needed.

\subsubsection{Group actions}
The group of \emph{algebra automorphisms} of $A$ is denoted by
$\text{Aut}_{\text{alg}}(A)$\index{notation}{a@$\text{Aut}_{\text{alg}}(A)$}, while the group of \emph{group
automorphisms} of $G$ is denoted by $\text{Aut}_{\text{grp}}(G)$\index{notation}{a@$\text{Aut}_{\text{grp}}(G)$}. One of
our base assumptions in \S\ref{subsec: theoryoftwist} was that $G$ acts on $A$ by $k$-algebra automorphisms, i.e. $G
\rightarrow \text{Aut}_{\text{alg}}(A)$. As $G$ is abelian we can assume that the group acts on the right without loss
of generality. We will denote the action of an element $g \in G$ by $-^g$. The embedding of $G$ in the automorphism
group means that for all $a,b \in A$, $g \in G$ and $\alpha \in k$ one has $(\alpha a)^g=\alpha a^g$ and
$(ab)^g=a^gb^g$. 

We will also define actions of $G$ on objects other than algebras later in the thesis. The same superscript notation
will be used in such circumstances, although the precise nature of the action will be given at the time.

Since $G$ is a finite abelian group, we can write it as a product of cyclic groups. We will denote the cyclic group of
order $n$ by $C_n$. The Klein-four group will be pivotal in all of our examples from Chapter \ref{chap: sklyanin}
onwards, and using this notation it can be denoted by $(C_2)^2$.

\subsubsection{Gradings}
One of the fundamental definitions required for the work in this thesis is the following. 
\begin{defn}\label{defn: ggradedalgebra}
The algebra $A$ is said to be \emph{$G$-graded}\index{term}{g@$G$-graded algebra} if it possesses a $k$-vector space
decomposition $A=\bigoplus_{g \in G}A_g$ for which $1 \in A_e$ and $A_g \cdot A_h \subset A_{gh}$ for all $g, h \in G$. 
\end{defn}

By replacing $G$ with $\N$ in this definition one obtains the definition of an \emph{\N-graded
algebra}\index{term}{n@$\N$-graded algebra}. 

We will make the additional assumption that $A$ is connected graded for certain results in \S\ref{sec: preservation}.
Moreover, all of our examples of cocycle twists (apart from \S\ref{subsec: homenvelopalg}) involve such algebras.

\begin{defn}\label{def: conngraded}
A $k$-algebra $A$ is said to be \emph{connected graded}\index{term}{connected graded, (c.g.)} or \emph{c.g.}\ if it has
an $\N$-grading $A = \bigoplus_{n \in \N}A_n$ for which $A_0=k$ and $\text{dim}_k A_n < \infty$ for all $n \in \N$.
\end{defn}

If $A$ is connected graded then one can define the notion of a Hilbert series.

\begin{defn}\label{def: hilbseries}
If $A$ is a connected graded algebra then its \emph{Hilbert series}\index{term}{Hilbert series} with respect to its
\N-grading is $H_A(t):=\sum_{n \in \N}(\text{dim}_kA_n)t^n$.
\end{defn}

From now on we will refer to connected graded algebras as c.g.\ algebras for brevity. When $A$ is \N-graded or a c.g.\
algebra we will assume that $G$ acts on $A$ by $\N$-graded algebra automorphisms. This means that in addition to the
requirement of acting by algebra automorphisms, if $a \in A_n$ then $a^g \in A_n$. The group of $\N$-graded algebra
automorphisms is denoted by $\text{Aut}_{\N-\text{alg}}(A)$.

If $A$ is $\N$-graded then we assume that it is generated by finitely many elements that are homogeneous with respect to
its $\N$-grading. We do not, however, necessarily assume that the generators lie in degree 1. Our algebras will often be
written as quotients of a free algebra $k\{x_0,\ldots,x_n\}$, which can also be described as the tensor algebra $T(V)$
over the vector space $V=kx_0 + \ldots + kx_n$. If we work with an algebra of the form $T(V)/I$ for some ideal $I$, then
we will refer to the generators of $I$ as the \emph{defining relations}\index{term}{defining relations} of the algebra.

At times we will focus our attention on quadratic algebras\index{term}{quadratic algebra}, particularly when studying
the Koszul property in \S\ref{subsec: koszul}. 

\begin{defn}\label{def: quadalg}
Let $V$ be a finite-dimensional vector space over $k$ and $T(V)$ be the tensor algebra over $V$. For an ideal $I$ whose
generators lie in $V \otimes V$ we say that the quotient $T(V)/I$ is a \emph{quadratic algebra} over $k$.  
\end{defn}

Note that our definition implies that quadratic algebras are c.g.\ $k$-algebras that are generated in degree 1 and whose
defining relations lie in degree 2. 

\subsubsection{Modules}
We will work with right modules unless otherwise stated. A module $M$ over the algebra $A$ may sometimes be written
$M_A$ (or $_AM$ if it is a left module). If a module is \emph{finitely generated} then we will abbreviate this to
\emph{f.g.}\index{term}{finitely generated, (f.g.)}. Although this abbreviation is also used for finitely generated
algebras, it will be clear from the context to which we refer.

Let us now assume that $A$ is $\N$-graded. 

\begin{defn}\label{def: gradedmodule}
A module $M$ over $A$ is an \emph{$\N$-graded right $A$-module}\index{term}{n@$\N$-graded module} if it has $k$-vector
space decomposition $M=\bigoplus_{n \in \N}M_n$ such that $M_n \cdot A_m \subset M_{n+m}$ for all $n, m \in \N$. 
\end{defn}

Using this definition allows us to define the category of modules in which we will sometimes work.

\begin{defn}\label{def: gradedmodulecategory}
By $\text{GrMod}(A)$\index{notation}{g@$\text{GrMod}(A)$} we denote the \emph{category of \N-graded right $A$-modules}.
The objects of this category are the \N-graded right $A$ modules as defined in Definition \ref{def: gradedmodule}. A
morphism $\phi: M \rightarrow N$ in this category is a homomorphism of right $A$-modules such that $\phi(M_n) \subset
N_n$ for all $n \in \N$. 
\end{defn}

Let $M \in \text{GrMod}(A)$. For any $d \in \N$ we define $M[d]:=\bigoplus_{n \in N} M_n'$\index{notation}{m@$M[d]$}
with $M_n'=M_{n+d}$ to be a \emph{degree shift of $M$ by degree $d$}\index{term}{degree shift of a graded module}. We
also define $M_{\geq d}:=\bigoplus_{n \geq d} M_n$\index{notation}{m@$M_{\geq d}$} to be a \emph{tail}\index{term}{tail
of a graded module} of $M$.

For $M, N \in \text{GrMod}(A)$ one can also consider the $\Z$-graded group
\begin{equation*}%\label{eq: HOM}
\text{HOM}_A(M,N) = \bigoplus_{j \in \Z} \text{Hom}_{\text{GrMod}(A)}(M,N[j]).
\end{equation*}
If $M$ and $N$ are f.g.\ modules then $\text{HOM}_A(M,N)$ coincides with $\text{Hom}_{A}(M,N)$, the abelian group of
ungraded left $A$-module homomorphisms. To prove this, consider a map $\phi \in \text{Hom}_A(M,N)$ and let
$m_1,\ldots,m_p$ denote the homogeneous generators of $M$ of degrees $d_1,\ldots, d_p$ respectively. Then $\phi(m_i)$
can be written as a sum of homogeneous elements in $N$, that is $\phi(m_i) \in \bigoplus_{j} N_{r_{i,j}}$ for some
integers $r_{i,j} \in \Z$. Hence $\phi \in \bigoplus_{i,j}  \text{Hom}_{\text{GrMod}(A)}(M,N[r_{i,j}-d_j])$.
Additionally, the abelian group $\text{Hom}_{A}(M,N)$ inherits a $\Z$-graded structure from the modules. 

For $i \in \N$ the derived functors
\begin{equation}\label{eq: EXT}
\text{EXT}^i_A(M,N) = \bigoplus_{j \in \Z} \text{Ext}^i_{\text{GrMod}(A)}(M,N[j]).
\end{equation}
have a $\Z$-graded structure. Again, when $M$ and $N$ are f.g.\ modules $\text{EXT}^i_A(M,N)$ coincides with the
ungraded $\text{Ext}^i_{A}(M,N)$. On several occasions we will use this graded structure on cohomology groups. In fact,
when either $M$ or $N$ has an $(A,B)$-bimodule structure for some $\N$-graded $k$-algebra $B$, the cohomology groups
$\text{Ext}^i_{A}(M,N)$ become $\Z$-graded $B$-modules (see \cite[Theorem 1.15]{rotman2008introduction} for a
description of the relevant module structures).

For a noetherian c.g.\ algebra $A$ and any f.g.\ module $M \in \text{GrMod}(A)$, one can define the Hilbert series of
$M$ in the same manner as in Definition \ref{def: hilbseries}. Such an object is denoted by $H_M(t)$.

One can define $G$-graded modules in an analogous manner to Definition \ref{def: gradedmodule} for \N-graded modules.
This allows the following category to be defined.

\begin{defn}\label{def: Ggradedmodcat}
Suppose that $A$ is $G$-graded. By $\text{GrMod}_{G}(A)$\index{notation}{g@$\text{GrMod}_{G}(A)$} we will denote the
\emph{category of $G$-graded right $A$-modules}. The objects in this category are the $G$-graded right $A$-modules. A
morphism $\phi: M \rightarrow N$ in this category is a homomorphism of right $A$-modules such that $\phi(M_g) \subset
N_g$ for all $g \in G$. 
\end{defn}

\subsubsection{Dimension functions}
Gelfand-Kirillov dimension, shortened to \emph{GK dimension}, is a dimension function for algebras and their modules
that is defined as follows.
\begin{defn}[{\cite[Definition pg. 14]{krause2000growth}}]\label{def: gkdim}
Let $k$ be a field and $A$ an associative $k$-algebra with identity. Let $V \subset A$ be a $k$-vector space for which
$1 \in V$ and define $V^n$ be the vector space spanned by all products of $n$ elements from $V$. The
\emph{Gelfand-Kirillov dimension}\index{term}{GK dimension} of $A$ is defined by
\begin{equation*}%\label{eq: }
\text{GKdim }A\index{notation}{GKdim $A$}= \underset{V}{\text{sup}} \underset{n \rightarrow \infty}{\varlimsup}
 \left(\frac{\text{log dim}_kV^n}{\text{log }n}\right).
\end{equation*}
Here the supremum is taken over all such vector subspaces $V$.
\end{defn}

For the basic properties of this dimension function, consult \cite{krause2000growth}. We will also use several
homological dimension functions, defined as follows.
\begin{defn}\label{def: homdimfunc}
Let $A$ be a $k$-algebra and $M$ be a right $A$-module. Then
\begin{itemize}
\item[(i)] $\text{pdim }M$\index{notation}{pdim $M$} denotes the \emph{projective dimension}\index{term}{projective
dimension} of $M$. This is the minimal length of a projective resolution of $M$.
\item[(ii)] $\text{lgldim }A$\index{notation}{lgldim $A$} and $\text{rgldim }A$\index{notation}{rgldim $A$} denote the
\emph{left} and \emph{right global dimension} of $A$ respectively. One has
\begin{equation*}%\label{eq: }
\text{rgldim }A = \text{sup}\{\text{pdim }M:\; M\text{ is a right }A\text{-module}\},
\end{equation*}
 with $\text{lgldim }A$ being defined analogously for left $A$-modules.
\item[(iii)] $\text{idim }M$\index{notation}{idim $M$} denotes the \emph{injective dimension}\index{term}{injective
dimension} of $M$. This is the minimal length of an injective resolution of $M$.
\end{itemize}
\end{defn}
A good reference for such dimension functions is \cite[\S 7.1]{mcconnell2001noncommutative}.

%%%%%%%%%%%%%%%%%%%%%%%%%%%%%%%%%%%%%%%%%%%%%%%%%%%

\chapter{Background material}\label{chap: background}
\chaptermark{Background}

%%%%%%%%%%%%%%%%%%%%%%%%%%%%%%%%%%%%%%%%%%%%%%%%%%%

\section{Noncommutative algebraic geometry}\label{sec: noncommgeom}
\sectionmark{Noncommutative algebraic geometry}
In this section we give an introduction to the field of noncommutative algebraic geometry in the sense of
\cite{artin1990some} and \cite{artin1991modules}. Our perspective is informed by the cocycle twists that we study in
this thesis. 

Broadly speaking, the aim of noncommutative algebraic geometry is to study noncommutative algebras using the techniques
and intuition used to study commutative algebras in algebraic geometry. Unfortunately, there are some fundamental
obstacles preventing the translation of the theory verbatim; prime ideals are crucial to the commutative theory, but are
relatively scarce in noncommutative algebras. 

Throughout \S\ref{sec: noncommgeom} we will assume that $k$ is an algebraically closed field of characteristic 0 unless
otherwise stated. We will work with noetherian c.g.\ algebras (see Definition \ref{def: conngraded}) that are assumed to
be f.g.\ as algebras by degree 1 elements, unless otherwise stated. While the latter assumption is not always necessary,
often results which do not require it cannot be stated as succinctly. 

\subsection{Noncommutative spaces}\label{subsec: noncommspaces}
Instead of working with prime ideals, one can approach the problem from a categorical perspective. Let $A$ denote a
noetherian c.g.\ $k$-algebra and recall the definition of $\text{GrMod}(A)$ given in Definition \ref{def:
gradedmodulecategory}. Consider the following subcategories of $\text{GrMod}(A)$:
\begin{itemize}
\item[(i)] $\text{grmod}(A)$\index{notation}{g@$\text{grmod}(A)$}, the full subcategory of all f.g.\ $\N$-graded
$A$-modules; 
\item[(ii)] $\text{fdmod}(A)$\index{notation}{g@$\text{fdmod}(A)$}, the full subcategory of finite-dimensional
$\N$-graded $A$-modules.
\end{itemize}

One can use these subcategories to define the following category, often referred to as the \emph{category of tails}. 
\begin{defn}[{\cite[cf. \S 1]{artin1990twisted}}]\label{defn: cattails}\index{term}{tails, category of}
The \emph{noncommutative space}\index{term}{noncommutative space} associated to $A$ is the quotient category
\begin{equation}\label{eq: qgrade}
\text{qgr}(A)\index{notation}{q@qgr$(A)$}:=\text{grmod}(A)/\text{fdmod}(A). 
\end{equation}
While $\text{qgr}(A)$ has the same objects as $\text{grmod}(A)$, its morphisms are different. Each morphism in
$\text{Hom}_{\text{qgr}(A)}(M,N)$ arises from a morphism in $\text{Hom}_{\text{gr}(A)}(M',N')$, where $M'$ is any graded
submodule of $M$ for which $M/M' \in \text{fdmod}(A)$ and $N'\cong N/N''$ in $\text{grmod}(A)$ for some $N'' \in
\text{fdmod}(A)$.
\end{defn}
\begin{rem}\label{rem: qgrdefn}
In \cite[\S 1]{artin1990twisted} the authors define a category $\text{QGr}(A)$ by taking the quotient of
$\text{GrMod}(A)$ by the subcategory consisting of direct limits of right bounded modules. If $A$ is right noetherian
then \cite[Proposition 2.3]{artin1994noncommutative} shows that $\text{QGr}(A)$ is determined up to equivalence by
$\text{qgr}(A)$. In this thesis we will work will noetherian algebras, hence Definition \ref{defn: cattails} is
sufficient for our needs.
\end{rem}

Definition \ref{defn: cattails} implies that two modules $M, N \in \text{grmod}(A)$ become isomorphic in $\text{qgr}(A)$
if they are isomorphic (as graded modules) in sufficiently high degree, that is $M_{\geq n}\cong N_{\geq n}$ for some $n
\in \N$. Using the terminology introduced in \S\ref{subsec: notation} we can say that their tails are isomorphic in high
degree. Finite-dimensional $\N$-graded modules become trivial in $\text{qgr}(A)$, thus are sometimes referred to as
\emph{torsion}.

There is a canonical functor related to the noncommutative space $\text{qgr}(A)$, namely $\pi: \text{grmod}(A)
\rightarrow \text{qgr}(A)$, which sends a module to the corresponding object in $\text{qgr}(A)$. For $M \in
\text{grmod}(A)$ we will refer to $\pi(M)$ as the \emph{tail} of $M$. This fits in with our previous definition of a
tail of a module since for $M \in \text{grmod}(A)$ one has $\pi(M) \cong \pi(M_{\geq n})$ in $\text{qgr}(A)$ for all $n
\in \N$. It will be clear from the context whether the term tail refers to a module of the form $M_{\geq n}$ or to
$\pi(M)$.

Let us now begin to motivate our use of the term \emph{noncommutative space} in reference to the category defined in
Definition \ref{defn: cattails}. We need to define a class of rings which is fundamental to noncommutative algebraic
geometry.
\begin{defn}[{\cite[\S 1]{artin1990twisted}}]\label{def: thcr}
Let $(X,\sigma,\mathcal{L})$ be a triple consisting of a projective scheme $X$, an automorphism $\sigma$ and an
invertible sheaf of $\mathcal{O}_X$-modules $\mathcal{L}$. The shorthand notation
$\mathcal{L}^{\sigma}=\sigma^{\ast}\mathcal{L}$ will be used for pullbacks. The \index{term}{twisted homogeneous
coordinate ring}\emph{twisted homogeneous coordinate ring} associated to this triple is 
\begin{equation*}%\label{eq: thcrdecomp}
B(X,\mathcal{L},\sigma)\index{notation}{b@$B(X,\mathcal{L},\sigma)$}=\bigoplus_{n \in \mathbb{N}}
\Gamma(X,\mathcal{L}_n), 
\end{equation*}
where $\mathcal{L}_0=\calO_X$ and the sheaf $\mathcal{L}_n$ for $n \geq 1$ is defined by
\begin{equation*}%\label{eq: sheafthcr}
\mathcal{L}_n=\mathcal{L} \otimes_{{\mathcal{O}_X}} \mathcal{L}^{\sigma} \otimes_{{\mathcal{O}_X}} \ldots
\otimes_{{\mathcal{O}_X}} \mathcal{L}^{{\sigma}^{n-1}}. 
\end{equation*}
There are natural morphisms of invertible sheaves
\begin{equation*}%\label{eq: thcrmult}
\mathcal{L}_n \otimes_{{\mathcal{O}_X}} \mathcal{L}_m \rightarrow \mathcal{L}_n \otimes_{{\mathcal{O}_X}}
\mathcal{L}_m^{\sigma^{n}} \rightarrow \mathcal{L}_{m+n}, 
\end{equation*}
which induce the multiplication in $B(X,\mathcal{L},\sigma)$ via taking global sections:
\begin{equation}\label{eq: globalsecs}
H^0(X,\mathcal{L}_n) \otimes H^0(X,\mathcal{L}_m) \rightarrow H^0(X,\mathcal{L}_n) \otimes
H^0(X,\mathcal{L}_m^{\sigma^{n}}) \rightarrow H^0(X,\mathcal{L}_{m+n}).
\end{equation}
\end{defn}

If the invertible sheaf associated to a twisted homogeneous coordinate ring is $\sigma$-ample --- whose definition
follows --- then the ring is noetherian.
\begin{defn}[{\cite[\S 1]{keeler2000criteria}}]\label{defn: sigmaample}
Let $(X,\mathcal{L},\sigma)$ be as in Definition \ref{def: thcr}. We say that $\mathcal{L}$ is
\emph{$\sigma$-ample}\index{term}{s@$\sigma$-ample} if for any coherent sheaf $\mathcal{F}$ of $\mathcal{O}_X$-modules
\begin{equation*}%\label{eq: rightsigample}
H^p(X,\mathcal{F}\otimes \mathcal{L} \otimes \mathcal{L}^{\sigma}\otimes \ldots \otimes \mathcal{L}^{\sigma^{n-1}})=0, 
\end{equation*}
for all $p>0$ and $n \gg 0$.
\end{defn}

This was originally the definition of a \emph{left $\sigma$-ample} invertible sheaf. Keeler proved in \cite[Theorem
1.2]{keeler2000criteria} that this definition is equivalent to a notion of \emph{right $\sigma$-ampleness} for
projective schemes over algebraically closed fields. We will mainly work with twisted homogeneous coordinate rings for
which the associated projective scheme is an elliptic curve defined over such a field. In that case --- as noted in the
proof of \cite[Theorem 4.7]{artin1995noncommutative} --- Corollary 1.6 from \cite{artin1990twisted} implies that ample
sheaves are always $\sigma$-ample and so we do not need to worry about this property.
\begin{rem}\label{rem: finesse}
We now remark upon another consequence of $\sigma$-ampleness, as demonstrated by \cite[Proposition
2.1]{smith1994center}: if $(X,\mathcal{L},\sigma)$ is a triple as in Definition \ref{def: thcr} for which $X$ is an
integral scheme and $\mathcal{L}$ is $\sigma$-ample, then there is an embedding of algebras $B:=B(X,\mathcal{L},\sigma)
\hookrightarrow k(X)[z,z^{-1};\sigma]$, where $k(X)$ is the function field of $X$. In this setting one may define
$\overline{B}_n:=\text{H}^0(X,\calL_n) \subset k(X)$\index{notation}{b@$\overline{B}_n$} for all $n \in \N$, in which
case one has $B_n=\overline{B}_n z^n$ upon considering $B$ as a subring of the Ore extension. This notational finesse
allows one to concretely see the twisted multiplication in $B$ arising from the structure of the Ore extension. We will
use such ideas in \S\ref{subsec: geomdata}.  
\end{rem}

Let us now try to justify our usage of the term noncommutative space in Definition \ref{defn: cattails} by first
considering what happens in the commutative situation. We therefore assume until further notice that $A$ is commutative
in addition to our base assumptions of being a noetherian c.g algebra that is generated in degree 1. 

Associated to $A$ is a projective scheme $X= \text{Proj(A)}$ and an embedding $j: X \hookrightarrow \proj{k}{n}$ for
some $n \in \N$. The following sheaf is crucial to the proof of our next result, Serre's theorem (Theorem \ref{thm:
serresthm}).
\begin{defn}[{\cite[cf. Chapter II, Definition pg. 117]{hartshorne1977algebraic}}]\label{def: twistingsheaf}
Let $A$, $X$ and $j$ be as in the previous paragraph. The \emph{twisting sheaf}\index{term}{twisting sheaf}
$\calO_X(1)$\index{notation}{o@$\calO_X(1)$} is defined to be the pullback $j^{\ast}\calO_{\proj{k}{n}}(1)$.
\end{defn}
\begin{thm}[{Serre's Theorem \cite[Chapter II, Exercise 5.9]{hartshorne1977algebraic}}]\label{thm: serresthm}
Let $A$ be a commutative, noetherian c.g.\ algebra generated in degree 1. Then there is an equivalence of categories
between $\text{qgr}(A)$ and the category of coherent $\mathcal{O}_X$-modules,
$\text{coh}(\calO_X)$\index{notation}{coh$(\calO_X)$}, where $X= \text{Proj(A)}$ is the projective scheme determined by
$A$.
\end{thm}

Serre's theorem can be proved by focusing on the distinguished element $\pi(A) \in \text{qgr}(A)$. The module $A_A$ can
be reconstructed in high degrees by taking global sections of the invertible sheaves
$\mathcal{O}_{X}(n)=\calO_X(1)^{\otimes n}$ for $n \gg 0$, along with a natural multiplication obtained from \eqref{eq:
globalsecs} by taking $\sigma=\text{id}$ and $\calL=\calO_{X}(1)$. In light of this observation it is not surprising
that twisted homogeneous coordinate rings form the gateway to a theory of noncommutative algebraic geometry.

Let us now return to the noncommutative setting and our base assumptions. Thus $A$ is a noetherian c.g.\ algebra that is
generated in degree 1. The following theorem, known as the \emph{Noncommutative Serre's
theorem}\index{term}{Noncommutative Serre's theorem}, shows that for twisted homogeneous coordinate rings the
noncommutative space $\text{qgr}(B(X,\mathcal{L},\sigma))$ is actually a commutative geometric object.
\begin{thm}[{Noncommutative Serre's theorem \cite[cf. Theorems 1.3 and 1.4]{artin1990twisted}}]\label{thm:
noncomserrethm}
Let $X$ be a projective scheme over a field $k$, $\sigma$ an automorphism of $X$, and $\mathcal{L}$ a $\sigma$-ample
invertible sheaf on $X$. Then the twisted homogeneous coordinate ring $B(X,\mathcal{L},\sigma)$ is a f.g.\ noetherian
algebra and there is an equivalence of categories 
\begin{equation*}%\label{eq: noncommserre}
\text{qgr}(B(X,\mathcal{L},\sigma)) \simeq \text{coh}(\mathcal{O}_X).
\end{equation*}
\end{thm}

Some of the central objects of study in noncommutative algebraic geometry are the \emph{Artin-Schelter-regular algebras}
(AS-regular). These noncommutative algebras are defined in Definition \ref{defn: asregular} and display many of the
properties of commutative polynomial rings. For this reason, an AS-regular algebra of global dimension $n$ is often
regarded as the coordinate ring of a noncommutative \proj{k}{n-1}\index{term}{noncommutative \proj{k}{n-1}}. The
noncommutative projective planes (where $n = 3$) were classified in \cite{artin1990some,artin1991modules} for algebras
generated in degree 1 and in \cite{stephenson1996artin} for the general case. Their classification relied heavily on
geometric techniques which we will soon discuss.

The results in this thesis, particularly those in Chapter \ref{chap: sklyanin}, relate to the classification of
noncommutative \proj{k}{3}'s, or, equivalently, AS-regular algebras of dimension 4. While no classification of such
algebras is on the horizon, many families of such algebras have been found. Examples include 4-dimensional Sklyanin
algebras and their close relations \cite{smith1992regularity,stafford1994regularity}, homogenisations of universal
enveloping algebras \cite{lebruyn1993homogenized}, algebras associated to quadrics
\cite{vancliff1994quadratic,vancliff1997embedding}, double Ore extensions \cite{zhang2008double}, algebras with an
additional $\Z^2$-grading \cite{lu2007regular,rogalski2012regular} and the algebras in
\cite{vancliff1998some,cassidy2006generalized,cassidy2010generlizations}.

We now consider an example of how a particular AS-regular algebra of low dimension fits in with Theorem \ref{thm:
noncomserrethm}, 
\begin{example}\label{eg: 3dimsklyanin}
Consider a 3-dimensional Sklyanin algebra $S(a,b,c)$, associated to the parameters $(a,b,c) \in \proj{k}{2}$ and
generated by three degree 1 elements. Such algebras are those of type $A$ with $r= 3$ in Artin and Schelter's
classification in \cite{artin1987graded} (see (10.14) op. cit. for the defining relations). 

Apart from 12 choices of parameters in $\proj{k}{2}$, $S(a,b,c)$ is AS-regular of global dimension 3. Thus
$\text{qgr}(S(a,b,c))$ is a noncommutative $\proj{k}{2}$. Furthermore, there is a graded surjection from $S(a,b,c)$ to
the twisted homogeneous coordinate ring $B(E,\mathcal{L},\sigma)$, where $E \subset \proj{k}{2}$ is an elliptic curve,
$\mathcal{L}=\calO_E(1)$ a very ample invertible sheaf with 3 global sections and $\sigma$ a translation automorphism. 

At the level of noncommutative spaces we see that the noncommutative \proj{k}{2} associated to $S(a,b,c)$ contains
$\text{qgr}(B(E,\mathcal{L},\sigma))$. By the Noncommutative Serre's Theorem the latter category is equivalent to
$\text{coh}(\calO_X)$, thus $\text{qgr}(S(a,b,c))$ contains a commutative geometric category.
\end{example}

In general, one can ask if there exists a graded surjection from a given c.g.\ algebra to a twisted homogeneous
coordinate ring that gives geometric information about the algebra. The following property is relevant.
\begin{defn}[{\cite[cf. \S 4]{artin1999generic}}] \label{defn: strongnoeth}
Let $A$ be a noetherian c.g.\ algebra. $A$ is \emph{strongly noetherian}\index{term}{strongly noetherian} if for all
commutative noetherian $k$-algebras $R$, $A \otimes R$ is noetherian. 
\end{defn}

Under the hypothesis that an algebra is strongly noetherian and generated in degree 1, one always has a map from it to
some noetherian twisted homogeneous coordinate ring \cite[Theorem 1.1]{rogalski2008canonical} (in fact, the scheme
involved is the \emph{point scheme} of the algebra, see Definition \ref{def: pointscheme}). In particular, many
AS-regular algebras of low dimension are strongly noetherian. Thus a noncommutative \proj{k}{n-1} will often contain a
bona fide projective scheme in the sense that $\text{coh}(\calO_X)\subset \text{qgr}(A)$ for some projective scheme $X$
(as in Example \ref{eg: 3dimsklyanin}). Unfortunately, not every noetherian algebra is strongly noetherian as Rogalski's
examples in \cite{rogalski2004generic} first illustrated.

Theorems \ref{thm: sklyanin} and \ref{thm: geometricthcrdesc} suggest that it may be possible to prove a result in the
vein of \cite[Theorem 1.1]{rogalski2008canonical} for \emph{twisted rings} rather than twisted homogeneous coordinate
rings. Twisted rings are defined in Definition \ref{defn: twistedhomring} and are one of the fundamental components of
Artin and Stafford's classification of noncommutative curves, begun in \cite{artin1995noncommutative} and completed in
\cite{artin2000semiprime}. Artin and Stafford's work in these papers allows all semiprime noetherian algebras of GK
dimension 2 to be described in terms of twisted rings (see Theorem \ref{thm: artinstaffordmain} for a more precise
formulation). We will return to the theme of canonical maps to twisted rings in \S\ref{subsec: irrinqgr}. 

Having associated a noncommutative space to a c.g.\ algebra, one can ask the following two questions to motivate the
development of the subject:
\begin{ques}\label{que: motivquenag}
 \begin{itemize}
 \item[(i)] What are the irreducible objects in $\text{qgr}(A)$?
 \item[(ii)] What are the noncommutative analogues of points, lines, etc. in $\text{qgr}(A)$?
 \end{itemize}
\end{ques}
In the next section we will give an overview of attempts to answer these questions.

\subsection{Irreducible objects in noncommutative spaces}\label{subsec: irrinqgr}

In order to answer Questions \ref{que: motivquenag} we will once again consider what happens in the commutative case.
Thus we assume until further notice that $A$ is a commutative noetherian c.g.\ algebra generated in degree 1. Associated
to $A$ is a projective scheme $X$, whose embedding in projective space gives rise to a twisting sheaf as in Definition
\ref{def: twistingsheaf}. 

Before considering irreducible objects in $\text{qgr}(A)$ we need the following definition.
\begin{defn}[{\cite[cf. Chapter II, Definition pg. 115]{hartshorne1977algebraic}}]\label{def: idealsheaf}
Consider a projective scheme $X$ and a closed point $p \in X$, along with the inclusion morphism $i: p \hookrightarrow
X$. The \emph{ideal sheaf}\index{term}{ideal sheaf} of $p$, denoted by
$\mathcal{I}_p$\index{notation}{i@$\mathcal{I}_p$}, is the kernel of the morphism $i^{\sharp}: \calO_X \rightarrow
i_{\ast}\calO_{X,p}$. Here $\calO_{X,p}$ denotes the local ring of $X$ at $p$.
\end{defn}

The equivalence of categories in Theorem \ref{thm: serresthm} tells us that the irreducible objects in $\text{qgr}(A)$
correspond to the irreducible objects in $\text{coh}(\calO_X)$. In the latter these are the skyscraper
sheaves\index{term}{skyscraper sheaf} $k_p:=\calO_X/\mathcal{I}_p$\index{notation}{k@$k_p$}, supported only at a closed
point $p \in X$. 

The functor $\text{coh}(\calO_X) \rightarrow \text{qgr}(A)$ is defined on a coherent sheaf of $\calO_X$-modules $\calF$
by
\begin{equation}\label{eq: serrecatequivdefn}
\mathcal{F} \mapsto \pi\left(\bigoplus_{n \in \N} H^0(X,\mathcal{F}\otimes_{\calO_X} \calO_X(n))\right),
\end{equation} 
where $\pi: \text{grmod}(A) \rightarrow \text{qgr}(A)$ is the canonical functor to the associated noncommutative space. 

The sheaf $k_p$ is sent under the equivalence described in \eqref{eq: serrecatequivdefn} to the tail of the module
$M_p:=A/I_p$\index{notation}{i@$I_p$}\index{notation}{m@$M_p$} for an appropriate ideal $I_p$. For any closed point $p
\in X$ the module $M_p$ has Hilbert series $1/(1-t)$. If $N$ denotes a nontrivial $\N$-graded submodule of $M_p$ then,
since the factor module $M_p/N$ is a finite-dimensional $k$-vector space, one has $\pi(N) \cong \pi(M_p)$ in
$\text{qgr}(A)$. Equivalently, $M_p$ is irreducible in $\text{qgr}(A)$.

Let us now return to our base assumptions, thus $A$ is a noetherian c.g.\ $k$-algebra generated in degree 1. Combining
our observations regarding the commutative case with the Noncommutative Serre's Theorem suggests that the following is a
good definition.
\begin{defn}[{\cite[Definition 3.8]{artin1990some}}]\label{defn: pointmodule}
Let $A$ be a noetherian c.g.\ algebra. An $\N$-graded right $A$-module $M$ is a \emph{point module}\index{term}{point
module} if it is cyclic and has Hilbert series $1/(1-t)$.
\end{defn}

Note that these conditions guarantee that any $\N$-graded submodule of a point module has finite codimension when $A$ is
generated in degree 1. Thus if $M_p$ is a point module then its tail $\pi(M_p)$ will be irreducible in $\text{qgr}(A)$.
However, there can exist irreducible objects in $\text{qgr}(A)$ other than the tails of point modules. As remarked at
the beginning of \cite[\S 7]{smith1994four}, this is comparable to the fact that noncommutative algebras can have simple
modules of dimension greater than 1. 
\begin{defn}[{\cite[pg. 8]{artin1990geometry}}]\label{defn: fatpointmodule}
Let $A$ be a noetherian c.g.\ algebra. An $\N$-graded right $A$-module $M$ is a \emph{fat point module of multiplicity
$e$}\index{term}{fat point module} if it has Hilbert series $e/(1-t)$, is generated in degree 0 and any non-zero
$\N$-graded submodule has finite codimension.
\end{defn}
\begin{rem}\label{rem: multiplicitymodule}
By \cite[Proposition 2.21(iii)]{artin1991modules}, the multiplicity of a f.g.\ $\N$-graded module $M$ can be defined
more generally: it is the leading coefficient of the expansion of the Hilbert series $H_M(t)$ in terms of powers of
$(1-t)$. It is clear that this more general definition of multiplicity coincides with the use of the term in Definition
\ref{defn: fatpointmodule}.
\end{rem}

One recovers the definition of a point module from that of a fat point module of multiplicity 1; as remarked after
Definition \ref{defn: pointmodule}, the finite codimension condition on $\N$-graded submodules automatically holds for a
cyclic module with Hilbert series $1/(1-t)$. For general fat point modules this condition is part of the definition to
ensure that the tails of such modules are irreducible in $\text{qgr}(A)$. We will refer to the tail (in $\text{qgr}(A)$)
of a fat point module as a \emph{fat point}\index{term}{fat point}, and similarly a \emph{point}\index{term}{point,
noncommutative notion of a} will refer to the tail of a point module. 

One could approach the question of irreducibility from another angle. The condition in Definition \ref{defn:
fatpointmodule} regarding every $\N$-graded submodule having finite codimension can be restated in terms of critical
modules.
\begin{defn}[{\cite[pg. 342]{artin1991modules}}]\label{def: criticalmodule}
A module $M \in \text{grmod}(A)$ is said to be \emph{$n$-critical}\index{term}{critical module} if $\text{GKdim}(M)=n$
and $\text{GKdim}(M/N)<n$ for any non-zero $\N$-graded submodule $N$.
\end{defn}

A basic property of GK dimension is that the finite-dimensional modules over any algebra have GK dimension 0. This has
two consequences for us; firstly, it means that 1-critical modules have irreducible tails in $\text{qgr}(A)$, and
secondly it implies that fat point modules of any multiplicity are 1-critical.

In relation to Question \ref{que: motivquenag}(i), irreducible objects in $\text{qgr}(A)$ are, unfortunately, not always
tails of 1-critical modules (nor fat point modules). An example of Smith \cite[pg. 5]{smith1992the} demonstrates this:
if $A=k[x,y]$ is given the grading $x \in A_1$, $y \in A_2$, then the tail of the module $A/xA$ is irreducible in
$\text{qgr}(A)$. This module has multiplicity 1/2 and so is not a fat point module. Smith's manuscript also gives an
example of a subring of a differential operator ring over which there is a module with Hilbert series $1/(1-t)^2$, whose
tail is irreducible in the associated noncommutative space.

Nonetheless, fat point modules do exhaust all irreducible objects in $\text{qgr}(A(\alpha,\beta,\gamma))$ for a large
class of 4-dimensional Sklyanin algebras by \cite[Proposition 7.1]{smith1994four} (recall Definition \ref{def: sklyanin
relations} for the construction of $A(\alpha,\beta,\gamma)$). We will elaborate on this point shortly.

Finding the point modules of a c.g.\ algebra $A$ is often the starting point to determining the geometry of
$\text{qgr}(A)$. A crucial object in relation to point modules is the \emph{point scheme} of an algebra, which was first
introduced in \cite[\S 3]{artin1990some}. We must define some ingredients to be able to describe the point scheme of an
algebra, namely truncated point modules and multilinearisations.
\begin{defn}[{\cite[Definition 3.8]{artin1990some}}]\label{def: truncated}
Let $A$ be a noetherian c.g.\ algebra generated by $n$ elements of degree 1. For $d \in \N$ a \emph{truncated point
module of length $d+1$}\index{term}{truncated point module} is a module $M \in \text{grmod}(A)$ that is cyclic and has
Hilbert series $H_M(t)=\sum_{i=0}^{d}t^i$.
\end{defn}

In the following definition we use the shorthand notations $\overline{\{n\}}:=
\{0,1,\ldots,n\}$\index{notation}{n@$\overline{\{n\}}$} and $\underline{i} :=
(i_0,\ldots,i_{d-1})$\index{notation}{i@$\underline{i}$} for a vector in $I:=\overline{\{n-1\}}^d$, where the length of
the vector is clear from the context.
\begin{defn}[{\cite[\S 3]{artin1990some}}]\label{def: multilin}
Let $A=T(V)/I$ be as in Definition \ref{def: truncated}, with $V = kx_0 + \ldots + k x_{n-1}$. Let $d \in \N$ and $f \in
I_d$ be a relation, where $f = \sum_{I}\alpha_{\underline{i}}x_{i_{0}}\ldots x_{i_{d-1}}$ for some
$\alpha_{\underline{i}} \in k$. Define $\tilde{f}$ to be the element of the homogeneous coordinate ring of
$(\proj{k}{n-1})^{\times d}$ given by 
\begin{equation*}%\label{}
\sum_{I}\alpha_{\underline{i}}x_{i_{0},0}x_{i_{1},1}\ldots x_{i_{d-1},d-1}, 
\end{equation*}
where the coordinate ring of the $(j+1)$'th copy of $\proj{k}{n-1}$ is $k[x_{0,j},\ldots, x_{n-1,j}]$. We say that
$\tilde{f}$ is the \emph{multilinearisation}\index{term}{multilinearisations} of $f$.
\end{defn}
  
Multilinearisations allow truncated point modules of length $d+1$ to be parameterised by a projective scheme $\Gamma_d$,
as proved in \cite[Proposition 3.9]{artin1990some}. Furthermore, this parameterisation is functorial in the sense that
there are natural maps in each situation that correspond to each other (formally, $\Gamma_d$ represents the functor of
flat families of truncated point modules of length $d+1$). To see this correspondence, consider the map that sends a
truncated point module of length $d+1$, $M = \bigoplus_{i=0}^{d}M_i$, to $\bigoplus_{i=0}^{d-1}M_i$ by factoring out the
highest degree component. This map sends $M$ to a truncated point module of length $d$. A morphism of schemes
$\Gamma_{d} \rightarrow \Gamma_{d-1}$ is induced and, as \cite[Proposition 3.7(i)]{artin1990some} shows, the induced
morphism is nothing more than the projection $(\proj{k}{n-1})^{\times d} \rightarrow (\proj{k}{n-1})^{\times d-1}$ which
`forgets' the final component of the product.
\begin{defn}\label{def: pointscheme}
Let $A$ be as in Definition \ref{def: truncated}. The \emph{point scheme}\index{term}{point scheme} of $A$ is the
inverse limit $\Gamma$\index{notation}{g@$\Gamma$} of the following diagram:
\begin{equation*}%\label{eq: pointscheme}
\begin{tikzpicture}
\draw [<-] (0.3,2) -- (1.2,2); 
\draw [<-] (1.8,2) -- (2.7,2); 
\draw [<-] (3.3,2) -- (4.1,2); 
\draw [<-] (4.9,2) -- (5.65,2); 
\draw [<-] (6.4,2) -- (7.2,2); 
\node at (0,2) {$\Gamma_0$};
\node at (1.5,2) {$\Gamma_1$};
\node at (3,2) {$\ldots$};
\node at (4.5,2) {$\Gamma_{d-1}$};
\node at (6,2) {$\Gamma_{d}$};
\node at (7.5,2) {$\ldots$};
\node at (3.75,0) {\large{$\Gamma$}};

\draw [->] (3.58,0.18) -- (0,1.7); 
\draw [->] (3.69,0.3) -- (1.5,1.7); 
\draw [->] (3.79,0.3) -- (4.5,1.7); 
\draw [->] (3.92,0.18) -- (6,1.7); 

\end{tikzpicture}
\end{equation*}
Here $\Gamma_d$ denotes the projective scheme parameterising truncated point modules of length $d+1$, with the maps
$\Gamma_{d}\rightarrow \Gamma_{d-1}$ as described in the previous paragraph.
\end{defn}

By \cite[Corollary 3.13]{artin1990some} the points of $\Gamma$ are in 1-1 correspondence with point modules of $A$. This
means that $\Gamma$ `parameterises' point modules in the following sense: $\Gamma$ represents the functor mapping a
f.g.\ commutative $k$-algebra $C$ to the set of isomorphism classes of factors of $A \otimes_k C$ with Hilbert series
$1/(1-t)$. The use of the word parameterises in Theorem \ref{thm: artzhanglimit} below reflects this.
 
In particularly nice cases one of the morphisms in the inverse limit structure of the point scheme is an isomorphism.
Proposition 3.7(ii) from \cite{artin1990some} then implies that all higher degree morphisms are also isomorphisms. 
Intrepreted in terms of modules, this means that given a truncated point module $M$ of sufficiently great length $d \gg
0$, there exists a unique point module $M'$ such that $M \cong M'/M'_{\geq d}$. 

For strongly noetherian c.g.\ algebras the following theorem shows that the point scheme has a particularly nice form.
\begin{thm}[{\cite[cf. Corollary E4.11]{artin2001abstract}}]\label{thm: artzhanglimit}
Assume that $A$ is a strongly noetherian c.g.\ algebra over a field $k$. Then the point scheme of $A$ is a projective
scheme that parameterises the point modules over $A$.
\end{thm}
\begin{rem}\label{rem: justifyscheme}
While the term \emph{point scheme} is misleading in the sense that such an object is not necessarily a scheme, Theorem
\ref{thm: artzhanglimit} shows that the term is justified under certain hypotheses. 
\end{rem}
All of the algebras we will study in this thesis are strongly noetherian, thus their point schemes are projective
schemes. 

Let us now consider Rogalski and Zhang's result \cite[Theorem 1.1]{rogalski2008canonical}, describing a canonical map
from a strongly noetherian c.g.\ algebra generated in degree 1 to a twisted homogeneous coordinate ring over the point
scheme of the algebra. Theorem \ref{thm: artzhanglimit} has implications for the relevance of Theorems \ref{thm:
sklyanin} and \ref{thm: geometricthcrdesc} to a generalisation of Rogalski and Zhang's result; before even considering a
canonical map to a twisted ring as defined in Definition \ref{defn: twistedhomring}, one should try to discover the
structure of the geometric object --- if such an object exists --- parameterising fat point modules of higher
multiplicity. In fact, \cite[Theorem E5.1]{artin2001abstract} provides an answer to this question for algebras
satisfying a strong homological condition, with the geometric objects shown to be countable unions of projective
schemes. 

Let us assume now that our algebra $A$ is strongly noetherian in addition to our other hypotheses. More information
about point modules can be uncovered by using the following technique, described in the introduction of
\cite{artin1991modules}. 
\begin{defn}\label{defn: ptschemeautomorphism}
Let $A$ be a strongly noetherian c.g.\ algebra with point scheme $\Gamma$, and let $M_p$ be a point module over $A$
corresponding to a point $p \in \Gamma$. The shift $M_p[1]_{\geq 0}$ is also a point module over $A$, therefore it
corresponds to some point $p^{\sigma}\in \Gamma$. This shifting operation induces a scheme automorphism $\sigma: \Gamma
\rightarrow \Gamma$ which we call the \emph{associated automorphism of the point scheme}\index{term}{point
scheme!associated automorphism}. 
\end{defn}

We will now state a result of Shelton and Vancliff concerning the point schemes of certain algebras. It will be useful
to us on several occassions, notably in \S\ref{subsec: 4sklypointscheme}. Any undefined properties in the statement of
the result will be defined in \S\ref{subsec: cohenmac}.
\begin{thm}[{\cite[Theorem 1.4]{shelton1999embedding}}]\label{thm: pointschemenice}
Assume that $k$ is an algebraically closed field for which $\text{char}(k) \neq 2$. Let $A$ be a c.g.\ $k$-algebra with
point scheme $\Gamma$, and suppose that the following are true for $A$:
\begin{itemize}
 \item[(i)] it is generated in degree 1 and has Hilbert series $1/(1-t)^4$;
 \item[(ii)] it is noetherian and Auslander regular of global dimension 4;
 \item[(iii)] it is Cohen-Macaulay.
\end{itemize}

Then the projective scheme $\Gamma_2$ corresponding to the multilinearisations of the defining relations of $A$ is
isomorphic to
\begin{equation}\label{eq: graphptscheme}
\Gamma_2 = \{(p,p^{\sigma}):\; p \in \Gamma\} \subset \proj{k}{3} \times \proj{k}{3},
\end{equation}
the graph of $\Gamma$ under the associated automorphism $\sigma$.
\end{thm}
\begin{notation}\label{not: projmap}
Theorem \ref{thm: pointschemenice} will be applicable to many of the algebras that we study. We introduce the notation
$\pi_i: \proj{k}{3} \times \proj{k}{3} \rightarrow \proj{k}{3}$\index{notation}{p@$\pi_1, \pi_2$} for $i =1,2$ to denote
the projection morphism to the $i$'th coordinate. In the setting of the theorem one may therefore write $\pi_1(\Gamma_2)
= \Gamma$.
\end{notation}

In the language of Definition \ref{def: pointscheme}, Theorem \ref{thm: pointschemenice} tells us that the projection
morphism $\Gamma_2 \rightarrow \Gamma_1$ defines an isomorphism from $\Gamma_2$ to the point scheme $\Gamma$. Translated
into a statement about modules, it implies that each truncated point module of length 3 over such an algebra uniquely
determines a point module; to study point modules it suffices to study the multilinearisations of the quadratic
relations in the algebra.

The point scheme of a c.g.\ algebra is an important invariant that is preserved by $\N$-graded isomorphisms. Moreover,
it is preserved by certain instances of a twisting operation on algebras called \emph{Zhang twists}. While we will
discuss Zhang twists fully in \S\ref{subsec: zhangtwist}, we define a special class of them here.
\begin{defn}[{cf. \cite[pg. 284]{zhang1998twisted}, Example \ref{eg: autzhangtwist}}]\label{def: algebraictwist}
Suppose that $A$ is a c.g.\ algebra and $\phi$ is an $\N$-graded algebra automorphism of $A$. One can define a new
multiplicative structure on $A$ by defining $a \ast_{\phi} b := a \phi^n(b)$ for all $a \in A_n$ and $b \in A_m$. The
algebra $(A,\ast_{\phi})$ is called a \emph{Zhang twist}\index{term}{Zhang twist} of $A$ by $\phi$.
\end{defn}

If a c.g.\ algebra $B$ is a Zhang twist of the $A$ as in Definition \ref{def: algebraictwist}, then by \cite[Theorem
3.1]{zhang1998twisted} one has $\text{GrMod}(A) \simeq \text{GrMod}(B)$, hence $\text{qgr}(A) \simeq \text{qgr}(B)$.
Thus the geometry encoded by point modules and fat point modules is preserved under such twists.

In general, Zhang twists can involve $G$-gradings for any semigroup $G$, and as Theorem \ref{thm: cocycleaszhang}
demonstrates, the cocycle twists in Proposition \ref{prop: twoconstrsaresame} can be phrased as Zhang twists. In that
situation, \cite[Theorem 3.1]{zhang1998twisted} shows that there is an equivalence of categories $\text{GrMod}_{G}(A)
\simeq \text{GrMod}_{G}(A^{G,\mu})$, where $G$ is the relevant semigroup. When $G \neq \N$ this equivalence gives
relatively little information about the relationship between $\text{qgr}(A)$ and $\text{qgr}(A^{G,\mu})$.  

Our results show that $\text{qgr}(A)$ and $\text{qgr}(A^{G,\mu})$ can be very different. Let us focus on the
4-dimensional Sklyanin algebra $A(\alpha, \beta, \gamma)$, under the additional assumption that the associated
automorphism $\sigma$ has infinite order. All fat point modules over $A(\alpha, \beta, \gamma)$ are known, and by
\cite[Proposition 7.1]{smith1994four} fat point modules constitute all of the irreducible objects in
$\text{qgr}(A(\alpha, \beta, \gamma))$. In particular:
\begin{itemize}
 \item[(i)] its points are parameterised by an elliptic curve $E$ and four extra points in \proj{k}{3}
\cite[Propositions 2.4 and 2.5]{smith1992regularity};
 \item[(ii)] fat points of all multiplicities were classified by Smith and Staniszkis in \cite{smith1993irreducible}.
\end{itemize}

By contrast, Theorem \ref{thm: sklyanin} shows that $A(\alpha, \beta, \gamma)^{G,\mu}$ has infinitely many fat point
modules of multiplicity 2 and only 20 point modules, the former valid without the assumption that $|\sigma|=\infty$.  

Before moving on we briefly address our method of constructing the fat point modules of multiplicity 2 over $A(\alpha,
\beta, \gamma)^{G,\mu}$. Fat point modules over $A(\alpha, \beta, \gamma)$ were realised in \cite{smith1993irreducible}
as factor modules of line modules (see Definition \ref{defn: linearmodules} with $n=1$). They also appear as a byproduct
of some categorical equivalences described in \cite[Proposition 7.5.2]{van1996translation}.

Our method uses the fact that $kG_{\mu} \cong M_2(k)$ for the Klein-four group $G$ and some normalised 2-cocycle $\mu$
(see Lemma \ref{lem: kgmuiso}). This means that $A(\alpha, \beta, \gamma)^{G,\mu}$ is a subring of $M_2(A(\alpha, \beta,
\gamma))$ via Construction 2 of the twist. It is then natural to take direct sums of modules over $A(\alpha, \beta,
\gamma)$ and consider them as $M_2(A(\alpha, \beta, \gamma))$-modules, before restricting them down to the cocycle
twist. For a point module over $A(\alpha, \beta, \gamma)$ corresponding to a point on the elliptic curve $E$, this
method produces a fat point module of multiplicity 2 over $A(\alpha, \beta, \gamma)^{G,\mu}$.

Let us now address Question \ref{que: motivquenag}(ii) -- what are the noncommutative analogues of lines, planes, etc.?
Once again we can look to the commutative case for inspiration, which suggests the following definition.
\begin{defn}[{\cite[cf. Definition pg. 51]{levasseur1993modules}}]\label{defn: linearmodules}
Let $A$ be a noetherian c.g.\ algebra. A module $M \in \text{grmod}(A)$ is a \emph{linear module of dimension
$n$}\index{term}{linear module} if it is cyclic and has Hilbert series $1/(1-t)^{n+1}$.
\end{defn}

For $n=0$ one recovers point modules, while modules in the $n=1,2$ cases are referred to as \emph{line
modules}\index{term}{line module} and \emph{plane modules}\index{term}{plane module} respectively. Thus lines and planes
in $\text{qgr}(A)$ can be defined as the tails of line and plane modules in $\text{grmod}(A)$ similarly. Another
indication that $\text{qgr}(A(\alpha,\beta,\gamma))$ is well-understood is given by Staniszkis's classification of
linear modules of all dimensions in \cite{staniszkis1996linear}. 

We saw in Theorem \ref{thm: artzhanglimit} that the point scheme of an algebra parameterises its point modules under the
strongly noetherian hypothesis. Is there an object which plays a similar role for larger linear modules? This is the
question addressed by the research in \cite{shelton2002schemes}. One of the results in that paper, namely Corollary 1.5
op. cit., shows that there is a scheme $\Omega_{\infty}(A,n)$ whose closed points parameterise the $n$-linear modules
over $A$. When $n=1$ this scheme is called the \emph{line scheme}\index{term}{line scheme} of the algebra, while for
$n=0$ one recovers the point scheme.

Now consider the following conditions on an algebra $A$.
\begin{conditions}[{\cite[Conditions 2.1-2.3]{shelton2002schemes}}]\label{cond: linescheme}
Suppose that $A$ satisfies the following conditions:
\begin{itemize}
 \item[(i)] $A$ has Hilbert series $H_A(t)=1/(1-t)^4$;
 \item[(ii)] $A$ is a domain;
 \item[(iii)] any plane module $M$ over $A$ is 3-critical, and any nontrivial $\N$-graded submodule of $M$ also has GK
dimension 3.
\end{itemize}
\end{conditions}
If $A$ satisfies these conditions then its line scheme will not only be a projective scheme, but can be characterised in
several different ways \cite[cf. Lemma 2.4]{shelton2002schemes}. Furthermore, by Corollary 2.6 op. cit. the irreducible
components of the line scheme of an algebra satisfying Conditions \ref{cond: linescheme} are at least 1-dimensional. 

There are some examples in the literature of AS-regular algebras of dimension 4 with a 1-dimensional line scheme; see
\cite[Proposition 3.5]{vancliff1998some} or the algebras defined in \cite[Example 5.1]{cassidy2010generlizations} for
example. Some of the algebras that we study later in this thesis will be shown to have a 1-dimensional line scheme (see
Propositions \ref{prop: linescheme1dim} and \ref{prop: staffordlinescheme}).

\subsection{Algebras with a 0-dimensional point scheme}\label{sec: finitedimptscheme}
One of the major themes of the papers
\cite{vancliff1998some,shelton1999some,cassidy2006generalized,stephenson2007constructing,cassidy2010generlizations} is
the study and construction of AS-regular algebras of dimension 4 with a 0-dimensional point scheme. As Theorem \ref{thm:
sklyanin} shows, some examples of cocycle twists that we study can be considered alongside these examples. 

The original example of a noncommutative algebra with a 0-dimensional point scheme is given in the unpublished
manuscript \cite{van1988example}. Within it, Van den Bergh shows that graded Clifford algebras with `generic' relations
have 20 point modules, related in pairs by the shifting automorphism of Definition \ref{defn: ptschemeautomorphism}.
These algebras were generalised in \cite{cassidy2010generlizations}, with a cocycle twist of such a generalisation being
studied in \S\ref{subsec: gradedskewclifford}.

As remarked in several papers, for example in \cite[\S 1.2]{shelton1999embedding}, an AS-regular algebra 
of dimension 4 that is generated in degree 1 with Hilbert series $1/(1-t)^4$ and `generic' relations has a 0-dimensional
point scheme. Furthermore, if such an algebra has a 0-dimensional point scheme then it has at most 20 point modules; the
precise number depends on the multiplicities of the closed points in the point scheme (see the remarks prior to
\cite[Definition II.1.7]{goetz2003noncommutative}). 

Since we have been unable to find a proof of the latter fact in the literature, we give one below. In the proof we
consider certain irreducible projective schemes as projective varieties in order that we may use the notion of degree
for such objects, as defined in \cite[Chapter I, Definition pg. 52]{hartshorne1977algebraic}.
\begin{prop}\label{prop: genericpointscheme}
Let $k$ be any algebraically closed field and assume that the $k$-algebra $A=T(V^{\ast})/I$ satisfies conditions
(i)-(iii) in Theorem \ref{thm: pointschemenice}. If the point scheme of $A$ is 0-dimensional then $A$ has 20 point
modules counting multiplicity.
\end{prop}
\begin{proof}
By assumption $A$ satisfies the hypotheses of Theorem \ref{thm: pointschemenice}. That result implies that the point
scheme of $A$ is $\Gamma :=\pi_1(\Gamma_2)$, where 
\begin{equation*}
\Gamma_2 \subset \mathbb{P}(V) \times \mathbb{P}(V) =  \proj{k}{3} \times \proj{k}{3}
\end{equation*}
is the scheme determined by the multilinearisations of the quadratic relations of $A$. Furthermore, $\Gamma_2$ is the
graph of $\Gamma$ under an automorphism. Combined with the assumption that $\Gamma$ is 0-dimensional, this implies that
$\Gamma_2$ must also be 0-dimensional. The closed points of $\Gamma$ parameterise point modules over $A$, thus it
suffices to prove that $\Gamma_2$ contains 20 points counting multiplicity.

We may use the Segre embedding to consider $\proj{k}{3} \times \proj{k}{3}$ (and thus $\Gamma_2$) as a closed subscheme
of
$\mathbb{P}(V \otimes V) = \proj{k}{15}$. Note that $\Gamma_2$ consists of precisely those points in $\proj{k}{3} \times
\proj{k}{3}$ at which all elements in $I_2$ vanish upon evaluation. One naturally has $I_2^{\perp} \subset V \otimes V$,
in which case $\Gamma_2$ can be described inside $\proj{k}{15}$ as the intersection of $\proj{k}{3} \times \proj{k}{3}$
with the projectivisation of $I_2^{\perp}$, namely $\mathbb{P}(I_2^{\perp})$. Thus $\Gamma_2=(\proj{k}{3} \times
\proj{k}{3}) \cap
\mathbb{P}(I_2^{\perp}) \subset \proj{k}{15}$. 

To see that $\Gamma_2$ consists of 20 points counting multiplicity we can apply Bezout's Theorem \cite[Chapter I,
Theorem 7.7]{hartshorne1977algebraic}. In the notation of that result we take $H=\mathbb{P}(I_2^{\perp})$ and $Y =
\proj{k}{3} \times \proj{k}{3}$, which can both be considered as projective varieties. By assumption, $I_2$ is a
6-dimensional $k$-vector space, hence $I_2^{\perp} \subset V \otimes V$ has dimension 10. It follows that
$\mathbb{P}(I_2^{\perp})= \proj{k}{9}$, whence it has degree 1. Using \cite[Chapter I, Exercise
7.1(b)]{hartshorne1977algebraic} one can see that $\proj{k}{3} \times \proj{k}{3}$ has degree $\binom{3+3}{3}$, which
equals 20. Bezout's theorem tells us that the sum of the multiplicities of the points in $\Gamma_2$ is equal to
$\text{deg}(\proj{k}{3} \times
\proj{k}{3})\cdot\text{deg}(\mathbb{P}(I_2^{\perp}))$, which is 20.
\end{proof}

In light of Proposition \ref{prop: genericpointscheme}, Theorem \ref{thm: sklyanin}(ii) tells us that the Sklyanin twist
$A(\alpha,\beta,\gamma)^{G,\mu}$ has the maximal number of point modules given that its point scheme is 0-dimensional. 

The following result shows that certain algebras are determined by the scheme $\Gamma_2$.
\begin{thm}[{\cite[cf. Theorem 4.1]{shelton2002schemes}}]\label{thm: pointschemereconstruct}
Fix a base field $k$ that is algebraically closed and for which $\text{char}(k) \neq 2$. Let $T(V)/I$  be a quadratic
algebra on four generators with six defining relations, and let $\Gamma_2 \subset \proj{k}{3} \times \proj{k}{3}$ denote
the zero locus of $I_2$ in the sense of Definitions \ref{def: multilin} and \ref{def: pointscheme}. If $\Gamma_2$ is
0-dimensional then
\begin{equation*}%\label{eq: vanshevanish}
 I_2 = \{ f \in V \otimes V : f|_{\Gamma_{2}} = 0 \}.
\end{equation*}
\end{thm}

Remarkably, this result does not require any hypotheses other than a specified number of generators and relations.
Applied to our results it implies that the Sklyanin twist $A(\alpha,\beta,\gamma)^{G,\mu}$ is determined up to
isomorphism by the 20 points in its point scheme and the associated automorphism, which together describe $\Gamma_2$ for
that algebra. We
show in Theorem \ref{thm: new} that $A(\alpha,\beta,\gamma)^{G,\mu}$ is not isomorphic to any of the existing examples
in the
literature of AS-regular algebras of dimension 4 with 20 point modules.

A similar result to Theorem \ref{thm: pointschemereconstruct} was proved for the line scheme of an AS-regular algebra of
dimension 4 in \cite[Theorem 4.3]{shelton2002schemes}. When the line scheme of such an algebra has dimension 1 its
relations can be reconstructed from those functions vanishing on the line scheme. 

Such results highlight the importance of finding examples of algebras that are AS-regular of dimension 4, have a
0-dimensional point scheme and a 1-dimensional line scheme, since information in $\text{qgr}(A)$ essentially determines
them. In addition to having a 0-dimensional point scheme, the Sklyanin twist $A(\alpha,\beta,\gamma)^{G,\mu}$ has a
1-dimensional line scheme by Proposition \ref{prop: linescheme1dim}.

\section{Twists of algebras}\label{sec: twistsofalgebras}
\sectionmark{Twists of algebras}
In this section we will describe three constructions relating to twisting algebras. The first two, namely cocycle twists
and Zhang twists, involve twisting the multiplicative structure in an algebra. These objects are the subject of sections
\S\ref{subsec: cocycletwists} and \S\ref{subsec: zhangtwist} respectively. Cocycle twists can be formulated as Zhang
twists, as shown in Theorem \ref{thm: cocycleaszhang}. A partial converse to this result is given by Proposition
\ref{prop: recoverztwist} in the next chapter, which gives a class of Zhang twists that can be phrased as cocycle
twists. The section ends with crossed products being defined in \S\ref{sec: crossedproduct}. Once again there is some
intersection between such objects and the twisting constructions previously introduced.
 
\subsection{Cocycle twists}\label{subsec: cocycletwists}
Cocycle twists are a classical construction and the subsequent material can be found in many places, such as \cite[\S
2.3]{karpilovsky1987schur}. The idea of twisting algebraic structures has been vastly extended; the recent paper
\cite{bazlov2012cocycle} studies cocycle twists in a monoidal category for example. In this section we expand on the
brief introduction to classical cocycle twists given in Construction 1 of \S\ref{subsec: theoryoftwist}. 

We will assume until further notice that $k$ is an algebraically closed field and $A$ is an associative $k$-algebra with
identity. By $G$ we will denote a finite group, not necessarily abelian, for which $\text{char}(k) \nmid |G|$ and $A$
admits a $G$-graded structure. Recall from Definition \ref{defn: ggradedalgebra} that this means that there is a direct
sum decomposition of $k$-vector spaces $A=\bigoplus_{g \in G} A_g$, with $A_g A_h \subset A_{gh}$ for all $g, h \in G$
and $1 \in A_e$. 

In Definition \ref{def: normalised2cocycle} we defined a \emph{normalised 2-cocycle} of $G$. All 2-cocycles that we use
will be normalised, therefore we make the following definition for the remainder of the thesis.
\begin{defn}[{cf. Definition \ref{def: normalised2cocycle}}]\label{def: 2cocycle}
A \emph{2-cocycle}\index{term}{2-cocycle} of $G$ with values in $k^{\times}$ is a function $\mu: G \times G
\rightarrow k^{\times}$ satisfying the following relations for all $g,h,l \in G$:
\begin{align}
\mu(g,h)\mu(gh,l) &=\mu(g,hl)\mu(h,l),\label{eq: cocycleidinto1}\\
\mu(e,g) =&\mu(g,e)=1.\label{eq: cocycleidinto2}
\end{align} 
\end{defn}

Let $Z^2(G,k^{\times})$\index{notation}{z@$Z^2(G,k^{\times})$} denote the set of 2-cocycles of $G$ with values in
$k^{\times}$. Under pointwise multiplication $Z^2(G,k^{\times})$ forms an abelian group whose identity element is
$\mu_e$, where $\mu_e(g,h)=1$ for all $g, h \in G$.

One can deform the $G$-graded multiplication in $A$ using a 2-cocycle as we saw in Construction 1 of \S\ref{subsec:
theoryoftwist}. However, we will now take a more general approach. Let $\mu: G \times G \rightarrow k^{\times}$ be any
function satisfying \eqref{eq: cocycleidinto2} for all $g \in G$. We define a map $\ast_{\mu}: A \times A \rightarrow A$
as follows. First, for homogeneous elements $a \in A_g$ and $b \in A_h$ define $a \ast_{\mu} b := \mu(g,h) ab$, where
juxtaposition denotes the original multiplication in $A$. This can then be extended by $k$-linearity to a function on
the whole of $A \times A$. One obtains a new algebra structure on $A$ with the same identity element, and we denote this
algebra by $(A,\ast_{\mu})$.

The next proposition shows that condition \eqref{eq: cocycleidinto1} in Definition \ref{def: 2cocycle} governs the
preservation of associativity under twisting.
\begin{prop}[{\cite[Lemma 2.3.3]{karpilovsky1987schur}}]\label{prop: preserveassociative}
Let $G$ be a finite group and $A$ be a $G$-graded associative $k$-algebra with identity. Consider a function $\mu: G
\times G \rightarrow k^{\times}$ satisfying \eqref{eq: cocycleidinto2}. Then $(A,\ast_{\mu})$ is a $k$-algebra with the
same algebra identity element as $A$. Moreover, $(A,\ast_{\mu})$ is associative if and only if $\mu$ satisfies
\eqref{eq: cocycleidinto1}. 
\end{prop}
\begin{proof}
To see that $(A,\ast_{\mu})$ is a $k$-algebra is suffices to check that $k$ is central under $\ast_{\mu}$. Let $\lambda
\in k^{\times}$. The $G$-grading of $A$ gives a decomposition into $k$-vector spaces, hence $k \subset A_e$ must hold.
Thus
\begin{equation}\label{eq: algident}
\lambda \ast_{\mu} a = \mu(e,g) \lambda a \;\text{ and }\; a \ast_{\mu} \lambda = \mu(g,e) \lambda a. 
\end{equation}
When \eqref{eq: cocycleidinto2} holds $\lambda$ must remain central under the new multiplication. Moreover, by taking
$\lambda = 1_A$ in \eqref{eq: algident} it is clear that the algebra identity in $(A,\ast_{\mu})$ is the same as that in
$A$.

Now let $a \in A_g$, $b \in A_h$ and $c \in A_l$ for some $g, h, l \in G$. One has
\begin{align*}
a \ast_{\mu} (b \ast_{\mu} c) &= a \ast_{\mu} (\mu(h,l) bc) = \mu(h,l)\mu(g,hl) a(bc),\\
(a \ast_{\mu} b) \ast_{\mu} c &= (\mu(g,h) ab) \ast_{\mu} c = \mu(g,h)\mu(gh,l) (ab)c.
\end{align*}
Using the associativity of $A$ it is clear that $\ast_{\mu}$ is associative if and only if \eqref{eq: cocycleidinto1} is
satisfied. 
\end{proof}

We gave the prototypical examples of cocycle twists in Example \ref{ex: twistedgroupalgebra}, namely \emph{twisted group
algebras}. Let us recall their definition.   
\begin{example}[{cf. \cite[\S 2.3]{karpilovsky1987schur}, Example \ref{ex: twistedgroupalgebra}}]\label{eg: groupalg}
Let $G$ be a finite group and $k$ a field. The \emph{group algebra}\index{term}{group algebra} $kG$ is a vector space
whose basis elements are indexed in a natural way by elements of $G$, that is, $kG= \bigoplus_{g \in G} kg$. The
multiplication in $kG$ is given by
\begin{equation*}
(\alpha g)(\beta h) = \alpha \beta (gh), 
\end{equation*}
for all $\alpha,\beta \in k$ and $g,h \in G$. 

Now let $A$ be an associative $k$-algebra. One can form a group algebra over $A$ by taking the tensor product of
algebras $AG:=A  \otimes kG$. The group algebra $AG$ has a $G$-grading given by $AG= \bigoplus_{g \in G} A \otimes g$,
and one can twist this grading using a 2-cocycle. For a 2-cocycle $\mu$, the algebra $AG_{\mu}:= (AG,\ast_{\mu})$ is a
\emph{twisted group algebra}\index{term}{twisted group algebra}. The multiplication in $AG_{\mu}$ is given by 
\begin{equation*}
(a \otimes g) \ast_{\mu} (b \otimes h) = \mu(g,h)ab \otimes gh, 
\end{equation*}
for all $a, b \in A$ and $g, h \in G$. It is clear that $AG_{\mu} = A \otimes kG_{\mu}$.
\end{example}

A natural question to ask is the following: for which 2-cocycles is $AG_{\mu}$ isomorphic to $AG$? In order to answer
this question let us assume that $A=k$. Consider a 2-cocycle $\mu$ for which there exists a function $\rho: G
\rightarrow k^{\times}$ such that for all $g, h \in G$,
\begin{equation}\label{eq: trivialcocycle}
\mu(g,h)=\rho(g)\rho(h)\rho(gh)^{-1}. 
\end{equation}

Let $B^2(G,k^{\times})$\index{notation}{b@$B^2(G,k^{\times})$} denote the subgroup of $Z^2(G,k^{\times})$ consisting of
all 2-cocycles satisfying \eqref{eq: trivialcocycle}. This is the set of
\emph{2-coboundaries}\index{term}{2-coboundary} of $G$ with coefficients in $k^{\times}$. If two 2-cocycles lie in the
same
coset modulo the coboundaries then we say they are \emph{cohomologous}. 

The next result shows that 2-coboundaries control isomorphisms between twisted group algebras. Although it is stated in
\cite{karpilovsky1987schur} for $k = \C$, the proof can be adapted to work when $k$ is only assumed to be algebraically
closed.
\begin{thm}[{\cite[cf. Theorem 2.3.4]{karpilovsky1987schur}}]\label{thm: coboundary}
The twisted group algebras $kG_{\mu}$ and $kG_{\phi}$ are isomorphic as $G$-graded algebras if and only if $\mu$ and
$\phi$ are cohomologous. In particular, if $\mu$ is a coboundary then $kG_{\mu} \cong kG$.
\end{thm}

By Theorem \ref{thm: coboundary} the abelian group $Z^2(G,k^{\times})/B^2(G,k^{\times})$, which we denote by
$H^2(G,k^{\times})$\index{notation}{h@$H^2(G,k^{\times})$}, can be viewed as an object that parameterises isomorphism
classes of $G$-graded deformations of $kG$. In fact, when $k = \C$ this quotient group has a special name.
\begin{defn}[{\cite[\S 2.1]{karpilovsky1987schur}}]\label{defn: schurmult}
The \emph{Schur multiplier}\index{term}{Schur multiplier} $M(G)$\index{notation}{m@$M(G)$} of a finite group $G$ is the
quotient group $H^2(G,\C^{\times})$.
\end{defn}
\begin{rem}\label{rem: otherschurmult}
The Schur multiplier of a group $G$ is also related to central extensions and projective representations of $G$, see
\cite[Theorem 1.4.1]{karpilovsky1987schur} for example.
\end{rem}

The Schur multiplier of a group $G$ --- and more generally $H^2(G,k^{\times})$ --- can be interpreted in terms of group
cohomology as in \cite[\S 1.3]{karpilovsky1987schur}. One can use this approach to show that cyclic groups have no
nontrivial 2-cocycles in the sense that $H^2(G,k^{\times})$ is isomorphic to the trivial group; for $k = \C$ this is
Proposition 2.1.1 op. cit.. It then follows from the next proposition that there are no nontrivial twists for cyclic
groups. The result is standard and is proved in a similar manner to one direction of Theorem \ref{thm: coboundary},
however we have been unable to find a reference for the form as we state it.
\begin{prop}\label{prop: trivialtwist}
Let $G$ be a finite group and $A$ a $G$-graded $k$-algebra. For a 2-cocycle $\mu$ that is cohomologous to the trivial
2-cocycle one has an isomorphism of $k$-algebras $(A,\cdot) \cong (A,\ast_{\mu})$, where $\cdot$ denotes the original
multiplicative structure of $A$.
\end{prop}
\begin{proof}
By assumption there exists some $\rho: G \rightarrow k^{\times}$ such that 
\begin{equation*}%\label{eq: coboundaryproof}
\mu(g,h)= \rho(g)\rho(h)\rho(gh)^{-1},
\end{equation*}
for all $g, h \in G$, as in \eqref{eq: trivialcocycle}. Define a map $\varphi:(A,\cdot) \rightarrow (A,\ast_{\mu})$ by
$\varphi(a)=\rho(g)^{-1}a$ for all $a \in A_g$ and extending by $k$-linearity. For homogeneous elements $a \in A_g$ and
$b \in A_h$ one has
\begin{equation*}
\varphi(a) \ast_{\mu} \varphi(b)= \mu(g,h)\rho(g)^{-1}\rho(h)^{-1} ab = \rho(gh)^{-1} ab = \varphi(ab).
\end{equation*}
Thus $\varphi$ is a $k$-algebra homomorphism, and since the two algebras have the same underlying vector space structure
it must in fact be an isomorphism.
\end{proof}

Another nice consequence of knowing that 2-cocycles over cyclic groups are trivial is given by the following result of
Yamazaki, given as an exercise in \cite{karpilovsky1987schur}. 
\begin{prop}[{\cite[Exercise 2.13.2]{karpilovsky1987schur}}]\label{prop: yamazaki}
Let $P$, $G_1$ and $G_2$ be abelian groups. Then 
\begin{equation*}
H^2(G_1 \times G_2, P) \cong H^2(G_1, P) \times H^2(G_2, P) \times \text{Hom}_{\text{grp}}(G_1 \otimes_{\Z} G_2, P)
\end{equation*}
\end{prop}

One can use Proposition \ref{prop: yamazaki} to calculate $H^2(G,k^{\times})$ for any finite abelian group. In
particular, the proposition implies that the Klein-four group $G=(C_2)^2$ is the smallest group for which
$H^2(G,k^{\times})$ is nontrivial. Our examples in Chapter \ref{chap: sklyanin} show that many interesting phenomena
occur for twists involving 2-cocycles over this group.

\subsection{Zhang twists}\label{subsec: zhangtwist}
We now move on to considering Zhang twists, a thorough treatment of which is given in \cite{zhang1998twisted}.
Throughout \S\ref{subsec: zhangtwist} we will retain the assumptions of \S\ref{subsec: cocycletwists}, except that $G$
can now be any semigroup, primarily to allow for the case $G=\N$. 

We will work with $G$-graded $k$-linear automorphisms of $A$. These are $k$-linear maps from $A$ to itself such that
the restriction to any graded piece $A_g$ is a vector space isomorphism. As Zhang notes prior to \cite[Definition
2.1]{zhang1998twisted}, $G$-graded $k$-linear automorphisms are not necessarily algebra automorphisms, although in
practice twists arising from algebra automorphisms are more commonly studied (see Example \ref{eg: autzhangtwist}). 

The main ingredient of a Zhang twist is the following.
\begin{defn}[{\cite[Definition 2.1]{zhang1998twisted}}]\label{defn: twistingsystem}
A set $\tau=\{\tau_g : g \in G \}$ of $G$-graded $k$-linear automorphisms of $A$ is
called a \emph{twisting system}\index{term}{twisting system} of $A$ if
\begin{equation}\label{eq: twistingsystem}
\tau_g(y\tau_h(z))=\tau_g(y)\tau_{gh}(z),
\end{equation}
for all $g,h,l \in G$ and all $y \in A_h$, $z \in A_l$.
\end{defn}

We proceed by giving an example of perhaps the simplest manner in which a nontrivial twisting system can arise. 
\begin{example}[{\cite[pg. 284]{zhang1998twisted}}]\label{eg: autzhangtwist}
Suppose that $A$ is a c.g.\ algebra and $f$ is an $\N$-graded algebra automorphism of $A$. Then $\{f^n : n \in \N \}$ is
a twisting system of $A$. This is an example of an \emph{algebraic} twisting system since it arises from a semigroup
homomorphism $(\N,+) \rightarrow \text{Aut}_{\N-\text{alg}}(A)$ where $n \mapsto f^n$. 
\end{example}

As for 2-cocycles, a twisting system can be used to define a new multiplication on the underlying $G$-graded vector
space structure of $A$. As the next proposition demonstrates, the condition in \eqref{eq: twistingsystem} means that
associativity is preserved by the new multiplication.
\begin{prop}[{\cite[Proposition and Definition 2.3]{zhang1998twisted}}]\label{prop: zhangtwist}
Let $A$ be a $G$-graded algebra and $\tau$ a twisting system. Then there is a new $G$-graded, associative
multiplication $\ast_{\tau}$ on the underlying $G$-graded $k$-vector space $A=\bigoplus_{g \in G} A_g$, defined by
\begin{equation*}%\label{eq: ztwistnewmult}
x \ast_{\tau} y := x \tau_{g}(y), 
\end{equation*}
for all $x \in A_g$ and $y \in A_h$. The element $1_{\tau}=\tau_{e}^{-1}(1)$ is the identity element with respect to
$\ast_{\tau}$. 
\end{prop}

The new algebra we have defined is called the \emph{Zhang twist of $A$ by $\tau$}\index{term}{Zhang twist} and is
denoted by $A^{G,\tau}$. The apparent conflict with this notation also being using in the previous section will be
explained shortly. By \cite[Proposition 2.4]{zhang1998twisted} we can assume without loss of generality that the
identity element in an algebra is preserved under a Zhang twist.

Given a $G$-graded right $A$-module there is a natural way to construct a $G$-graded right $A^{G,\tau}$-module from it. 
\begin{defn}[{\cite[Proposition and Definition 2.6]{zhang1998twisted}}]\label{defn: ztwistmodule} 
Let $A$ be a $G$-graded algebra and $\tau$ a twisting system with $A^{G,\tau}$ the associated Zhang twist. If $M=
\bigoplus_{g \in G} M_g$ is a $G$-graded right $A$-module then there is a graded right $A^{G,\tau}$-module structure on
the underlying $G$-graded $k$-vector space structure of $M$, defined by 
\begin{equation*}%\label{eq: ztwistmodule}
m \ast_{\tau} z:= m \tau_h(z), 
\end{equation*}
for all $m \in M_h$ and $z \in A_l$. The graded right $A^{G,\tau}$-module $(M,\ast_{\tau})$ is called a \emph{twist of
the module $M$ by $\tau$}\index{term}{twisted module}, and is denoted by $M^{\tau}$.
\end{defn}

Before stating some of the main results of Zhang's paper, we make the following remark. In the context of
noncommutative algebraic geometry the $G$-grading used when twisting usually comes from the semigroup $G=\N$, arising
from the underlying connected graded structure of the algebra. Although Zhang twists preserve several properties for
gradings by
general semigroups, when $G=\N$ stronger results can be proved. One such result is the following Morita-type theorem.
\begin{thm}[{\cite[Theorem 3.1]{zhang1998twisted}}]\label{thm: ztwistgmodequiv}
Let $A$ and $B$ be two $G$-graded algebras where $G$ is a semigroup. If $B$ is a Zhang twist of the $G$-grading on $A$
then the categories $\text{GrMod}_{G}(A)$ and $\text{GrMod}_{G}(B)$ are equivalent.
\end{thm}

When such algebras are connected graded and $A_1 \neq 0$, this result becomes an if and only if statement by
\cite[Theorem 3.5]{zhang1998twisted}.
 
We now illustrate how cocycle twists can be formulated as Zhang twists. 
\begin{thm}[{\cite[cf. Example 2.9]{zhang1998twisted}}]\label{thm: cocycleaszhang}
Cocycle twists as described in \S\ref{subsec: cocycletwists} can be formulated as Zhang twists.
\end{thm}
\begin{proof}
Let $A$ be a $G$-graded algebra and $\mu$ be a 2-cocycle. Define a twisting system $\tau=\{\tau_g : g \in G \}$ as
follows: for all $g,h \in G$ and $y \in A_h$, let
$\tau_{g}(y):=\mu(g,h)y$. Such maps are clearly $k$-linear and one can check that the cocycle condition \eqref{eq:
cocycleidinto1} implies that \eqref{eq: twistingsystem} is satisfied. We claim that there is an algebra isomorphism
between the
Zhang twist $A^{G,\tau}$ and the cocycle twist $A^{G,\mu}$. 

To prove this, consider the map which is the identity on the underlying vector space
of $A$. For homogeneous elements $x \in A_g$ and $y \in A_h$ we have
\begin{equation*}%\label{eq: cocycletwistaszhang}
x \ast_{\tau} y= x\tau_g(y)=\mu(g,h)xy=x \ast_{\mu} y. 
\end{equation*}
The multiplication in the two twists is the same and thus the map is an isomorphism. For cocycle twists of $G$-graded
algebras one therefore has an equivalence of categories of graded modules by Theorem \ref{thm: ztwistgmodequiv}. 
\end{proof}

\subsection{Crossed products}\label{sec: crossedproduct}
In Theorem \ref{thm: cocycleaszhang} it was shown that cocycle twists of a $G$-graded algebra can be expressed as Zhang
twists. In this short section we will define the notion of a crossed product and see that cocycle twists are related to
a special case of their construction. Viewing cocycle twists in this manner will be useful to us in certain
circumstances, allowing us to use results in the literature on crossed products.

The definition of a crossed product is as follows.
\begin{defn}[{\cite[\S 1.5.8]{mcconnell2001noncommutative}}]\label{def: crossedproduct}
Let $R$ be a ring and $G$ a group. Let $S \supset R$ be a ring containing a set of units $\overline{G}=\{\overline{g} :
g \in G\}$, isomorphic as a set to $G$. $S$ is said to be a \emph{crossed product of $R$ and $G$}\index{term}{crossed
product} if the following conditions hold:
\begin{itemize}
\item[(i)] $S$ is a free right $R$-module with basis $\overline{G}$ and $\overline{1}_G=1_S$;
\item[(ii)] for all $g_1,g_2 \in G$, $\overline{g}_{1} R=R\overline{g}_{1}$ and $\overline{g}_{1} \overline{g}_{2}
R=\overline{g_{1}g_{2}}R$.
\end{itemize}
Such a ring is often written $S=R \ast G$\index{notation}{r@$R \ast G$}.
\end{defn}

We will see in Chapter \ref{chap: cocycletwists} that the group ring $AG$ and the skew group ring $AG_{\mu}$ play a key
role in many proofs concerning the cocycle twist $A^{G,\mu}$. Such objects are a special case of crossed products; the
free basis indexed by $G$ is central in $AG_{\mu}$ rather than just normal as required by condition (ii) in Definition
\ref{def: crossedproduct}.

\section{Goldie theory}\label{sec: goldietheory}
\sectionmark{Goldie theory}
In this section we will describe the results from Goldie theory that will be needed in this thesis. 

Artin-Wedderburn theory shows that any semisimple artinian ring is isomorphic to a direct sum of matrix rings over
division rings \cite[Theorems 0.1.10 and 0.1.11]{mcconnell2001noncommutative}. Goldie theory answers the question of
when a ring has a semisimple artinian classical quotient ring. The classical quotient ring of a ring $R$ is the ring
obtained by inverting the set of non-zero regular elements in $R$. 
 
The rings to which Goldie's theory of noncommutative localisation apply are defined as follows.
\begin{defn}[{\cite[cf. pg. 115]{goodearl2004introduction}}]
A ring $R$ is \emph{Goldie}\index{term}{Goldie ring} if the following conditions hold for the modules $R_R$ and $_RR$:
\begin{itemize}
 \item[(i)] any direct sum of submodules is finite;
 \item[(ii)] any ascending chain of annihilators is also finite.
\end{itemize}
\end{defn}

A noetherian ring is easily seen to be Goldie, therefore the next result is applicable to the rings that we study in
this thesis. 
\begin{thm}[{Goldie's Theorem, \cite[Theorems 6.15 and 6.18]{goodearl2004introduction}}]\label{thm: goldie}
A ring $R$ has a semisimple artinian classical quotient ring if and only if it is semiprime Goldie. Moreover, $R$ has a
simple artinian classical quotient ring if and only if it is prime Goldie.
\end{thm}

The rings we study will often be graded, and there is a version of Goldie's theorem for such rings \cite[Theorem
C.I.1.6]{nastasescu1982graded}. We will only need the `prime' version of that result, stated as Theorem \ref{thm:
grgoldie} below. In the theorem the ring of fractions refers to inverting the non-zero homogeneous regular elements. Any
property that is prefixed by gr- has the same definition as in the ungraded case but only applied to graded modules. 
\begin{thm}[{Graded Goldie's Theorem, \cite[cf. Theorem 1]{goodearl2000graded}}]\label{thm: grgoldie}
Let $G$ be an abelian group and $R$ a $G$-graded, gr-prime, gr-Goldie ring. Then $R$ has a gr-simple, gr-artinian ring
of fractions, denoted by $Q_{\text{gr}}(R)$\index{notation}{q@$Q_{\text{gr}}(R)$}. We call $Q_{\text{gr}}(R)$ the
\emph{graded quotient ring}\index{term}{graded quotient ring} of $R$.
\end{thm}

We now describe such graded quotient rings explicitly when $G = \Z$, which also encompasses $\N$-graded algebras. By
\cite[Corollary A.I.4.3 and Theorem A.I.5.8]{nastasescu1982graded}, such a ring is isomorphic to a ring of the form
$M_n(D)[z,z^{-1};\sigma]$, where $D$ is a division ring, $\sigma$ is some automorphism of $D$ and $n \geq 1$. Thus the
graded quotient ring of a graded prime ring is a skew polynomial ring over a simple artinian ring. 

When $R$ is a domain, $Q_{\text{gr}}(R)=D[z,z^{-1};\sigma]$ for some division ring $D$. This division ring is called the
\emph{graded division ring}\index{term}{graded division ring} of $R$. One could choose another element in $Dz$ to be the
skew parameter, in which case the corresponding automorphism is obtained from the previous one by a conjugation map.
This is demonstrated by the following calculation. For all $d, d' \in D$ one has
\begin{equation*}%\label{eq: changeskewgen}
dzd'=d\sigma(d')z=(d\sigma(d')d^{-1})dz. 
\end{equation*}

Thus, if one takes $dz$ to be the new skew parameter then $\sigma$ is replaced by $c_{d} \circ \sigma$ in the skew
structure, where $c_d: D \rightarrow D$ is defined by $c_d(x)=dxd^{-1}$ for all $x \in D$. In our case $D$ will often be
a field, in which case the conjugation is trivial and one can change the skew parameter without changing the associated
automorphism.

%%%%%%%%%%%%%%%%%%%%%%%%%%%%%%%%%%%%%%%%%%%%%%%%%%%%%%%%%%%%%%%%%%%%%

\chapter{Cocycle twists of automorphism-induced $G$-gradings}\label{chap: cocycletwists}
\chaptermark{Cocycle twists}

In this chapter we will elaborate on Constructions 1 and 2 from \S\ref{subsec: theoryoftwist} and study them in more
detail. These two twisting constructions are shown to be the same in Proposition \ref{prop: twoconstrequal}. In
\S\ref{subsec: odesskiiegtwist} we then formulate Odesskii's example (Example \ref{ex: odesskii}) in terms of a cocycle
twist. This is followed by \S\ref{subsec: zhangtwistascocycle}, in which we show that some Zhang twists of $\N$-gradings
can be described as cocycle twists. To end \S\ref{sec: construction} we discuss the effect of group automorphisms and
choices of duality on twisting in \S\ref{subsec: twistggrading}.

In \S\ref{sec: preservation} we show that many properties are preserved under cocycle twists, for example AS-regularity
in Corollary \ref{cor: asreg}. Our main tool is Proposition \ref{prop: fflat}, which allows the use of faithful flatness
arguments.

The chapter concludes with \S\ref{sec: modules}, in which we study the interplay between 1-critical modules over an
algebra $A$ and a cocycle twist $A^{G,\mu}$. Under some hypotheses, our work demonstrates that point modules over $A$
can be used to construct fat point modules of multiplicity 2 over $A^{G,\mu}$ (see Proposition \ref{prop: fatpoints}).

\section{Constructions of the twists}\label{sec: construction}
\sectionmark{Construction}
We begin this section by fixing the base assumptions under which we will work.
\begin{hyp}[General case]\label{hyp: generalcase}
Let $A$ be a $k$-algebra where $k$ is an algebraically closed field. Assume that a finite abelian group $G$ acts on $A$
by algebra automorphisms, where $\text{char}(k) \nmid |G|$. Fix an isomorphism between $G$ and its group of characters
$G^{\vee}$, mapping $g \mapsto \chi_g$.  
\end{hyp}

Our primary interest in cocycle twists is to apply them to $\N$-graded algebras. As such, we record the following
additional assumptions that will be used when dealing with properties related to $\N$-graded algebras.
\begin{hyp}[$\N$-graded case]\label{hyp: gradedcase}
Further to Hypotheses \ref{hyp: generalcase}, assume that $A$ is $\N$-graded and $G$ acts on $A$ by $\N$-graded algebra
automorhisms, i.e. $G \rightarrow \text{Aut}_{\N\text{-alg}}(A)$.
\end{hyp}

\subsection{Two constructions}\label{subsec: twoconstruct}
The first construction is essentially Odesskii's twist from \cite{odesskii2002elliptic} in greater generality. Given a
2-cocycle $\mu \in Z^2(G,k^{\times})$ we can form the twisted group algebra $AG_{\mu}=A \otimes kG_{\mu}$, as defined in
Example \ref{eg: groupalg}. As noted in \S\ref{sec: crossedproduct}, this algebra has a crossed product structure.

We will define an action of $G$ on $kG_{\mu}$ by $g^{h}:=\chi_g(h)g$ for all $g, h \in G$ and extending $k$-linearly. To
see that this is indeed an action by algebra automorphisms, note that 
\begin{equation}\label{eq: Gactiongroupalgebra}
g^{hk}=\chi_g(hk)g=\chi_g(h)(\chi_g(k)g)=(g^k)^h, 
\end{equation}
and
\begin{equation*}
(g \ast_{\mu} h)^k = (\mu(g,h) gh)^k = \mu(g,h)\chi_{gh}(k)gh= \mu(g,h)\chi_{g}(k)\chi_{h}(k)gh = g^k \ast_{\mu} h^k.
\end{equation*}
In \eqref{eq: Gactiongroupalgebra} we have used the fact that $G$ is abelian. 

While this is not the obvious action of $G$ on the twisted group algebra, choosing it will simplify our work. Observe
that under this action $kG_{\mu}$ affords the regular representation of $G$, with isotypic components of the form
$\left(kG_{\mu}\right)^{\chi_{g}}=kg$. We then define a diagonal action of $G$ on the tensor product $A \otimes
kG_{\mu}$, where
\begin{equation}\label{eq: diagonalaction}
\left(\sum_i a_i \otimes g_i \right)^h=\sum_i a_i^h \otimes  g_i^h, 
\end{equation}
for all $a_i \in A$, $g_i, h \in G$. The algebra in which we are interested is the invariant ring under this action, $(A
\otimes kG_{\mu})^G=(AG_{\mu})^G$.

We now describe the second construction, which is a special case of \cite[\S 7.5.1]{montgomery1993hopf} obtained by
taking $H=kG$. We will avoid Hopf algebra terminology and give the construction in some detail.

Let $G$ and $A$ be as in Hypotheses \ref{hyp: generalcase}. We must first address how a $G$-grading is induced on $A$ by
the action of $G$. Maschke's Theorem \cite[Theorem 1.9]{isaacs1976character} tells us that since $A$ is a $kG$-module it
is completely reducible and therefore splits into a possibly infinite direct sum of irreducible submodules, $A =
\bigoplus_{i \in I}A_i$. Since $G$ is abelian, each submodule $A_i$ is 1-dimensional, thus $A_i=k a_i$ for some $a_i \in
A$. 

Having discussed these preliminaries we may now define the $G$-grading on $A$ that is induced by the action of $G$.
\begin{lemma}\label{lem: ggrading}
Define $A_g:=A^{\chi_{g^{-1}}}$ for all $g\in G$, where $A^{\chi_{g^{-1}}}$ is the isotypic component of $A$
corresponding to the character $\chi_{g^{-1}}$. Then $A= \bigoplus_{g \in G}A_g$ defines a $G$-grading on $A$.
\end{lemma}
\begin{proof}
Since $G$ acts by $k$-algebra automorphisms one clearly has $1 \in A_e$. To complete the proof, consider homogeneous
elements $a \in A_{g_{1}}, b \in A_{g_{2}}$ and $h \in G$, and calculate
\begin{equation*}%\label{eq: twistgrad}
(ab)^h=a^hb^h=\chi_{g_{1}^{-1}}(h)a \chi_{g_{2}^{-1}}(h)b=\chi_{(g_1g_2)^{-1}}(h)ab,
\end{equation*}
which implies that $ab \in A_{g_{1}g_{2}}$. Thus $A_{g_{1}} \cdot A_{g_{2}} \subset A_{g_{1} g_{2}}$ for all $g_1, g_2
\in G$.
\end{proof}
We denote the cocycle twist\index{term}{cocycle twist} of $A$ by $\mu$ under the induced grading by $A^{G,\mu}$, as in
\S\ref{subsec: cocycletwists}. 

The two constructions we have given are in fact the same, as the next proposition shows. We remark that this fact was
also noted more generally in \cite[\S 3.4]{bazlov2012cocycle} via the coaction of a group algebra on another algebra.
\begin{prop}[{cf. \cite[Lemma 3.6]{bazlov2012cocycle}}]\label{prop: twoconstrequal}
If the same isomorphism $G \cong G^{\vee}$ is used in both constructions, then $A^{G,\mu} \cong (AG_{\mu})^G$ as
$k$-algebras.
\end{prop}
\begin{proof}
Consider the $G$-grading on $A$ defined in Lemma \ref{lem: ggrading}. Since $A^{G,\mu}$ is a deformation of the
$G$-graded multiplication on $A$, it is clear that it has the same $G$-grading on the underlying vector space structure
it shares with $A$. We define a $k$-algebra homomorphism on $G$-homogeneous elements by
\begin{equation}\label{eq: isobetweentwoconstrs}
\phi: A^{G,\mu} \rightarrow (AG_{\mu})^G,\;\;\; a \in A_g \mapsto a \otimes g, 
\end{equation}
and extend by $k$-linearity. To see that this is well-defined, observe that for all $g,h \in G$ one has
\begin{equation*}%\label{eq: isotypicmult}
A^{\chi_{g}} \otimes \left(kG_{\mu}\right)^{\chi_{h}}=\left(A \otimes kG_{\mu} \right)^{\chi_{gh}}
\end{equation*}
under the diagonal action of $G$. Since in \eqref{eq: isobetweentwoconstrs} one has $a \in A_g=A^{\chi_{g^{-1}}}$ and $g
\in \left(kG_{\mu}\right)^{\chi_{g}}$, the element $a \otimes g $ is indeed invariant in $AG_{\mu}$.

We now check that this map is also a $k$-algebra homomorphism. It is enough to check for homogeneous elements with
respect to the $G$-grading, therefore suppose that $a \in A_g$ and $b \in A_h$. Then
\begin{equation*}%\label{eq: twoconstrssame}
\phi(a) \phi(b) = (a \otimes g)(b \otimes h) =ab \otimes (g \ast_{\mu} h) =
\mu(g,h) ab \otimes gh =\phi(\mu(g,h)ab) = \phi(a \ast_{\mu} b).
\end{equation*}
Define a map $\psi: (AG_{\mu})^G \rightarrow A^{G,\mu}$ by $\psi(a \otimes g)=a$ on pure tensors (where $a \in A_{g}$ by
definition of the grading). This shows that $\phi$ must be an isomorphism.
\end{proof}

We end this section with a remark.
\begin{rem}\label{rem: diagrelns}
The decomposition of $A$ into 1-dimensional $kG$-modules makes it clear that one can always choose a vector space basis
of $A$ such that $G$ acts diagonally on it. We will assume from now on that the algebra structure is defined in terms of
such a basis, and in particular that the relations of the algebra are written using it. The action of $G$ must preserve
the relations of $A$, therefore in a similar manner one can choose a basis for the relations on which $G$ acts
diagonally -- we will assume that this holds as well. 
\end{rem}

\subsection{Odesskii's example as a cocycle twist}\label{subsec: odesskiiegtwist}
Let us recall Example \ref{ex: odesskii}. For a specific choice of algebra, this example encompasses Odesskii's example
from \cite{odesskii2002elliptic}, and it is this situation on which we wish to focus.

Let $G$ be the Klein-four group $(C_2)^2= \langle g_1, g_2 \rangle$ and let $A$ be a 4-dimensional Sklyanin algebra over
$k = \C$. Consider the action of $G$ on $A$ by $\N$-graded algebra automorphisms as described on a set of algebra
generators for $A$ in \eqref{eq: odesskiiaction}. We note that the algebra generators appearing there are not those used
in Definition \ref{def: sklyanin relations}. 

To satisfy Hypotheses \ref{hyp: gradedcase} we must fix an isomorphism between $G$ and $G^{\vee}$. The character table
of $G$ in \eqref{eq: chartable} below is labelled to reflect such a choice. 
\begin{equation}\label{eq: chartable}
  \begin{array}{c|cccc}
 & e & g_1 & g_2 & g_1 g_2 \\
\hline
\chi_e & 1 & 1 & 1 & 1 \\
\chi_{g_{1}} & 1 & 1 & -1 & -1 \\
\chi_{g_{2}} & 1 & -1 & 1 & -1 \\
\chi_{g_{1}g_{2}} & 1 & -1 & -1 & 1 \end{array}
\end{equation}

The next lemma is a key result, not only in relation to Odesskii's example, but also for our work in \S\ref{sec:
modules}.
\begin{lemma}\label{lem: kgmuiso}
Let $k$ be an algebraically closed field with $\text{char}(k)\neq 2$. Consider the 2-cocycle $\mu$ defined by 
\begin{equation}\label{eq: mucocycledefn}
\mu(g_1^p g_2^q, g_1^r g_2^s) := (-1)^{ps}\; \text{ for all }p, q, r, s \in \{ 0 , 1 \},
\end{equation}
the action of $G$ on $kG_{\mu}$ defined in \S\ref{subsec: twoconstruct} and that on $M_2(k)$ described in \eqref{eq:
matrixaction}. Then there exists an isomorphism of $k$-algebras $\phi: kG_{\mu} \rightarrow M_2(k)$ such that 
\begin{equation}\label{eq: preserveisocomps}
\phi((kG_{\mu})^{\chi_{g}}) = M_2(k)^{\chi_{g}},
\end{equation}
for all $g \in G$.
\end{lemma}
\begin{proof}
Consider the map $\phi: kG_{\mu} \rightarrow M_2(k)$ which sends
\begin{equation}\label{eq: diagmat}
e \mapsto \twomat{1}{0}{0}{1},\; g_{1} \mapsto \twomat{1}{0}{0}{-1},\; g_{2} \mapsto \twomat{0}{1}{1}{0},\; g_{1}g_{2}
\mapsto \twomat{0}{-1}{1}{0}.
\end{equation}

The multiplication table for the standard vector space basis of $kG_{\mu}$ is
\begin{equation}\label{eq: multtable}
  \begin{array}{c|cccc}
kG_{\mu} & e & g_{1} & g_{2} & g_{1}g_{2} \\
\hline
e & e & g_{1} & g_{2} & g_{1}g_{2} \\
g_{1} & g_{1} & e & -g_{1}g_{2} & -g_{2} \\
g_{2} & g_{2} & g_{1}g_{2} & e & g_{1} \\
g_{1}g_{2} & g_{1}g_{2} & g_{2} & -g_{1} & -e \end{array}
\end{equation}
One can use \eqref{eq: multtable} to show that $\phi$ is a $k$-algebra homomorphism. We give an example of the
calculations needed:
\begin{equation*}%\label{eq: maptomatrices}
\phi(g_1 \ast_{\mu} g_2)= \phi(-g_1g_2)= \twomat{0}{1}{-1}{0} =\twomat{1}{0}{0}{-1}\twomat{0}{1}{1}{0} =
\phi(g_1)\phi(g_2).
\end{equation*}
The remaining verifications are similar and we omit them. Since the two algebras are 4-dimensional over $k$ and the
matrices in \eqref{eq: diagmat} are linearly independent, $\phi$ must be an isomorphism.

Recall the action of $G$ on $kG_{\mu}$, which was defined by $g^{h}:=\chi_g(h)g$ for all $g, h \in G$. Under this action
the isotypic component of $kG_{\mu}$ corresponding to $\chi_g$ is spanned by $g$. Thus, to complete the proof we must
check that $\phi(g)$ spans $M_2(k)^{\chi_{g}}$ for all $g \in G$. Once again, we omit some of the necessary calculations
but give an example to show how to fill in the details. 

To see that $\phi(g_1)$ belongs to $M_2(k)^{\chi_{g_{1}}}$, note that
\begin{align*}%\label{}
 \begin{pmatrix}
 1 & 0 \\
 0 & -1 
 \end{pmatrix}^{g_{1}} &=\begin{pmatrix}
 -1 & 0 \\
 0 & 1 
 \end{pmatrix} 
\begin{pmatrix}
1 & 0 \\
0 & -1 
 \end{pmatrix} \begin{pmatrix}
 -1 & 0 \\
 0 & 1 
 \end{pmatrix} = \begin{pmatrix}
 1 & 0 \\
 0 & -1 
 \end{pmatrix}= \phi(g_1), \\ 
 \begin{pmatrix}
  1 & 0 \\
  0 & -1 
 \end{pmatrix}^{g_{2}} &= \begin{pmatrix}
 0 & 1 \\
 1 & 0 
 \end{pmatrix} \begin{pmatrix}
 1 & 0 \\
 0 & -1 
 \end{pmatrix} \begin{pmatrix}
 0 & 1 \\
 1 & 0 
 \end{pmatrix} = \begin{pmatrix}
 -1 & 0 \\
 0 & 1 
\end{pmatrix}= -\phi(g_1).
\end{align*}
Thus $g_1$ fixes $\phi(g_1)$ while the other two non-identity elements of $G$ act by the scalar -1. But this implies
precisely what we wanted to show: $\phi(g_1)$ belongs to $M_2(k)^{\chi_{g_{1}}}$. Repeating such calculations for each
element of the group implies that $\phi$ respects the isotypic decompositions of $kG_{\mu}$ and $M_2(k)$.
\end{proof}
\begin{rem}\label{rem: cocycle}
The 2-cocycle $\mu$ used in Lemma \ref{lem: kgmuiso} is cohomologous to the one which will be used in Lemma \ref{subsec:
hbergcentext}. Although we did not prove that $\mu$ is a 2-cocycle, one can adapt the proof of Lemma \ref{subsec:
hbergcentext} easily.
\end{rem}

We can now prove the main result in this section.
\begin{prop}\label{prop: odesskiieg}
Consider Odesskii's example from \cite[Introduction]{odesskii2002elliptic}: $A$ is a 4-dimensional Sklyanin algebra for
which the Klein four-group $G$ acts on $M_2(A)$ by algebra automorphisms. The invariant ring under this action,
$M_2(A)^G$, can be expressed as a cocycle twist $A^{G,\mu}$ for some 2-cocycle $\mu$.
\end{prop}
\begin{proof}
As discussed at the beginning of \S\ref{subsec: odesskiiegtwist}, the example from
\cite[Introduction]{odesskii2002elliptic} is described in Example \ref{ex: odesskii} when $A$ is chosen to be a
4-dimensional Sklyanin algebra. Under this assumption, we wish to study the algebra $(A \otimes_{k} M_2(k))^G$, where
$G$ acts diagonally on the tensor product. 

By Lemma \ref{lem: kgmuiso}, $kG_{\mu} \cong M_2(k)$ as a $k$-algebra, and moreover this isomorphism respects the action
of $G$ on these algebras; the action on $kG_{\mu}$ is that defined in \S\ref{subsec: twoconstruct}, while the action on
$M_2(k)$ is defined in \eqref{eq: matrixaction}. The result follows from these observations upon considering the
invariant construction of cocycle twists from \S\ref{subsec: twoconstruct}.
\end{proof}

\subsection{Zhang twists as cocycle twists}\label{subsec: zhangtwistascocycle}
We recall the definition of Zhang twists given in Proposition \ref{prop: zhangtwist}. It will now be shown that --- as a
partial converse to Theorem \ref{thm: cocycleaszhang} --- some Zhang twists can be formulated as cocycle twists. 

To do this we require a 2-cocycle of $(C_n)^2$ for any $n\geq 2$, with such 2-cocycles being described in the following
lemma. When $n=2$ the 2-cocycle we construct is cohomologous to that appearing in Odesskii's example in \S\ref{subsec:
odesskiiegtwist}.
\begin{lemma}\label{subsec: hbergcentext}
Fix an integer $n \geq 2$ and an algebracially closed field $k$ of characteristic coprime to $n$. For $G=(C_n)^2=\langle
g_1,g_2 \rangle$ and a primitive $n$'th root of unity $\lambda \in k$, the function 
\begin{equation}\label{eq: heisenbergcocycle}
\mu(g_1^p g_2^q, g_1^r g_2^s)=\lambda^{qr} \; \text{ for all }p,q,r,s \in \{0,1,\ldots, n-1 \},
\end{equation} 
defines a 2-cocycle over $G$.
\end{lemma}
\begin{proof}
We must check the conditions in Definition \ref{def: 2cocycle}. First, note that $\mu$ certainly satisfies \eqref{eq:
cocycleidinto2}, thus we only give the necessary calculations to check that \eqref{eq:
cocycleidinto1} is satisfied. Consider elements $g,h,l \in G$ of the form $g=g_1^p g_2^q$, $h=g_1^r g_2^s$ and $l=g_1^t
g_2^u$ for some $p,q,r,s,t,u \in \{0,1,\ldots, n-1\}$. We verify that
\begin{align*}%\label{eq: }
 \mu(g,h)\mu(gh,l) &=\lambda^{qr}\mu(g_1^{p+r} g_2^{q+s},g_1^t g_2^u) = \lambda^{qr+qt+st},\\
\mu(h,l)\mu(g,hl) &=\lambda^{st}\mu(g_1^p g_2^q,g_1^{r+t} g_2^{s+u}) = \lambda^{qr+qt+st}.
\end{align*}
This means that $\mu$ satisfies \eqref{eq: cocycleidinto1} and therefore is a 2-cocycle.
\end{proof}

We also need the following lemma, which shows that if the action of $G$ on $A$ preserves a grading on $A$ unrelated to
the induced $G$-grading, then the twist $A^{G,\mu}$ will also possess the additional grading. This result will be useful
to us primarily when the extra grading is an \N-grading, although it will also be used in \S\ref{subsec:
rogzhangalgebras} in relation to $\Z^{2}$-gradings.
\begin{lemma}\label{lem: autpresgrad}
Suppose that $A$ has a $H$-grading for some group $H$ and that a finite abelian group $G$ acts on $A$ by $H$-graded
algebra automorphisms. Then any cocycle twist $A^{G,\mu}$ will admit the same $H$-grading as $A$ on their shared
underlying vector space structure.
\end{lemma}
\begin{proof}
We must show that for all $h_1,h_2 \in H$ and homogeneous elements $x \in A_{h_{1}}$ and $y \in A_{h_{2}}$ one has $x
\ast_{\mu} y \in A_{h_{1}h_{2}}$, since then $A^{G,\mu}_{h_{1}} \cdot A^{G,\mu}_{h_{2}} \subset A^{G,\mu}_{h_{1}h_{2}}$.
As $G$ acts on $A$ by $H$-graded algebra automorphisms, one can apply Maschke's theorem to the $H$-graded components of
$A$ under the action of $G$. This allows us to further assume that $x$ and $y$ are homogeneous with respect to the
$G$-grading, thus $x \in A_{g_{1}}$ and $y \in A_{g_{2}}$ for some $g_1,g_2 \in G$. Then
\begin{equation*}%\label{eq: preservegradingaut}
x \ast_{\mu} y = \mu(g_1,g_2)xy \in A_{h_{1}h_{2}},
\end{equation*}
which completes the proof.
\end{proof}
\begin{rem}\label{rem: asdfsdaf}
One interpretation of this lemma is that the action of $G$ induces a $(G,H)$-bigrading on $A$; the induced $G$-grading
is compatible with the pre-existing $H$-grading. To apply Lemma \ref{lem: autpresgrad} under Hypotheses \ref{hyp:
gradedcase}, one can consider an $\N$-grading as a $\Z$-grading whose negative components are all zero.
\end{rem}

We now assume until the end of \S\ref{subsec: zhangtwistascocycle} that $A$ is a c.g.\ and f.g.\ algebra, although not
necessarily in degree 1. 

In the following proposition --- which is the main result of \S\ref{subsec: zhangtwistascocycle} --- we consider an
algebraic twisting system as defined in Example \ref{eg: autzhangtwist}. 
\begin{prop}\label{prop: recoverztwist}
Let $\phi$ be an $\N$-graded algebra automorphism of $A$ which has finite order. Denote the associated twisting system
by $\tau=\{\tau_n := \phi^n :\;n \in \N \}$. Then there exists a finite abelian group $G$ and a 2-cocycle $\mu$ for
which:
\begin{itemize}
\item[(i)] the group $G$ acts on $A$ by $\N$-graded algebra automorphisms;
\item[(ii)] the Zhang twist $A^{\N,\tau}$ is isomorphic to $A^{G,\mu}$ as an $(\N,G)$-bigraded $k$-algebra.
\end{itemize}
\end{prop}
\begin{proof} 
Let $G=(C_n)^2=\langle g_1,g_2\rangle$ and $\mu$ be the 2-cocycle of $G$ from Lemma \ref{subsec: hbergcentext}. Our aim
is to write the Zhang twist as a cocycle twist involving this data. 

Fix a primitive $n$'th root of unity $\lambda \in k$. Let $g_1$ act on $A$ by $\phi$ and $g_2$ act on homogeneous
elements by scalar multiplication by $\lambda$, thus $a^{g_2}=\lambda^m a$ for all $a \in A_m$. This extends naturally
to an action on the whole of $A$. These two actions commute and combine to give a graded action of $G$ on $A$. Once we
have chosen an isomorphism $G \cong G^{\vee}$, this action will induce an $(\N,G)$-bigrading on $A$ by Lemma \ref{lem:
autpresgrad} and Remark \ref{rem: asdfsdaf}. For $0 \leq i,j \leq n-1$ we will let $\chi_{g_{1}^{i}g_{2}^{j}}$ denote
the character for which $\chi_{g_{1}^{i}g_{2}^{j}}(g_1^rg_2^s)=\lambda^{ir+js}$ for all $0 \leq r,s \leq n-1$. It is an
easy check to confirm that this defines an isomorphism $G \cong G^{\vee}$.

To prove the result it suffices to show that the multiplications in the two twists are the same on homogeneous elements
with respect to the $(\N,G)$-bigrading on $A$. Assume therefore that $a \in A_{m_1,g_{1}^{p}g_{2}^{q}}$ and $b \in
A_{m_2,g_{1}^{r}g_{2}^{s}}$ for some $0 \leq p,q,r,s \leq n-1$ and $m_1,m_2 \in \N$. One can assume that $q=n-m_1$ and
$s=n-m_2$ by definition of the action of $g_2$ and the chosen isomorphism $G \cong G^{\vee}$. Under the Zhang twist
multiplication one has 
\begin{equation*}%\label{eq: cocyclezhangmult}
a \ast_{\tau} b=a \tau_{m_{1}}(b)=a\phi^{n-q}(b)=a
b^{g_1^{n-q}}=\chi_{g_{1}^{n-r}g_{2}^{n-s}}(g_1^{n-q})ab=\lambda^{qr}ab.
\end{equation*}

On the other hand, in the cocycle twist we have the multiplication
\begin{equation*}%\label{eq: cocyclezhangmult1}
a \ast_{\mu} b = \mu(g_1^{p}g_2^{q},g_1^{r}g_2^{s}) ab=\lambda^{qr}ab.
\end{equation*}
Thus $A^{\N,\tau} \cong A^{G,\mu}$ as $(\N,G)$-bigraded $k$-algebras. 
\end{proof} 

Proposition \ref{prop: recoverztwist} says that $A^{G,\mu}$ is a Zhang twist of both the $\N$- and $G$-graded structures
on $A$. Thus by Theorem \ref{thm: ztwistgmodequiv} there are equivalences of categories 
\begin{equation*}%\label{eq: ztwistcatequivs}
\text{GrMod}(A^{G,\mu}) \simeq \text{GrMod}(A) \text{   and   }\text{GrMod}_{G}(A^{G,\mu}) \simeq
\text{GrMod}_{G}(A). 
\end{equation*}

Henceforth we will primarily use the notation $A^{G,\mu}$ for a cocycle twist. On occasions we will write $(AG_{\mu})^G$
for a cocycle twist if we wish to emphasise the invariant construction. Having two different constructions of the twist
will be advantageous, since in some places it will be preferable to use one instead of the other.

\subsection{Twisting the $G$-grading by an automorphism}\label{subsec: twistggrading}
In this section we illustrate how seemingly different actions can give rise to isomorphic twists.

We begin with a lemma which relates the action of a group automorphism on two different objects; on the graded
structure of an algebra and on a 2-cocycle over the group. Such actions are presumably well-known, but we have not been
able to find a reference.
\begin{lemma}\label{lem: autoncocycle}
Let $G$ be a finite abelian group and $A = \bigoplus_{g \in G} A_g$ a $G$-graded $k$-algebra. For a group automorphism
$\sigma$
one can define a new $G$-grading on $A$ by $A_{\sigma} := \bigoplus_{g \in G} B_g$, where $B_g=A_{\sigma(g)}$. Moreover,
for a 2-cocycle $\mu$ the cocycle twist $(A_{\sigma},\ast_{\mu})$ is isomorphic as a $k$-algebra to
$(A,\ast_{\mu^{\left(\sigma^{-1}\right)}})$, where
$\mu^{(\sigma^{-1})}$ is the 2-cocycle defined by
\begin{equation}\label{eq: autactoncocycle}
\mu^{(\sigma^{-1})}(g,h) := \mu(\sigma^{-1}(g),\sigma^{-1}(h))\index{notation}{m@$\mu^{(\sigma^{-1})}$}, 
\end{equation}
for all $g, h \in G$.
\end{lemma}
\begin{proof}
Since $\sigma$ is a group automorphism it follows that the decomposition $A_{\sigma}= \bigoplus_g B_g$ is a
$G$-grading and that $\mu^{(\sigma^{-1})}$ as defined in \eqref{eq: autactoncocycle} is a 2-cocycle of $G$.

In the twist $(A_{\sigma},\ast_{\mu})$ consider homogeneous elements $a \in B_g$ and $b \in
B_h$. Under the graded structure in $(A,\ast_{\mu^{\left(\sigma^{-1}\right)}})$ one has $a \in A_{\sigma(g)}$ and $b \in
A_{\sigma(h)}$. Writing the multiplication of $a$ and $b$ in $(A_{\sigma},\ast_{\mu})$ gives
\begin{equation}\label{eq: autactoncocyclemult}
a \ast_{\mu} b = \mu(g,h)ab = \mu^{(\sigma^{-1})}(\sigma(g),\sigma(h))ab.
\end{equation}
Notice that the right-hand side of \eqref{eq: autactoncocyclemult} is precisely the multiplication $a
\ast_{\mu^{\left(\sigma^{-1}\right)}} b$ in $(A,\ast_{\mu^{\left(\sigma^{-1}\right)}})$. This is sufficient to complete
the proof.
\end{proof}

We will use Lemma \ref{lem: autoncocycle} in combination with Proposition \ref{prop: trivialtwist} in \S\ref{subsec:
permuteaction}, where cocycle twists related to different gradings are shown to be isomorphic.
                           
We now show that the choice of isomorphism $G \cong G^{\vee}$ is not totally  benign, but that its consequences can be
explained using Lemma \ref{lem: autoncocycle}. Let us use the notation $(A,\phi,\mu)$ for a triple consisting of an
algebra, an isomorphism $G \rightarrow G^{\vee}$, and a 2-cocycle respectively. When $G$ acts on $A$ by algebra  
automorphisms each such triple can be naturally associated to a cocycle twist.
\begin{prop}\label{prop: benign}
Let $G$ act on $A$ by algebra automorphisms. Let $\phi$ and $\rho$ be isomorphisms $G \cong G^{\vee}$ and $\mu$ be a
2-cocycle. Then there exists an automorphism of $G$, $\tau$ say, such that the cocycle twists corresponding to the
triples $(A,\phi,\mu)$ and $(A,\rho,\mu^{(\tau^{-1})})$ are isomorphic as $k$-algebras.
\end{prop}
\begin{proof}
Given $\phi$, we will identify $\rho$ with an automorphism of $G$ as follows. Firstly, there exists an automorphism
$\psi: G^{\vee} \rightarrow G^{\vee}$ such that $\phi = \psi \circ \rho$. Suppose we have an element $x \in A_g$, where
the grading is determined under the duality given by $\phi$. This means that for all $h \in G$,
\begin{equation*}%\label{eq: dualgrading}
x^h = \phi(g)^{-1}(h)x=\psi(\rho(g))^{-1}(h)x.
\end{equation*}

Since all maps involved are isomorphisms, for all $g \in G$ there exists $k_g \in G$ such that
$\psi(\rho(g))=\rho(k_g)$. We claim that the map $\tau:\; g \mapsto k_g$ defines an isomorphism of $G$. To see this,
note that for all $g,h \in G$ one has
\begin{equation*}%\label{eq: }
\rho(k_{gh})=\psi(\rho(gh))=\psi(\rho(g)) \psi(\rho(h)) =\rho(k_{g})\rho(k_{h}).
\end{equation*}

As $\rho$ and $\psi$ and are isomorphisms, it follows that $\tau$ is a homomorphism and bijective as claimed. Under the
duality isomorphism $\rho$ one has $x \in A_{k_{g}}$, since 
\begin{equation}\label{eq: dualgrading1}
x^h =\psi(\rho(g))^{-1}(h)x=\rho(k_g)^{-1}(h)x,
\end{equation}
for all $h \in G$.

We are now in a position to show that for homogeneous elements the two twists give the same multiplication, from which
the result follows. Suppose that $x \in A_g$ and $y \in A_h$ for some $g,h \in G$ under the duality given by $\phi$.
Thus $x
\ast_{\mu} y=\mu(g,h)xy$ in $(A,\phi,\mu)$. Under the duality given by $\rho$ one has $x \in A_{k_{g}}$ by \eqref{eq:
dualgrading1}, and in a similar manner $y \in A_{k_{h}}$. Thus in $(A,\rho,\mu^{(\tau^{-1})})$ the multiplication is
\begin{equation*}%\label{eq: zhangcocylemult}
x \ast_{\mu^{\left(\tau^{-1}\right)}} y = \mu^{(\tau^{-1})}(k_g,k_h)xy=\mu(\tau^{-1}(k_g),\tau^{-1}(k_h))xy=\mu(g,h)xy, 
\end{equation*}
which completes the proof.
\end{proof}

%%%%%%%%%%%%%%%%%%%%%%%%%%%%%%%%%%%%%%%%%%%%%%%%%%%%%%%%%%%

\section{Preservation of properties}\label{sec: preservation}
\sectionmark{Preservation of properties}
In this section we prove that many properties are preserved by the twists defined in \S\ref{sec: construction}. A table
giving a summary of our results can be found in Appendix \ref{app: preserve}. We begin with some basic results and then
prove Proposition \ref{prop: fflat}, which allows us to use faithful flatness arguments; these will be our main tools
when dealing with more advanced properties such as AS-regularity in \S\ref{subsec: asregular}.

Most results will use either Hypotheses \ref{hyp: generalcase} or Hypotheses \ref{hyp: gradedcase}, with this being
made clear in the statement of each result. Under the former, more general hypotheses, the fact that a cocycle twist can
be formulated as a Zhang twist means that certain properties are preserved under twisting, as we saw in Proposition
\ref{prop: propspreservedalready}. However, the stronger results from \cite{zhang1998twisted} in this vein are only
valid when the grading being twisted is an $\N$-grading, and this only occurs as a degenerate case for us (see
Proposition \ref{prop: recoverztwist}). 

Our first result is a useful lemma regarding the behaviour of regular and normal elements under a twist. Although this
result is not stated explicitly in \cite{zhang1998twisted}, its proof is essentially contained in the proof of
Proposition 2.2(1) op. cit.. Nevertheless, since it will be used several times later in the thesis we state it in a
manner that will be most useful for our purposes and give a proof. 

Before stating it we make the following definition: an element $a \in A$ is \emph{right regular} if whenever $ab=0$ for
some $b \in A$, then $b=0$. There is an analogous definition for being left regular, with a \emph{regular element}
satisfying both conditions.
\begin{lemma}[{\cite[cf. Proposition 2.2(1)]{zhang1998twisted}}]\label{prop: stillregular}
Assume Hypotheses \ref{hyp: generalcase}. Any element $a \in A$ that is homogeneous with respect to the $G$-grading is
regular (normal) in $A$ if and only if it is regular (normal) in $A^{G,\mu}$.
\end{lemma}
\begin{proof}
Suppose that $a \in A_g$ is regular in $A$, but not in $A^{G,\mu}$. There must exist some $b \in A^{G,\mu}$ such that $a
\ast_{\mu} b=0$. We can further assume that $b$ is homogeneous with respect to the $G$-grading on $A$. However, this
implies that $\mu(g,h)ab = 0$ for some $h \in G$ by definition of the new multiplication, which contradicts the
regularity of $a$ in $A$. The proof of the other direction is identical.

An element $a \in A_g$ is normal if and only if for all $G$-homogeneous elements $b \in A_h$ we have $ba=ab'$ for some
$b' \in A_h$. The latter statement is equivalent to $b \ast_{\mu} a = \frac{\mu(h,g)}{\mu(g,h)} a \ast_{\mu} b'$ in the
twist, i.e. $a$ is normal with respect to the new multiplication also.
\end{proof}

We now prove a lemma concerning the behaviour of $G$-graded ideals under a cocycle twist.
\begin{lemma}\label{lem: defrelns}
Assume Hypotheses \ref{hyp: generalcase}. Let $I$ be a $G$-graded ideal of $A$, with a homogeneous generating set
$\{f_i\}_{i \in I}$. Then $I$ remains an ideal in $A^{G,\mu}$, and moreover it is still generated by $\{f_i\}_{i \in I}$
with respect to the new multiplication.
\end{lemma}
\begin{proof}
That $I$ is still an ideal in the twist is proved in \cite[Proposition 3.1(2)]{montgomery2005algebra}. To complete the
proof it suffices to deal with the case that $I=(f)$ for some homogeneous element $f \in A_g$. Suppose that $a \in A$
with homogeneous decomposition $a = \sum_{h \in G} a_h$. One has
\begin{equation*}
fa = f \ast_{\mu} \left(\sum_{h \in G} \frac{a_h}{\mu(g,h)} \right)\; \text{ and }\; af = \left(\sum_{h \in G}
\frac{a_h}{\mu(h,g)} \right) \ast_{\mu} f,
\end{equation*}
which proves the result.
\end{proof}
\begin{rem}\label{rem: defrelns}
It is clear from the proof of Lemma \ref{lem: defrelns} that the statement is also true for one-sided ideals.
\end{rem}

Our next result is related to a remark in \cite{odesskii2002elliptic} in which Odesskii asserts that the algebra in his
example --- which we studied in \S\ref{subsec: odesskiiegtwist} --- has the same Hilbert series as the 4-dimensional
Sklyanin algebra of which it is a twist. We state our result for general cocycle twists.
\begin{lemma}\label{lem: hilbseries}
Assume Hypotheses \ref{hyp: gradedcase}. Then the Hilbert series of $A$ is preserved under twisting, that is
$H_A(t)=H_{A^{G,\mu}}(t)$.
\end{lemma}
\begin{proof}
By Lemma \ref{lem: autpresgrad}, $A^{G,\mu}$ possesses the same $\N$-graded structure as $A$, thus the dimensions of the
graded components are the same for both.
\end{proof}

The following result of Montgomery is particularly useful in the graded context of our examples, although it holds under
more general hypotheses.
\begin{lemma}[{\cite[cf. Proposition 3.1(1)]{montgomery2005algebra}}]\label{lemma: finitelygenerated}
Assume Hypotheses \ref{hyp: generalcase}. Then $A$ is f.g.\ as a $k$-algebra if and only if $A^{G,\mu}$ is
too. 
\end{lemma}
\begin{rem}\label{rem: montfingenremark}
By consulting the proof of this result in \cite{montgomery2005algebra}, one can see that a generating set for
$A^{G,\mu}$ can be obtained as follows: take a generating set of $A$ and find a vector space $V$ which contains this
generating set and is preserved by the action of $G$. Then $A^{G,\mu}$ will be generated by $V$ under the new
multiplication on the shared underlying vector space. 

In most of our examples we will assume that, in addition to Hypotheses \ref{hyp: gradedcase}, $A$ is a c.g.\
algebra that is generated in degree 1. In that case we can conclude using Lemma \ref{lemma: finitelygenerated} that
$A^{G,\mu}$ is also generated in degree 1. A basis of $V$ which generates both $A$ and $A^{G,\mu}$ as algebras (under
their respective multiplications) and on which $G$ acts diagonally will be called a \emph{diagonal
basis}\index{term}{diagonal basis}. 
\end{rem}

In the next proposition --- which will be crucial for our subsequent results --- it will be useful to consider the
invariant construction of the twist. While its consequences regarding faithful flatness are used in many results, its
information on bimodule structures is particularly useful in relation to global dimension (see Proposition \ref{prop:
gldim}). 

Before stating the result, we recall the concept of twisting a module by an automorphism. Let $A$ be an algebra and
$\phi$ be an algebra automorphism. For a right $A$-module $M$, one can define a new right $A$-module $M^{\phi}$ through
the multiplication $m \ast_{\phi} a=m\phi(a)$ for all $a \in A$, $m \in M$. This can be recovered from Definition
\ref{defn: ztwistmodule} by using a trivial grading. 

One can twist both sides of an $(A,A)$-bimodule in this manner simultaneously. Suppose that $A$ is graded and
$_{A}M_{A}$ is a graded bimodule that is free of rank 1 on both sides. By considering the action of homogeneous elements
on the left free generator $m$, we can see that $a\cdot m=m \cdot \phi(a)$ for all $a \in A_n$, where $\phi$ is some
$\N$-graded algebra automorphism. Such a bimodule therefore has the form ${^{\text{id}}}A^{\phi}$, that is, an
$(A,A)$-bimodule that is twisted on one side.
\begin{prop}\label{prop: fflat}
Assume Hypotheses \ref{hyp: generalcase}. As an $(A^{G,\mu},A^{G,\mu})$-bimodule there is a decomposition
\begin{equation}\label{eq: decompbimod}
AG_{\mu} \cong \bigoplus_{g \in G} {^{\text{id}}(A^{G,\mu})^{\phi_{g}}},
\end{equation}
for some automorphisms $\phi_{g}$ of $A^{G,\mu}$, with $\phi_e=\text{id}$. Each summand is free of rank 1 as a left and
right $A^{G,\mu}$-module. Consequently, $AG_{\mu}$ is a faithfully flat extension of $A^{G,\mu}$ on both the left and
the right. Similarly, $_A(AG_{\mu})$ and $(AG_{\mu})_A$ are free modules of finite rank, thus $AG_{\mu}$ is a faithfully
flat extension of $A$ on both the left and the right. 
\end{prop}
\begin{proof}
We will proceed as in the proof of the main theorem of \cite{smith1989can}. Let $AG_{\mu}=\bigoplus_{g \in G}
M^{\chi_{g}}$ be the isotypic decomposition\index{term}{isotypic decomposition} of $AG_{\mu}$ under the action
of $G$. Observe that $A^{G,\mu}=M^{\chi_{e}}$ and $M^{\chi_{g}}M^{\chi_{h}}=M^{\chi_{gh}}$ for all $g,h \in G$, since
$G$ acts by algebra automorphisms. This means that each isotypic component $M^{\chi_{g}}$ has an
$(A^{G,\mu},A^{G,\mu})$-bimodule
structure. 

The isotypic component $M^{\chi_{g}}$ contains the element $1 \otimes g$. An arbitrary element in this component is a
sum of pure tensors of the form $a \otimes h$ for some $a \in A_{g^{-1}h}=A^{\chi_{gh^{-1}}}$. Thus $a \otimes g^{-1}h
\in A^{G,\mu}$ and therefore
\begin{equation*}%\label{eq: freedecomp}
a \otimes h= \left(\frac{a \otimes g^{-1}h}{\mu(g^{-1}h,g)}\right) \cdot (1 \otimes g)=(1 \otimes g) \cdot \left(\frac{a
\otimes g^{-1}h}{\mu(g,g^{-1}h)}\right).
\end{equation*}
Thus $M^{\chi_{g}}$ is cyclic as a left or a right $A^{G,\mu}$-module. Note that $1 \otimes g$ is regular in $AG$,
therefore by Lemma \ref{prop: stillregular} it is also regular in $AG_{\mu}$. This proves that $M^{\chi_{g}}$ is a free
$A^{G,\mu}$-module of rank 1 on both the left and the right.

By the discussion prior to the statement of the proposition, we know that the bimodule generated by $1 \otimes g$ is
isomorphic to ${^{\text{id}}}(A^{G,\mu})^{\phi_{g}}$ for some algebra automorphism $\phi_{g}$. To describe $\phi_g$ it
suffices to look at the left action of a homogeneous element in $A^{G,\mu}$ on $1 \otimes g$, which can be taken to be a
free generator for the left $A^{G,\mu}$-module structure. Consider a homogeneous element $a \otimes h \in A^{G,\mu}_h$.
One has
\begin{equation*}%\label{eq: twistbimod}
(a \otimes h) \cdot (1 \otimes g) = \mu(h,g) a \otimes hg = (1 \otimes g) \cdot  \frac{\mu(h,g)}{\mu(g,h)}(a \otimes h).
\end{equation*}

Define a map $\phi_g: A^{G,\mu} \rightarrow A^{G,\mu}$ by $a \otimes h \mapsto \frac{\mu(h,g)}{\mu(g,h)}(a \otimes h)$
on homogeneous elements and extending $k$-linearly. To see that this is a $G$-graded automorphism, consider homogeneous
elements $a \otimes h \in A^{G,\mu}_h$ and $b \otimes l \in A^{G,\mu}_l$. Then
\begin{equation}\label{eq: phighom}
\phi_g(a \otimes h)\phi_g(b \otimes l) = \frac{\mu(h,g)\mu(l,g)\mu(h,l)}{\mu(g,h)\mu(g,l)}(ab\otimes hl).
\end{equation}

On the other hand one can use \eqref{eq: cocycleidinto1} to see that
\begin{equation}\label{eq: phighom1}
\phi_g(\mu(h,l)(ab\otimes hl)) = \frac{\mu(hl,g)\mu(h,l)}{\mu(g,hl)} (ab\otimes hl) =
\frac{\mu(h,lg)\mu(l,g)\mu(h,l)}{\mu(g,h)\mu(gh,l)} (ab\otimes hl).
\end{equation}
Observe that $\frac{\mu(h,lg)}{\mu(gh,l)} = \frac{\mu(h,g)}{\mu(g,l)}$, which follows from $G$ being abelian together
with another use of \eqref{eq: cocycleidinto1}. Substituting this expression into \eqref{eq: phighom1} produces the
expression in \eqref{eq: phighom}. It is clear that $\phi_g$ is injective, therefore it must be a $G$-graded
automorphism of $A^{G,\mu}$ as claimed.

The result is trivial for $A$ by the definition of $AG_{\mu}$.
\end{proof}

We now state the definition of a property which generalises that of being strongly noetherian, which was defined in
Definition \ref{defn: strongnoeth}. 
\begin{defn}[{\cite[\S 4]{artin1999generic}}]
Let $A$ be a noetherian $R$-algebra. $A$ is \emph{universally noetherian}\index{term}{strongly noetherian} if for any
$R$-algebras $R'$ the following conditions are satisfied:
\begin{itemize}
 \item[(i)] if $R'$ is right noetherian then so is $A \otimes_R R'$;
 \item[(ii)] if $R'$ is left noetherian then so is $A \otimes_R R'$. 
\end{itemize}
\end{defn}

One can replace the word universal in the following corollary with strongly --- or indeed omit it completely --- without
changing the veracity of the statement. It is a generalisation (in some sense) of \cite[Proposition
3.1(3)]{montgomery2005algebra}, where the preservation of being noetherian is proved in the Hopf algebra setting (see
\cite[Proposition 5.1]{zhang1998twisted} also).  
\begin{cor}\label{cor: uninoeth}\index{term}{strongly noetherian!preservation of}
Assume Hypotheses \ref{hyp: generalcase}. Then $A$ is universally noetherian if and only if $A^{G,\mu}$ is.
\end{cor}
\begin{proof}
We make no distinction between left and right since the proof is identical. Assume that $A$ is universally noetherian.
Then $AG_{\mu}$ is a f.g.\ $A$-module by the proof of Proposition \ref{prop: fflat}, hence by \cite[Proposition
4.1(1a)]{artin1999generic} $AG_{\mu}$ is also universally noetherian. Using Proposition \ref{prop: fflat} again, the
extension $A^{G,\mu} \hookrightarrow AG_{\mu}$ is faithfully flat on both sides, therefore we can apply
\cite[Proposition 4.1(2a)]{artin1999generic} to reach the desired conclusion.

In the other direction we can use the same argument but with $AG_{\mu}$ replaced with $A^{G,\mu}G_{\mu^{-1}}$;
Proposition \ref{prop: fflat} tells us that $A^{G,\mu}G_{\mu^{-1}}$ is a faithfully flat extension of both $A^{G,\mu}$
and of $A$.
\end{proof}

\subsection{AS-regularity}\label{subsec: asregular}
The aim of this section is to prove that the AS-regular property is preserved under cocycle twists. As was discussed in
\S\ref{sec: noncommgeom}, this property is very important in relation to a geometric theory of noncommutative algebras.
Before giving its definition we must first define the AS-Gorenstein property. 
\begin{defn}[{\cite[\S 0]{artin1987graded}}]\label{defn: asgor}
An algebra $A$ is said to be \emph{AS-Gorenstein}\index{term}{AS-Gorenstein} (of global dimension $d$) if it is
connected graded and satisfies the condition 
\begin{equation*}%\label{eq: asgorensteinreln}
\text{Ext}^i_A(k,A)= \left\{ \begin{array}{cl} k & \text{if }i=d, \\ 0 & \text{if }i \neq d, \end{array}\right. 
\end{equation*}
when $k$ and $A$ are considered as left $\N$-graded $A$-modules.
\end{defn}

One can calculate this particular Ext group in either $\text{Mod}(A)$ or in $\text{GrMod}(A)$; since $A$ and $k$ are
f.g.\ modules, the discussion in \cite[\S 1.4]{levasseur1992some} shows that the two Ext groups in question are the
same. 

We can now give the definition of an AS-regular algebra.
\begin{defn}[{\cite[\S 0]{artin1987graded}}]\label{defn: asregular}
Let $d \in \mathbb{N}$. A connected graded $k$-algebra $A$ is said to be \emph{AS-regular of dimension
$d$}\index{term}{AS-regular} if the following conditions are satisfied:
\begin{itemize}
 \item[(i)] $A$ has finite GK dimension;
 \item[(ii)] $\text{gldim }A=d$;
 \item[(iii)] $A$ is AS-Gorenstein.
\end{itemize}
\end{defn}
\begin{rem}\label{rem: fgldim}
Note that condition (ii) means that the left and right global dimensions of $A$ agree and equal $d$ (see Definition
\ref{def: homdimfunc}). The phrase \emph{`$A$ has finite global dimension'}\index{term}{finite global dimension} will
therefore be used to encapsulate this. 
\end{rem}

Under Hypotheses \ref{hyp: gradedcase}, Proposition \ref{lem: hilbseries} implies that $A$ and $A^{G,\mu}$ have the same
Hilbert series. In particular, this means that their GK dimensions are equal. Whilst we will mostly be concerned with
$\N$-graded algebras in this section, we can prove that condition (i) of Definition \ref{defn: asregular} is preserved
in the ungraded case, where one does not have recourse to Hilbert series arguments.
\begin{prop}\label{prop: gkdim}
Under Hypotheses \ref{hyp: generalcase}, $\text{GKdim }A= \text{GKdim }A^{G,\mu}$. 
\end{prop}
\begin{proof}
By Proposition \ref{prop: fflat}, $AG_{\mu}$ is a f.g.\ module over $A$ and $A^{G,\mu}$ on both sides. Applying
\cite[Proposition 5.5]{krause2000growth} twice, first to $A \subset AG_{\mu}$, then $A^{G,\mu} \subset AG_{\mu}$, proves
the result.
\end{proof}

Let us move on to global dimension. For the purposes of this section, we only need to show that finite global dimension
is preserved under cocycle twists. However, we will prove the more general result that left and right global dimension
are preserved, regardless of whether they are equal or not. The algebras we will consider are all noetherian, in which
case left and right global dimensions coincide \cite[\S 7.1.11]{mcconnell2001noncommutative}.

We will need the following technical result to compare the global dimension of $A^{G,\mu}$ with that of the twisted
group algebra $AG_{\mu}$. There is an analogous version for left modules.
\begin{thm}[{\cite[Theorem 7.2.8(i)]{mcconnell2001noncommutative}}]\label{thm: mcrobtechnical}
Let $R,S$ be rings with $R \subset S$ such that $R$ is an $(R,R)$-bimodule direct summand of $S$. Then
\begin{equation*}%\label{eq: gldimbimodstr}
\text{rgldim }R \leq \text{rgldim }S +\text{pdim }S_R. 
\end{equation*}
\end{thm}

Without further ado, we show that left and right global dimension are preserved under cocycle twists.
\begin{prop}\label{prop: gldim}
Assume Hypotheses \ref{hyp: generalcase}. One has $\text{rgldim }A=\text{rgldim }A^{G,\mu}$ and $\text{lgldim
}A=\text{lgldim }A^{G,\mu}$.
\end{prop}
\begin{proof}
We will give the proof for right global dimension, from which a left-sided proof can easily be derived. Recall from
\S\ref{sec: crossedproduct} that $AG_{\mu}$ has the structure of a crossed product. We can therefore apply \cite[Theorem
7.5.6(iii)]{mcconnell2001noncommutative} to conclude that $\text{rgldim }A=\text{rgldim }AG_{\mu}$. By Proposition
\ref{prop: fflat} we know that $A^{G,\mu}$ is an $(A^{G,\mu},A^{G,\mu})$-bimodule direct summand of $AG_{\mu}$. We may
therefore apply Theorem \ref{thm: mcrobtechnical}, which tells us that
\begin{equation}\label{eq: gldiminequality}
\text{rgldim }A^{G,\mu} \leq \text{rgldim }AG_{\mu} + \text{pdim }(AG_{\mu})_{A^{G,\mu}}= \text{rgldim }AG_{\mu}.
\end{equation}

Here $\text{pdim }(AG_{\mu})_{A^{G,\mu}}=0$ since the module is free by Proposition \ref{prop: fflat}. We have proved
that $\text{rgldim }A^{G,\mu} \leq \text{rgldim }A$.

Now we repeat the argument with the roles of $A$ and $A^{G,\mu}$ reversed, considering them as subrings of
$A^{G,\mu}G_{\mu^{-1}}$. One obtains the opposite inequality, which proves the result.
\end{proof}
We have addressed two of the conditions necessary to be AS-regular without any assumption regarding an \N-grading.
However, we must now assume that $A$ is c.g.\ and that $G$ acts by $\N$-graded algebra automorphisms.

Our main tool for criterion (iii) in Definition \ref{defn: asregular} will be a result of Brown and Levasseur for left
modules. We also state a version of their result for right modules which can be proved in an analogous manner.
\begin{prop}[{\cite[cf. Proposition 1.6]{brown1985cohomology}}]\label{prop: brownlevass}
Let $R$ and $S$ be rings and $R \rightarrow S$ a ring homomorphism such that $S$ is flat as a left and right
$R$-module. 
\begin{itemize}
\item[(i)]Let $X$ be an $(R,R)$-bimodule such that the $(R,S)$-bimodule $X \otimes_R S$ is an
$(S,S)$-bimodule. Then for every f.g.\ left $R$-module $M$ and all $i \geq 0$, there are isomorphisms
of right $S$-modules,
\begin{equation*}%\label{eq: exttensorreln}
\text{Ext}_R^i(M,X)\otimes_R S \cong \text{Ext}_S^i(S \otimes_R M,X \otimes_R S). 
\end{equation*}
\item[(ii)] Similarly, suppose that $X$ is an $(R,R)$-bimodule such that the $(S,R)$-bimodule $S \otimes_R X$ is an
$(S,S)$-bimodule. Then for every f.g.\ right $R$-module $M$ and all $i \geq 0$, there are isomorphisms
of left $S$-modules,
\begin{equation*}%\label{eq: exttensorreln}
S \otimes_R \text{Ext}_R^i(M,X) \cong \text{Ext}_S^i(M \otimes_R S,S \otimes_R X). 
\end{equation*} 
\end{itemize}
\end{prop}
\begin{rem}\label{rem: x=s}
For some of our applications of Proposition \ref{prop: brownlevass} we will take $R=X$. In that case $X \otimes_R S
\cong S$ as an $(R,S)$-bimodule, from which $X \otimes_R S$ inherits a natural $(S,S)$-bimodule structure. 
\end{rem}

Before stating the result regarding the AS-Gorenstein property, we need a preliminary lemma. 
\begin{lemma}\label{lem: tensor}
In addition to Hypotheses \ref{hyp: gradedcase}, assume that $A$ is a c.g.\ algebra. As right $AG_{\mu}$-modules we
have isomorphisms
\begin{equation*}%\label{eq: degree1factor}
k \otimes_A AG_{\mu} \cong k \otimes_{A^{G,\mu}} AG_{\mu} \cong kG_{\mu}. 
\end{equation*}
As left $AG_{\mu}$-modules we have isomorphisms
\begin{equation*}%\label{eq: degree1factor1}
AG_{\mu} \otimes_A k \cong AG_{\mu} \otimes_{A^{G,\mu}} k \cong kG_{\mu}.
\end{equation*}
\end{lemma}
\begin{proof}
Under the $\N$-grading on $AG_{\mu}$ one has $A_{\geq 1} \cdot AG_{\mu} = (AG_{\mu})_{\geq 1}$. Thus
\begin{gather}
\begin{align*}%\label{eq: degree1factorwork}
k \otimes_A AG_{\mu}  \cong \frac{A}{A_{\geq 1}} \otimes_A AG_{\mu}  \cong \frac{AG_{\mu}}{A_{\geq 1} \cdot AG_{\mu}} =
\frac{AG_{\mu}}{(AG_{\mu})_{\geq 1}} \cong kG_{\mu}.
\end{align*}
\end{gather}

To complete the proof we must first show that the following equality holds:
\begin{equation}\label{eq: genindegreeonecocycletwist}
(A^{G,\mu})_{\geq 1} \cdot AG_{\mu} = (AG_{\mu})_{\geq 1}.
\end{equation}
Clearly $(A^{G,\mu})_{\geq 1} \cdot AG_{\mu} \subseteq (AG_{\mu})_{\geq 1}$. To prove the opposite inclusion, observe
that
the action of $G$ on $AG_{\mu}$ in \eqref{eq: diagonalaction} respects its $\N$-graded structure under our hypotheses.
Remark \ref{rem: asdfsdaf} then tells us that $AG_{\mu}$ is endowed with an $(\N,G)$-bigrading. By linearity it
therefore suffices to consider a bihomogeneous element $a \otimes g \in (AG_{\mu})_{(n,h)}$ for some $g,h \in G$ and
integer $n \geq 1$. Then $a \otimes gh^{-1} \in (AG_{\mu})_{(n,e)} \subset A^{G,\mu}_{\geq 1}$, and thus 
\begin{equation*}
\frac{1}{\mu(gh^{-1},h)}(a \otimes gh^{-1})(1 \otimes h) = a \otimes g,
\end{equation*}
which proves that $(A^{G,\mu})_{\geq 1} \cdot AG_{\mu} \supseteq (AG_{\mu})_{\geq 1}$. One can then use \eqref{eq:
genindegreeonecocycletwist} to see that 
\begin{gather}
\begin{align*}%\label{eq: degree1factorworkmore}
k \otimes_{A^{G,\mu}} AG_{\mu}  \cong \frac{A^{G,\mu}}{(A^{G,\mu})_{\geq 1}} \otimes_{A^{G,\mu}} AG_{\mu}  \cong
\frac{AG_{\mu}}{(A^{G,\mu})_{\geq 1}\cdot AG_{\mu}} = \frac{AG_{\mu}}{(AG_{\mu})_{\geq 1}} \cong kG_{\mu}.
\end{align*}
\end{gather}

We omit the proof of the statement for left $AG_{\mu}$-modules.
\end{proof}

This lemma can be interpreted as saying that under the $(\N,G)$-bigrading on $AG_{\mu}$, the subalgebra consisting of
elements that have degree zero is the twisted group algebra $kG_{\mu}$.
\begin{prop}\label{prop: asgor}
Assume in addition to Hypotheses \ref{hyp: gradedcase} that $A$ is a c.g.\ algebra. Then $A$ is AS-Gorenstein of
global dimension $d$ if and only if $A^{G,\mu}$ shares this property.
\end{prop}

We will give the proof in the only if direction when $k$ and $A$ are considered as left $A$-modules. The proof in the
opposite direction is identical by untwisting, while the proof for right modules is almost identical to that below; the
only difference is that it requires the use of Proposition \ref{prop: brownlevass}(ii) rather than part (i) of that
result.
\begin{proofof}{Proposition \ref{prop: asgor}}
We wish to apply Proposition \ref{prop: brownlevass}(i) with $R=X=A$, $S=AG_{\mu}$ and $M=k$. To see that the hypotheses
of that result are satisfied, observe that $A \subset AG_{\mu}$ is flat by Proposition \ref{prop: fflat} and recall
Remark \ref{rem: x=s}. Applying Proposition \ref{prop: brownlevass}(i) gives
\begin{equation}\label{eq: asgorfirststep}
\text{Ext}^i_A(k,A) \otimes_A AG_{\mu} \cong \text{Ext}^i_{AG_{\mu}}(AG_{\mu} \otimes_A k,A \otimes_A AG_{\mu}) \cong 
\text{Ext}^i_{AG_{\mu}}(kG_{\mu},AG_{\mu}),
\end{equation}
by using Lemma \ref{lem: tensor}. Since $A$ is AS-Gorenstein of global dimension $d$ we know that the left hand side is
non-zero only for $i=d$. For this value of $i$ it is equal to $k \otimes_A AG_{\mu} \cong kG_{\mu}$ using Lemma
\ref{lem: tensor} once again. 

We would now like to apply Proposition \ref{prop: brownlevass}(i) a second time, using $R=X=A^{G,\mu}$, $S=AG_{\mu}$ and
$M=k$. To see that these data satisfy the hypotheses of that proposition we may apply the same argument as used earlier
in the proof, mutatis mutandis. Applying Proposition \ref{prop: brownlevass}(i) we obtain
\begin{gather}
\begin{aligned}\label{eq: asgorsecondstep}
\text{Ext}^i_{A^{G,\mu}}(k,A^{G,\mu}) \otimes_{A^{G,\mu}} AG_{\mu} & \cong  \text{Ext}^i_{AG_{\mu}}(AG_{\mu}
\otimes_{A^{G,\mu}} k, A^{G,\mu} \otimes_{A^{G,\mu}} AG_{\mu}) \\
   & \cong  \text{Ext}^i_{AG_{\mu}}(kG_{\mu}, AG_{\mu}).
\end{aligned}
\end{gather}
Combining the information from \eqref{eq: asgorfirststep} and \eqref{eq: asgorsecondstep} gives
\begin{equation*}%\label{eq: asgorthirdstep}
\text{Ext}^i_{A^{G,\mu}}(k,A^{G,\mu}) \otimes_{A^{G,\mu}} AG_{\mu} \cong \left\{ \begin{array}{cl} kG_{\mu} & \text{if
}i=d, \\ 0 & \text{if }i \neq
d. \end{array}\right. 
\end{equation*}

Since $A^{G,\mu} \subset AG_{\mu}$ is a faithfully flat extension on the left, $\text{Ext}^i_{A^{G,\mu}}(k,A^{G,\mu})$
must vanish in all degrees for which $i \neq d$. When $i=d$ we have
\begin{equation}\label{eq: asgorfinalstep}
\text{Ext}^d_{A^{G,\mu}}(k,A^{G,\mu}) \otimes_{A^{G,\mu}} AG_{\mu} \cong kG_{\mu}, 
\end{equation}
as right $AG_{\mu}$-modules. 

Recall from \S\ref{subsec: notation} that $\text{Ext}^i_{A^{G,\mu}}(k,A^{G,\mu})$ is a $\Z$-graded group since $k$ and
$A^{G,\mu}$ are f.g.\ $\N$-graded left $A^{G,\mu}$-modules. This $\Z$-grading is compatible with the right
$A^{G,\mu}$-module structure, in which case the graded module structure on $\text{Ext}^i_{A^{G,\mu}}(k,A^{G,\mu})$
allows us to complete the proof as follows. One may use the $(A^{G,\mu},A^{G,\mu})$-bimodule structure of $AG_{\mu}$
described in Proposition \ref{prop: fflat} to see that upon restricting the isomorphism in \eqref{eq: asgorfinalstep} to
$A^{G,\mu}$, one obtains
\begin{equation*}%\label{eq: asgorfinalstep}
\bigoplus_{g\in G} \text{Ext}^d_{A^{G,\mu}}(k,A^{G,\mu})^{\phi_{g}} \cong (kG_{\mu})_{A^{G,\mu}}.
\end{equation*}
By considering the $G$-graded components of this isomorphism and noting that $\phi_e$ is the identity, one obtains the
isomorphism of right $A^{G,\mu}$-modules $\text{Ext}^d_{A^{G,\mu}}(k,A^{G,\mu}) \cong k$, which proves the result.
\end{proofof}
\begin{cor}\label{cor: asreg}\index{term}{AS-regular!preservation of}
Assume in addition to Hypotheses \ref{hyp: gradedcase} that $A$ is a c.g.\ algebra. Then $A$ is AS-regular if and only
if $A^{G,\mu}$ is. Moreover, if $A$ has global and GK dimension less than or equal to 4, then $A$ is a domain if and
only if $A^{G,\mu}$ is a domain.
\end{cor}
\begin{proof}
The statement about AS-regularity follows from Lemma \ref{lem: hilbseries} and Propositions \ref{prop: gldim} and
\ref{prop: asgor}. The second part of the corollary follows from \cite[Theorem 3.9]{artin1991modules}. 
\end{proof}
\begin{rem}\label{rem: domain}
Remark \ref{rem: bgmunilp} will illustrate that being a domain is \emph{not} preserved in general by cocycle twists.
This should be expected since cocycle twists are Zhang twists by Theorem \ref{thm: cocycleaszhang}; such twists
preserve being a domain when $G$ is an ordered semigroup but not in general (see \cite[\S 5]{zhang1998twisted}
and Proposition 5.2 op. cit. in particular). One should contrast this with Lemma \ref{prop: stillregular}, which shows
that regular homogeneous elements remain regular under twisting.
\end{rem}

\subsection{The Koszul property}\label{subsec: koszul}
Our next result will be to show that the \emph{Koszul} property is preserved under cocycle twists. This
property is often studied in relation to quadratic algebras as defined in Definition \ref{def: quadalg}, although it can
be defined for more general algebras too. Whilst we will cite \cite{koszul1996beilinson} for any definitions, we suggest
\cite{krahmernotes} as a good introduction to the subject.

Recall from \S\ref{subsec: notation} that for $M, N \in \text{grmod}(A)$ the cohomology group $\text{Ext}_{A}^i(M,N)$ is
$\Z$-graded. We use this idea to define a Koszul algebra.
\begin{defn}[{\cite[cf. Proposition 2.1.3]{koszul1996beilinson}}]\label{defn: koszulcomplex} 
A c.g.\ $k$-algebra is \emph{Koszul}\index{term}{Koszul property} if and only if for all $i \geq 0$ the $\Z$-graded
components of $\text{Ext}_A^i(k,k)$ vanish in all degrees other than degree $i$.
\end{defn}

\begin{prop}\label{prop: koszul}
In addition to Hypotheses \ref{hyp: gradedcase} assume that $A$ is a quadratic algebra over $k$. Then $A$ is Koszul if
and only $A^{G,\mu}$ is.
\end{prop}
\begin{proof}
We wish to apply Proposition \ref{prop: brownlevass}(i) with $R=A$, $S=AG_{\mu}$, $X={_A}k_A$ and $M={_A}k$. To do so we
must check that the hypotheses of that result are satisfied. Observe that $A \subset AG_{\mu}$ is flat by Proposition
\ref{prop: fflat}, while $X \otimes_R S = kG_\mu$ by Lemma \ref{lem: tensor}, whence it has a natural
$(AG_{\mu},AG_{\mu})$-bimodule structure. We may therefore apply Proposition \ref{prop: brownlevass}(i), in which case
one has
\begin{gather}
\begin{aligned}\label{eq: koszulbrown}
\text{Ext}_A^i(k,k)\otimes_A AG_{\mu} &\cong \text{Ext}_{AG_{\mu}}^i(AG_{\mu} \otimes_A k,k \otimes_A AG_{\mu}) \\
 &\cong \text{Ext}_{AG_{\mu}}^i(kG_{\mu},kG_{\mu}). 
\end{aligned}
\end{gather}
Note that Lemma \ref{lem: tensor} has been used to pass from the first line of \eqref{eq: koszulbrown} to the second.

Now set $R= A^{G,\mu}$, $S=AG_{\mu}$, $X={_{A^{G,\mu}}}k_{A^{G,\mu}}$ and $M={_{A^{G,\mu}}}k$. One can use the same
argument as earlier in the proof, mutatis mutandis, to see that these data also satisfy the hypotheses of Proposition
\ref{prop: brownlevass}(i). Applying that result we obtain 
\begin{gather}
\begin{aligned}\label{eq: koszulbrown1}
\text{Ext}_{A^{G,\mu}}^i(k,k)\otimes_{A^{G,\mu}} AG_{\mu} &\cong \text{Ext}_{AG_{\mu}}^i(AG_{\mu} \otimes_{A^{G,\mu}}
k,k \otimes_{A^{G,\mu}} AG_{\mu}) \\
 &\cong \text{Ext}_{AG_{\mu}}^i(kG_{\mu},kG_{\mu}), 
\end{aligned}
\end{gather}
where we have used Lemma \ref{lem: tensor} once again. 

The $\Z$-grading on $\text{Ext}_A^i(k,k)$ and $\text{Ext}_{A^{G,\mu}}^i(k,k)$ is compatible with their right $A$- and
$A^{G,\mu}$-module structures respectively. Thus the tensor products $\text{Ext}_A^i(k,k)\otimes_A AG_{\mu}$ and
$\text{Ext}_{A^{G,\mu}}^i(k,k)\otimes_{A^{G,\mu}} AG_{\mu}$ are naturally $\Z$-graded right $AG_{\mu}$-modules. The
$\Z$-grading on the cohomology group $\text{Ext}_{AG_{\mu}}^i(kG_{\mu},kG_{\mu})$ is also compatible with its right
$AG_{\mu}$-module structure. Moreover, one can see from the proof of Proposition \ref{prop: brownlevass}(i) (from
\cite[Proposition 1.6]{brown1985cohomology}) that the isomorphisms in \eqref{eq: koszulbrown} and \eqref{eq:
koszulbrown1} respect these $\Z$-graded structures. We may therefore conclude that there is an isomorphism 
\begin{equation}\label{eq: comparedims}
\text{Ext}_A^i(k,k)\otimes_A AG_{\mu} \cong \text{Ext}_{A^{G,\mu}}^i(k,k)\otimes_{A^{G,\mu}} AG_{\mu} 
\end{equation}
of $\Z$-graded right $AG_{\mu}$-modules. 

Using the free module structures of $_A(AG_{\mu})$ and  $_{A^{G,\mu}}(AG_{\mu})$ described in Proposition \ref{prop:
fflat}, we may express the isomorphism in \eqref{eq: comparedims} as
\begin{equation}\label{eq: comparedims1}
\bigoplus_{|G|}\text{Ext}_A^i(k,k) \cong \bigoplus_{|G|}\text{Ext}_{A^{G,\mu}}^i(k,k), 
\end{equation}
at the level of vector spaces. Furthermore, since $AG_{\mu}$ is an $\N$-graded left module over $A$ and over
$A^{G,\mu}$, the isomorphism in \eqref{eq: comparedims1} respects the $\Z$-graded structure. 

Since $A$ is Koszul, we know that the $\Z$-graded components of the left hand side of \eqref{eq: comparedims1} vanish in
all degrees other than degree $i$. It follows that $\text{Ext}_{A^{G,\mu}}^i(k,k)$ must also vanish in all degrees other
than degree $i$, hence $A^{G,\mu}$ must be Koszul.
\end{proof}

Let us now define the Koszul dual of a c.g.\ algebra with quadratic relations.
\begin{defn}[{\cite[cf. Definition 2.8.1]{koszul1996beilinson}}]\label{def: kdual}
Let $A=T(V)/(R)$ be a quadratic algebra over $k$. The \emph{Koszul dual}\index{term}{Koszul
dual} of $A$ is the $k$-algebra $A^!:=T(V^{\ast})/(R^{\perp})$\index{notation}{a@$A^{\exclaim}$}. Here $R^{\perp}
\subset V^{\ast} \otimes V^{\ast}$ is the space of functions which vanish on the quadratic relations of $A$.
\end{defn}

If $x \in V$ then we will denote by $\overline{x}$ the element in $V^{\ast}$ that vanishes away from the 1-dimensional
space spanned by $x$ and for which $\overline{x}(x)=1$. 

Our next aim is to show that taking the Koszul dual of an algebra and applying a cocycle twist almost commute with each
other. To achieve this aim we first need to define a graded action of $G$ on the Koszul dual given a graded action on
$A$. Such an
action can be defined on generators by $\overline{x}^g = \overline{x^g}$ for all $g \in G$ and $\overline{x} \in
V^{\ast}$. This action induces a $G$-grading on $A^{!}$ in the manner of \S\ref{subsec: twoconstruct}; we claim that if
$x \in A_g$ then $\overline{x} \in A^{!}_g$. Indeed, observe that for all $h \in G$ we have $\overline{x}^h=
\overline{x^{h}}=\chi_{g^{-1}}(h)\overline{x}$.

We are now in a position to prove the result.
\begin{prop}\label{prop: koszuldualtwistcommute}
In addition to Hypotheses \ref{hyp: gradedcase}, assume that $A$ is a quadratic algebra over $k$ that is Koszul. Then
there is an
isomorphism of $k$-algebras
\begin{equation}\label{eq: kdualcocyclecommute}
(A^{!})^{G,\mu^{-1}} \cong (A^{G,\mu})^{!}.
\end{equation}
\end{prop}
\begin{proof}
We first fix our notation. Let $V=\text{span}_k(x_0,\ldots,x_{n})$ be a diagonal basis of generators of $A$, where $x_i
\in A_{g_{i}}$ for some $g_i \in G$. We will denote the generators of the other algebras involved by 
\begin{equation*}%\label{eq: kdualgenerators}
(A^!,\overline{x_{i}}),\;\; ((A^!)^{G,\mu^{-1}},y_i),\;\; (A^{G,\mu},v_i),\;\text{ and }\;
((A^{G,\mu})^!,\overline{v_{i}}). 
\end{equation*}

Although twists have the same underlying vector space structure, the new generators allow us to write the twisted
multiplication by juxtaposition. Thus
\begin{equation}\label{eq: twistedmultx2}
v_iv_j =x_i \ast_{\mu} x_j=\mu(g_i,g_j)x_i x_j\;\text{ and }\; y_iy_j =\overline{x_i} \ast_{\mu^{-1}}
\overline{x_j}=\frac{1}{\mu(g_i,g_j)} \overline{x_i}\overline{x_j}.
\end{equation}

Consider the map $\phi: (A^{!})^{G,\mu^{-1}} \rightarrow (A^{G,\mu})^{!}$ sending $y_i \mapsto \overline{v_i}$. Note
that 
\begin{equation*}%\label
\overline{v_i}\overline{v_j} (v_kv_l) = \left\{
  \begin{array}{l l}
    1 & \quad \text{if $(k,l)=(i,j)$},\\
    0 & \quad \text{otherwise}.
  \end{array} \right. 
\end{equation*}
One can regard elements of both $(A^{!})^{G,\mu^{-1}}$ and $(A^{G,\mu})^{!}$ as functions on the underlying vector
space
of $A$ by untangling the twists involved. Using \eqref{eq: twistedmultx2} we can interpret
$\overline{v_i}\overline{v_j}$ as follows:
\begin{equation*}%\label
\overline{v_i}\overline{v_j} (x_k x_l) = \left\{
  \begin{array}{l l}
    \frac{1}{\mu(g_{i},g_{j})} & \quad \text{if $(k,l)=(i,j)$},\\
    0 & \quad \text{otherwise}.
  \end{array} \right. 
\end{equation*}

By \eqref{eq: twistedmultx2} one can also conclude that
$\overline{v_i}\overline{v_j}=\frac{1}{\mu(g_{i},g_{j})}\overline{x_i}\overline{x_j}=y_i y_j$. This expression implies
that the relations in both $(A^{!})^{G,\mu^{-1}}$ and $(A^{G,\mu})^{!}$ are controlled by those in $A^{!}$, therefore
$\phi$ is both well-defined and a surjection. 

Since $A$ is Koszul, the Hilbert series of $A^!$ depends only on the Hilbert series of $A$; the precise relationship
between their Hilbert series was first given in \cite[\S 4, Theorem]{froberg1975determination}. By Proposition
\ref{prop: koszul} the same is true of the Hilbert series of $(A^{G,\mu})^{!}$ in relation to that of $A^{G,\mu}$. In
combination with Lemma
\ref{lem: hilbseries}, this shows that the two algebras in \eqref{eq: kdualcocyclecommute} must have the same Hilbert
series. Thus the map we have defined is an isomorphism, which completes the proof.
\end{proof}

For an application of this proposition, see \S\ref{par: kdual}.

\subsection{The Cohen-Macaulay property and Auslander regularity}\label{subsec: cohenmac} 
In this section we will prove that several more homological properties of rings are preserved under cocycle twists.
Their definitions can be stated in the graded situation, however we state them --- and prove their preservation --- in
full generality.

The definitions that follow can all be found in \cite[\S 1.2]{levasseur1993modules}. The first concept that we need is
the notion of the grade of a module. The \emph{grade}\index{term}{grade of a module} of a f.g.\ left or right $A$-module
$M$ is defined to be the value
\begin{equation*}%\label{eq: cmgradedefn}
j_A(M)=\text{inf}\{i: \text{Ext}_A^i(M,A)\neq 0\} \in \N \cup \{+\infty\}. 
\end{equation*}
\begin{defn}\label{def: cm}
A ring $A$ is said to satisfy the \emph{Cohen-Macaulay property} or be \emph{Cohen-Macaulay}\index{term}{Cohen-Macaulay,
(CM)} if for all non-zero f.g.\ $A$-modules $M$, one has
\begin{equation*}%\label{eq: cohenmacprop}
\text{GKdim }M+j_A(M)=\text{GKdim }A.
\end{equation*} 
\end{defn}

The \emph{Auslander condition}\index{term}{Auslander condition} also uses the notion of the grade of a module; it is
satisfied by ring $A$ if for every f.g.\ left or right module $M$, all $i \geq 0$ and every $A$-submodule $N$ of
$\text{Ext}^i_A(M,A)$, one has $j_A(N)\geq i$.
\begin{defn}\label{def: auslanderprops}
A ring $A$ is said to be \emph{Auslander-Gorenstein}\index{term}{Auslander-Gorenstein} if in addition to satisfying the
Auslander condition it has finite left and right injective dimension. It is said to be \emph{Auslander
regular}\index{term}{Auslander regular} if in addition to satisfying the Auslander condition it has finite global
dimension. 
\end{defn}

Before proceeding, recall that Remark \ref{rem: fgldim} applies to our usage of the phrase finite global dimension.

The following result shows that the properties defined in Definitions \ref{def: cm} and \ref{def: auslanderprops} are
preserved under a cocycle twist.
\begin{prop}\label{prop: cohenmac}\index{term}{Auslander regular!preservation of}
In addition to Hypotheses \ref{hyp: generalcase}, assume that $A$ is noetherian. Then $A$ has one of the following
properties if and only if $A^{G,\mu}$ does as well:
\begin{itemize}
 \item[(i)] it is Cohen-Macaulay;
 \item[(ii)] it is Auslander-Gorenstein;
 \item[(iii)] it is Auslander regular.
\end{itemize}
\end{prop}
\begin{proof}
(i) Assume that $A$ is Cohen-Macaulay. We first show that $AG_{\mu}$ shares this property. As we saw in Proposition
\ref{prop: gkdim}, $\text{GKdim }A=\text{GKdim }AG_{\mu}=\text{GKdim }A^{G,\mu}$. Let $M$ be a f.g\ right
$AG_{\mu}$-module. It must also be f.g.\ as an $A$-module since the extension $A \subset AG_{\mu}$ is
finite by Proposition \ref{prop: fflat}. By \cite[Lemma 5.4]{ardakov2007primeness} it is clear
that the grades of $M_{AG_{\mu}}$ and $M_A$ are equal. One can then apply \cite[Lemma
1.6]{lorenz1988on} to conclude that $\text{GKdim }M_{AG_{\mu}}=\text{GKdim }M_{A}$. Piecing this together, we find that
\begin{gather}
\begin{aligned}\label{eq: cohenaagmu}
\text{GKdim }M_{AG_{\mu}}+ j_{AG_{\mu}}(M)= \text{GKdim }M_A+j_A(M) &=\text{GKdim }A \\ &=\text{GKdim }AG_{\mu},
\end{aligned}
\end{gather}
and therefore $AG_{\mu}$ is Cohen-Macaulay. 

Now let $M$ be a f.g.\ right $A^{G,\mu}$-module. By applying Proposition \ref{prop: brownlevass}(ii) with
$R=X=A^{G,\mu}$ and $S=AG_{\mu}$ we obtain
\begin{equation}\label{eq: brownlevasscohenmacappl}
AG_{\mu} \otimes_{A^{G,\mu}} \text{Ext}_{A^{G,\mu}}^i \left(M,A^{G,\mu}\right) \cong \text{Ext}_{AG_{\mu}}^i\left(M
\otimes_{A^{G,\mu}} AG_{\mu},AG_{\mu}\right).
\end{equation}

When combined with faithful flatness of the extension $A^{G,\mu} \subset AG_{\mu}$ (by Proposition \ref{prop: fflat}),
this implies that
\begin{equation*}%\label{eq: gradereln}
j_{A^{G,\mu}}(M)=j_{AG_{\mu}}(M \otimes_{A^{G,\mu}} AG_{\mu}). 
\end{equation*}

If we can show that $\text{GKdim }M_{A^{G,\mu}}=\text{GKdim }(M \otimes_{A^{G,\mu}} AG_{\mu})_{AG_{\mu}}$ then the
result follows from an equality like that in \eqref{eq: cohenaagmu}. By faithful flatness of the extension $A^{G,\mu}
\subset AG_{\mu}$, $M$ is contained in $M \otimes_{A^{G,\mu}} AG_{\mu}$. Therefore by the definition of GK dimension one
has 
\begin{equation}\label{eq: gkineq1}
\text{GKdim }M_{A^{G,\mu}} \leq \text{GKdim }(M \otimes_{A^{G,\mu}} AG_{\mu})_{A^{G,\mu}}. 
\end{equation}
By \cite[Proposition 5.6]{krause2000growth} one has the inequality 
\begin{equation}\label{eq: gkineq2}
\text{GKdim }M_{A^{G,\mu}} \geq \text{GKdim }(M \otimes_{A^{G,\mu}} AG_{\mu})_{AG_{\mu}}.
\end{equation}

Applying \cite[Lemma 1.6]{lorenz1988on} to $M \otimes_{A^{G,\mu}} AG_{\mu}$ allows one to combine the inequalities in
\eqref{eq: gkineq1} and \eqref{eq: gkineq2}, from which the preservation of the Cohen-Macaulay property under cocycle
twists follows.

(ii) Using \cite[Proposition 3.9(i)]{yi1995injective} one can see that if $A$ satisfies the Auslander condition then so
must $AG_{\mu}$. The twist $A^{G,\mu}$ then satisfies the Auslander condition by \cite[Theorem
2.2(iv)]{teo1996homological}, since the only hypothesis needed is that the extension be flat -- this is true by
Proposition \ref{prop: fflat}. 

It remains to show that finite left and right injective dimension are preserved. Consider the $G$-grading on $AG_{\mu}$
for which $(AG_{\mu})_g = A \otimes g$ for all $g \in G$. Under this grading $AG_{\mu}$ is a strongly $G$-graded ring,
thus one can apply \cite[Corollary 2.7]{nastasescu1983strongly} with $R = N = AG_{\mu}$ and $\sigma = e$. That result
implies that
\begin{equation*}%\label{eq: leftinjdim}
\text{idim }_{AG_{\mu}}AG_{\mu} = \text{idim }_{(AG_{\mu})_{e}}(AG_{\mu})_{e}= \text{idim }_{A}A.
\end{equation*}
A right-sided analogue of N\u{a}st\u{a}sescu's result shows that the same is also true for right injective dimension.

Now consider the $G$-grading on $AG_{\mu}$ under which $(AG_{\mu})_g = A^{G,\mu}(1 \otimes g)$ for all $g \in G$. This
$G$-grading is induced by the diagonal action of $G$ on $AG_{\mu}$. It is clear that $AG_{\mu}$ is a strongly $G$-graded
ring under this grading as well. One can therefore apply \cite[Corollary 2.7]{nastasescu1983strongly} with this new
$G$-grading on $R = N = AG_{\mu}$, together with $\sigma = e$, to see that 
\begin{equation*}%\label{eq: leftinjdim}
\text{idim }_{AG_{\mu}}AG_{\mu} = \text{idim }_{(AG_{\mu})_{e}}(AG_{\mu})_{e}= \text{idim }_{A^{G,\mu}}A^{G,\mu}.
\end{equation*}
An analogous argument, used in conjunction with a right-sided analogue of the corollary of N\u{a}st\u{a}sescu, proves
that $AG_{\mu}$ and $A^{G,\mu}$ also have equal right injective dimensions. This completes the proof.

(iii) We saw in the proof of (ii) that the Auslander condition is preserved. One can then see from Proposition
\ref{prop: gldim} that the global dimensions of $A$ and $A^{G,\mu}$ are equal, which completes the proof.
\end{proof}

%%%%%%%%%%%%%%%%%%%%%%%%%%%%%%%%%%%%%%%%%%%%%%%%%%%%%%%%%%%

\section{Modules under twisting}\label{sec: modules}
\sectionmark{Modules under twisting}
In this section we explore the interplay between modules over $A$ and those over a cocycle twist $A^{G,\mu}$. Our
ultimate aim is to apply the results we obtain to the algebras studied in Chapter \ref{chap: sklyanin}. Let us define
the hypotheses that we will work under for the duration of this section.
\begin{hyp}\label{hyp: genhypforfatpts}
Let $k$ be an algebraically closed with $\text{char}(k)\neq 2$. Assume that $A$ is a $k$-algebra that satisfies the
hypotheses of Theorem \ref{thm: pointschemenice}, with degree 1 generators $x_0, x_1 ,x_2$ and $x_3$. Let
$G=\langle g_1,g_2 \rangle$ be the Klein four-group, which acts on $A$ by $\N$-graded algebra automorphisms.
Furthermore, assume that
the action of $G$ on $A_1$ affords the regular representation, inducing the $G$-grading on generators:
\begin{equation}\label{eq: gengrading}
x_0 \in A_{e},\; x_1 \in A_{g_{1}},\; x_2 \in A_{g_{2}},\; x_3 \in A_{g_{1}g_{2}}.
\end{equation}
Finally, assume that $\mu$ is the 2-cocycle of $G$ defined by $\mu(g_1^p g_2^q, g_1^r
g_2^s) = (-1)^{ps}$ for all $p, q, r, s \in \{ 0 , 1\}$. 
\end{hyp}

Recall that $kG_{\mu}\cong M_2(k)$ for the group $G$ and 2-cocycle $\mu$ by Lemma \ref{lem: kgmuiso}. Under Hypotheses
\ref{hyp: genhypforfatpts} one therefore has a cocycle twist $A^{G,\mu}$ that embeds inside $M_2(A)\cong A \otimes
M_2(k)$. The following result describes the generators of the twist under this matrix embedding.
\begin{lemma}\label{lem: matrixgens}
The degree 1 generators of $A^{G,\mu}$, denoted by $v_0, v_1, v_2$ and $v_3$, are given by the following matrices in
$M_2(A)$:
\begin{align}\label{eq: matrixembedding}
v_0 = \begin{pmatrix} x_0 & 0 \\ 0 & x_0 \end{pmatrix},\; v_1 = \begin{pmatrix} x_1 & 0 \\ 0 &
-x_1 \end{pmatrix},\; v_2 = \begin{pmatrix} 0 & x_2 \\ x_2 & 0 \end{pmatrix}, \;\;
v_3 = \begin{pmatrix} 0 & -x_3 \\ x_3 & 0 \end{pmatrix}.
\end{align}
\end{lemma}
\begin{proof}
Consider the embedding $A^{G,\mu} \hookrightarrow A\otimes kG_{\mu}$ from Proposition \ref{prop: twoconstrequal}. A
homogeneous element $x \in A_{g}$ is sent to $x \otimes g$. An isomorphism $kG_{\mu}\cong M_2(k)$ was described in
\eqref{eq: diagmat}. By Lemma \ref{lem: kgmuiso} this isomorphism respects the isotypic components of each algebra under
their respective $G$-actions. One can pair up the generators $x_i$ with matrices in \eqref{eq: diagmat} using the
grading in \eqref{eq: gengrading}. The result follows upon using the isomorphism $M_2(A)\cong A \otimes M_2(k)$.
\end{proof}

Let us now introduce some notation for point modules over $A$. Under Hypotheses \ref{hyp: genhypforfatpts} one may apply
Theorem \ref{thm: pointschemenice} to conclude that point modules over $A$ are parameterised by the point scheme $\Gamma
\subset \proj{k}{3}$. Furthermore, the shifting operation on point modules is controlled by a scheme automorphism
$\sigma$. 

We will denote the point module corresponding to a point $p=(p_0,p_1,p_2,p_3) \in \Gamma$ by $M_p= \bigoplus_{j \in \N}
km_j^p$. For the action of the generators of $A$ on $M_p$ we will use the notation $m_j^p \cdot x_i= \alpha_{j,i}^p
m_{j+1}^p$, where $\alpha_{j,i}^p \in k$. By standard point module theory one has $M_p[j]_{\geq 0} \cong
M_{\sigma^j(p)}$ for all $j \in \N$. The scalar $\alpha_{j,i}^p$ is therefore determined by the $(i+1)$'th coordinate of
$\sigma^j(p)$, thus $\alpha_{0,i}^p=p_i$ for $i=0,1,2,3$ in particular.

For a point $p \in \Gamma$ it is clear that $M_p^2$ is an $\N$-graded right $M_2(A)$-module with Hilbert series
$2/(1-t)$, whose action is given by matrix multiplication on the right. It is natural to restrict such a module to the
subring
$A^{G,\mu}$, whose generators act by the matrices in \eqref{eq: matrixembedding}. We can prove the following proposition
about such modules.
\begin{prop}\label{prop: fatpoints} 
Suppose that at least three coordinates of $p \in \Gamma$ are non-zero. Then $M_p^2$ is a fat point module over
$A^{G,\mu}$ of multiplicity 2.
\end{prop}
\begin{proof}
We must show that $M_p^2$ is 1-critical and generated in degree 0; let us proceed by proving the latter statement. Let
$N$ denote the submodule generated by $(M_p^2)_0$. We prove by induction that $(M_p^2)_j \subset N$ for all $j \geq 0$,
where the base case $j=0$ is clear. Suppose that $(M_p^2)_j \subset N$ for some $j \geq 0$. Since $M_p[j]_{\geq 0} \cong
M_{\sigma^j(p)}$, the action of the generators of $A$ on $m_j^p$ is given by the coordinates of $\sigma^j(p)$. At least
one generator does not annihilate $m_j^p$, in which case letting the corresponding generator of $A^{G,\mu}$ act on
$(m_j^p,0)$ and $(0,m_j^p)$ shows that $(M_p^2)_{j+1} \subset N$. By induction, $N=M_p^2$ and so $M_p^2$ is generated in
degree 0.

To prove that $M_p^2$ is 1-critical it is sufficient to show that any cyclic $\N$-graded submodule has finite
codimension. Consider the submodule generated by an element $(m_j^p,\lambda m_j^p) \in M_p^2$, where $\lambda \in
k^{\times}$.  We will show that $(m_{j+1}^p,0)$ and $(0,m_{j+1}^p)$ belong to the submodule. Since $M_p^2[j+1]_{\geq 0}
\cong M_{\sigma^{j+1}(p)}^2$, the argument of the previous paragraph implies that these two elements generate $M_p^2$ in
high degree, in which case the submodule generated by $(m_j^p,\lambda m_j^p)$ must have finite codimension.

By assumption $\alpha_{j,i}^p \neq 0$ for at least three of the generators. This means that either
$\alpha_{j,0}^p,\alpha_{j,1}^p \neq 0$ or $\alpha_{j,2}^p,\alpha_{j,3}^p \neq 0$. In the former case $x_0$ and $x_1$ do
not annihilate $m_j^p$ and thus
%\begin{gather}
\begin{align*}%\label{eq: finitecodim}
(m_j^p,\lambda m_j^p) \cdot \left(v_0+ \frac{\alpha_{j,0}^p}{\alpha_{j,1}^p} v_1\right) &=
(2\alpha_{j,0}^p m_{j+1}^p,0), \\
(m_j^p, \lambda m_j^p) \cdot \left(v_0 - \frac{\alpha_{j,0}^p}{\alpha_{j,1}^p} v_1 \right) &=
(0,2 \lambda \alpha_{j,0}^p m_{j+1}^p).
\end{align*}
%\end{gather}

On the other hand, if $\alpha_{j,2},\alpha_{j,3} \neq 0$ then $x_2$ and $x_3$ do not annihilate $m_j^p$ and so
%\begin{gather}
\begin{align*}%\label{eq: finitecodim1}
(m_j^p,\lambda m_j^p) \cdot \left(v_2 + \frac{\alpha_{j,2}^p}{\alpha_{j,3}^p} v_3 \right) &=
(2 \lambda \alpha_{j,3}^p m_{j+1}^p,0), \\
(m_j^p, \lambda m_j^p) \cdot \left(v_2 - \frac{\alpha_{j,2}^p}{\alpha_{j,3}^p} v_3 \right) &=
(0,2 \alpha_{j,2}^p m_{j+1}^p). 
\end{align*}
%\end{gather}
These equations prove that the submodule generated by $(m_j^p, \lambda m_j^p)$ has finite codimension in $M_p^2$.

It remains to show that the submodules generated by either $(m_j^p,0)$ or $(0,m_j^p)$ have finite codimension. We give
the argument for $(m_j^p,0)$, the argument for $(0,m_j^p)$ being similar. By assumption either
$\alpha_{j,0}^p,\alpha_{j,2}^p \neq 0$ or $\alpha_{j,1}^p,\alpha_{j,3}^p \neq 0$. If $\alpha_{j,0}^p,\alpha_{j,2}^p \neq
0$ then one has
\begin{equation*}
(m_j^p,0)\cdot v_0 = (\alpha_{j,0}^p m_{j+1}^p,0)\; \text{ and }\; (m_j^p,0)\cdot v_2 = (0,\alpha_{j,2}^p m_{j+1}^p),
\end{equation*}
while if $\alpha_{j,1}^p,\alpha_{j,3}^p \neq 0$ one has
\begin{equation*}
(m_j^p,0)\cdot v_1 = (\alpha_{j,1}^p m_{j+1}^p,0)\; \text{ and }\;(m_j^p,0)\cdot v_3 = (0,-\alpha_{j,3}^p m_{j+1}^p).
\end{equation*} 
Once again, this is sufficient to show the submodule generated by $(m_j^p,0)$ has finite codimension.
\end{proof}

There is a natural action of $G$ on point modules of $A$, which we now describe. Since $G$ acts on $A$ by $\N$-graded
algebra automorphisms, one can use Definition \ref{defn: ztwistmodule} to twist such a module by the action of a
particular element of the group. Explicitly, for a point module $M$ and $g \in G$, define a new $A$-module $M^g$ by the
multiplication $m \ast_g a=ma^g$ for all $m \in M$ and $a \in A$. The underlying $\N$-graded vector space structure of
$M$ remains unchanged and therefore $M^g$ still has the same Hilbert series as a point module. Moreover, $g$ acts by an
automorphism and so the twisted module remains cyclic, therefore $M^g$ must also be a point module.

This action on point modules induces an action of $G$ on $\Gamma$: for all $g \in G$, $p \in \Gamma$, define $p^g$
to be the point $q \in \Gamma$ for which $M_q \cong (M_p)^g$. Before stating a result concerning this action, let us
introduce some notation.
\begin{notation}\label{not: idealaut}
Let $I$ be a right ideal in an algebra $A$, on which a finite group $G$ acts by algebra automorphisms. Then for all $g
\in G$ we define $g(I) := \{a^{g} : \; a \in I\}$\index{notation}{g@$g(I)$}. 
\end{notation}
\begin{lemma}\label{lem: actiononpoints}
Assume Hypotheses \ref{hyp: genhypforfatpts}. Then the group $G$ acts on a point $p=(p_0,p_1,p_2,p_3) \in \Gamma$ in the
following manner:
\begin{gather}
\begin{aligned}\label{eq: Gactonpoints}
&p^e =p,\;\; p^{g_{1}}=(p_0,p_1,-p_2,-p_3),\;\; p^{g_{2}}=(p_0,-p_1,p_2,-p_3),\\ &p^{g_{1}g_{2}} =(p_0,-p_1,-p_2,p_3).
\end{aligned}
\end{gather}
In particular, this action preserves the condition on a point having at least three non-zero coordinates. 
\end{lemma}
\begin{proof}
Consider the point module $M_p=A/I_p$, where $I_p$ is a right ideal. For an element $g \in G$ we claim that $(M_p)^g
\cong A/g^{-1}(I_p)$. To see this, recall that $(M_p)^g$ has the same underlying $\N$-graded vector space structure as
$M_p$. Thus we may consider the map $\varphi: (M_p)^g \rightarrow A/g^{-1}(I_p)$ defined by $a + I_p \mapsto a^{g^{-1}}
+ g^{-1}(I_p)$ for $a \in A$. This map is well-defined on the coset structures involved and preserves the $\N$-graded
vector space structures, thus it remains to check that it is an isomorphism of right $A$-modules. For $a, b \in A$ one
has
\begin{align*}%\label{eq: ?}
\varphi((a+I_p) \ast_g b) &= \varphi(ab^{g} + I_p) = (ab^g)^{g^{-1}} + g^{-1}(I_p) = a^{g^{-1}}b + g^{-1}(I_p),\\
\varphi(a+I_p)b &= (a^{g^{-1}} + g^{-1}(I_p))b = a^{g^{-1}}b + g^{-1}(I_p).
\end{align*}
Thus $\varphi$ is a homomorphism of $\N$-graded right $A$-modules. As $G$ acts by automorphisms, the modules $(M_p)^g$
and $A/g^{-1}(I_p)$ have the same Hilbert series. Consequently, $\varphi$ is an isomorphism.

The right ideal $I_p$ is generated by the degree 1 elements 
\begin{equation}\label{eq: ptmoddeg1gens}
p_0x_1 - p_1 x_0,\;\;\;  p_0x_2 - p_2 x_0,\;\;\; p_0x_3 - p_3 x_0.
\end{equation}
The isomorphism $(M_p)^g \cong A/g^{-1}(I_p)$ and Lemma \ref{lem: defrelns} indicate that the behaviour of the three
generators in \eqref{eq: ptmoddeg1gens} under the action of $g^{-1}$ govern $p^g$. The result is clear when $g=e$ since
the identity acts trivially. We give a proof for $g=g_1$, with the remaining two cases being similar. Noting that $g_1$
has order 2, one has
\begin{gather}
\begin{aligned}\label{eq: ptmoddeg1gensact}
(p_0x_1 - p_1 x_0)^{g_{1}} &= p_0x_1^{g_{1}} - p_1 x_0^{g_{1}} = p_0x_1 - p_1 x_0,\\
(p_0x_2 - p_2 x_0)^{g_{1}} &= p_0x_2^{g_{1}} - p_2 x_0^{g_{1}} = -p_0x_2 - p_2 x_0,\\
(p_0x_3 - p_3 x_0)^{g_{1}} &= p_0x_3^{g_{1}} - p_1 x_3^{g_{1}} = -p_0x_3 - p_3 x_0.
\end{aligned} 
\end{gather}
The right ideal generated by the three elements on the right-hand side of \eqref{eq: ptmoddeg1gensact} corresponds to
the point $q = (p_0,p_1,-p_2,-p_3)$. Thus $p^{g_{1}}=q$ as in the statement of the lemma.
\end{proof}

In Proposition \ref{prop: fatpoints} we constructed fat point modules over the twist $A^{G,\mu}$. One can repeat the
same trick, using the embedding of $A=(A^{G,\mu})^{G,\mu}$ inside $M_2(A^{G,\mu})$ to prove the following result. 
\begin{prop}\label{prop: fatpointsotherwaygen}
Consider $M_p^2$, a fat point module over $A^{G,\mu}$ constructed in Proposition \ref{prop: fatpoints}. The direct sum
$(M_p^2)^2$ is an $\N$-graded right $M_2(A^{G,\mu})$-module. On restriction to a module over the subalgebra $A$, one has
a decomposition of $\N$-graded right $A$-modules
\begin{equation}\label{eq: orbitdecomp}
(M_p^2)^2 \cong \bigoplus_{g \in G} M_{p^{g}}.
\end{equation}
\end{prop}
\begin{proof}
The fat point module $M_p^2$ was obtained by restricting the $M_2(A)$-module $M_p^2$ to $A^{G,\mu}$. Thus
$(M_p^2)^2$ can be considered as the $M_4(A)$-module $M_p^4$, which becomes an $M_2(A^{G,\mu})$-module
upon restriction. One can then regard $(M_p^2)^2$ as an $A$-module by restricting a second time, with the
action on homogeneous pieces given by $4 \times 4$ matrices. 

Indeed, one can write the action of the generators of $A$ on $(M_p^2)^2$ explicitly by unravelling the composite
embedding of $A$ into $M_4(A)$. Let us use the notation $m_j^p \cdot x_i= \alpha_{j,i}^p m_{j+1}^p$ --- which
was introduced before Proposition \ref{prop: fatpoints} --- for the action of the generators of $A$ on the point module
$M_p$. For all $j \in \N$ and $c_l \in k$, the action of the generators of $A$ on $(M_p^2)^2$ is given by
\begin{equation*}
(c_0 m_j^p, c_1 m_j^p,c_2 m_j^p,c_3 m_j^p) \cdot x_i := (c_0 m_{j+1}^p, c_1 m_{j+1}^p,c_2 m_{j+1}^p,c_3 m_{j+1}^p) \cdot
Q_i,
\end{equation*}
where
\begin{gather}
\begin{aligned}\label{eq: actonfatptdoubles}
Q_0 &= \begin{pmatrix} \alpha_{j,0}^p &0&0&0 \\0& \alpha_{j,0}^p &0&0 \\ 0&0& \alpha_{j,0}^p &0 \\ 0&0&0& \alpha_{j,0}^p
\end{pmatrix},\;\; 
Q_1 = \begin{pmatrix} \alpha_{j,1}^p &0&0&0 \\0& -\alpha_{j,1}^p &0&0 \\ 0&0& -\alpha_{j,1}^p &0 \\ 0&0&0&
\alpha_{j,1}^p \end{pmatrix},\\ 
Q_2 &= \begin{pmatrix} 0&0&0& \alpha_{j,2}^p \\0&0& \alpha_{j,2}^p &0 \\ 0& \alpha_{j,2}^p &0&0 \\ \alpha_{j,2}^p &0&0&0
\end{pmatrix},\;\; 
Q_3 = \begin{pmatrix} 0&0&0& \alpha_{j,3}^p \\0&0& -\alpha_{j,3}^p &0 \\ 0& -\alpha_{j,3}^p &0&0 \\ \alpha_{j,3}^p
&0&0&0 \end{pmatrix}.
\end{aligned}
\end{gather}

One can see from the matrices in \eqref{eq: actonfatptdoubles} that there is a decomposition of $A$-modules
\begin{equation}\label{eq: forbitdecomp}
(M_p^2)^2=(M_p,0,0,M_p) \oplus (0,M_p,M_p,0). 
\end{equation}

We claim that the decomposition in the statement of the proposition is then given by the isomorphisms of right
$A$-modules
\begin{gather}
\begin{aligned}\label{eq: decompasptmods}
(M_p,0,0,M_p) &\cong (m_0,0,0,m_0)A \oplus (m_0,0,0,-m_0)A \cong M_p \oplus M_{p^{g_{1}}}, \\
(0,M_p,M_p,0) &\cong (0,m_0,m_0,0)A \oplus (0,m_0,-m_0,0)A \cong M_{p^{g_{2}}} \oplus M_{p^{g_{1}g_{2}}}.
\end{aligned}
\end{gather}
To see this, note that the submodules 
\begin{equation*}
(m_0,0,0,m_0)A,\; (m_0,0,0,-m_0)A, \; (0,m_0,m_0,0)A \; \text{ and }\;(0,m_0,-m_0,0)A 
\end{equation*}
certainly have the correct Hilbert series to be point modules and are cyclic. By calculating which degree 1 elements of
$A$ annihilate them, one can see that these cyclic submodules are indeed isomorphic to the point modules indicated in
\eqref{eq: decompasptmods}.
\end{proof}

Proposition \ref{prop: fatpointsotherwaygen} allows us to determine the isomorphisms in $\text{grmod}(A^{G,\mu})$
between the fat point modules constructed in Proposition \ref{prop: fatpoints}. 
\begin{cor}\label{cor: fatpointisoclasses}
The only isomorphisms in $\text{grmod}(A^{G,\mu})$ between the fat point modules over $A^{G,\mu}$ described in
Proposition \ref{prop: fatpoints} are of the form $M_p^2 \cong M_{p^g}^2$ for all $g \in G$. 
\end{cor}
\begin{proof}
To see that $M_p^2 \cong M_{p^g}^2$ holds for all $g \in G$, we first use the final part of Lemma \ref{lem:
actiononpoints}: if $p \in \Gamma$ has three non-zero coordinates then so does $p^g$ for all $g \in G$. Thus one can
construct $M_{p^g}^2$ as in Proposition \ref{prop: fatpoints} for any $g \in G$. The isomorphisms are governed by the
matrices in \eqref{eq: diagmat}:
\begin{itemize}
 \item[(i)] $M_p^2 \cong M_{p^{g_{1}}}^2$ via right multiplication by $\twomat{1}{0}{0}{-1}$;
 \item[(ii)] $M_p^2 \cong M_{p^{g_{2}}}^2$ via right multiplication by $\twomat{0}{1}{1}{0}$;
 \item[(iii)] $M_p^2 \cong M_{p^{g_{1}g_{2}}}^2$ via right multiplication by $\twomat{0}{-1}{1}{0}$.
\end{itemize}

Suppose now that $p, q \in \Gamma$ both have at least three non-zero coordinates and their associated fat point modules
are isomorphic, thus $M_p^2 \cong M_{q}^2$. By Proposition \ref{prop: fatpointsotherwaygen}, we can take direct sums of
these fat point modules and consider them as right $M_2(A^{G,\mu})$-modules. On restricting them down to the subalgebra
$A$ we obtain via \eqref{eq: orbitdecomp} the following isomorphism:
\begin{equation}\label{eq: orbitdecomp2}
\bigoplus_{g \in G} M_{p^{g}} \cong \bigoplus_{g \in G} M_{q^{g}}.
\end{equation}

Both modules in \eqref{eq: orbitdecomp2} are f.g.\ $\N$-graded modules of GK dimension 1. By \cite[Proposition
1.5]{smith1992the} the factors in a critical composition series of such a module are, when considered in high degree,
unique up to permutation and isomorphism. In our case this implies that we must have $\pi(M_q) \cong \pi(M_{p^g})$ for
some $g \in G$. Thus there exists some $n \in \N$ such that $M_q[n]_{\geq 0}\cong M_{p^g}[n]_{\geq 0}$ in
$\text{grmod}(A)$. But $M_q[n]_{\geq 0} \cong M_{q^{\sigma^{n}}}$ for any point in $\Gamma$, in which case one has an
isomorphism in $\text{grmod}(A)$ of the form $M_{q^{\sigma^{n}}} \cong M_{(p^g)^{\sigma^{n}}}$. Since point modules over
$A$ are parameterised up to isomorphism by the closed points of $\Gamma$, we may conclude that
$q^{\sigma^{n}}=(p^g)^{\sigma^{n}}$. As $\sigma$ is an automorphism it follows that $q=p^g$.
\end{proof}

To end this section we recap the ideas we have used. For $G$ and $\mu$ as in Hypotheses \ref{hyp: genhypforfatpts} we
can play the following game: take a 1-critical $A$-module $M$ and consider $M^2$ as a module over $A^{G,\mu} \cong
M_2(A)^G$. This module has GK dimension 1 and, by replacing it with $M^2_{\geq n}$ for some $n \in \N$ if necessary, has
a critical composition series in which each composition factor is also 1-critical. This technique provides an effective
way of discovering 1-critical modules over a cocycle twist $A^{G,\mu}$ given knowledge of those over $A$.

%%%%%%%%%%%%%%%%%%%%%%%%%%%%%%%%%%%%%%%%%%%%%%%%%%%%%%%%%%%%%%%%%%%%%

\chapter{Twists of Sklyanin algebras}\label{chap: sklyanin} 
\chaptermark{Sklyanin twists}
In this chapter we will consider twists of 4-dimensional Sklyanin algebras\index{term}{Sklyanin algebra, 4-dimensional}.
Such algebras can be presented as the quotient of the free $k$-algebra $k\{x_0,x_1,x_2,x_3\}$ by the ideal generated by
the relations:
\begin{equation}\label{eq: 4sklyaninrelns}
\begin{array}{ll}
f_1:=[x_0,x_1]-\alpha[x_2,x_3]_+, &  f_2:=[x_0,x_1]_+ -[x_2,x_3], \\ \relax
f_3:=[x_0,x_2]-\beta[x_3,x_1]_+, & f_4:=[x_0,x_2]_+ -[x_3,x_1], \\ \relax
f_5:=[x_0,x_3]-\gamma[x_1,x_2]_+, & f_6:=[x_0,x_3]_+ -[x_1,x_2],  
\end{array}
\end{equation}
where $\alpha, \beta, \gamma \in k$ satisfy 
\begin{equation}\label{eq: 4sklyanincoeffcond}
\alpha + \beta + \gamma + \alpha \beta \gamma = 0 \;\text{ and }\; \{\alpha,\beta, \gamma\} \cap \{0,\pm 1\}=\emptyset.
\end{equation}
As noted in \cite[Equation 1.1.1]{smith1992regularity}, the first condition in \eqref{eq: 4sklyanincoeffcond} can be
rewritten as 
\begin{equation}\label{eq: 4sklyanincoeffcond1}
(\alpha+1)(\beta+1)(\gamma+1)=(1-\alpha)(1-\beta)(1-\gamma),
\end{equation}
and this form will be useful in some of our computations. The relations in \eqref{eq: 4sklyaninrelns} were first given
in \cite[Equation 0.2.2]{smith1992regularity}.

Many properties of the algebra $A(\alpha, \beta, \gamma)$ are controlled by an elliptic curve $E \subset \proj{k}{3}$
defined by the following elements in $k[y_0, y_1, y_2, y_3]$, the homogeneous coordinate ring of $\proj{k}{3}$
\cite[Proposition 2.4]{smith1992regularity}:
\begin{equation*}
 y_0^2+y_1^2+y_2^2+y_3^2, \;\;\; y_3^2 - \left(\frac{1-\gamma}{1+\alpha}\right)y_1^2
+\left(\frac{1+\gamma}{1-\beta}\right)y_1^2
\end{equation*}

One can consider the algebra defined by the relations \eqref{eq: 4sklyaninrelns} for any $\alpha, \beta, \gamma \in k$
satisfying \eqref{eq: 4sklyanincoeffcond}, where $k$ is an algebraically closed field with $\text{char}(k)\neq 2$. Such
algebras will be referred to as \emph{4-dimensional Sklyanin algebras} and be denoted by
$A(\alpha,\beta,\gamma)$\index{notation}{a@$A(\alpha,\beta,\gamma)$} (or simply $A$ if we can omit the parameters
without ambiguity). The characteristic assumption on $k$ will be made throughout this chapter and the next. Since the
group that will be used in these chapters has order 4, the hypothesis that $\text{char}(k)\nmid |G|$ which was used in
Chapter \ref{chap: cocycletwists} will always hold.

Before beginning our study of twists of such algebras we make some general remarks pertaining to the notation we will
use for elements in cocycle twists. Our notation will be like that used in the proof of Proposition \ref{prop:
koszuldualtwistcommute}: for the remaining chapters (apart from \S\ref{subsec: homenvelopalg}) we will denote the
generators of the algebra that we wish to twist by $x_i$. If such generators are not acted on diagonally by the group
then we will use the notation $w_i$ for a diagonal basis of the generating space. Despite the fact that $A$ and
$A^{G,\mu}$ share the same underlying $k$-vector space structure, we will denote the corresponding generators of the
twist by $v_i$ to avoid confusion and allow us to write the twisted multiplication as juxtaposition. Thus if $x_i \in
A_g$ and $x_j \in A_h$ for some $g, h \in G$, then $v_iv_j =x_i \ast_{\mu} x_j=\mu(g,h)x_i x_j$.

\section{Properties of the twist}\label{subsec: 4sklyanintwistandptscheme}
Although there are other twists of the 4-dimensional Sklyanin algebra --- as discussed in \S\ref{subsec: permuteaction}
and Remark \ref{rem: otheractions} --- the one that we will focus our attention upon is the following. Consider the
isomorphism between $G=(C_2)^2 = \langle g_1,g_2 \rangle$ and its dual $G^{\vee}$ given by the character table
\eqref{eq: chartable}. Define an
action of the generators of $G$ on those of $A(\alpha,\beta,\gamma)$ by
\begin{equation}\label{eq: sklyaninactionIuse}
x_0^{g_{1}}=x_0,\;\;\; x_1^{g_{1}}=x_1,\;\;\; x_2^{g_{1}}=-x_2,\;\;\; x_3^{g_{1}}=-x_3\;\; \text{ and }\;\; x_i^{g_{2}}
= (-1)^i x_i,
\end{equation}
for $i = 0,1,2,3$. The generators of $A(\alpha,\beta,\gamma)$ are acted on diagonally and therefore are homogeneous
with respect to the induced $G$-grading, lying in the following components:
\begin{equation}\label{eq: 4sklyaningrading}
x_0 \in A(\alpha,\beta,\gamma)_{e},\; x_1 \in A(\alpha,\beta,\gamma)_{g_{1}},\; x_2 \in
A(\alpha,\beta,\gamma)_{g_{2}},\; x_3 \in A(\alpha,\beta,\gamma)_{g_{1}g_{2}}.
\end{equation}

The action in \eqref{eq: sklyaninactionIuse} is not the same as that used in the example studied in \S\ref{subsec:
odesskiiegtwist}. Nevertheless, we will show in Proposition \ref{prop: odesskiiegdone} that Odesskii's original example
is isomorphic to an algebra in the same family as that which we study.

The next lemma describes the relations of the algebra that we will primarily focus our attention on. We will use the
2-cocycle that is used here throughout the rest of the thesis. 
\begin{lemma}\label{lem: relationsoftwist}
Let $\mu$ be the 2-cocycle of $G$ defined in \eqref{eq: mucocycledefn}. Then the algebra
$A(\alpha,\beta,\gamma)^{G,\mu}$\index{notation}{a@$A(\alpha,\beta,\gamma)^{G,\mu}$} is the quotient of
the free $k$-algebra $k\{v_0,v_1,v_2,v_3\}$ by the ideal generated by the following six quadratic relations:
\begin{equation}\label{eq: twistrelns}
\begin{array}{ll}
f_1^{\mu}:=[v_0,v_1]-\alpha[v_2,v_3], &  f_2^{\mu}:=[v_0,v_1]_{+} -[v_2,v_3]_{+}, \\ \relax
f_3^{\mu}:=[v_0,v_2]-\beta[v_3,v_1], &  f_4^{\mu}:=[v_0,v_2]_{+} -[v_3,v_1]_{+}, \\ \relax
f_5^{\mu}:=[v_0,v_3]+\gamma[v_1,v_2], &  f_6^{\mu}:=[v_0,v_3]_{+} +[v_1,v_2]_{+}.\end{array}
\end{equation} 
\end{lemma}
\begin{proof}
We begin by computing how the defining relations of $A(\alpha,\beta,\gamma)$ from \eqref{eq: 4sklyaninrelns} behave
under the twist:
\begin{align*}
0 &=  x_0x_1-x_1x_0-\alpha x_2 x_3-\alpha x_3 x_2  \\
&= \frac{x_0 \ast_{\mu} x_1}{\mu(e,g_1)}-\frac{x_1\ast_{\mu} x_0}{\mu(g_1,e)}-\alpha \frac{x_2 \ast_{\mu}
x_3}{\mu(g_2,g_1 g_2)}-\alpha \frac{x_3 \ast_{\mu} x_2}{\mu(g_1 g_2,g_2)} \\
&= x_0 \ast_{\mu} x_1-x_1 \ast_{\mu}x_0-\alpha x_2 \ast_{\mu}x_3+\alpha x_3 \ast_{\mu}x_2,\\
0 &=  x_0x_1+x_1x_0- x_2 x_3+ x_3 x_2  \\
&= \frac{x_0 \ast_{\mu} x_1}{\mu(e,g_1)}+\frac{x_1\ast_{\mu} x_0}{\mu(g_1,e)}-\frac{x_2 \ast_{\mu} x_3}{\mu(g_2,g_1
g_2)}+\frac{x_3 \ast_{\mu} x_2}{\mu(g_1 g_2,g_2)} \\
&= x_0 \ast_{\mu}x_1+x_1 \ast_{\mu}x_0- x_2 \ast_{\mu}x_3- x_3 \ast_{\mu}x_2,\\
0 &=  x_0x_2-x_2x_0-\beta x_3 x_1-\beta x_1 x_3  \\
&= \frac{x_0 \ast_{\mu} x_2}{\mu(e,g_2)}-\frac{x_2\ast_{\mu} x_0}{\mu(g_2,e)}-\beta \frac{x_3 \ast_{\mu} x_1}{\mu(g_1
g_2,g_1)}-\beta \frac{x_1 \ast_{\mu} x_3}{\mu(g_1,g_1 g_2)} \\
&= x_0 \ast_{\mu}x_2-x_2 \ast_{\mu}x_0-\beta x_3 \ast_{\mu}x_1+\beta x_1 \ast_{\mu}x_3,\\
0 &=  x_0 x_2+x_2 x_0- x_3 x_1 + x_1 x_3  \\
&= \frac{x_0 \ast_{\mu} x_2}{\mu(e,g_2)}+\frac{x_2 \ast_{\mu} x_0}{\mu(g_2,e)}-\frac{x_3 \ast_{\mu} x_1}{\mu(g_1
g_2,g_1)}+ \frac{x_1 \ast_{\mu} x_3}{\mu(g_1, g_1 g_2)} \\
&= x_0 \ast_{\mu}x_2+x_2 \ast_{\mu}x_0-x_3 \ast_{\mu}x_1 - x_1 \ast_{\mu}x_3,\\
0 &=  x_0x_3-x_3x_0-\gamma x_1 x_2-\gamma x_2 x_1  \\
&= \frac{x_0 \ast_{\mu} x_3}{\mu(e,g_1 g_2)}-\frac{x_3\ast_{\mu} x_0}{\mu(g_1 g_2,e)}-\gamma \frac{x_1 \ast_{\mu}
x_2}{\mu(g_1, g_2)}-\gamma \frac{x_2 \ast_{\mu} x_1}{\mu(g_2,g_1)} \\
&= x_0 \ast_{\mu}x_3-x_3 \ast_{\mu}x_0+\gamma x_1 \ast_{\mu}x_2-\gamma x_2 \ast_{\mu}x_1,\\
0 &=  x_0x_3+x_3x_0-x_1 x_2+ x_2 x_1  \\
&= \frac{x_0 \ast_{\mu} x_3}{\mu(e,g_1 g_2)}+\frac{x_3\ast_{\mu} x_0}{\mu(g_1 g_2,e)}-\frac{x_1 \ast_{\mu}
x_2}{\mu(g_1,g_2)}+ \frac{x_2 \ast_{\mu} x_1}{\mu(g_2, g_1)} \\
&= x_0 \ast_{\mu}x_3+x_3 \ast_{\mu}x_0+ x_1 \ast_{\mu}x_2+ x_2 \ast_{\mu}x_1.
\end{align*}
Recall our remarks regarding notation for cocycle twists at the beginning of Chapter \ref{chap: sklyanin}; rewriting the
relations above in terms of the new generators $v_i$ produces the relations in \eqref{eq: twistrelns}. By regarding
$A(\alpha,\beta,\gamma)^{G,\mu}$ as a twist of the presentation of $A(\alpha,\beta,\gamma)$ by the relations in
\eqref{eq: 4sklyaninrelns}, one can use Lemma \ref{lem: defrelns} to see that the ideal of relations in the twist is
generated by the relations in \eqref{eq: twistrelns}.

Notice that $A(\alpha,\beta,\gamma)$ has a generating set which is homogeneous with respect to the $G$-grading by
\eqref{eq: 4sklyaningrading}. Thus Remark \ref{lemma: finitelygenerated} implies that $A(\alpha,\beta,\gamma)^{G,\mu}$
is generated as an algebra by the same generating set, hence by the $v_i$.
\end{proof}
\begin{rem}\label{rem: reducestudy}
In \S\ref{subsec: permuteaction} we will show that the study of a whole family of cocycle twists can be reduced to
studying the algebra defined by the relations in \eqref{eq: twistrelns}. 
\end{rem}

By virtue of being a cocycle twist, we immediately get the following result.
\begin{thm}\label{thm: 4sklytwistprops}
Assume that the parameter triple $(\alpha, \beta, \gamma)$ satisfies \eqref{eq: 4sklyanincoeffcond}. Then the algebra
$A(\alpha,\beta,\gamma)^{G,\mu}$ is f.g.\ in degree 1 and has the following properties:
\begin{itemize}
 \item[(i)] it is a universally noetherian domain;
 \item[(ii)] it has Hilbert series $1/(1-t)^4$;
 \item[(iii)] it is AS-regular of global dimension 4;
 \item[(iv)] it is Auslander regular;
 \item[(v)] it satisfies the Cohen-Macaulay property;
 \item[(vi)] it is Koszul.
\end{itemize}
\end{thm}
\begin{proof}
Note that $A(\alpha,\beta,\gamma)^{G,\mu}$ is f.g.\ in degree 1 by Lemma \ref{lem: relationsoftwist}. We proceed by
showing that part of (i) holds. By \cite[Corollary 4.12]{artin1999generic}, $A(\alpha,\beta,\gamma)$ is universally
noetherian. Using Corollary \ref{cor: uninoeth} shows that the same is true for $A(\alpha,\beta,\gamma)^{G,\mu}$.

By \cite[Thm 5.5]{smith1992regularity} and \cite[Corollary 1.9]{levasseur1993modules}, $A(\alpha,\beta,\gamma)$ has the
properties stated in (ii)-(vi). We can then use Lemma \ref{lem: hilbseries} (for (ii)), Corollary \ref{cor: asreg} and
Proposition \ref{prop: gldim} (for (iii)), Proposition \ref{prop: cohenmac} (for (iv) and (v)) and Proposition
\ref{prop: koszul} (for (vi)) to show that $A(\alpha,\beta,\gamma)^{G,\mu}$ also has the respective properties. To
complete the proof one can use \cite[Theorem 3.9]{artin1991modules} to conclude that $A(\alpha,\beta,\gamma)^{G,\mu}$ is
a domain.
\end{proof}
\begin{rem}
In fact Theorem \ref{thm: 4sklytwistprops} is valid for any parameter triple $(\alpha, \beta, \gamma)$ that is not of
the form $(-1,1,\gamma)$, $(\alpha,-1,1)$ or $(1,\beta,-1)$. For those three triples the associated 4-dimensional
Sklyanin algebra is not a domain, while for the remaining triples not covered by the theorem one obtains an iterated Ore
extension over $k$ (see \cite[\S 1]{smith1992regularity}). We leave the proof of the outstanding cases to the reader.
\end{rem}

\subsection{Permuting the $G$-action}\label{subsec: permuteaction}
Our aim in this section is to determine the behaviour of some other twists of $A(\alpha,\beta,\gamma)$. We will show
that such twists are isomorphic up to a change of parameters to the algebra whose relations are given in \eqref{eq:
twistrelns}. 
\begin{prop}\label{lem: 24to1}
Let $G$ be the Klein-four group and $(\alpha,\beta,\gamma)$ a parameter triple satisfying \eqref{eq:
4sklyanincoeffcond}. There are 24 actions of $G$ on $A(\alpha,\beta,\gamma)$ by $\N$-graded algebra
automorphisms for which the following hold:
\begin{itemize}
 \item[(i)] $G$ acts diagonally on the generators $x_0,x_1,x_2$ and $x_3$;
 \item[(ii)] the action of $G$ on $A(\alpha,\beta,\gamma)_1$ affords the regular representation.
\end{itemize}
For any such action of $G$, consider the twist of the induced $G$-grading on $A(\alpha,\beta,\gamma)$ by the 2-cocycle
$\mu$ from Lemma \ref{lem: relationsoftwist}. The algebra obtained is isomorphic to the twist studied in that lemma up
to a change of parameters which still satisfy \eqref{eq: 4sklyanincoeffcond}.
\end{prop}
\begin{proof}
In the previous section we studied the action on $A(\alpha,\beta,\gamma)$ corresponding to the grading
\begin{equation}\label{eq: 4sklyaninusualgrading}
x_0 \in A(\alpha,\beta,\gamma)_{e},\; x_1 \in A(\alpha,\beta,\gamma)_{g_{1}},\; x_2 \in
A(\alpha,\beta,\gamma)_{g_{2}},\; x_3 \in A(\alpha,\beta,\gamma)_{g_{1}g_{2}}. 
\end{equation}
We identify each $G$-graded component of $A$ with the corresponding index of the generator it contains in \eqref{eq:
4sklyaninusualgrading}. In this manner, any other grading corresponding to an action of $G$ affording the regular
representation gives rise to a permutation in the symmetric group $S_4$. For example, the action which induces the
grading 
\begin{equation*}%\label{eq: 4sklyanin01grading}
x_1 \in A(\alpha,\beta,\gamma)_{e},\; x_0 \in A(\alpha,\beta,\gamma)_{g_{1}},\; x_3 \in
A(\alpha,\beta,\gamma)_{g_{2}},\; x_2 \in A(\alpha,\beta,\gamma)_{g_{1}g_{2}}. 
\end{equation*}
corresponds to the permutation $(01)(23)$. It is trivial to check using the relations in \eqref{eq: 4sklyaninrelns} that
each permutation in $S_4$ corresponds to a genuine $G$-grading.

The action which induces the grading in \eqref{eq: 4sklyaninusualgrading}
corresponds to the identity permutation, hence \emph{for this proof only the associated twist will be denoted
$A(\alpha,\beta,\gamma)^{G,\mu;(id)}$}. Our aim can therefore be reformulated as trying to understand
$A(\alpha,\beta,\gamma)^{G,\mu;\sigma}$\index{notation}{a@$A(\alpha,\beta,\gamma)^{G,\mu;\sigma}$} for other
permutations $\sigma \in S_4$. 

Let us now assume that any $G$-grading is one of the 24 arising from an action of $G$ on the generators by the regular
representation. The automorphism group of $G$ is isomorphic to $S_3$, which acts by permuting the order 2 elements. We
will use the convention that products of permutations are applied from right to left.

Notice that two $G$-gradings are twists of each other by an automorphism of $G$ if and only if the components of
their $G$-gradings corresponding to $e$ are equal. This is a consequence of the automorphism group of $G$ being
isomorphic to $S_3$. Recast in terms of permutations, we obtain a partition of $S_4$ by the subsets 
\begin{equation*}
H_j=\{\sigma \in S_4 : \; \sigma^{-1}(0)=j\} \;\text{ for }\;j=0,1,2,3. 
\end{equation*}
Let us choose the identity map and the transpositions $(0j)$ for $j = 1,2,3$ as representatives of these subsets. 

We now show that the action of $\text{Aut}_{\text{grp}}(G)$ on 2-cocycles of $G$ is trivial. To do this we first
identify the subgroup of $S_4$ given by
\begin{equation}\label{eq: 0stabperms}
H_0=\{(id),(12),(23),(13),(123),(132)\},
\end{equation}
with $\text{Aut}_{\text{grp}}(G)$. The identification we use arises naturally from our prior identification of elements
of $G$ with the set $\{0,1,2,3\}$ using \eqref{eq: 4sklyaninusualgrading}.

Observe that the following table describes for each permutation $\sigma \in H_0$ how the 2-cocycles $\mu$ and
$\mu^{(\sigma^{-1})}$ are cohomologous via the function $\rho: G \rightarrow k^{\times}$, where $i$ denotes a primitive
$4^{\text{th}}$ root of unity in $k$:
\begin{equation}\label{eq: sigmaactcoboundary}
\begin{array}{c|cccc}
\sigma & \rho(e) & \rho(g_1) & \rho(g_2) & \rho(g_1g_2) \\
\hline
(12) & 1 & -1 & 1 & 1 \\ 
(13) & 1 & i & 1 & i \\ 
(23) & 1 & 1 & i & i \\ 
(123) & 1 & i & -1 & i \\ 
(132) &  1 & 1 & i & -i 
\end{array}
\end{equation}

Now consider $\sigma \in H_j$. We already know that the $G$-grading associated to $\sigma$ is obtained from that
associated to $(0j)$ by twisting the grading by an automorphism of $G$. But \eqref{eq: sigmaactcoboundary} shows that
the action of $\text{Aut}_{\text{grp}}(G)$ on 2-cocycles is trivial. We may conclude by Lemma \ref{lem: autoncocycle}
that the twists $A(\alpha,\beta,\gamma)^{G,\mu;\sigma}$ and $A(\alpha,\beta,\gamma)^{G,\mu;(0j)}$ are twists of the same
$G$-grading by cohomologous 2-cocycles. By Proposition \ref{prop: trivialtwist} it follows that these algebras are
isomorphic. 

To complete the proof we describe the changes of parameters needed to show that the three algebras of the form
$A(\alpha,\beta,\gamma)^{G,\mu;(0j)}$ belong to the family $A(\alpha',\beta',\gamma')^{G,\mu;(id)}$ for some
$(\alpha',\beta',\gamma')$ satisfying \eqref{eq: 4sklyanincoeffcond}.

For the permutation $(01)$, one can use the rescaling 
\begin{equation*}%\label{eq: 01cob}
(v_0,v_1,v_2,v_3) \mapsto \left(v_0,\frac{i}{\sqrt{\beta
\gamma}}v_1,-\frac{1}{\sqrt{\gamma}}v_2,-\frac{i}{\sqrt{\beta}}v_3\right),
\end{equation*}
to show that there is an isomorphism of $\N$-graded $k$-algebras
\begin{equation*}\label{eq: 01iso}
A(\alpha,\beta,\gamma)^{G,\mu;(01)} \cong A\left(\alpha, \frac{1}{\beta}, \frac{1}{\gamma}\right)^{G,\mu;(id)}.
\end{equation*}

Similarly, for $(02)$ there is an isomorphism of $\N$-graded $k$-algebras
\begin{equation*}%\label{eq: 0203isos}
A(\alpha,\beta,\gamma)^{G,\mu;(02)} \cong A\left(\frac{1}{\alpha}, \beta, \frac{1}{\gamma}\right)^{G,\mu;(id)}, 
\end{equation*}
given by the rescaling
\begin{equation*}
(v_0,v_1,v_2,v_3) \mapsto
\left(v_0,\frac{i}{\sqrt{\gamma}}v_1,\frac{i}{\sqrt{\alpha\gamma}}v_2,\frac{1}{\sqrt{\alpha}}v_3\right).
\end{equation*}

Likewise, for $(03)$ there is an isomorphism of $\N$-graded $k$-algebras
\begin{equation*}%\label{eq: 0203isos1}
A(\alpha,\beta,\gamma)^{G,\mu;(03)} \cong A\left(\frac{1}{\alpha}, \frac{1}{\beta}, \gamma\right)^{G,\mu;(id)}, 
\end{equation*}
given by the rescaling
\begin{equation*}
(v_0,v_1,v_2,v_3) \mapsto
\left(v_0,\frac{i}{\sqrt{\beta}}v_1,\frac{1}{\sqrt{\alpha}}v_2,\frac{i}{\sqrt{\alpha\beta}}v_3\right).
\end{equation*} 

If the parameters $(\alpha, \beta, \gamma)$ satisfy \eqref{eq: 4sklyanincoeffcond} then so does the triple obtained by
negation, permutation, or taking the reciprocal of two of the three parameters. This completes the proof.
\end{proof}
\begin{rem}\label{rem: otheractions}
There are of course other actions of $G$ on $A(\alpha,\beta,\gamma)$. For example, we could consider the action inducing
the grading 
\begin{equation*}%\label{eq: 4sklyaninotheraction}
x_0, x_1 \in A(\alpha,\beta,\gamma)_{g_{1}},\; x_2,x_3 \in A(\alpha,\beta,\gamma)_{g_{1}g_{2}}. 
\end{equation*}
The cocycle twist of this grading by $\mu$ is a Zhang twist of $A(\alpha,\beta,\gamma)$ by the automorphism by which
$g_1$
acts. This is an artefact of $g_2$ acting by scalar multiplication, which brings to mind the proof of Proposition
\ref{prop: recoverztwist}; we defined a group action where one generator acted by scalar multiplication in order to
recover a Zhang twist as a cocycle twist.
\end{rem}

Proposition \ref{lem: 24to1} implies that our study of $A(\alpha,\beta,\gamma)^{G,\mu;(id)}$ encompasses 24 twists up to
an allowable change of parameters. Thus, from the beginning of \S\ref{subsec: 4sklypointscheme} onwards, the notation
$A^{G,\mu}$ will refer to that twist.

We now address the question of whether Odesskii's example from \S\ref{subsec: odesskiiegtwist} is isomorphic to one of
the twists we have just analysed. We therefore assume that $k=\C$ for the rest of \S\ref{subsec: permuteaction}. Recall
that in Proposition \ref{prop: odesskiieg} it was shown that Odesskii's example can be formulated as a cocycle twist.

We will use results of Staniszkis and Smith from \cite{smith1993irreducible}. Although the primary focus of their paper
is to
classify the finite-dimensional simple $A$-modules, they also describe the $\N$-graded automorphism group of $A$,
denoted $\text{Aut}_{\N\text{-alg}}(A)$, which we describe below. For this purpose it is useful to consider the
associated elliptic curve $E$ as the quotient group $\C/\Lambda$, where $\Lambda$ is an integer lattice in $\C$
generated by a complex number $\nu$ such that $\text{im}(\nu)\neq 0$ (see the section on elliptic functions in
\cite[Chapter IV, pgs. 326-329]{hartshorne1977algebraic}).

Following \cite[\S 2.10]{smith1992regularity}, let us introduce the following holomorphic functions on $\C$ associated
to the lattice $\Lambda$. That such \emph{theta functions} are related to the Sklyanin algebra associated to $E$ is
clear from \eqref{eq: thetafunctionsoldnew} below.
\begin{defn}\label{def: thetafunctions}
For $a,b \in \{0,1\}$ define $\Theta_{ab}: \C \rightarrow \C$ to be a holomorphic function satisfying the relations
\begin{equation}
 \Theta_{ab}(z+1) = (-1)^a \Theta_{ab}(z), \;\; \Theta_{ab}(z+\nu)=\text{exp}(-\pi i \nu - 2\pi i z - \pi i b)
\Theta_{ab}(z),
\end{equation}
for all $z \in \C$. 
\end{defn}

Let the automorphism associated to the point scheme of $A$ be translation by the point $\tau \in E$. It is shown in
\cite[Theorem 2.2]{smith1993irreducible} that apart from one exceptional case there is a short exact sequence of groups
\begin{equation}\label{eq: sesAutA}
\sesge{\C^{\times}}{\text{Aut}_{\N\text{-alg}}(A)}{E_4},
\end{equation}
where $E_4$ denotes the 4-torsion on $E$. When $E$ is regarded as a lattice, such torsion can easily be described as the
16 cosets of the form $\frac{n}{4}+\frac{m}{4}\nu + \Lambda$ for $n,m \in \{0,1,2,3\}$. By \cite[Chapter IV, Theorem
4.16]{hartshorne1977algebraic}, $E_4$ is isomorphic as a group to $(\Z/4\Z)^2$.  For any $n \in \N$ the automorphism
associated to $\lambda  \in \C^{\times}$ acts on $A_n$ by scalar multiplication by $\lambda^n$.

The exceptional case mentioned above occurs when $|\tau|=3$ and $E$ has a special form. In that case there is still an
exact sequence like in \eqref{eq: sesAutA}, but with $E_4$ replaced with $E_4 \times (\Z/3 \Z)$ (see \cite[Theorem
2.2(c)]{smith1993irreducible}). As we are mainly interested in the situation when $\tau$ has infinite order, we will not
consider this case.

On pg. 64 op. cit. some automorphisms of $A$ are given which are labelled by the generators of $E_4$, and one can in
fact associate all points in $E_4$ with automorphisms. The exact sequence in \eqref{eq: sesAutA} is non-split because
the composition of automorphisms parameterised by $E_4$ only respects the addition on $E$ coming from their
corresponding points up to scalar automorphism.

By taking compositions of the automorphisms corresponding to the two generators of $E_4$ (which are exhibited on
\cite[pg. 64]{smith1993irreducible}), we can recover all of the automorphisms parameterised by $E_4$ up to scalar. This
is enough for our purposes: to prove Proposition \ref{prop: odesskiiegdone} we will only need to know whether certain
automorphisms act diagonally on the generators, thus working up to scalar is sufficient. 

Some of the automorphisms exhibited in Smith and Staniszkis's paper can be written in terms of the theta functions
introduced above in Definition \ref{def: thetafunctions}. The parameters in the relations \eqref{eq: 4sklyaninrelns} can
be written in terms of those functions as follows:
\begin{equation}\label{eq: thetafunctionsoldnew}
\alpha =\left(\frac{\Theta_{11}(\tau)\Theta_{00}(\tau)}{\Theta_{01}(\tau)\Theta_{10}(\tau)} \right)^2 ,\;\; \beta=-
\left(\frac{\Theta_{11}(\tau)\Theta_{01}(\tau)}{\Theta_{00}(\tau)\Theta_{10}(\tau)} \right)^2,\;\; \gamma=-
\left(\frac{\Theta_{11}(\tau)\Theta_{10}(\tau)}{\Theta_{00}(\tau)\Theta_{01}(\tau)} \right)^2. 
\end{equation}

Using these expressions, one can rewrite the automorphisms in terms of the new parameters, as we now illustrate.
Consider the automorphism corresponding to $\frac{1}{4} + \frac{1}{4}\nu \in E_4$, which sends
\begin{equation*}%\label{eq: rewriteautoldparam}
(X_{11},X_{00},X_{01},X_{10}) \mapsto \left(\frac{\Theta_{11}(\tau)}{\Theta_{00}(\tau)} X_{00},i
\frac{\Theta_{00}(\tau)}{\Theta_{11}(\tau)} X_{11},\frac{\Theta_{01}(\tau)}{\Theta_{10}(\tau)} X_{10}, -i
\frac{\Theta_{10}(\tau)}{\Theta_{01}(\tau)} X_{01}\right).
\end{equation*}

Using the identification $(X_{11},X_{00},X_{01},X_{10})=(x_0,x_1,x_2,x_3)$ and \eqref{eq: thetafunctionsoldnew}, we can
rewrite this as
\begin{equation*}%\label{eq: rewriteautnewparam}
(x_0,x_1,x_2,x_3) \mapsto \left( (\beta \gamma)^{\frac{1}{4}} x_1,i (\beta \gamma)^{-\frac{1}{4}}
x_0,\left(\frac{\beta}{\gamma}\right)^{\frac{1}{4}} x_3, -i \left(\frac{\beta}{\gamma}\right)^{-\frac{1}{4}} x_2 
\right).
\end{equation*}

We now prove that Odesskii's twist must be one that we have previously encountered.
\begin{prop}\label{prop: odesskiiegdone}
Suppose that $|\tau|=\infty$. The example of Odesskii in \S\ref{subsec: odesskiiegtwist} is isomorphic to an algebra
with relations \eqref{eq: twistrelns} for some choice of parameters $(\alpha',\beta', \gamma')$ satisfying \eqref{eq:
4sklyanincoeffcond}.
\end{prop}
\begin{proof}
Note that the algebra being twisted in Odesskii's example, $A(\alpha,\beta,\gamma)$ say, must correspond to parameters
satisfying \eqref{eq: 4sklyanincoeffcond} due to its construction using a smooth elliptic curve. Smith and Stafford's
work in \cite{smith1992regularity} shows that the Sklyanin algebra is only `elliptic' in this sense under the
restriction on parameters.

The action in Odesskii's example --- which was described in \eqref{eq: odesskiiaction} --- affords the regular
representation when restricted to the degree 1 generators. We will show that the action of $G$ on $A$ in this example is
by automorphisms that are diagonal in the new basis (for which the relations are given by \eqref{eq: 4sklyaninrelns}).
In that case, the induced $G$-grading on $A$ will be one of the 24 considered Proposition \ref{lem: 24to1}. That result
implies that the twist must be isomorphic to $A(\alpha', \beta', \gamma')^{G,\mu}$ for some choice of parameters
satisfying the required condition. It is therefore sufficient to show that all automorphisms of order 2 act diagonally
on our generators.

We are not in the exceptional case, hence there is an exact sequence as given in \eqref{eq: sesAutA}. Let $\rho \in
\text{Aut}_{\N\text{-alg}}(A)$ have order 2. As discussed prior to the proposition, there exists $\lambda \in
\C^{\times}$ and one of the automorphisms parameterised by $E_4$, $\phi$ say, such that $\lambda \rho= \phi$. Since
$\lambda$ commutes with all automorphisms, one has $\phi^2= \lambda^2 \in \C^{\times}$. Knowing the automorphisms
parameterised by $E_4$ up to scalar allows us to conclude that the only such automorphisms whose square is a scalar
automorphism are those corresponding to the 2-torsion points. It can be seen from \cite[pg. 64]{smith1993irreducible})
that such automorphisms act diagonally on the generators, from which we can conclude that $\rho$ must also act
diagonally on them. As was shown in the previous paragraph, this is sufficient to prove the result.
\end{proof}

\subsection{The point scheme}\label{subsec: 4sklypointscheme}
We now begin to study the point scheme of $A(\alpha,\beta,\gamma)^{G,\mu}$, which we will denote by $\Gamma$. Our
ultimate aim is to prove the following result.
\begin{thm}[{cf. Theorem \ref{prop: finitepointscheme}}]
Suppose that $A(\alpha,\beta,\gamma)$ is associated to the elliptic curve $E$ and automorphism $\sigma$, which has
infinite order. Then the point scheme $\Gamma$ of $A(\alpha,\beta,\gamma)^{G,\mu}$ consists of 20 points.
\end{thm}

The proof of this result will appear later as Theorem \ref{prop: finitepointscheme}. The delay is needed because the
proof uses the interplay between 1-critical modules over $A(\alpha,\beta,\gamma)$ and $A(\alpha,\beta,\gamma)^{G,\mu}$
-- this is studied in \S\ref{sec: modulesoversklytwist}.

Let us begin by recalling to the reader's attention Theorem \ref{thm: pointschemenice}, regarding the
multilinearisations of the quadratic relations in certain algebras. 
\begin{lemma}\label{lem: multilins}
Consider the multilinearisations\index{term}{multilinearisations} of the relations of $A(\alpha,\beta,\gamma)^{G,\mu}$
given in \eqref{eq: twistrelns}:
\begin{gather}
\begin{aligned}\label{eqn: multilins}
&m_1:=v_{01}v_{12}-v_{11}v_{02}-\alpha v_{21}v_{32}+\alpha v_{31}v_{22},\;\;\;\;
m_2:=v_{01}v_{12}+v_{11}v_{02}-v_{21}v_{32}-v_{31}v_{22}, \\
&m_3:=v_{01}v_{22}-v_{21}v_{02}+\beta v_{11}v_{32}-\beta v_{31}v_{12},\;\;\;\;
m_4:=v_{01}v_{22}+v_{21}v_{02}-v_{31}v_{12}-v_{11}v_{32}, \\
&m_5:=v_{01}v_{32}-v_{31}v_{02}+\gamma v_{11}v_{22}-\gamma v_{21}v_{12},\;\;\;\;
m_6:=v_{01}v_{32}+v_{31}v_{02}+v_{11}v_{22}+v_{21}v_{12}. 
\end{aligned}
\end{gather}
The closed subscheme $\Gamma_2 \subset \proj{k}{3} \times \proj{k}{3}$ determined by the equations in \eqref{eqn:
multilins} is
isomorphic to the graph of the point scheme $\Gamma$ under a scheme
automorphism $\phi$. Furthermore, the closed points of $\Gamma$ parameterise point modules over
$A(\alpha,\beta,\gamma)^{G,\mu}$.
\end{lemma}
\begin{proof}
One can use Theorem \ref{thm: 4sklytwistprops} to see that $A(\alpha,\beta,\gamma)^{G,\mu}$ satisfies the hypotheses of
Theorem \ref{thm: pointschemenice}, from which the result follows.
\end{proof}

We will often study $\Gamma_2$, since it incorporates information about both the point scheme and the associated
automorphism. Lemma \ref{lem: multilins} allows us to write
\begin{equation}\label{eq: 4sklyaninpointlocus}
\Gamma_2=\{q = (p,p^{\phi}) \in \proj{k}{3} \times \proj{k}{3}: m_i(q)=0\;\text{ for $i=1,\ldots,6$} \}.
\end{equation}

Our next result exhibits 20 closed points of the scheme $\Gamma_2$. Under some additional hypotheses, such points will
be shown in Theorem \ref{prop: finitepointscheme} to compose all of the closed points of $\Gamma_2$. Let us introduce
the following notation for certain points in $\proj{k}{3}$ prior to stating the result:
\begin{equation}\label{eq: 4sklyanineipts}
e_0:=(1,0,0,0),\;\; e_1:=(0,1,0,0),\;\; e_2:=(0,0,1,0),\;\; e_3:=(0,0,0,1),
\end{equation}
\begin{lemma}\label{lem: ptschemecontains}
The scheme $\Gamma_2$ contains the closed points $(e_j,e_j)$ for $j = 0,1,2,3$, as well as the
following 16 points, where $i^2= -1$:
\begin{gather}
\begin{aligned}\label{eq: 4sklyanintwistpts}
&((1,\pm i, \pm i ,1),(1, \pm i, \pm i , 1)),\;\;\; ((1,\pm i, \mp i ,-1),(1, \pm i, \mp i , -1)),\\
&\left(\left(1,-(\beta \gamma)^{-\frac{1}{2}},\mp \gamma^{-\frac{1}{2}},\mp\beta^{-\frac{1}{2}} \right),\left(1,-(\beta
\gamma)^{-\frac{1}{2}},\pm \gamma^{-\frac{1}{2}},\pm\beta^{-\frac{1}{2}} \right)\right), \\
&\left(\left(1,(\beta \gamma)^{-\frac{1}{2}},\mp \gamma^{-\frac{1}{2}},\pm\beta^{-\frac{1}{2}} \right),\left(1,(\beta
\gamma)^{-\frac{1}{2}},\pm \gamma^{-\frac{1}{2}},\mp\beta^{-\frac{1}{2}} \right)\right), \\
&\left(\left(1,\pm i\gamma^{-\frac{1}{2}}, (\alpha \gamma)^{-\frac{1}{2}},\pm i \alpha^{-\frac{1}{2}} \right),\left(1,
\mp i\gamma^{-\frac{1}{2}},(\alpha \gamma)^{-\frac{1}{2}},\mp i \alpha^{-\frac{1}{2}} \right)\right), \\
&\left(\left(1, \mp i\gamma^{-\frac{1}{2}},- (\alpha \gamma)^{-\frac{1}{2}},\pm i \alpha^{-\frac{1}{2}} \right),\left(1,
\pm i\gamma^{-\frac{1}{2}},- (\alpha \gamma)^{-\frac{1}{2}},\mp i \alpha^{-\frac{1}{2}} \right)\right), \\
&\left(\left(1,\pm \beta^{-\frac{1}{2}},\pm i \alpha^{-\frac{1}{2}},i (\alpha \beta)^{-\frac{1}{2}} \right),\left(1,\mp
\beta^{-\frac{1}{2}},\mp i \alpha^{-\frac{1}{2}}, i (\alpha \beta)^{-\frac{1}{2}} \right)\right), \\
&\left(\left(1,\pm \beta^{-\frac{1}{2}},\mp i \alpha^{-\frac{1}{2}},- i (\alpha \beta)^{-\frac{1}{2}} \right),\left(1,
\mp \beta^{-\frac{1}{2}},\pm i \alpha^{-\frac{1}{2}},- i (\alpha \beta)^{-\frac{1}{2}} \right)\right).
\end{aligned}
\end{gather}
In particular, $\Gamma = \pi_1(\Gamma_2)$ contains at least 20 points.
\end{lemma}
\begin{proof}
It is easy to see that $(e_j,e_j) \in \Gamma_2$ for $j = 0,1,2,3$, since these points satisfy the
multilinearisations in \eqref{eqn: multilins}. A routine verification --- using \eqref{eq:
4sklyanincoeffcond} where necessary --- confirms that the points in \eqref{eq: 4sklyanintwistpts} also satisfy the
equations in \eqref{eqn: multilins}, and therefore belong to $\Gamma_2$.

To see that the 20 points are distinct, note that by \eqref{eq: 4sklyaninpointlocus} it suffices to show that the
points in $\Gamma$ are distinct. Our assumptions on $\alpha$, $\beta$, $\gamma$ from \eqref{eq: 4sklyanincoeffcond}
imply that $e_j \notin \Gamma'$ for $j = 0,1,2,3$. It is easy to partition $\Gamma'$ into four sets of order 4 based on
which pair from the parameters $\alpha$, $\beta$, $\gamma$ (if any) is needed to describe the coordinates of the point:
\begin{gather}
\begin{aligned}\label{eq: blah}
 &\{(1,\pm i, \pm i ,1), (1, \pm i, \mp i , 1)\},\\
 &\left\lbrace\left(1,-(\beta \gamma)^{-\frac{1}{2}},\mp \gamma^{-\frac{1}{2}},\mp\beta^{-\frac{1}{2}} \right),
 \left(1,(\beta \gamma)^{-\frac{1}{2}},\mp \gamma^{-\frac{1}{2}},\pm\beta^{-\frac{1}{2}} \right)\right\rbrace ,\\
 &\left\lbrace\left(1,\pm i\gamma^{-\frac{1}{2}}, (\alpha \gamma)^{-\frac{1}{2}},\pm i \alpha^{-\frac{1}{2}} \right),
 \left(1, \mp i\gamma^{-\frac{1}{2}},- (\alpha \gamma)^{-\frac{1}{2}},\pm i \alpha^{-\frac{1}{2}}
\right)\right\rbrace,\\
 &\left\lbrace\left(1,\pm \beta^{-\frac{1}{2}},\pm i \alpha^{-\frac{1}{2}},i (\alpha \beta)^{-\frac{1}{2}} \right),
 \left(1,\pm \beta^{-\frac{1}{2}},\mp i \alpha^{-\frac{1}{2}},- i (\alpha \beta)^{-\frac{1}{2}} \right)\right\rbrace.
\end{aligned}
 \end{gather}

The points in each set are clearly distinct and differ from each other only by multiplying two of their coordinates by
-1. Now consider the following function $\proj{k}{3} \rightarrow k$, which sends
\begin{equation*}
(a,b,c,d) \mapsto \left(\frac{ab}{cd}, \frac{ac}{bd},\frac{ad}{bc}\right) \text{ where }a,b,c,d \in k,\; abcd \neq 0.
\end{equation*}
This function takes a different value on each of the four sets in \eqref{eq: blah}, as the following table shows (again
we use the assumptions from \eqref{eq: 4sklyanincoeffcond} on the parameters):

\begin{table}[h]
\centering
\begin{tabular}{| c | c | c | c |}
 \hline Representative & $\left(\frac{ab}{cd},\frac{ac}{bd},\frac{ad}{bc}\right)$ \\[1.2ex]
 
 \hline $(1,i,i,1)$ & $(1,1,-1)$ \\
 \hline $\left(1,-(\beta \gamma)^{-\frac{1}{2}}, -\gamma^{-\frac{1}{2}},\beta^{-\frac{1}{2}} \right)$ &
$(-1,-\beta,-\gamma)$ \\
 
 \hline $\left(1, i\gamma^{-\frac{1}{2}}, (\alpha \gamma)^{-\frac{1}{2}}, i \alpha^{-\frac{1}{2}} \right)$ &
$(\alpha,-1,\gamma)$ \\
 
 \hline $\left(1,\beta^{-\frac{1}{2}}, i \alpha^{-\frac{1}{2}},i (\alpha \beta)^{-\frac{1}{2}} \right)$ &
$(-\alpha,\beta,1)$ \\
 \hline
\end{tabular}
\end{table}

This shows that the four sets in \eqref{eq: blah} are disjoint, from which the result follows.
\end{proof}
\begin{rem}\label{rem: computercalcs}
The points in \eqref{eq: 4sklyanintwistpts} were discovered through the use of computer calculations for specific
parameter choices. The solutions obtained indicated the general form of the points stated above.
\end{rem}

We now introduce some new notation.
\begin{notation}\label{not: gammaprime}
Recall the definition of $\pi_1$ given in Notation \ref{not: projmap}. Define $\Gamma'$ to be the set of points
$\pi_1(q)$ for $q$ in \eqref{eq: 4sklyanintwistpts}. Thus $|\Gamma'|=16$, with Lemma \ref{lem: ptschemecontains}
implying that $\Gamma' \cup \{e_0,e_1,e_2,e_3\} \subset \Gamma$.
\end{notation}

Whilst we have not yet proved that the 20 points in Lemma \ref{lem: ptschemecontains} constitute the whole of
$\Gamma_2$, we can still find the order of $\phi$.
\begin{lemma}\label{lem: phiaut}
The scheme automorphism $\phi$ has order 2. In particular, it fixes 8 of the points given in Lemma \ref{lem:
ptschemecontains}.
\end{lemma}
\begin{proof}
We will exploit some of the observations made after \cite[Definition 1.4]{vancliff1998some} with regard to graded skew
Clifford algebras, which remain valid in our situation. Notice that the multilinearisations in \eqref{eqn: multilins}
are invariant under the map $v_{i1} \leftrightarrow v_{i2}$. Thus $\Gamma_2$ is invariant under the automorphism which
switches components of the ambient space $\proj{k}{3} \times \proj{k}{3}$. 

Combining this observation with the description of $\Gamma_2$ in \eqref{eq: 4sklyaninpointlocus}, we conclude
that $(p,p^{\phi}) \in \Gamma_2$ if and only if $(p^{\phi},p) \in \Gamma_2$. But $\Gamma_2$ is the graph of $\Gamma$
under the automorphism $\phi$, hence we must have $p=p^{\phi^{2}}$. Consequently, $\phi$ has order at most 2. 

It is clear by observation that 8 of the points in $\Gamma$ exhibited in Lemma \ref{lem:
ptschemecontains} are fixed by $\phi$, including those of the form $e_j$. The remaining 12 points of $\Gamma$ that we
have given each have order 2 under this automorphism. Consequently, $\phi$ has order 2.
\end{proof}

We end this section by remarking that the conclusion of Lemma \ref{lem: phiaut} is surprising, since generically the
associated automorphism of the point scheme of $A(\alpha,\beta,\gamma)$ will have infinite order.

\section{Modules over the twist}\label{sec: modulesoversklytwist}
In this section we study modules over $A(\alpha,\beta,\gamma)^{G,\mu}$, which will be denoted by $A^{G,\mu}$ when it is
possible to omit the parameters. Our first result does not fit neatly into one of the sections that follows, thus we
state
it here. It concerns the dimension of the line scheme of $A^{G,\mu}$ -- we will use the formulation of the line scheme
that is used in \cite[Lemma 2.5]{shelton2002schemes}.
\begin{prop}\label{prop: linescheme1dim}\index{term}{line scheme!calculations}
Assume that $\text{char}(k) = 0$. Then the line scheme of $A(\alpha,\beta,\gamma)^{G,\mu}$ is 1-dimensional for generic
parameters.
\end{prop}
\begin{proof}
A relation in $A^{G,\mu}$ can be written in the form $\sum_{i=1}^{6} t_i f_i^{\mu}$ for some scalars $t_i \in k$. The
line scheme of $A^{G\,\mu}$ is the locus $(t_1,\ldots,t_6) \in \proj{k}{5}$ for which the matrix
\begin{equation*}%\label{eq: lineschemeAGmumatrix} 
\begin{pmatrix}
 0 & t_1+t_2 & t_3+t_4 & t_5+t_6 \\
 t_2-t_1 & 0 & \gamma t_5 + t_6 & \beta t_3 -t_4 \\
 t_4 - t_3 & t_6 - \gamma t_5 & 0 & -\alpha t_1 - t_2 \\
 t_6 - t_5 & -\beta t_3 - t_4 & \alpha t_1 - t_2 & 0
\end{pmatrix}  
\end{equation*}
has rank less than 3 by \cite[Lemma 2.5]{shelton2002schemes}. The Macaulay2 code given by Code \ref{code:
lineschemeagmu} in Appendix \ref{sec: sklyanincalcs} shows that this scheme is indeed 1-dimensional. Of course, this
code is only valid in characteristic 0 and for generic parameters, hence our assumptions.
\end{proof}

\subsection{Point modules and their annihilators}\label{subsec: pointmodbehav}
We already know the existence of 20 isomorphism classes of point modules through their parameterisation by $\Gamma$ (see
Notation \ref{not: gammaprime}). We study the point modules corresponding to these 20 points, although some of our
results hold for a general point module. Since we show that there are no further point modules up to isomorphism in
Theorem \ref{prop: finitepointscheme}, such considerations will eventually prove superfluous.

Let us first describe the behaviour of point modules of $A^{G,\mu}$ under the twisting operation. To avoid confusion we
introduce the following notation.
\begin{notation}\label{not: tilde}
Point modules over $A$ will be denoted by $M_p$, while those over $A^{G,\mu}$ will be denoted by $\widetilde{M}_p$.
\end{notation}
\begin{prop}\label{prop: ptmodbehav}
Let $\widetilde{M}_p=A^{G,\mu}/I_p$ denote the point module corresponding to the point $p \in \Gamma$. If $p$ is fixed
by $\phi$ then $I_p$\index{notation}{m@$\widetilde{M}_p$} is a two-sided ideal and $\widetilde{M}_{p} \cong 
\widetilde{M}_p[1]_{\geq 0}$. If $p^{\phi}\neq p$ then we have isomorphisms
\begin{equation*}%\label{eq: shiftptmodiso}
\widetilde{M}_{p^{\phi}} \cong  \widetilde{M}_p[1]_{\geq 0} \; \text{ and } \; \widetilde{M}_p \cong
\widetilde{M}_{p^{\phi}}[1]_{\geq 0}.
\end{equation*}
\end{prop}
\begin{proof}
Since $\phi$ has order 2, this is a consequence of Theorem \ref{thm: pointschemenice} and Definition \ref{defn:
ptschemeautomorphism}.
\end{proof}

We will now consider the annihilators of these modules. Our first observation is that all such ideals are prime; point
modules are 1-critical and therefore by \cite[Proposition 2.30(vi)]{artin1991modules} their annihilators are prime
ideals. 

For a point module $\widetilde{M}_p$ with $p \in \Gamma$, one has
\begin{equation*}%\label{eq: ptmodannih}
\text{Ann}_{A^{G,\mu}}(\widetilde{M}_p)=\bigcap_{n \in \N} \text{Ann}_{A^{G,\mu}}(\widetilde{M}_p)_{n}.
\end{equation*}

Suppose that $p$ is fixed by $\phi$. The annihilator of each graded piece of $\widetilde{M}_p$ is the same, since it is
isomorphic to its shifts by Proposition \ref{prop: ptmodbehav}. Thus the annihilator of such a point module is precisely
the defining ideal of the point module, which is two-sided by the same result. 

On the other hand, the annihilator of the point modules which have order 2 under the shift is
$\text{Ann}_{A^{G,\mu}}(\widetilde{M}_p)_0 \cap \text{Ann}_{A^{G,\mu}}(\widetilde{M}_p)_1$. By Proposition \ref{prop:
ptmodbehav} the annihilator of such a module is $I_p \cap I_{p^{\phi}}$. In particular, this means that this
intersection of right ideals is a two-sided ideal.

From now until the beginning of \S\ref{subsec: fatpoints} let us assume that $p$ has order 2 under $\phi$. We would like
to determine the answers to the following two questions:
\begin{ques}\label{que: annihilator}
\begin{itemize}
 \item[(i)] What is the Hilbert series of $I_p \cap I_{p^{\phi}}$?
 \item[(ii)] Can generators for this ideal be found?
\end{itemize}
\end{ques}
In order to answer Question \ref{que: annihilator}(i), we will use the isomorphism of vector spaces
\begin{equation}\label{eq: 2ndisovs}
\frac{I_p}{I_p \cap I_{p^{\phi}}} \cong \frac{I_p + I_{p^{\phi}}}{I_{p^{\phi}}}.
\end{equation}
This allows us to prove the following lemma.
\begin{lemma}\label{lem: hilbseriesannih}
Let $p \in \Gamma$ be a point such that $p^{\phi}\neq p$. Then $I_p + I_{p^{\phi}}$ contains two elements of
degree 1 and has codimension 2 inside $A^{G,\mu}$ in all higher degrees.
\end{lemma}
\begin{proof}
Since $(I_{p})_{1}$ and $(I_{p^{\phi}})_{1}$ both have codimension
1 inside $A^{G,\mu}_{1}$, if $A^{G,\mu}_{1} \not\subset I_p + I_{p^{\phi}}$ then we would have $I_p=I_{p^{\phi}}$,
which is a contradiction. Thus $A^{G,\mu}_{1} \subset I_p + I_{p^{\phi}}$. By Theorem \ref{thm:
4sklytwistprops}, $A^{G,\mu}$ is generated in degree 1, thus $A^{G,\mu}_{\geq 1}=I_p + I_{p^{\phi}}$.

Using this information together with \eqref{eq: 2ndisovs} allows us to obtain the following relation of
Hilbert series:
\begin{gather}
\begin{align*}%\label{eq: hilbseriesreln}
 H_{I_{p} \cap I_{p^{\phi}}}(t) = H_{I_{p}}(t)+H_{I_{p^{\phi}}}(t)-H_{A^{G,\mu}_{\geq 1}}(t) &=
2\left(\frac{1}{(1-t)^4}- \frac{1}{(1-t)} \right) - \left( \frac{1}{(1-t)^4}-1 \right) \\
&=\frac{1-(1-t)^3(1+t)}{(1-t)^4}.
\end{align*}
\end{gather}
Thus $H_{A^{G,\mu}/(I_{p} \cap I_{p^{\phi}})} = \frac{1+t}{1-t}$. Consequently, the ideal $I_{p} \cap I_{p^{\phi}}$ has
codimension 2 in all degrees greater than 0, which proves the result. 
\end{proof}

We can now describe the generators of the intersection $I_{p} \cap I_{p^{\phi}}$, thus providing an
answer to Question \ref{que: annihilator}(ii).
\begin{prop}\label{prop: ptmodintersectgens}
Let $p \in \Gamma'$ be a point for which $p^{\phi}\neq p$. Then $I_{p} \cap I_{p^{\phi}}$ is generated as a two-sided
ideal by two degree 1 elements. The generators of this ideal for each of the 6 orbits of order 2 are given respectively
by:
\begin{itemize}
 \item[(i)] $(rs v_0 + v_1, r v_3 - s v_2)$ for $r=\gamma^{-\frac{1}{2}},\; s=\pm \beta^{-\frac{1}{2}}$;
 \item[(ii)] $(rs v_0 + v_2,s v_1 - r v_3)$ for $r=i \gamma^{-\frac{1}{2}},\; s= \pm i \alpha^{-\frac{1}{2}}$;
 \item[(iii)] $(rs v_0 - v_3, r v_2 + s v_1)$ for $r=-\beta^{-\frac{1}{2}},\; s= \pm i \alpha^{-\frac{1}{2}}$.
\end{itemize}
Moreover, the factor ring $A^{G,\mu}/I_{p} \cap I_{p^{\phi}}$ is isomorphic to the quotient of a skew polynomial ring by
a central regular element of degree 2.
\end{prop}
\begin{proof}
One can see that $(I_{p} \cap I_{p^{\phi}})_1$ certainly contains the two elements given in the statement of the
proposition in each case. We will factor out the two-sided ideal generated by these two elements and show that the
factor ring obtained has the Hilbert series $\frac{1+t}{1-t}$. By Lemma \ref{lem: hilbseriesannih}, this is sufficient
to prove the result.

Let us first consider case (i). It is clear that the factor ring is generated as an algebra by $v_0$ and $v_2$. The
relations of $A^{G,\mu}$ can be rewritten in terms of these two elements; the relations $f_i^{\mu}$ for $i=1,3,5$ lie in
the ideal $(rs v_0 + v_1, r v_3 - s v_2)$, while the remaining relations can be rewritten in the following manner:
\begin{equation*}%\label{eq: relsinintersectptmod}
f_2^{\mu} \rightsquigarrow  rs v_0^2 + \frac{s}{r}v_2^2,\;\; f_4^{\mu} \rightsquigarrow (1+s^2)[v_0,v_2]_+,\;\;
f_6^{\mu} \rightsquigarrow \left(\frac{1}{r}-r \right)[v_0,v_2]_+.
\end{equation*} 
Note that any cancellation involving $r$ and $s$ needed to determine these `new' relations does not depend on the sign
of the scalar $s$, hence is valid for both choices. 

Since $\beta \neq 0,-1$ and $\gamma^2 \neq 1$, we must have 
\begin{equation*}%\label{eq: intersectptmodiso}
\frac{A^{G,\mu}}{(rs v_0 + v_1, r v_3 - s v_2)} \cong \frac{k\{v_0,v_2\}}{(v_0 v_2+v_2 v_0, v_0^2 + \gamma v_2^2)},
\end{equation*}
as $k$-algebras. This is a factor ring of the quantum plane $k_{-1}[v_0,v_2]$ by the ideal generated by a central
regular element of degree 2. Such a ring has Hilbert series $(1+t)/(1-t)$, therefore the annihilator must be generated
by the two elements as claimed. 

The remainder of the proof comprises the same argument repeated in the other cases. In case (ii) the factor ring
obtained by factoring out the two-sided ideal generated by $rs v_0 + v_2$ and $s v_1 - r v_3$ is generated as an algebra
by $v_0$ and $v_1$. The relations of $A^{G,\mu}$ can be rewritten in terms of the two algebra generators; as in case
(i), the
relations $f_i^{\mu}$ for $i=1,3,5$ lie in the ideal $(rs v_0 + v_2,s v_1 - r v_3)$, while the others can be rewritten
as follows:
\begin{equation*}%\label{eq: relsinintersectptmod.1}
f_2^{\mu} \rightsquigarrow  (1+s^2)[v_0,v_1]_+,\;\; f_4^{\mu} \rightsquigarrow rs v_0^2 + \frac{s}{r} v_1^2,\;\;
f_6^{\mu} \rightsquigarrow \left(\frac{1}{r}+r\right)[v_0,v_1]_+.
\end{equation*} 

Since $\alpha \neq 0,1$ and $\gamma^2 \neq 1$, we must have 
\begin{equation*}%\label{eq: intersectptmodiso.1}
\frac{A^{G,\mu}}{(rs v_0 + v_2,s v_1 - r v_3)} \cong \frac{k\{v_0,v_1\}}{(v_0 v_1+v_1 v_0, v_0^2 - \gamma v_1^2)},
\end{equation*}
as $k$-algebras. Once again, the factor ring is a factor of a skew polynomial ring by a central regular element of
degree 2 and has the correct Hilbert series.

Moving on to case (iii), the factor ring obtained by factoring out the two-sided ideal generated by $rs v_0 - v_3$ and
$r v_2 + s v_1$ is generated as an algebra by $v_0$ and $v_1$. Rewriting the relations of $A^{G,\mu}$ in terms of the
two algebra generators we find that the relations $f_i^{\mu}$ for $i=1,3,5$ lie in the ideal $(rs v_0 - v_3,r v_2 + s
v_1)$, while the remaining relations can be transformed to the forms given below:
\begin{equation*}%\label{eq: relsinintersectptmod.2}
f_2^{\mu} \rightsquigarrow   (1-s^2) [v_0,v_1]_+,\;\; f_4^{\mu} \rightsquigarrow  
\left(r-\frac{1}{r}\right)[v_0,v_1]_+,\;\; f_6^{\mu} \rightsquigarrow  rs v_0^2+\frac{s}{r} v_2^2.
\end{equation*} 

Since $\alpha \neq 0,-1$ and $\beta^2 \neq 1$, we have an isomorphism of $k$-algebras
\begin{equation*}%\label{eq: intersectptmodiso.2}
\frac{A^{G,\mu}}{(rs v_0 - v_3, r v_2 + s v_1)} \cong \frac{k\{v_0,v_1\}}{(v_0 v_1+v_1 v_0, v_0^2 + \beta v_1^2)}.
\end{equation*}
As in the previous cases, this factor ring is a quotient of a skew polynomial ring by a central regular element of
degree 2. The factor ring therefore has the correct Hilbert series, which completes the proof.
\end{proof}

The Hilbert series of these factor rings --- after removing the degree 0 piece --- is that of a fat point module of
multiplicity 2. Despite this, one can see that they are not 1-critical as modules over $A^{G,\mu}$ because they have GK
dimension 1 factor modules, namely the point modules $\widetilde{M}_p$ and $\widetilde{M}_{p^{\phi}}$.

\subsection{Fat point modules of multiplicity 2}\label{subsec: fatpoints}
Our aim in this section is to apply the results from \S\ref{sec: modules} to $A$ and $A^{G,\mu}$. To see that both of
these algebras satisfy Hypotheses \ref{hyp: genhypforfatpts} one can use Theorem \ref{thm: 4sklytwistprops}, as well as
noting that the action of $G$ on $A^{G,\mu}$ is chosen to be that which induces the $G$-grading inherited from $A$. 

Before giving our first result we recap some of the geometry associated to $A$. We recall that the parameters associated
to $A$ satisfy \eqref{eq: 4sklyanincoeffcond}, with the following information depending upon this fact. As proved in
\cite[Propositions 2.4 and 2.5]{smith1992regularity}, the point modules over $A$ are parameterised by points on a smooth
elliptic curve $E \subset \proj{k}{3}$ and four extra points $e_j$ as in \eqref{eq: 4sklyanineipts}. In \cite[Corollary
2.8]{smith1992regularity} it is shown that the automorphism associated to the point scheme fixes the four exceptional
points and is given by a translation $\sigma$ on $E$. 

In order to apply the machinery of \S\ref{sec: modules} we need to following result.
\begin{lemma}\label{lem: threegensannihilate}
For all $p \in E$ at least three coordinates of $p$ are non-zero. Furthermore, the action of $G$ on the point scheme of
$A$ restricts to $E$.
\end{lemma}
\begin{proof}
Assume that $p= (p_0,p_1,p_2,p_3) \in E$. We will use \cite[Proposition 2.5]{smith1992regularity}, which describes the
homogeneous coordinate ring of $E$:
\begin{equation}\label{eq: thcroreextn}
\frac{k[y_0,y_1,y_2,y_3]}{\left(y_0^2+y_1^2+y_2^2+y_3^2,y_3^2+\left(\frac{1-\gamma}{1+\alpha}
\right)y_1^2+\left(\frac{1+\gamma}{1-\beta}\right) y_2^2\right)}.
\end{equation}

The action of $G$ on the point scheme of $A$ is described prior to Lemma \ref{lem: actiononpoints}, see \eqref{eq:
Gactonpoints} in particular. It is clear from that equation and \eqref{eq: thcroreextn} above that if $p \in E$ then
$p^g \in E$ for all $g \in G$.

Let us now prove the other part of the result. Since $e_j \notin E$ for $j=0,1,2,3$ we can assume that at least two
coordinates of $p$ are non-zero. If there were two non-zero
entries, $p_l$ and $p_m$ say, then the equations defining $E$ would reduce to the form 
\begin{equation*}%\label{eq: twogenszero}
p_l^2+p_m^2=p_l^2+ \lambda p_m^2=0, 
\end{equation*}
for some $\lambda \in k$. The only solution when $\lambda \neq 1$ is $p_l=p_m=0$, which results in a contradiction. If
$\lambda=1$ then either $\frac{1-\gamma}{1+\alpha}=1$ or $\frac{1+\gamma}{1-\beta}=1$ or
$\frac{1-\gamma}{1+\alpha}=\frac{1+\gamma}{1-\beta}$. 

If $\frac{1-\gamma}{1+\alpha}=1$ then one has $\alpha=-\gamma$, in which case $\beta=0$ or $\gamma=\pm 1$, contradicting
\eqref{eq: 4sklyanincoeffcond}. Similarly, if $\frac{1+\gamma}{1-\beta}=1$ then $\beta=-\gamma$, whereupon $\alpha=0$
or $\gamma=\pm 1$. Once again, such parameters are not permitted by \eqref{eq: 4sklyanincoeffcond}. Finally, assume that
$\frac{1-\gamma}{1+\alpha}=\frac{1+\gamma}{1-\beta}$. In this case one has $(1-\beta)(1-\gamma)=(1+\alpha)(1+\gamma)$,
which can be rearranged to $\alpha=-\beta$ by using \eqref{eq: 4sklyanincoeffcond1}. This forces $\gamma=0$ or
$\beta=\pm 1$, contradicting \eqref{eq: 4sklyanincoeffcond} again. 
\end{proof}

The parameterisation described in the following result also has a geometric interpretation, as can be seen from
Proposition \ref{prop: fatpointsincohE}.
\begin{prop}\label{claim: fatpoints}
$A^{G,\mu}$ has a family of fat point modules of multiplicity 2 parameterised up to isomorphism by the $G$-orbits of
$E$.
\end{prop}
\begin{proof}
Let $p \in E$. By Lemma \ref{lem: threegensannihilate} at least three coordinates of $p$ are non-zero and the action of
$G$ on the point scheme preserves $E$. Thus one may apply Proposition \ref{prop: fatpoints} to obtain a fat point module
$M_p^2$ over $A^{G,\mu}$. By Corollary \ref{cor: fatpointisoclasses} the only isomorphisms between such modules in
$\text{grmod}(A^{G,\mu})$ are of the form $M_p^2 \cong M_{p^g}^2$ for $g \in G$. 
\end{proof}
\begin{rems}\label{rem: qgrisoslater}
\begin{itemize}
\item[(i)] In fact, there are no further isomorphisms between the corresponding fat points in $\text{qgr}(A^{G,\mu})$,
although we will only prove this in Corollary \ref{cor: qgrisos}. 

\item[(ii)] If one tries to apply the construction of Proposition \ref{prop: fatpoints} to the point modules
$M_{e_{j}}$, one does not obtain any fat point modules. In fact, the right $M_2(A)$-module $M^2_{e_{j}}$ becomes
isomorphic to the direct sum $\widetilde{M}_{e_{j}}^2$ upon restriction to $A^{G,\mu}$. 
\end{itemize}
\end{rems}

We will use $[p]$\index{notation}{p@$[p]$} to denote the $G$-orbit of a point $p \in \proj{k}{3}$; when necessary it
will be made clear  whether this point lies on $E$ or in $\Gamma$, where $\Gamma$ is the point scheme of $A^{G,\mu}$.
This notation allows us to define for all $p \in E$ the right $A^{G,\mu}$-module $\widetilde{F}_{[p]}:=
M_{p}^2$\index{notation}{f@$\widetilde{F}_{[p]}$}; Proposition \ref{claim: fatpoints} implies that this is
well-defined and furthermore this notation is in harmony with that introduced in Notation \ref{not: tilde}.

Our next aim is to reverse the process of Proposition \ref{claim: fatpoints} by taking direct sums of point modules over
$A^{G,\mu}$. Before doing so we need to introduce some preliminary material, including the following result regarding
the action of $G$ on the point scheme $\Gamma$.
\begin{lemma}\label{lem: gammaorbits}
Consider $\Gamma' \cup \{e_0,e_1,e_2,e_3\}$, which is contained in the point scheme $\Gamma$. Under the action of $G$ on
$\Gamma$ described prior to Lemma \ref{lem: actiononpoints}, $\Gamma' \cup \{e_0,e_1,e_2,e_3\}$ decomposes as the union
of 8 $G$-orbits: 
\begin{gather}
\begin{aligned}\label{eq: orbitsunderG}
\textit{Singleton orbits: }\; &[e_0],\;\;[e_1],\;\;[e_2],\;\;[e_3].\\
\textit{Order 4 orbits: }\; &\left[\left(1,(\beta \gamma)^{-\frac{1}{2}}, -\gamma^{-\frac{1}{2}},\beta^{-\frac{1}{2}}
\right)\right], \;\; \left[\left(1, i\gamma^{-\frac{1}{2}}, (\alpha \gamma)^{-\frac{1}{2}}, i \alpha^{-\frac{1}{2}}
\right)\right], \\
&\left[\left(1, \beta^{-\frac{1}{2}}, i \alpha^{-\frac{1}{2}},i (\alpha \beta)^{-\frac{1}{2}} \right)\right],\;\;
\left[(1,i,i,1)\right].
\end{aligned}
\end{gather}
Furthermore, if $[p]$ is an order 4 orbit in \eqref{eq: orbitsunderG} then there exists $h \in G$ such that
$(p^g)^{\phi}=(p^{g})^h$ for all $g \in G$. That is, the restriction of $\phi$ to each orbit of order 4 coincides with
the action of a particular element of $G$. 
\end{lemma}
\begin{proof}
One can use \eqref{eq: Gactonpoints} to verify that the orbits under the action of $G$ are as stated in \eqref{eq:
orbitsunderG}.

Let us now address the second part of the statement of the lemma. Since $[(1,i,i,1)]$ is fixed pointwise by $\phi$, it
is clear that for $p \in[(1,i,i,1)]$ one has $p^{\phi}=p^e$. Thus the identity element is associated to this orbit. For
representatives of the remaining three orbits of order 4, we exhibit in \eqref{eq: orbitcorrespondence} the group
element associated to their orbit: 
\begin{gather}
\begin{aligned}\label{eq: orbitcorrespondence}
\left(1,(\beta \gamma)^{-\frac{1}{2}}, -\gamma^{-\frac{1}{2}},\beta^{-\frac{1}{2}} \right)^{\phi} &= 
\left(1,(\beta \gamma)^{-\frac{1}{2}},- \gamma^{-\frac{1}{2}},\beta^{-\frac{1}{2}} \right)^{g_{1}}, \\
\left(1, i\gamma^{-\frac{1}{2}}, (\alpha \gamma)^{-\frac{1}{2}}, i \alpha^{-\frac{1}{2}} \right)^{\phi} &=
\left(1, i\gamma^{-\frac{1}{2}}, (\alpha \gamma)^{-\frac{1}{2}}, i \alpha^{-\frac{1}{2}} \right)^{g_{2}}, \\
\left(1, \beta^{-\frac{1}{2}}, i \alpha^{-\frac{1}{2}},i (\alpha \beta)^{-\frac{1}{2}} \right)^{\phi} &=
\left(1, \beta^{-\frac{1}{2}}, i \alpha^{-\frac{1}{2}},i (\alpha \beta)^{-\frac{1}{2}} \right)^{g_{1}g_{2}}.
\end{aligned} 
\end{gather}
One can see that there is a 1-1 correspondence between elements of $G$ and order 4 orbits of $\Gamma$ that we have
discovered so far.
\end{proof}

Assume now that $\sigma$, the associated automorphism of the point scheme of $A$, has infinite order. In
\cite{smith1993irreducible}, Smith and Staniszkis classify fat point modules of all multiplicities over $A$ under this
assumption on $\sigma$. Their classification is a by-product of their work classifying the finite-dimensional
simple $A$-modules. As the final remark in \S 4 op. cit. states, their work shows that there are four non-isomorphic fat
point modules of each multiplicity $m \geq 2$. These modules can be denoted by
$F(\omega^{\sigma^{m-1}})$\index{notation}{f@$F(\omega^{\sigma^{m-1}})$} for some 2-torsion point $\omega \in E_2$. This
notation is natural since $F(\omega^{\sigma^{m-1}})$ arises as the quotient of any line module $M_{p,q}$ associated to
points
$p, q \in E$ such that $p+q = \omega^{\sigma^{m-1}}$. 

Let us now state a result relating point modules over $A^{G,\mu}$ with fat point modules of multiplicity 2 over $A$.
Beforehand, recall from Notation \ref{not: tilde} that we use tildes to denote point modules over $A^{G,\mu}$.
\begin{prop}\label{prop: sklyaninfatpts}
Consider $A$ as the invariant subring $M_2(A^{G,\mu})^G$. Four isomorphism classes of fat point modules of multiplicity
2 over $A$ arise as the restriction of modules of the form $\widetilde{M}_p^2$, where $\widetilde{M}_p$ is a point
module over $A^{G,\mu}$. When $|\sigma|=\infty$ one recovers in this manner all four fat point modules of multiplicity 2
over $A$.
\end{prop}
\begin{proof}
We remark that $A^{G,\mu}$ satisfies Hypotheses \ref{hyp: genhypforfatpts}, in which case we have the tools of
\S\ref{sec: modules} at our disposal. Consider the 16 points in $\Gamma'$, which arise from the projection to the first
coordinate of those in \eqref{eq: 4sklyanintwistpts}. By our assumption on scalars from \eqref{eq: 4sklyanincoeffcond},
each of these points has at least three non-zero coordinates. Thus by Proposition \ref{prop: fatpoints} one can
construct 16 fat point modules of multiplicity 2 over $A$, of the form $\widetilde{M}_p^2$. 

The 16 points we are considering are partitioned into the four $G$-orbits described in \eqref{eq: orbitsunderG}. Thus by
Corollary \ref{cor: fatpointisoclasses} there are precisely four isomorphism classes of such fat point modules. When
$|\sigma|=\infty$, the work in \cite{smith1993irreducible} shows that there are precisely four isomorphism classes of
fat point modules of multiplicity 2. Thus, under that hypothesis we recover each of these classes using the construction
of Proposition \ref{prop: fatpoints}.
\end{proof}
\begin{rem}\label{rem: degenfatptno}
Applying the construction of Proposition \ref{prop: fatpoints} to the point modules $\widetilde{M}_{e_{j}}$ produces
behaviour like that explained in Remarks \ref{rem: qgrisoslater}(ii). The right $M_2(A^{G,\mu})$-module
$\widetilde{M}^2_{e_{j}}$ becomes isomorphic to the direct sum $M_{e_{j}}^2$ upon restriction to $A$. 
\end{rem}

Let $p \in \Gamma'$. Proposition \ref{prop: sklyaninfatpts} allows us to associate a 2-torsion point
$\omega_{[p]}\index{notation}{o@$\omega_{[p]}$} \in E$ to $p$. Thus
$F(\omega_{[p]}^{\sigma})$\index{notation}{f@$F(\omega_{[p]}^{\sigma})$} denotes the
fat point module over $A$ which is isomorphic to the restriction of the $M_2(A^{G,\mu})$-module
$\widetilde{M}_{p^{g}}^2$ for all $g \in G$.

Let us also introduce the notation $N_p:=A^{G,\mu}/I_p \cap I_{p^{\phi}}$\index{notation}{n@$N_p$} for points $p \in
\Gamma'$ such that $p^{\phi}\neq p$. These right $A^{G,\mu}$-modules were studied in \S\ref{subsec: pointmodbehav}.
Although such modules are not 1-critical --- as we observed after the proof of Proposition \ref{prop:
ptmodintersectgens} --- one can use them to recover fat point modules over $A$.
\begin{cor}\label{cor: anncompseries}
Assume that $|\sigma|= \infty$ and let $p \in \Gamma'$ have order 2 under $\phi$. When regarded as a right $A$-module by
restriction, the $M_2(A^{G,\mu})$-module $(N_{p}[1]_{\geq 0})^2$ has a critical composition series of length 2. Both of
the composition factors are isomorphic to the fat point module $F(\omega_{[p]}^{\sigma})$.
\end{cor}
\begin{proof}
Consider the following chain of $M_2(A^{G,\mu})$-submodules of $(N_{p}[1]_{\geq 0})^2$:
\begin{equation}\label{eq: dasfds}
(N_{p}[1]_{\geq 0})^2=\left( \left(\frac{A^{G,\mu}}{I_p \cap I_{p^{\phi}}}\right)[1]_{\geq 0} \right)^2 \supsetneq 
\left( \left(\frac{I_p}{I_p \cap I_{p^{\phi}}}\right)[1]_{\geq 0}\right)^2 \supsetneq 0.
\end{equation}

Recall from the proof of Lemma \ref{lem: hilbseriesannih} that $A^{G,\mu}_{\geq 1}=I_p + I_{p^{\phi}}$. Using this fact
and the second isomorphism theorem for modules, the middle term in \eqref{eq: dasfds} can be
rewritten as follows:
\begin{equation*}%\label{eq: adfdsa}
\left( \left(\frac{I_p}{I_p \cap I_{p^{\phi}}}\right)[1]_{\geq 0}\right)^2 \cong \left( \left(\frac{I_p+
I_{p^{\phi}}}{I_{p^{\phi}}}\right)[1]_{\geq 0}\right)^2 = \left( \left(\frac{A^{G,\mu}_{\geq
1}}{I_{p^{\phi}}}\right)[1]_{\geq 0}\right)^2.
\end{equation*}
By Proposition \ref{prop: ptmodbehav} this module is isomorphic to $\widetilde{M}_p^2$ as a right
$M_2(A^{G,\mu})$-module. The other factor of the chain of submodules in \eqref{eq: dasfds} can be seen to be isomorphic
to
$\widetilde{M}_{p^{\phi}}^2$ by using the third isomorphism theorem and Proposition \ref{prop: ptmodbehav} once again.
Lemma \ref{lem: gammaorbits} implies that $p^{\phi} \in [p]$, since $p^{\phi}=p^g$ for the group element $g$ associated
to the orbit $[p]$. 

One can therefore use Proposition \ref{prop: sklyaninfatpts} to conclude that, on restriction to modules over $A$, both
of the factors of the chain in \eqref{eq: dasfds} are isomorphic to a fat point module of the form
$F(\omega_{[p]}^{\sigma})$. Thus the chain in \eqref{eq: dasfds} is a critical composition series for $(N_{p}[1]_{\geq
0})^2$ of length 2.
\end{proof}

We have seen in Propositions \ref{claim: fatpoints} and \ref{prop: sklyaninfatpts} that by taking direct sums of point
modules over $A$ and $A^{G,\mu}$ we can recover fat point modules of multiplicity 2 over the other algebra. The next
result shows that this relationship is reciprocal -- direct sums of fat point modules over one algebra decompose as
direct sums of point modules on restriction to the other algebra. 
\begin{prop}\label{prop: fatpointsotherway}
Assume that $|\sigma| = \infty$.
\begin{itemize}
\item[(i)] Let $p \in E$ and $\omega \in E_2$. On restriction to a module over $A=M_2(A^{G,\mu})^G$, one has
$(\widetilde{F}_{[p]})^2 \cong \bigoplus_{g \in G}M_{p^{g}}$. 

\item[(ii)] Now let $p \in \Gamma'$. The $M_2(A)$-module $F(\omega_{[p]}^{\sigma})^2$ is isomorphic upon restriction to
$A^{G,\mu}$ to the direct sum $\bigoplus_{g \in G}\widetilde{M}_{p^{g}}$. In this manner one recovers all 16 point
modules of $A^{G,\mu}$ associated to points in $\Gamma'$.
\end{itemize}
\end{prop}
\begin{proof}
One can prove both (i) and (ii) by a direct application of Proposition \ref{prop: fatpointsotherwaygen}. We elaborate
the argument with respect to the final statement of (ii). This holds because each $G$-orbit $[p] \subset \Gamma'$
corresponds uniquely to a fat point module $F(\omega_{[p]}^{\sigma})^2$ over $A$ by Proposition \ref{prop:
sklyaninfatpts}. Proposition \ref{prop: fatpointsotherwaygen} then implies that taking a direct sum of this fat point
module and restricting returns the $A^{G,\mu}$-module $\bigoplus_{g \in G}\widetilde{M}_{p^{g}}$.
\end{proof}

We can now prove that the point scheme of $A^{G,\mu}$ consists of precisely 20 points.
\begin{thm}\label{prop: finitepointscheme}
Assume that $|\sigma|= \infty$. Then $\Gamma$, the point scheme\index{term}{point scheme!calculations of} of
$A^{G,\mu}$, has the form
\begin{equation*}
\Gamma = \Gamma' \cup \{e_0,e_1,e_2,e_3\}. 
\end{equation*}
Thus $\Gamma$ contains 20 points, each of which has multiplicity 1.
\end{thm}
\begin{proof}
Suppose that $\widetilde{M}_p$ is a point module over $A^{G,\mu}$, where $p \in \proj{k}{3}$. We claim that if $p \neq
e_j$ then at least three coordinates of $p$ are non-zero. To see this, consider the following manner of expressing the
multilinearisations in \eqref{eqn: multilins}:
\begin{equation}\label{eq: 4sklyaninmatrixform}
\begin{pmatrix} 
-v_{11} & v_{01} & \alpha v_{31} & -\alpha v_{21} \\
v_{11} & v_{01} & -v_{31} & -v_{21} \\ 
-v_{21} & -\beta v_{31} & v_{01} & \beta v_{11} \\
v_{21} & -v_{31} & v_{01} & -v_{11} \\
-v_{31} & -\gamma v_{21} & \gamma v_{11} & v_{01} \\
v_{31} & v_{21} & v_{11} & v_{01}  \end{pmatrix} \cdot \begin{pmatrix} v_{02} \\ v_{12} \\ v_{22} \\ v_{32}
\end{pmatrix} =0.
\end{equation}
We know that $(p,p^{\phi})$ is a solution to these equations, i.e. $(p,p^{\phi}) \in \Gamma_2$. Furthermore, as $\phi$
is an automorphism this is the unique point $q \in \Gamma_2$ for which $p=\pi_1(q)$. In particular, this implies that
when the coordinates of $p$ are substituted into the left-hand matrix in \eqref{eq: 4sklyaninmatrixform}, that matrix
must have rank less than or equal to 3. 

Suppose, therefore, that two coordinates of $p$ are zero. We may assume that the remaining two coordinates are non-zero,
otherwise $p$ would be one of the four points of the form $e_j$. Recall that we have assumed that $\{\alpha,\beta,
\gamma\} \cap
\{0,\pm 1\}=\emptyset$. This assumption implies that the matrix in \eqref{eq: 4sklyaninmatrixform} has rank at least 4,
regardless of which two coordinates of $p$ are zero. Thus we may assume that at least three coordinates of $p$ are
non-zero.

The construction of Proposition \ref{prop: fatpoints} shows that $\widetilde{M}_p^2$ is a fat point
module of multiplicity 2 over $A$. But when $|\sigma|=\infty$ there are only four such modules up to isomorphism, hence
we know that $\widetilde{M}_p^2$ must be isomorphic to one of them. Proposition \ref{prop: fatpointsotherway} implies
that there
exists some point $z \in \Gamma'$ for which $\widetilde{M}_p^2 \cong F(\omega_{[z]}^{\sigma})$ as right $A$-modules.

Applying Proposition \ref{prop: fatpointsotherwaygen} to $\widetilde{M}_p^2$ produces the right $A^{G,\mu}$-module
$\bigoplus_{g \in G}\widetilde{M}_{p^{g}}$. Using the decomposition of the right $A^{G,\mu}$-module
$(F(\omega_{[z]}^{\sigma}))^2$ from Proposition \ref{prop: fatpointsotherway}, we must have an isomorphism
\begin{equation}\label{eq: dualdecomp}
\bigoplus_{g \in G}\widetilde{M}_{p^{g}} \cong \bigoplus_{g \in G}\widetilde{M}_{z^{g}},
\end{equation}
of right $A^{G,\mu}$-modules. But by \cite[Proposition 1.5]{smith1992the} the factors in a critical composition series
of a f.g.\ $\N$-graded module of GK dimension 1 are unique up to permutation and isomorphism in high degree. 

That result applies to the modules in \eqref{eq: dualdecomp}, hence we may conclude that $\pi(\widetilde{M}_{p}) \cong
\pi(\widetilde{M}_{z^{g}})$ in $\text{qgr}(A^{G,\mu})$ for some $g \in G$. This implies the existence of $n \in \N$ such
that $\widetilde{M}_{p^{\phi^{n}}} \cong \widetilde{M}_{\left(z^{g}\right)^{\phi^{n}}}$ in $\text{gr}(A^{G,\mu})$. Since
$\Gamma$ parameterises isomorphism classes of point modules over $A^{G,\mu}$ and $\phi$ is an automorphism, we may
conclude that $p= z^g$.

To finish the proof, observe that Proposition \ref{prop: genericpointscheme} implies that each of the 20 points in
$\Gamma$ has multiplicity 1.
\end{proof}

Recall the discussion in \S\ref{sec: finitedimptscheme} regarding algebras with a 0-dimensional point scheme and a
1-dimensional line scheme being determined by this data. Propositions \ref{prop: linescheme1dim} and \ref{prop:
finitepointscheme} imply that $A^{G,\mu}$ is a desirable AS-regular algebra to study in light of its
associated geometry.

Figure \ref{fig: 1critical} illustrates the behaviour of 1-critical modules of small multiplicity over $A$ and
$A^{G,\mu}$ that we have uncovered in \S\ref{subsec: fatpoints}. We write $E^G:=E/G$ to denote the orbit space of $E$
under the group action -- this will be shown in Lemma \ref{lem: ellcurve} to be an elliptic curve. Bold points and lines
represent modules of multiplicity 2, while those which are in regular font represent modules of multiplicity 1. The
isomorphisms above right-pointing arrows arise by restricting $M_2(A)$-modules to $A^{G,\mu}$, whereas those above
arrows pointing to the left arise by restricting $M_2(A^{G,\mu})$-modules to $A$. For example, the top right-facing
arrow illustrates the construction of fat point modules over $A^{G,\mu}$ given in Proposition \ref{claim: fatpoints},
while the left-facing arrow directly below
it refers to the `reverse process' of Proposition \ref{prop: fatpointsotherway}.
\begin{figure}[ht]
\centering
\begin{tikzpicture}
%\draw (0,15) rectangle (12,0);
%Titles
\node at (3,14.5) {\underline{\Large{$\text{grmod}_{\text{1-crit}}(A)$}}};
\node at (9,14.5) {\underline{\Large{$\text{grmod}_{\text{1-crit}}(A^{G,\mu})$}}};
% Fat Sklyanin points - to go middle left
\node at (1.5,5.5) {\large{$\bullet$}};
\node at (3.5,7.5) {\large{$\bullet$}};
\node at (1.5,7.5) {\large{$\bullet$}};
\node at (3.5,5.5) {\large{$\bullet$}};
\node at (2.5,6.5) {\Large{$F(\omega_{[p]}^{\sigma})$}};

% AGmu point modules to go mid right
\node at (9,5) {\large{$\cdot$}};
\node at (9,6) {\large{$\cdot$}};
\node at (9,7) {\large{$\cdot$}};
\node at (9,8) {\large{$\cdot$}};
\node at (9.5,6.5) {\Large{$\widetilde{M}_{p}$}};
\node at (8,5) {\large{$\cdot$}};
\node at (8,6) {\large{$\cdot$}};
\node at (8,7) {\large{$\cdot$}};
\node at (8,8) {\large{$\cdot$}};
\node at (10,5) {\large{$\cdot$}};
\node at (10,6) {\large{$\cdot$}};
\node at (10,7) {\large{$\cdot$}};
\node at (10,8) {\large{$\cdot$}};
\node at (11,5) {\large{$\cdot$}};
\node at (11,6) {\large{$\cdot$}};
\node at (11,7) {\large{$\cdot$}};
\node at (11,8) {\large{$\cdot$}};
 
% e_i two lots - to go at the bottom
\node at (2,1) {\large{$\cdot$}};
\node at (4,3) {\large{$\cdot$}};
\node at (2,3) {\large{$\cdot$}};
\node at (4,1) {\large{$\cdot$}};
\node at (3,2) {\Large{$M_{e_{j}}$}};
\node at (8,1) {\large{$\cdot$}};
\node at (10,3) {\large{$\cdot$}};
\node at (8,3) {\large{$\cdot$}};
\node at (10,1) {\large{$\cdot$}};
\node at (9,2) {\Large{$\widetilde{M}_{e_{j}}$}};

% Bottom curves and isomorphisms
\node at (6,3) {\large{$M_{e_{j}}^2 \cong \widetilde{M}_{e_{j}}^2$}};
\draw [->] (4.5,2.25) .. controls (5.5,2.75) and (6.5,2.75) .. (7.5,2.25); 

\node at (6,1) {\large{$\widetilde{M}_{e_{j}}^2 \cong M_{e_{j}}^2$}};
\draw [<-] (4.5,1.75) .. controls (5.5,1.25) and (6.5,1.25) .. (7.5,1.75); 

% Middle curves and isomorphisms
\node at (6,7.6) {\normalsize{$F(\omega_{[p]}^{\sigma})^2 \cong \bigoplus_{g \in G} \widetilde{M}_{p^{g}}$}};
\draw [->] (4.5,6.75) .. controls (5.5,7.25) and (6.5,7.25) .. (7.5,6.75); 

\node at (6,5.4) {\large{$\widetilde{M}_p^2 \cong F(\omega_{[p]}^{\sigma})$}};
\draw [<-] (4.5,6.25) .. controls (5.5,5.75) and (6.5,5.75) .. (7.5,6.25); 

% Top curves and isomorphisms
\node at (6,12.6) {\large{$M_{p}^2 \cong \widetilde{F}_{[p]}$}};
\draw [->] (4.5,11.75) .. controls (5.5,12.25) and (6.5,12.25) .. (7.5,11.75); 

\node at (6,10.4) {\normalsize{$(\widetilde{F}_{[p]})^2 \cong \bigoplus_{g \in G} M_{p^{g}}$}};
\draw [<-] (4.5,11.25) .. controls (5.5,10.75) and (6.5,10.75) .. (7.5,11.25); 

% Top stuff
\node at (2.7,12.3) {\Large{$E$}};
\node at (9.25,12.3) {\Large{$E^{G}$}};

\node[right] at (3.5,12.6) {\large{$M_{p}$}};
\node at (3.48,12.5) {\Large{$\cdot$}};
\node[right] at (10.14,12) {\large{$\widetilde{F}_{[p]}$}};
\node at (10.12,12) {\large{$\bullet$}};

\draw plot [smooth] coordinates {(1,10) (3.6,12) (2.75,13) (1.9,12) (4.2,10)};
 
\draw [ultra thick] plot [smooth] coordinates {(7.8,10) (10.1,12) (9.25,13) (8.4,12) (11,10)};
\end{tikzpicture}
\caption{1-critical modules of small multiplicities over $A$ and $A^{G,\mu}$}
\label{fig: 1critical}
\end{figure}
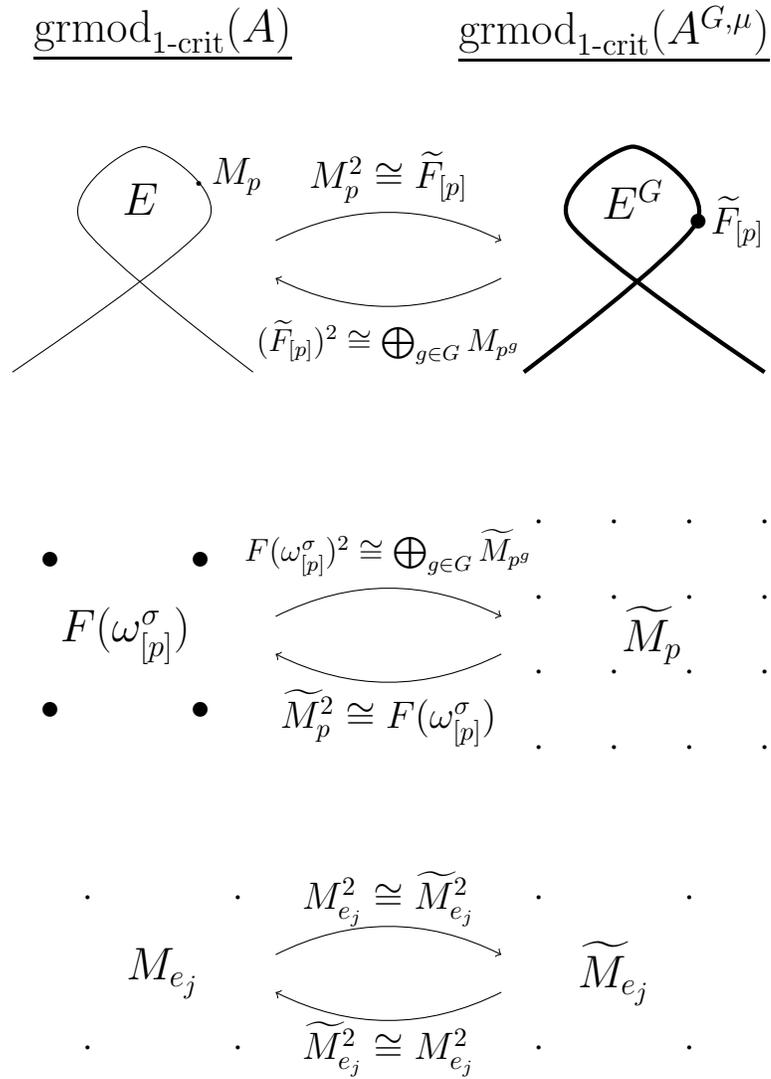 

%%%%%%%%%%%%%%%%%%%%%%%%%%%%%%%%%%%%%%%%%%%%%%%%%%%%%%%%%%%%%%%%%%%%%

\chapter{Twisting a factor ring of the Sklyanin algebra} \label{chap: thcrtwist}
\chaptermark{Twisting a geometric factor ring}
The 4-dimensional Sklyanin algebra $A(\alpha,\beta,\gamma)$ contains two central elements of degree 2; this is proved in
\cite[Corollary 3.9]{smith1992regularity}, with such elements having the form
\begin{equation}\label{eq: 4sklyanincentre}
 \Omega_1:=-x_0^2+x_1^2+x_2^2+x_3^2,\;\; \Omega_2:=x_1^2+\left(\frac{1+\alpha}{1-\beta}
\right)x_2^2+\left(\frac{1-\alpha}{1+\gamma} \right)x_3^2.
\end{equation}

The same result shows that there is a isomorphism between the ring obtained by factoring out the ideal generated by
these two central elements and the twisted homogeneous coordinate ring
$B(E,\calL,\sigma)$\index{notation}{b@$B(E,\calL,\sigma)$}. Recall that $E$ is a smooth elliptic curve,
$\calL:=\calO_E(1)$ a very ample invertible sheaf with 4 global sections, and $\sigma$ an automorphism of $E$ given by
translation by a point. 

In this chapter we will assume the following hypothesis, which we highlight to avoid repetition. We remark that there
are no known conditions on the parameter triple $(\alpha,\beta,\gamma)$ which imply that it is true.
\begin{hypsing}\label{hypsing: siginforder}
The automorphism $\sigma$ has infinite order.
\end{hypsing}
Since $\sigma$ is given by a translation by a point, under Hypothesis \ref{hypsing: siginforder} there are no points
fixed by $\sigma$. The ring $B(E,\calL,\sigma)$ will be denoted by $B$ whenever there is no ambiguity in doing so. 

The aim of the chapter is to investigate a twist of $B$ related to the twist that we studied in Chapter \ref{chap:
sklyanin}. The twist --- which we denote by $B^{G,\mu}$ --- is shown in Theorem \ref{thm: thcrtwistprops} to be
Auslander-Gorenstein and Cohen-Macaulay. We then establish the centre of the twist in Proposition \ref{prop:
centrethcrtwist}, which allows us to determine the centre of $A^{G,\mu}$ in Corollary \ref{cor: centretwistskly}. The
main result of this chapter appears in \S\ref{subsubsec: geomdescrthcrtwist}, in which we study $B^{G,\mu}$ in terms of
Artin and Stafford's classification of noncommutative curves; Theorem \ref{thm: geomdescthcrtwist} describes the twist
in terms of geometry, and this description allows us to determine the irreducible objects in $\text{qgr}(B^{G,\mu})$ in
Proposition \ref{prop: onlyfatpoints}.

\section{Properties of the twist}\label{subsec: propthcrtwist}
In this section we give some elementary properties of $B^{G,\mu}$. Let us consider the grading defined on $A$ by
\eqref{eq: 4sklyaningrading}. The two central elements exhibited in \eqref{eq: 4sklyanincentre} are homogeneous with
respect to both this grading and the \N-grading. This means that the cocycle twist performed in \S\ref{subsec:
4sklyanintwistandptscheme} can also be applied to $B$. In addition, $\Omega_1$ and $\Omega_2$ lie in the identity
component of the $G$-grading. Using Lemma \ref{prop: stillregular} we can see that these elements must remain central
and regular in the twist. In fact, we can say more: by \cite[Theorem 5.4]{smith1992regularity} $\Omega_1$ and $\Omega_2$
form a regular sequence of central elements in $A$. Applying Lemma \ref{prop: stillregular} to successive factor rings
shows that this remains true for $\Theta_1$ and $\Theta_2$ in $A^{G,\mu}$.

From the discussion of the previous paragraph, the following lemma is straightforward.
\begin{lemma}\label{lem: twistthcr}
The degree 2 elements
\begin{equation}\label{eq: 4sklyanintwistcentre}
\Theta_1:=-v_0^2+v_1^2+v_2^2-v_3^2,\;\; \Theta_2:=v_1^2+\left(\frac{1+\alpha}{1-\beta}
\right)v_2^2-\left(\frac{1-\alpha}{1+\gamma} \right)v_3^2,\index{notation}{t@$\Theta_1, \Theta_2$}
\end{equation}
in $A^{G,\mu}$ are central and form a regular sequence. Moreover, there is an isomorphism of $k$-algebras
\begin{equation}\label{eq: isofactors}
B^{G,\mu}:=B(E,\calL,\sigma)^{G,\mu}\index{notation}{b@$B(E,\calL,\sigma)^{G,\mu}$}=\left(\frac{A(\alpha,
\beta,\gamma)}{(\Omega_1,\Omega_2)} \right)^{G,\mu} \cong \frac{A(\alpha, \beta,\gamma)^{G,\mu}}{(\Theta_1,\Theta_2)}.
\end{equation} 
\end{lemma}
\begin{proof}
To complete the proof we only need to apply Lemma \ref{lem: defrelns} to $A(\alpha, \beta,\gamma)$ and the ideal
$(\Omega_1,\Omega_2)$. This confirms that the isomorphism in \eqref{eq: isofactors} holds.
\end{proof}
\begin{rem}\label{rem: genindeg1Bgmu}
Since $A^{G,\mu}$ is generated in degree 1, the isomorphism in \eqref{eq: isofactors} implies that $B^{G,\mu}$ must
share this property. This fact will be crucial in the proof of Theorem \ref{thm: geomdescthcrtwist} later on. 
\end{rem}

We will now prove that the twist has several further properties. Note that the result holds regardless of Hypothesis
\ref{hypsing: siginforder}, which is assumed throughout this chapter.
\begin{thm}\label{thm: thcrtwistprops}
$B^{G,\mu}$ has the following properties:
\begin{itemize}
 \item[(i)] it is universally noetherian;
 \item[(ii)] it has GK dimension 2;
 \item[(iii)] it is Auslander-Gorenstein of dimension 2;
 \item[(iv)] it satisfies the Cohen-Macaulay property;
 \item[(v)] it is Koszul.
\end{itemize}
\end{thm}
\begin{proof}
For (ii), observe that Lemma \ref{lem: hilbseries} implies that $\text{GKdim}(B^{G,\mu})=\text{GKdim}(B)$, with the
latter being equal to 2 by \cite[Proposition 1.5]{artin1990twisted}. To prove (iv) we note that $A$ is Cohen-Macaulay
and therefore so is $B$ by \cite[Corollary 6.7]{levasseur1992some}. Applying Proposition \ref{prop: cohenmac}(i) shows
that $B^{G,\mu}$ is also Cohen-Macaulay. 

By Theorem \ref{thm: 4sklytwistprops}, $A^{G,\mu}$ has properties (i) and (iii), therefore it suffices to show that
these two properties are preserved by going to the factor ring. Lemma \ref{lem: twistthcr} shows that $\Theta_1$ and
$\Theta_2$ form a regular sequence of normal elements in $B^{G,\mu}$, in which case using \cite[Proposition
4.9(1)]{artin1999generic} shows that (i) is true. Applying \cite[Theorem 3.6(2)]{levasseur1992some} twice proves (iii).
Finally, $B$ is Koszul by \cite[Theorem 3.9]{stafford1994regularity}. Using Proposition \ref{prop: koszul} shows that
this is also true for $B^{G,\mu}$.
\end{proof}
\begin{rem}\label{rem: bgmunilp}
Whereas $B$ is a domain --- which follows from the irreducibility of $E$ --- $B^{G,\mu}$ contains nilpotent elements in
degree 1. Let $v=v_0-i v_1 - i v_2 -v_3$. One can compute by hand or using the computer program Affine that
\begin{equation}\label{eq: nilpotentelement}
v^2 = -\Theta_1-i f_{2}^{\mu}-i f_4^{\mu} - f_6^{\mu} = 0.
\end{equation}

Since $v^2$ belongs to the relations of $B^{G,\mu}$ (which are $G$-invariant), so must $(v^g)^2=(v^2)^g$ for all $g \in
G$. Thus the other elements in the same $G$-orbit as $v$ are also nilpotent. Curiously, these elements are linearly
independent and therefore they generate $A^{G,\mu}$ as an algebra. In that ring the square of each generator is central,
since only terms of the form $f_i^{\mu}$ vanish in \eqref{eq: nilpotentelement}.
\end{rem}

The fact that $B^{G,\mu}$ is not a domain illustrates that Zhang twists do not necessarily preserve this property (see
Remark \ref{rem: domain} for more discussion on this point).

While it may not be a domain, $B^{G,\mu}$ is prime -- this will be proved in Corollary \ref{cor: actuallyprime}. We are
not in a position to prove this yet, however we prove that $B^{G,\mu}$ is semiprime in the next result. The following
definition is needed beforehand: given a ring $R$ on which a finite group $G$ acts by ring automorphisms, we say that
$R$ has no \emph{additive $|G|$-torsion}\index{term}{additive $|G|$-torsion} if the $|G|r \neq 0$ for all $r \in R$.
\begin{prop}\label{prop: semiprime}
$B^{G,\mu}$ is semiprime. 
\end{prop}
\begin{proof}
As it is a matrix ring over a domain, $M_2(B)$ is certainly semiprime. Note that $\text{char}(k)\nmid |G|$, hence $|G|
\in k$ acts faithfully on $M_2(B)$ and thus there is no additive $|G|$-torsion. Since $B^{G,\mu}=M_2(B)^G$, one can
apply \cite[Corollary 1.5(1)]{montgomery1980fixd} to obtain the result.
\end{proof}

The next result of this section will show that the centre of $B^{G,\mu}$ is trivial. As a consequence we can determine
the centre of $A^{G,\mu}$ as well. To do so we will need to use the following property.
\begin{defn}[{\cite[Introduction]{rogalski2006proj}}]\label{defn: projsimple}
A $k$-algebra $A$ is \emph{projectively simple}\index{term}{projectively simple} if for any two-sided ideal of $A$, $I$
say, one has $\text{dim}_k(A/I) < \infty$.  
\end{defn}
The ring $B$ is projectively simple when $\sigma$ has infinite order by \cite[Proposition 0.1]{rogalski2006proj}. Let us
now determine the centre of $B^{G,\mu}$.
\begin{prop}\label{prop: centrethcrtwist}
The centre of $B^{G,\mu}$ is trivial, that is $Z(B^{G,\mu})=k$.
\end{prop}
\begin{proof}
Suppose that $f \in Z(B^{G,\mu})$. We can assume without loss of generality that $f$ is homogeneous with respect to both
the $\N$-grading and the $G$-grading: the homogeneous components (with respect to either grading) of a central element
must also be central. Assume further that it has strictly positive \N-degree. We can untwist the factor ring
$B^{G,\mu}/(f)$ to obtain $B/(f')$, where $f'$ is normal by Proposition \ref{prop: stillregular}. Moreover, $f'$ is
regular since $B$ is a domain, hence $f$ is regular in $B^{G,\mu}$ by the same result. One may therefore apply
\cite[Lemma 5.7]{levasseur1992some}, which implies that $\text{GKdim }B/(f')=\text{GKdim }B -1$. By \cite[Proposition
1.5]{artin1990twisted}, $B$ has GK dimension 2, hence $\text{GKdim }B/(f')=1$. However, $B$ is projectively simple as
noted prior to the proposition, therefore by definition any factor ring must be finite-dimensional. This contradiction
implies that $Z(B^{G,\mu})=k$.
\end{proof}
\begin{cor}\label{cor: centretwistskly}
The centre of $A^{G,\mu}$ is $Z(A^{G,\mu})=k[\Theta_1,\Theta_2]$.
\end{cor}
\begin{proof}
We will follow the method of \cite[Proposition 6.12]{levasseur1993modules}. Lemma \ref{lem: twistthcr} shows that
$B^{G,\mu}$ is obtained from $A^{G,\mu}$ by factoring out the ideal $(\Theta_1,\Theta_2)$, while Proposition \ref{prop:
centrethcrtwist} implies that $Z(B^{G,\mu})=k$. Combining these facts together enables one to conclude that
$Z(A^{G,\mu}) \subset k+(\Theta_1,\Theta_2)$.

Our next step is to show that $Z(A^{G,\mu}/(\Theta_1))=k[\Theta_2]$ where, once again, knowledge of the centre of
$B^{G,\mu}$ tells us that $Z(A^{G,\mu}/(\Theta_1)) \subset k + (\Theta_2)$. Suppose that the centre strictly contains
$k[\Theta_2]$ and choose a central element $w \notin k[\Theta_2]$ of minimal degree. We can write $w=\Theta_2 w'$ for
some $w'$ of strictly smaller degree. Appealing to Proposition \ref{prop: stillregular} we know that $\Theta_2$ is
regular and therefore $w'$ must also be central. Since it is of lower degree than $w$ we must have $w' \in k
[\Theta_2]$, which gives a contradiction.

Now let $z \in Z(A^{G,\mu})$ be homogeneous of minimal degree such that $z \notin k[\Theta_1,\Theta_2]$. By the previous
paragraph, $z=\Theta_1 y + f$ for some $y \in A^{G,\mu}$ and $f \in k[\Theta_2]$. Note that $\Theta_1 y$ and $f$ have
the same degree as $z$ if they are non-zero. Since $z$ and $f$ are central, $\Theta_1 y$ must be too. As argued in the
previous paragraph, $\Theta_1$ is regular which implies that $y$ is central as well. But $y$ has strictly lower degree
than $z$, contradicting our original assumption.
\end{proof}

Corollary \ref{cor: centretwistskly} allows us to show that $A^{G,\mu}$ is not isomorphic to one of the previously
discovered examples of AS-regular algebras with 20 point modules. Since we have used the computer program Affine in the
proof, the result holds only for generic parameters values.
\begin{thm}\label{thm: new}
Assume that a 4-dimensional Sklyanin algebra $A(\alpha,\beta,\gamma)$ is associated to the elliptic curve $E$ and
automorphism $\sigma$, with $|\sigma|=\infty$. Then for generic parameters $\alpha,\beta$ and $\gamma$,
$A(\alpha,\beta,\gamma)^{G,\mu}$ is not isomorphic to any of the previously studied examples in the literature of
AS-regular algebras of dimension 4 with 20 point modules.
\end{thm}
\begin{proof}
As remarked in \cite[\S 4]{lebruyn1995central}, graded Clifford algebras are finite over their centre. By Corollary
\ref{cor: centretwistskly} the centre of $A^{G,\mu}$ has GK dimension 2, and one can then see from \cite[Proposition
5.5]{krause2000growth} that $A^{G,\mu}$ cannot be finite over its centre. Thus $A^{G,\mu}$ is not a graded Clifford
algebra. It therefore suffices to consider the algebras studied in Examples 5.1 and 5.2 from
\cite{cassidy2010generlizations}, since they encompass the remaining examples in the literature that are not finite over
their centre (see \cite{shelton2001koszul}).

Let us begin with the algebras in \cite[Example 5.1]{cassidy2010generlizations}, letting $k$ be a field of
characteristic 0. Over $k$, the algebras in question are generated by four degree 1 elements $x_1,x_2,x_3$ and $x_4$,
subject to the defining relations 
\begin{equation*}%\label{eq: agammarelns}
\begin{array}{lll}
x_4x_1 - i x_1 x_4, &  x_3^2-x_1^2, & x_3x_1 - x_1x_3 + x_2^2, \\ \relax
x_3x_2 - i x_2 x_3, & x_4^2-x_2^2, & x_4x_2 - x_2x_4 + \gamma x_1^2,
\end{array}
\end{equation*}
where $i^2 =-1$ and $\gamma \in k^\times$. We follow Cassidy and Vancliff's notation and denote such an algebra by
$A(\gamma)$. If $\gamma = \pm 2$ then $A(\gamma)$ does not have 20 point modules, therefore let us assume that $\gamma
\neq \pm 2$. Computer calculations show that the elements $x_1^4$ and $x_2^4$ are central in $A(\gamma)$. Moreover,
there are no nontrivial central elements of smaller degree. By Corollary \ref{cor: centretwistskly} there are three
central elements of degree 4 in $A^{G,\mu}$, namely $\Theta_1^2$, $\Theta_2^2$ and $\Theta_1\Theta_2$. If there was an
$\N$-graded isomorphism $A(\gamma) \cong A^{G,\mu}$ then either $x_1^4$ or $x_2^4$ would map to some $\Theta_i^2$. But
there are no central elements of degree 2 in $A(\gamma)$, which is a contradiction.

It remains to deal with the algebras studied in \cite[Example 5.2]{cassidy2010generlizations}. Once again, such algebras
are generated over the base field $k$ by four degree 1 elements $x_1,x_2,x_3$ and $x_4$, subject to the relations 
\begin{equation*}%\label{eq: aeg5.2relns}
\begin{array}{lll}
x_3x_1 +  x_1 x_3 - \beta_2 x_2^2, &  x_4x_2 + x_2x_4 - x_3^2, & x_4x_1 + x_1x_4 - \alpha_2x_3^2, \\ \relax
\alpha_1 x_3^2 + \beta_1 x_2^2 - x_1^2, & x_2x_3-x_3x_2, & x_2^2 - x_4^2.
\end{array}
\end{equation*}
Here $\alpha_1,\alpha_2,\beta_1, \beta_2 \in k$, with these parameters satisfying the conditions 
\begin{equation*}%\label{eq: parametersgradedcliff}
\alpha_2 (\alpha_2 -1) = 0 \;\text{ and } \; (\alpha_1^2 + \alpha_2^2 \beta_1)(\beta_1^2 + \beta_2^2 \alpha_1) \neq 0.
\end{equation*}

One can use the computer program Affine to show that for any choice of parameters the following degree 4 elements are
central:
\begin{equation*}
x_1^4, \;\; x_3^4, \;\; x_1^2x_3^2 , \;\; (x_2x_1)^2+(x_1x_2)^2, \;\; (x_4x_3)^2+(x_3x_4)^2.
\end{equation*}
As for the previous class of examples, this contradicts the fact that $\text{dim}_kZ(A^{G,\mu})_4=3$.
\end{proof}

\section{Modules over the twist}\label{subsec: modulesthcrtwist}
We will continue our study of $B^{G,\mu}$ by showing that it does not possess any point modules. Further information
about 1-critical modules over the twist will be given in \S\ref{subsubsec: geomdescrthcrtwist} once we have described a
geometric construction of it.

Although $B$ has point modules parameterised by the elliptic curve $E$, when the associated automorphism $\sigma$ has
infinite order $B^{G,\mu}$ can have at most 20 point modules by Lemma \ref{lem: twistthcr} and Theorem \ref{prop:
finitepointscheme}. We will show in Proposition \ref{prop: bgmunoptmodules} that the point scheme of $B^{G,\mu}$ is
empty, although we need a useful, more general preliminary result. When applied to point modules it can be used to
understand the action given in \eqref{eq: Gactonpoints}. The lemma uses the notation we introduced in Notation \ref{not:
idealaut}.
\begin{lemma}\label{lem: annGinvariant}
Let $M$ be a module over an algebra $A$, on which a finite group $G$ acts by algebra automorphisms. Then
$\text{Ann}_A(M^g)=g^{-1}(\text{Ann}_A(M))$, and the ideal $\bigcap_{g \in G} \text{Ann}_A(M^g)$ is $G$-invariant. In
particular, if $G$ acts by graded automorphisms then the ideal consisting of elements which annihilate all point modules
over $A$ is $G$-invariant.
\end{lemma}
\begin{proof}
Let $a \in \text{Ann}_A(M^g)$. Then for all $m \in M$, $ma^g=0$, hence $a^g \in \text{Ann}_A(M)$. This implies that $a
\in g^{-1}(\text{Ann}_A(M))$. The other inclusion is proved by the reverse argument. By the first part of the result 
\begin{equation}\label{eq: annptmod}
\bigcap_{g \in G} \text{Ann}_A(M^g) = \bigcap_{g \in G} g^{-1}(\text{Ann}_A(M)).
\end{equation}
Writing the ideal in this manner makes it clear that it is preserved under the action of $G$. 

As we saw prior to Notation \ref{not: idealaut}, if $G$ acts by $\N$-graded algebra automorphisms and $M$ is a point
module over $A$, then $M^g$ is also a point module for all $g \in G$. The result is now clear.
\end{proof}

Having proved this general lemma we can now return to our standard notational assumptions, thus $A$ denotes a
4-dimensional Sklyanin algebra associated to the elliptic curve $E$ and automorphism of infinite order $\sigma$, while
$G=(C_2)^2$.

There are two ways to show that the point scheme of $B^{G,\mu}$ is empty. The first is by abstract arguments and the
second is more computational. We relegate the latter to Appendix \ref{app: calcptmodnonproof}. 
\begin{prop}\label{prop: bgmunoptmodules}
$B^{G,\mu}$ has no point modules.
\end{prop}
\begin{proof}
Let $\Gamma''$ denote the point scheme of $B^{G,\mu}$ and suppose that it is non-empty. By Theorem \ref{thm:
thcrtwistprops}(i), $B^{G,\mu}$ is strongly noetherian. We can therefore use \cite[Theorem 1.1]{rogalski2008canonical}
to conclude that there exists a graded homomorphism 
\begin{equation}\label{eq: canmaptothcr}
B^{G,\mu} \rightarrow B(\Gamma'',\mathcal{M},\phi),
\end{equation}
that is surjective in high degree, where $\phi$ is an automorphism of $\Gamma''$ and $\mathcal{M}$ is a $\phi$-ample
invertible sheaf. We know that the point scheme of $B^{G,\mu}$ is 0-dimensional, therefore by \cite[Proposition
1.5]{artin1990twisted} the GK dimension of $B(\Gamma'',\mathcal{M},\phi)$ must be 1.

Let $I$ denote the kernel of the map in \eqref{eq: canmaptothcr}. Since it is surjective in high degree there must exist
$n \in \N$ such that 
\begin{equation*}%\label{eq: canmaptothcrkernel}
\left(\frac{B^{G,\mu}}{I}\right)_{\geq n} \cong B(\Gamma'',\mathcal{M},\phi)_{\geq n},
\end{equation*}
as $k$-vector spaces. Thus $B^{G,\mu}/I_{\geq n}$ also has GK dimension 1. 

In addition to the existence of the homomorphism in \eqref{eq: canmaptothcr}, \cite[Theorem 1.1]{rogalski2008canonical}
states that $I_{\geq n}$ consists of the elements of $B^{G,\mu}$ that annihilate all of its point modules. By Lemma
\ref{lem: annGinvariant} this ideal is $G$-invariant, therefore one can twist the $G$-grading on $B^{G,\mu}/I_{\geq n}$
by the 2-cocycle $\mu$. As $\mu$ has order 2, upon twisting one obtains a factor ring $B/I'$ for some ideal $I'$. The
factor ring has GK dimension 1 by Proposition \ref{prop: gkdim}. 

By \cite[Theorem 4.4]{artin1995noncommutative}, ideals of $B$ correspond to $\sigma$-invariant
closed subschemes of $E$. Since $\sigma$ is given by translation by a point of infinite order on $E$, there are no
nontrivial such subschemes. Thus $I'$ must equal $0$ or have finite
codimension in $B$. In either case we get a contradiction: in the former because $B$ has GK dimension 2, and in the
latter because $B/I'$ would be finite-dimensional.
\end{proof}

We saw in Remark \ref{rem: bgmunilp} that $B^{G,\mu}$ contains zero divisors. The existence of such elements could also
be proved using the previous proposition and \cite[Thm 0.2]{artin1995noncommutative} as follows. The latter result
implies that if $B^{G,\mu}$ were a domain then there would be an equivalence of categories $\text{qgr}(B^{G,\mu}) \simeq
\text{coh}(Y)$ for some projective curve $Y$. But this would mean that $B^{G,\mu}$ had a family of pairwise
non-isomorphic point modules parameterised by $Y$, contradicting the conclusion of Proposition \ref{prop:
bgmunoptmodules}.

While $B^{G,\mu}$ does not possess any point modules when $\sigma$ has infinite order, it does have a family of fat
point modules of multiplicity 2 parameterised by $G$-orbits on $E$. As we will see in Proposition \ref{prop:
onlyfatpoints}, such modules arise as the restriction to $B^{G,\mu}$ of $M_2(B)$-modules of the form $M_p^2$, where
$M_p$ is a point module over $B$.

\section{Structure in relation to Artin-Stafford theory}\label{subsubsec: geomdescrthcrtwist}
In this section we discuss $B^{G,\mu}$ from the viewpoint of Artin and Stafford's classification of noncommutative
projective curves in \cite{artin1995noncommutative} and \cite{artin2000semiprime}. In the first of these papers the
authors classify connected graded $k$-algebras which are noetherian domains of quadratic growth using geometric
techniques. This classification is extended in the latter paper to semiprime algebras. 

By Theorem \ref{thm: thcrtwistprops} and Proposition \ref{prop: semiprime}, $B^{G,\mu}$ is a c.g.\ semiprime noetherian
algebra of GK dimension 2. Thus we can study it from the viewpoint of the classification in \cite{artin2000semiprime}.
Although a priori we only know that this ring is semiprime, it will be shown Corollary \ref{cor: actuallyprime} that it
is in fact a prime ring. Thus we will only describe the geometry needed in the prime case, since that will be sufficient
for our needs. 

\subsection{Geometric data and twisted rings}\label{subsec: geomdata}
Our first task is to define the geometric rings that feature in the classification. To this end, consider a projective
curve $Y$ over $k$, whose function field $K$ has transcendence degree 1 over the base field. As we will describe at the
end of \S\ref{subsec: geomdata}, it is perhaps more natural in our situation to obtain such a field $K$ as the graded
division ring of a prime noetherian algebra (along with some extra conditions) and then use \cite[Chapter I, Corollary
6.12]{hartshorne1977algebraic} to recover the curve $Y$.

Now let $T$ be a central simple $K$-algebra that is finite-dimensional over $K$. The structure sheaf $\calO_Y$ is a
subsheaf of the constant sheaf of rational functions on $Y$, namely $K$, therefore the sections of $\calO_Y$ lie inside
$K$. Since $K \subset T$ such sections are also contained in $T$ .
\begin{defn}[{\cite[pg. 75]{artin2000semiprime}}]\label{defn: lattice}
We say that $\mathcal{\calL}$ is an \emph{$\calO_Y$-lattice}\index{term}{o@$\calO_Y$-lattice} in $T$ if it is a sheaf of
finitely generated $\calO_Y$-submodules of $T$ for which $\calL K = T$. 
\end{defn}

When the term lattice is used in future we will mean the term in the sense of this definition. Given such a lattice, one
can define the \emph{left} and \emph{right orders}\index{term}{order associated to a lattice} of $\calL$
by\index{notation}{e@$E(\calL)$, $E'(\calL)$}
\begin{equation*}%\label{eq: leftrightorder}
E(\calL)=\{\alpha \in T:\alpha \calL \subseteq \calL\}\;\text{ and }\;E'(\calL)=\{\alpha \in T: \calL \alpha\subseteq
\calL\},
\end{equation*}
respectively. 

A lattice is said to be \emph{invertible}\index{term}{invertible lattice} if it is a locally free left $E(\calL)$-module
of rank 1. Lemma 1.10 from \cite{artin2000semiprime} implies that an invertible lattice is also a locally free right
$E'(\calL)$-module of rank 1. 

A crucial point for us is that for invertible lattices $\calL$ and $\mathcal{M}$ the product lattice $\calL\mathcal{M}$
--- where the product takes places inside $T$ --- is isomorphic to the tensor product $\calL \otimes_{E(\mathcal{M})}
\mathcal{M}$ if $E'(\calL)=E(\mathcal{M})$. Following the notation in \cite{artin2000semiprime}, we will denote such a
tensor product by $\calL \cdot \mathcal{M}$.

We can now introduce some more geometric objects that will be needed. Let $\mathcal{E}$ be a coherent sheaf of
$\calO_Y$-orders inside $T$, $\tau$ an automorphism of $T$, and $\mathcal{B}_1$ an invertible lattice in $T$ such that
$E(\mathcal{B}_1)=\mathcal{E}$ and $E'(\mathcal{B}_1)=\mathcal{E}^{\tau}$. One can then define a sequence of sheaves
$\{\mathcal{B}_n\}$ by $\mathcal{B}_n=\mathcal{B}_1 \otimes_{\mathcal{E}} \mathcal{B}_1^{\tau}
\otimes_{\mathcal{E}^{\tau}} \ldots  \otimes_{\mathcal{E}^{\tau^{n-1}}} \mathcal{B}_1^{\tau^{n-1}}$. The conditions on
$E(\mathcal{B}_1)$ and $E'(\mathcal{B}_1)$ imply that $\mathcal{B}_n=\mathcal{B}_1 \cdot \mathcal{B}_1^{\tau} \cdot
\ldots \cdot \mathcal{B}_1^{\tau^{n-1}}$. Using this data one can define a \emph{sheaf of bimodule
algebras}\index{term}{sheaf of bimodule algebras} $\mathbb{B}(\mathcal{E},\mathcal{B}_1,\tau)=\bigoplus_{n \in \N}
\mathcal{B}_n$\index{notation}{b@$\mathbb{B}(\mathcal{E},\mathcal{B}_1,\tau)$} with $\mathcal{B}_0=\mathcal{E}$ and
multiplication given by the formula $\mathcal{B}_i \cdot \mathcal{B}_j^{\tau^{i}} = \mathcal{B}_{i+j}$ for all $i, j
\geq 1$ (see \cite[pg. 102]{artin2000semiprime}).

There is a notion of ampleness for sequences of sheaves such as $\{\mathcal{B}_n\}$, and this condition plays the same
role for the twisted rings we will construct as $\sigma$-ampleness does for twisted homogeneous coordinate rings. While
one can make such a definition in the context of coherent left $\mathcal{E}$-modules, by \cite[Lemma
6.1]{artin2000semiprime} one can reduce to the more familiar definition stated below. 
\begin{defn}[{\cite[pg. 99]{artin2000semiprime}}]\label{def: amplebimodule}
Let $\{\mathcal{L}_n\}$ be a sequence of coherent $\mathcal{O}_Y$-modules. The sequence $\{\mathcal{L}_n\}$ is
\emph{ample}\index{term}{ample sequence of sheaves} if for all coherent sheaves $\mathcal{G}$ of $\mathcal{O}_Y$-modules
and all $n \gg 0$, the sheaf $\mathcal{G} \otimes_{\mathcal{O}_Y} \mathcal{L}_n$ is generated by global sections, and
$H^1(Y,\mathcal{G} \otimes_{\mathcal{O}_Y} \mathcal{L}_n)=0$.
\end{defn}

A priori we will not know that the sequence of sheaves we work with is ample, but this is not a problem as we discuss in
Remark \ref{rem: ample}.

Suppose now that $\mathcal{B}_1$ is an invertible lattice such that $E(\mathcal{B}_1)=\mathcal{E}$,
$E'(\mathcal{B}_1)=\mathcal{E}^{\tau}$, and the sequence of sheaves $\{\mathcal{B}_n\}$ defined above is ample in the
sense of Definition \ref{def: amplebimodule}. We say that $\mathcal{B}_1$ is an \emph{ample lattice}\index{term}{ample
lattice}. One does not need an ample lattice to construct the rings in the next definition, however this condition does
imply that the objects constructed are noetherian. The rings in question are formed by essentially taking global
sections of sheaves of bimodule algebras. In the definition we use the ideas discussed in Remark \ref{rem: finesse},
which allow us to view the twisted multiplicative structure in a more concrete setting.
\begin{defn}[{\cite[pgs. 103-104]{artin2000semiprime}}]\label{defn: twistedhomring}
Let $Y$, $\tau$, $\mathcal{E}$ and $\{\mathcal{B}_n\}$ be as above. The \emph{generalised twisted homogeneous coordinate
ring}\index{term}{generalised THCR|see{twisted ring}} associated to this data is the ring 
\begin{equation*}%\label{eq: twistedring}
B(\mathcal{E},\mathcal{B}_1,\tau)=\bigoplus_{n \in \N} H^0(Y,\mathcal{B}_n)z^n,  
\end{equation*}
whose the multiplication is induced by that in the corresponding sheaf of bimodule algebras, i.e.
$z\beta=\beta^{\tau}z$.
\end{defn}

Such rings were originally studied in \cite{van1996translation}, although Van den Bergh's definitions differ slightly.
For brevity we will refer to such rings as \emph{twisted rings}\index{term}{twisted ring} in future.

Artin and Stafford's main results need several technical hypotheses, however we can summarise them in a simplified
case. 
\begin{thm}[{cf. \cite[Theorems 0.3 and 0.5, Corollary 0.4]{artin2000semiprime}}]\label{thm: artinstaffordmain}
Suppose that $R$ is a semiprime noetherian c.g.\ algebra of GK dimension 2. Then there is a Veronese subring of $R$
which in high degree is a left ideal in a twisted ring for which $\mathcal{B}_1$ is an ample lattice. Moreover, this
twisted ring is a finite left module over the Veronese ring of $R$. 
\end{thm}

The geometric data needed to construct such a twisted ring from $R$ is obtained in part from certain divisors associated
to the graded components $R_n$, generalising the methods of \cite{artin1995noncommutative}. One piece of data that is
easy to recover is the curve $Y$, and we now indicate in the prime case how one can do this. Let $R$ be as in Theorem
\ref{thm: artinstaffordmain}, with the additional hypotheses of being prime and generated in degree 1. As we saw in
\S\ref{sec: goldietheory}, such a ring has a graded quotient ring $Q_{\text{gr}}(R)=M_n(D)[z,z^{-1};\tau]$. Define
$K:=Z(M_n(D))\cong Z(D)$. By \cite[Theorem 1.1]{artin1995noncommutative}, $K$ has transcendence degree 1 over $k$ when
the base field is algebraically closed. Using \cite[Chapter I, Corollary 6.12]{hartshorne1977algebraic} enables one to
associate $K$ to a unique smooth projective curve over $k$ by considering the set of all discrete valuation rings inside
$K$.

We will show that $B^{G,\mu}$ fits into the classification in a particularly nice way, in the respect that it is
\emph{isomorphic} to a twisted ring. Another nice feature in our situation is that $\mathcal{E}=\mathcal{E}^{\tau}$,
which is not true in general.

\subsection{A geometric description of the twist}\label{subsec: geomdesc}
To obtain a geometric description of $B^{G,\mu}$, we first go back to $B$. By \cite[Proposition 2.1]{smith1994center},
this ring embeds in the Ore extension $k(E)[z,z^{-1};\sigma]$, which is its graded quotient ring. The function field
$k(E)$ is the graded division ring of the homogeneous coordinate ring of the elliptic curve, defined in \eqref{eq:
thcroreextn}. 

The action of $G$ on $B$ extends in a natural way to an action on $k(E)[z,z^{-1};\sigma]$ by \Z-graded algebra
automorphisms, where $(ab^{-1})^g=a^g (b^g)^{-1}$ for all $a, b \in B$ with $b$ homogeneous and $g \in G$. This action
induces a $G$-grading on the localisation under which $B$ is a $G$-graded subring. 

As shown in \S\ref{sec: goldietheory}, one can replace the skew generator $z$ with another element in $k(E)z$ up to
changing $\sigma$ by a conjugation. In our situation $k(E)$ is a field and so any conjugation is trivial. Moreover, $B_1
\subset k(E)z$ and $x_0 \in B_1 \cap B_e$, hence we can assume that $z$ is fixed by the action of $G$ without changing
$\sigma$.

Note that the action of $G$ preserves the skew structure of $k(E)[z,z^{-1};\sigma]$. Indeed, for all $f \in k(E)$ and $g
\in G$ we have
\begin{equation*}
0=(zf-f^{\sigma}z)^g=z^gf^g-(f^{\sigma})^gz^g=zf^g-(f^{\sigma})^gz=((f^g)^{\sigma}-(f^{\sigma})^g)z.
\end{equation*}
We highlight for future reference that for all $f \in k(E)$ one has
\begin{equation}\label{eq: sigmagcommute}
(f^g)^{\sigma}=(f^{\sigma})^g. 
\end{equation}

As \cite[Chapter I, Theorem 4.4]{hartshorne1977algebraic} shows, an algebra automorphism on $k(E)$ induces a unique
automorphism on the curve itself. By \eqref{eq: sigmagcommute} the induced actions of $G$ and $\sigma$ must also commute
on $E$.
\begin{lemma}\label{lem: actionscoincide}
This action of $G$ on $E$ coincides with that induced by the action on point modules of $A$, given in \eqref{eq:
Gactonpoints}. 
\end{lemma}
\begin{proof}
Both actions are a consequence of the action of $G$ on $B$; one by localisation, the other through the action on ideals
defining points in $\mathbb{P}^3_k$. 
\end{proof}

This lemma, in conjunction with Proposition \ref{claim: fatpoints}, allows us to extend the conclusions of that
proposition to $\text{qgr}(A^{G,\mu})$.
\begin{cor}\label{cor: qgrisos}
Let $p, q \in E$. There is an isomorphism $\pi(M_p^2)\cong \pi(M_q^2)$ in $\text{qgr}(A^{G,\mu})$ if and only if $q \in
[p]$.
\end{cor}
\begin{proof}
By Corollary \ref{cor: fatpointisoclasses} one has $M_p^2 \cong M_{p^{g}}^2$ in $\text{grmod}(A^{G,\mu})$ for all $g \in
G$. Thus suppose that $\pi(M_p^2)\cong \pi(M_q^2)$ for some $p,q \in E$. As in the proof of Proposition \ref{claim:
fatpoints}, this implies the existence of $n \in \N$ such that $M_{p^{\sigma^{n}}}^2\cong M_{q^{\sigma^{n}}}^2$.
Corollary \ref{cor: fatpointisoclasses} then tells us that $\sigma^n(p)$ and $\sigma^n(q)$ lie in the same $G$-orbit. By
\eqref{eq: sigmagcommute} and Lemma \ref{lem: actionscoincide} the actions of $G$ and $\sigma$ on $E$ commute, therefore
$p$ and $q$ lie in the same $G$-orbit, which completes the proof.
\end{proof}

As $B$ is a $G$-graded subring of its graded quotient ring, we can twist both rings simultaneously and use the invariant
ring construction to see that
\begin{equation}\label{eq: twistthcrinvariant}
B^{G,\mu}=M_2(B)^G \hookrightarrow M_2(k(E)[z,z^{-1};\sigma])^G = M_2(k(E))^G[z,z^{-1};\tilde{\tau}],
\end{equation}
where $\tilde{\tau}$ denotes the induced action of $\sigma$. Here we have used the fact that $G$ acts trivially on $z$.

This embedding suggests that $M_2(k(E))^G$ plays a role in governing the underlying geometry of the algebra. Indeed,
this is the case as our main result shows. 
\begin{thm}\label{thm: geomdescthcrtwist}\index{term}{twisted ring}
There is an isomorphism of $k$-algebras
\begin{equation*}
B^{G,\mu} \cong B(\mathcal{E},\mathcal{B}_1,\tilde{\tau})=\bigoplus_{n \in \N} H^0(E^G,\mathcal{B}_n)z^n \subset
M_2(k(E))^G[z,z^{-1};\tilde{\tau}], 
\end{equation*}
for the following geometric data:
\begin{itemize}
 \item[(i)] an elliptic curve $E^G := E/G$, with $\tau$ the action induced by $\sigma$, along with the morphism of
curves $r:E \rightarrow E^G$ sending $p \mapsto [p]$;
 \item[(ii)] $\mathcal{E}=M_2(r_{\ast}\calO_E)^G$, a sheaf of $\mathcal{O}_{E^{G}}$-orders inside $M_2(k(E))^G$ on which
the automorphism $\tilde{\tau}$ acts. This automorphism restricts to the induced action of $\tau$ on $k(E^G)$;
 \item[(iii)] $\mathcal{B}_1=M_2(r_{\ast}\calL)^G$, an invertible lattice in $M_2(k(E))^G$, where $\calL$ is as defined
after \eqref{eq: 4sklyanincentre}.
\end{itemize}
\end{thm}

Since the proof of Theorem \ref{thm: geomdescthcrtwist} is quite technical, we will first prove several preliminary
lemmas. Note that the sheaf $\mathcal{E}$ is defined on an open set $U \subseteq E^G$, with $V=r^{-1}(U)$, by
\begin{equation*}
\mathcal{E}(U):=\left[M_2(r_{\ast}\calO_E)^G\right](U)=M_2((r_{\ast}\calO_E)(U))^G = M_2(\calO_E(V))^G. 
\end{equation*}
Sheaves of invariant objects are discussed in \cite[Exercises 2.2.14, 2.3.20 and 2.3.21]{liu2002algebraic} for example.

Before stating the first lemma, recall that $\sigma$ is given by translation by a point of infinite order on $E$.
\begin{lemma}\label{lem: ellcurve}
Define $E^G:= E/G$, the orbit space of $E$ under the action of $G$. Then $E^G$ is a smooth elliptic curve, with an
associated automorphism $\tau$ that is induced by $\sigma$. Furthermore, $\tau$ has infinite order and does not fix any
points.
\end{lemma}
\begin{proof}
The orbit space $E^G$ is the curve associated to the fixed field $k(E)^G$, which has transcendence degree 1 over $k$. In
particular we have $k(E^G)=k(E)^G$. As remarked after \eqref{eq: sigmagcommute}, the actions of $G$ and $\sigma$ commute
on $E$. One may therefore conclude that there is an induced action of $\sigma$ on $E^G$. Denote this map by $\tau$,
which is defined by $[p]^{\tau}:=[p^{\sigma}]$ for all $p \in E$, and suppose that $\tau$ has a fixed point $[p]$. Since
$[p]$ contains 4 or fewer points and $\sigma$ is given by translation, this implies that $\sigma$ has finite order which
is a contradiction. A similar argument proves that $\tau$ has infinite order. 

Let us now show that $E^G$ is smooth. The singular locus of $E^G$ must be finite and preserved by $\tau$. If it were
non-empty then this would imply that $\sigma$ has finite order, which is a contradiction. One can now apply Hurwitz's
Theorem \cite[Chapter IV, Corollary 2.4]{hartshorne1977algebraic} to see that $E^G$ has genus 0 or 1; if it had genus 0
then it would be birationally equivalent to \proj{k}{1}, whose automorphisms always fix at least one point. Thus $E^G$
must have genus 1 and hence be an elliptic curve. 
\end{proof}

Before our next lemma we need to define outer and $X$-outer actions of a group. Our statement of the latter property is
given for prime Goldie rings, in which case \cite[Examples 3.6 and 3.7]{montgomery1980fixd} allow us to give the
formulation below.
\begin{defn}\label{def: Xouter}
Let $G$ be a finite group acting by ring automorphisms on a ring $R$. We say that $G$ is \emph{outer}\index{term}{outer
action} if no nontrivial element of $G$ acts by conjugation by an element of $R$. If, in addition, we assume that $R$ is
prime Goldie, we say that $G$ is \emph{X-outer}\index{term}{X-outer action} on $R$ if it is outer when extended to the
Goldie quotient ring $Q(R)$.
\end{defn}
\begin{lemma}\label{lem: Xouter}
The group $G$ is outer on $M_2(k(E))$ and therefore X-outer on the ring $M_2(\calO_E(V))$ for each $G$-invariant affine
open set $V \subset E$. Consequently, $M_2(\calO_E(V))^G$ is prime.
\end{lemma}
\begin{proof}
Recall that the action of $G$ on $M_2(k)$ is given in \eqref{eq: matrixaction}. The action of $G$ on $A$ induces the
$G$-grading in \eqref{eq: 4sklyaningrading}. This means that the action of $G$ on the degree 1 component of the factor
ring $B=A/(\Omega_1,\Omega_2)$ affords the regular representation. Note that $k(E)$ is graded division ring of $B$ (by
\cite[Proposition 2.1]{smith1994center} for example). Combining these facts, we observe that each graded component of
$k(E)$ under the induced $G$-grading is non-empty. Hence we can find a non-zero element $y \in k(E)_{g_{2}}$, in which
case:
\begin{equation*}%\label{eq: conjtopleft}
\begin{pmatrix} y & 0 \\ 0 & 0 \end{pmatrix}^{g_{1}}=\begin{pmatrix} -y & 0 \\ 0 & 0 \end{pmatrix}.
\end{equation*}

Suppose that this action were given by conjugation by $\left( \begin{smallmatrix} a&b\\ c&d \end{smallmatrix} \right)
\in M_2(k(E))$:
\begin{gather}
\begin{aligned}\label{eq: conjg1calc}
\begin{pmatrix} -y & 0 \\ 0 & 0 \end{pmatrix} 
&= \frac{1}{ad-bc} \begin{pmatrix} a & b \\ c & d \end{pmatrix} \begin{pmatrix} y & 0 \\ 0 & 0 \end{pmatrix}
\begin{pmatrix} d & -b \\ -c & a \end{pmatrix}  \\
&= \frac{1}{ad-bc} \begin{pmatrix} ad y & -ab y \\ c d y & -bc y \end{pmatrix}.
\end{aligned}
\end{gather}
Certainly we would need $b=c=0$ for the correct entries of the matrix to vanish. But then \eqref{eq: conjg1calc} would
reduce to
\begin{equation}\label{eq: matricesneq}
\begin{pmatrix} -y & 0 \\ 0 & 0 \end{pmatrix} = \begin{pmatrix} y & 0 \\ 0 & 0 \end{pmatrix}, 
\end{equation}
which is a contradiction.

We must repeat these calculations for the other two non-identity elements in the group. For $g=g_2$ or $g_1g_2$, choose
a non-zero element $y' \in k(E)_{g_{1}}$. In both cases we have
\begin{equation*}%\label{eq: conjotherlements}
\begin{pmatrix} y' & 0 \\ 0 & 0 \end{pmatrix}^{g}= \begin{pmatrix} 0 & 0 \\ 0 & -y' \end{pmatrix}.
\end{equation*}

Using the conjugation in \eqref{eq: conjg1calc}, one can see that $a=d=0$ must hold in order that the correct entries of
the matrix vanish. However, regardless of the values of $b$ and $c$ chosen, the conjugation does not coincide with the
group action. Thus the actions of $g_2$ and $g_1g_2$ do not arise from conjugation in $M_2(k(E))$ either. 

Now consider a $G$-invariant affine open set $V \subset E$ and $M_2(\calO_E(V))$. As $\calO_E(V)$ is a domain it is
clear that $M_2(\calO_E(V))$ is prime. Furthermore, it has Goldie quotient ring $M_2(k(E))$, on which the action of $G$
is outer by the argument above. By Definition \ref{def: Xouter} the action of $G$ must be X-outer on $M_2(\calO_E(V))$,
whereupon we may apply \cite[Theorem 3.17(2)]{montgomery1980fixd} to show that $M_2(\calO_E(V))^G$ is a prime ring.
\end{proof}

We may now define the sheaf of orders which will be needed for the geometric description of $B^{G,\mu}$. 
\begin{lemma}\label{lem: maxorder}
Define $\mathcal{E}:=M_2(r_{\ast}\calO_E)^G$, which is considered as a subsheaf of the constant sheaf $M_2(k(E))^G$.
Then
\begin{itemize}
\item[(i)] the natural action of $\sigma$ on $M_2(k(E))$ restricts to an automorphism $\tilde{\tau}$ on $\mathcal{E}$,
with $\tilde{\tau}$ restricting to the induced action of $\tau$ on $k(E^G)$;
\item[(ii)] $\mathcal{E}$ is a sheaf of Dedekind prime rings\index{term}{Dedekind prime ring} and therefore a sheaf of
maximal orders\index{term}{maximal order}.
\end{itemize}
\end{lemma}
\begin{proof}
The automorphism $\sigma$ extends naturally from $k(E)$ to $M_2(k(E))$ under the trivial action on the matrix units.
Since this automorphism commutes with $G$, it restricts to an automorphism $\tilde{\tau}$ on $\mathcal{E}$. It follows
from the construction that the action of $\tilde{\tau}$ on $k(E)^G$ coincides with that of $\tau$.

To prove (ii), let $U \subset E^G$ be an affine open set with $V=r^{-1}(U)$. The open set $V \subset E$ is
$G$-invariant, hence $\mathcal{E}(U)= M_2(\calO_E(V))^G$ is prime by Lemma \ref{lem: Xouter}. As $G$ acts on
$\calO_E(V)$ by $k$-algebra automorphisms, one can view $\mathcal{E}(U)$ as a cocycle twist of the form
$\calO_E(V)^{G,\mu}$, where $G=(C_2)^2$ and $\mu$ is the same 2-cocycle as used to twist $B$. Consequently, we have the
results of Chapter \ref{chap: cocycletwists} at our disposal. In particular, $\calO_E(V)$ is noetherian and has global
dimension 1 since $E$ is a smooth curve, therefore we may apply Corollary \ref{cor: uninoeth} and Proposition \ref{prop:
gldim} to see that $\mathcal{E}(U)$ is a hereditary noetherian prime (HNP)\index{term}{HNP ring} ring.

We will now show that $\mathcal{E}(U)$ contains no idempotent maximal ideals, which by \cite[Proposition
5.6.10]{mcconnell2001noncommutative} will imply that it contains no nontrivial idempotent ideals. The existence of
idempotent maximal ideals is a local condition: given a sheaf $\mathcal{M}$ of maximal ideals of $\mathcal{E}$, the
sheaf $\mathcal{E}/\mathcal{M}$ is supported at a single point $p \in E^G$, and the stalk $\mathcal{M}_p$ is an
idempotent maximal ideal of $\mathcal{E}_p$ if and only if $\mathcal{M}(V)$ is an idempotent maximal ideal of
$\mathcal{E}(V)$ for an affine open set $V \ni p$.

Let $\bigcup_{i=1}^n U_i$ be an affine open cover of $E^G$ for some $n \in \N$, which gives rise to a corresponding
cover $\bigcup_{i=1}^n V_i$ of $E$ by $G$-invariant open sets with $V_i = r^{-1}(U_i)$. As the rings $\mathcal{E}(U_i)
\subset M_2(k(E))$ are PI and noetherian, they can have at most finitely many idempotent maximal ideals by \cite[Theorem
13.7.15]{mcconnell2001noncommutative}. By the remarks of the preceding paragraph, there can be at most finitely many
idempotent maximal ideals in the totality of the stalks of $\mathcal{E}$. We know that $\tau$ has no finite orbits on
$E^G$. This means that if a stalk $\mathcal{E}_p$ contains an idempotent maximal ideal $\mathcal{M}_p$, then there are
infinitely many other stalks with such an ideal, of the form $\tilde{\tau}^n(\mathcal{M}_p)$ for $n \in \N$. This is a
contradiction.

Now we may apply \cite[Theorem 5.6.3]{mcconnell2001noncommutative} to $\mathcal{E}(U)$: it contains no nontrivial
idempotent ideals and therefore it is a Dedekind prime ring. In particular, condition (ii) of Theorem 5.2.10 op. cit.
implies that $\mathcal{E}(U)$ is a maximal order.
\end{proof}

\begin{lemma}\label{lem: twistedring}
Define $\mathcal{B}_1:=M_2(r_{\ast}\calL)^G$ which, like $\mathcal{E}$, is considered as a subsheaf of the constant
sheaf $M_2(k(E))^G$. Then 
\begin{itemize}
 \item[(i)] $\mathcal{B}_1$ is an invertible $\calO_{E^G}$-lattice;
 \item[(ii)]  $\mathcal{E}$ is coherent sheaf of maximal $\calO_{E^G}$-orders for which
$\mathcal{E}=\mathcal{E}^{\tilde{\tau}}$;
 \item[(iii)] $\mathcal{E}= E(\mathcal{B}_1) = E'(\mathcal{B}_1)$.
\end{itemize}
\end{lemma}
\begin{proof}
We first show that $\mathcal{B}_1$ is an $\calO_{E^G}$-lattice. To do this we need to prove that $M_2(k(E))^G =
\mathcal{B}_1 \cdot k(E^G)$, where it suffices to work at the level of the global sections of $\mathcal{B}_1$. Observe
that the embedding in \eqref{eq: twistthcrinvariant} sends $B^{G,\mu}_1$ to $M_2(\text{H}^0(E,\calL))^Gz$, which is
precisely $\text{H}^0(E^G,\mathcal{B}_1)z$ since
\begin{equation*}%\label{eq: twistsections}
\text{H}^0(E^G,\mathcal{B}_1) =  \left[M_2(r_{\ast}\calL)^G\right](E^G) = M_2(\calL(E))^G=M_2(\text{H}^0(E,\calL))^G.
\end{equation*} 
Thus $\text{H}^0(E^G,\mathcal{B}_1)$ contains the elements
\begin{equation}\label{eq: sorted}
\begin{pmatrix} x_0 & 0 \\ 0 & x_0 \end{pmatrix}z^{-1},\begin{pmatrix} x_{1} & 0 \\ 0 &
 -x_{1} \end{pmatrix}z^{-1},\begin{pmatrix} 0 & x_{2} \\ x_{2} & 0 \end{pmatrix}z^{-1},\begin{pmatrix} 0 & -x_3 \\ x_{3}
& 0 \end{pmatrix}z^{-1},
\end{equation}
where $x_0,x_1,x_2$ and $x_3$ are the degree 1 generators of $B$. The elements in \eqref{eq: sorted} are linearly
independent over $k(E^G)$. But $M_2(k(E))^G$ is a 4-dimensional vector space over $k(E^G)$, which implies that
$M_2(k(E))^G = \mathcal{B}_1 \cdot k(E^G)$ must hold and therefore $\mathcal{B}_1$ is an $\calO_{E^G}$-lattice.

Notice that $\mathcal{B}_1$ is an $\mathcal{E}$-module on the right since
\begin{equation*}%\label{eq: sheafmodule}
M_2(r_*\mathcal{L})^G\cdot M_2(r_*\mathcal{O}_E)^G \subseteq 
 M_2(r_*\mathcal{L}\cdot r_*\mathcal{O}_E)^G \subseteq M_2(r_*\mathcal{L})^G. 
\end{equation*}
In particular, one has $\mathcal{E} \subseteq E'(\mathcal{B}_1)$. One may argue in a similar manner for left modules to
see that $\mathcal{E} \subseteq E(\mathcal{B}_1)$. Applying \cite[Lemma 3.1.12(i)]{mcconnell2001noncommutative} shows
that the orders $E(\mathcal{B}_1)$ and $E'(\mathcal{B}_1)$ are equivalent to $\mathcal{E}$. However, we showed in Lemma
\ref{lem: maxorder} that $\mathcal{E}$ is a sheaf of maximal orders, therefore we must have $\mathcal{E}=
E(\mathcal{B}_1) = E'(\mathcal{B}_1)$. By \cite[Lemma 1.10(i)]{artin2000semiprime}, $\mathcal{E}$ must be a coherent
sheaf of $\calO_{E^G}$-modules. The equality $\mathcal{E}=\mathcal{E}^{\tilde{\tau}}$ follows from Lemma \ref{lem:
maxorder}(i). 

Finally, we show that $\mathcal{B}_1$ is invertible over $\mathcal{E}$. Note that $\mathcal{B}_1$ is generated by its
global sections (since $\calL$ is), therefore $\mathcal{B}_1$ is a sheaf of f.g.\ modules over $\mathcal{E}$. Since
$\mathcal{E}$ is a sheaf of Dedekind prime rings by Lemma \ref{lem: maxorder}, \cite[Lemma
5.7.4]{mcconnell2001noncommutative} implies that $\mathcal{B}_1 \subset M_2(k(E))^G$ must be a sheaf of torsionfree,
projective modules contained in the Goldie quotient ring of $\mathcal{E}$. These facts imply that it must be an
invertible lattice.
\end{proof}

We are now in a position to prove Theorem \ref{thm: geomdescthcrtwist}.
\begin{proofof}{Theorem \ref{thm: geomdescthcrtwist}}
To begin, note that in Lemma \ref{lem: twistedring} we proved that $\mathcal{E}= E(\mathcal{B}_1) = E'(\mathcal{B}_1)$,
whereupon
\begin{equation}\label{eq: productsheaf}
\mathcal{B}_n = \mathcal{B}_1 \cdot \mathcal{B}_1^{\tilde{\tau}} \cdot \ldots \cdot \mathcal{B}_1^{\tilde{\tau}^{n-1}},
\end{equation}
for all $n \in \N$. As all sheaves involved lie inside $M_2(k(E))^G$, we may construct the twisted ring
$B(\mathcal{E},\mathcal{B}_1,\tilde{\tau})$ as a subring of the Ore extension $M_2(k(E))^G[z,z^{-1};\tilde{\tau}]$, as
described in the statement of the theorem.

The degree 1 piece of $B(\mathcal{E},\mathcal{B}_1,\tilde{\tau})$ is $\text{H}^0(E^G,\mathcal{B}_1)z =
M_2(\text{H}^0(E,\calL))^Gz$. Recall that in the proof of Lemma \ref{lem: twistedring} we observed that this is
precisely the image of $B_1^{G,\mu}$ under its embedding in the Ore extension. Remark \ref{rem: genindeg1Bgmu} tells us
that $B^{G,\mu}$ is generated in degree 1, therefore we obtain an embedding of $k$-algebras $B^{G,\mu} \hookrightarrow
B(\mathcal{E},\mathcal{B}_1,\tau)$. In particular, there is an injection of vector spaces $B_n^{G,\mu} \hookrightarrow
H_0(E^G,\mathcal{B}_n)z^n$ for all $n \in \N$. 

Let us now consider an open set $U \subset E^G$, with $V=r^{-1}(U)$. Since $\tau$ is induced by $\sigma$, we have
$r^{-1}(\tau^n(U)) = \sigma^n(V)$ for all $n \in \N$. Note that $\tilde{\tau}$ commutes with the action of $G$ on
$M_2(k(E))^G$, which is a consequence of $\sigma$ commuting with the group action. Observe finally that as $\mathcal{E}$
is a sheaf of maximal orders and $\mathcal{B}_1$ is an invertible lattice contained in $M_2(k(E))^G$, the tensor
products appearing in the definition of $\mathcal{B}_n$ are actual products inside $M_2(k(E))^G$. Thus
\begin{align*}%\label{eq: stupidbloodyequation}
\mathcal{B}_n(U) &= \mathcal{B}_1(U)\cdot \mathcal{B}_1^{\tilde{\tau}}(U) \cdot  \ldots \cdot
\mathcal{B}_1^{\tilde{\tau}^{n-1}}(U)\\
&=  \left[M_2(r_{\ast}\calL)^G\right](U) \cdot \left[M_2((r_{\ast}\calL)^{\tau})^G\right](U) \cdot \ldots \cdot 
\left[M_2((r_{\ast}\calL)^{\tau^{n-1}})^G\right](U) \\
&= M_2(\calL(V))^G \cdot M_2(\calL(\sigma(V)))^G \cdot \ldots \cdot M_2(\calL(\sigma^{n-1}(V)))^G\\
&\subseteq M_2(\calL(V)\calL(\sigma(V)) \ldots \calL(\sigma^{n-1}(V)))^G\\
&= M_2(\calL\calL^{\sigma} \ldots \calL^{\sigma^{n-1}}(V))^G\\
&= \left[M_2(r_{\ast}\calL_n)^G\right](U).
\end{align*}

We therefore have an inclusion $\mathcal{B}_n \hookrightarrow M_2(r_{\ast}\calL_n)^G$, and by taking global sections one
has $\text{H}^0(E^G,\mathcal{B}_n) \subset  M_2(H^0(E,\calL_n))^G$. But since the latter vector space is precisely
$B^{G,\mu}_n$, this implies that the injection $B_n^{G,\mu} \hookrightarrow H_0(E^G,\mathcal{B}_n)z^n$ must be a
bijection, which completes the proof.
\end{proofof}
\begin{rem}\label{rem: ample}
Although we did not show during the proof of Theorem \ref{thm: geomdescthcrtwist} that the sequence $\{\mathcal{B}_n\}$
is ample, it follows from that result as we now indicate. Since $B^{G,\mu}$ is semiprime noetherian of GK dimension 2,
the sheaves associated to its graded components form an ample sequence \cite[Proposition 6.4]{artin2000semiprime};
Theorem \ref{thm: geomdescthcrtwist} shows that these sheaves coincide with the sequence $\{\mathcal{B}_n\}$.
\end{rem}

Our first corollary is due to a result which generalises the categorical part of the Noncommutative Serre's theorem
(Theorem \ref{thm: noncomserrethm}).
\begin{cor}\label{cor: astfequivcat}
There is an equivalence of categories $\text{qgr}(B^{G,\mu}) \simeq \text{coh}(\mathcal{E})$, where
$\text{coh}(\mathcal{E})$\index{notation}{coh$(\mathcal{E})$} denotes the category of coherent sheaves over
$\mathcal{E}$.
\end{cor}
\begin{proof}
Combine Theorem \ref{thm: geomdescthcrtwist} with \cite[Corollary 6.11]{artin2000semiprime}.
\end{proof}

We also obtain another corollary, which shows that the twist is in fact a prime ring.
\begin{cor}\label{cor: actuallyprime}
$B^{G,\mu}$ is prime.
\end{cor}
\begin{proof}
By Theorem \ref{thm: geomdescthcrtwist} we know that the degree 0 component of the graded quotient ring of $B^{G,\mu}$
is $M_2(k(E))^G$. If we can show that this ring is simple artinian then $B^{G,\mu}$ must be prime. By Lemma \ref{lem:
Xouter} the action of $G$ on $M_2(k(E))$ is outer. As $M_2(k(E))$ is simple artinian, we may apply \cite[Theorem
2.7(1)]{montgomery1980fixd} to conclude that $M_2(k(E))^G$ is also simple artinian, which completes the proof.
\end{proof}

One can explicitly describe the structure of $M_2(k(E))^G$, as the following result shows. In the proof we use the
notion of \emph{PI degree}\index{term}{PI degree} defined in \cite[\S 13.6.7]{mcconnell2001noncommutative}.
\begin{cor}\label{cor: shapeofsimpleart}
There is an isomorphism of $k$-algebras $M_2(k(E))^G \cong M_2(k(E^G))$, where $k(E^G)=k(E)^G$.
\end{cor}
\begin{proof}
By Corollary \ref{cor: actuallyprime}, $M_2(k(E))^G$ is simple artinian and therefore isomorphic to a ring of the form
$M_n(D)$, where $D$ is some division ring. Considering the PI degree\index{term}{PI degree} of $M_2(k(E))^G$ enables one
to conclude that either $n=2$ and $D$ is a field or $M_2(k(E))^G$ is a division ring that is 4-dimensional over its
centre. We can rule out the latter case since $M_2(k(E))^G$ is the degree 0 component of the graded quotient ring of
$B^{G,\mu}$, which contains non-regular homogeneous elements by Remark \ref{rem: bgmunilp}. Thus $D$ must be a field,
and therefore the centre of $M_2(D)$ is isomorphic to $D$ itself. Since $k(E^G)$ is central in $M_2(k(E))^G$, it must
embed in $D$ under the isomorphism $M_2(k(E))^G \cong M_2(D)$. 

To complete the proof, observe that $M_2(k(E))$ is a 4-dimensional module over $M_2(k(E))^G$ (on either the left or the
right), which itself is 4-dimensional over its centre $Z(M_2(k(E))^G) \cong D$. However, it is clear that $M_2(k(E))$ is
16-dimensional over $k(E^G)$, in which case one must have $D=k(E^G)$. 
\end{proof}

Our work in the remainder of this section concerns the irreducible objects in $\text{qgr}(B^{G,\mu})$. We begin with a
preliminary lemma.
\begin{lemma}\label{lem: irrtail1crit}
The irreducible objects in $\text{qgr}(B^{G,\mu})$ are precisely the tails of 1-critical $B^{G,\mu}$-modules. 
\end{lemma}
\begin{proof}
Let us consider an irreducible object in $\text{qgr}(B^{G,\mu})$. Such an object will have the form $\pi(M)$ for some $M
\in \text{grmod}(B^{G,\mu})$. Theorem \ref{thm: thcrtwistprops}(ii) implies that $B^{G,\mu}$ has GK dimension 2, in
which case \cite[Proposition 5.1(d)]{krause2000growth} tells us that $\text{GKdim }M=0,1$ or 2. Note that $M$ cannot
have GK dimension 0 since $\pi(M)$ is nontrivial. Let us assume therefore that $\text{GKdim }M=2$. Since $\pi(M)$ is
irreducible this means that $M$ must be 2-critical. However, the only 2-critical $B^{G,\mu}$-module is $B^{G,\mu}$
itself, and $\pi(B^{G,\mu})$ is reducible in $\text{qgr}(B^{G,\mu})$. We reach a contradiction, upon which the result is
proved. 
\end{proof}

We can now give a more concrete description of the irreducible objects in $\text{qgr}(B^{G,\mu})$. 
\begin{prop}\label{prop: onlyfatpoints}
Any irreducible object in $\text{qgr}(B^{G,\mu})$ is isomorphic to $\pi(M_p^2)$ for some $p \in E$, where $M_p^2$ is a
direct sum of point modules over $B$. In particular, such modules are 1-critical and are therefore fat point modules of
multiplicity 2.
\end{prop}
\begin{proof}
By Lemma \ref{lem: irrtail1crit}, any irreducible object in $\text{qgr}(B^{G,\mu})$ has the form $\pi(N)$ for some
1-critical $B^{G,\mu}$-module $N$. Consider the extension $N \otimes_{B^{G,\mu}} M_{2}(B)$ as a right
$M_2(B)$-module, which is $\N$-graded via the standard grading on a tensor product of graded modules. One may argue in
the same manner as in the proof of Proposition \ref{prop: cohenmac}(i) to show that $N \otimes_{B^{G,\mu}} M_{2}(B)$ has
GK dimension 1. This $M_{2}(B)$-module is noetherian, hence there is a 1-critical $M_{2}(B)$-module, $I$ say, that
embeds in $\left(N \otimes_{B^{G,\mu}} M_{2}(B)\right)_{\geq n}$ for some sufficiently large $n \in \N$. One therefore
has $\pi(I) \hookrightarrow \pi(N \otimes_{B^{G,\mu}} M_{2}(B))$ in $\text{qgr}(M_2(B))$. The 1-critical modules over
$B$ are isomorphic in high degree to point modules by the equivalence of categories $\text{qgr}(B) \simeq
\text{coh}(\calO_E)$ from \cite[Theorem 1.3]{artin1990twisted}. Thus we may replace $\pi(I)$ with $\pi(M_p^2)$ for some
$p \in E$.

Let us now restrict from $\text{qgr}(M_2(B))$ to $\text{qgr}(B^{G,\mu})$. By Proposition \ref{prop: fflat} one has 
\begin{equation}\label{eq: restinject}
\pi(M_p^2) \hookrightarrow \pi(N \otimes_{B^{G,\mu}} M_2(B)) \cong \pi\left(N \otimes_{B^{G,\mu}} \left(\bigoplus_{g \in
G} {^{\text{id}}}(B^{G,\mu})^{\phi_g} \right)\right) \cong \pi\left(\bigoplus_{g \in G} N^{\phi_g}\right). 
\end{equation}   

The summands $N^{\phi_g}$ are 1-critical $B^{G,\mu}$-modules, and moreover they are the factors in a critical
composition series for $(N \otimes_{B^{G,\mu}} M_2(B))_{B^{G,\mu}}$. By \cite[Proposition 1.5]{smith1992the}, these
factors are unique up to permutation and isomorphism in high degree. However, we claim that the right $M_2(B)$-module
$M_p^2$ remains 1-critical upon restriction to a module over $B^{G,\mu}$, in which case $\pi(M_p^2) \cong
\pi(N^{\phi_{g}})$ for some $g \in G$. To prove this claim, observe that the module $M_p$ is also a point module over
$A$, thus $M_p^2$ can be considered as an $M_2(A)$-module that has been restricted first to $A^{G,\mu}$ and then to the
factor ring $B^{G,\mu}$. Thus Proposition \ref{prop: fatpoints} is applicable and implies that $M_p^2$ is a 1-critical
$B^{G,\mu}$-module. Since such a module has Hilbert series $2/(1-t)$, it must be a fat point module of multiplicity 2.

By untwisting one has $\pi((M_p^2)^{\phi_g^{-1}}) \cong \pi(N)$ in $\text{qgr}(B^{G,\mu})$. Note that the $G$-graded
automorphism $\phi_g$ extends naturally from $B^{G,\mu}$ to $M_2(B)$ via the $G$-grading on $B$ (i.e. the grading
induced by the action of $G$ on $B$ on which the action on $M_2(k)$ has no bearing). Thus we may consider
$(M_p^2)^{\phi_g^{-1}}$ as a right $M_2(B)$-module, in which case there exists $L \in \text{grmod}(B)$ such that
$(M_p^2)^{\phi_g^{-1}} \cong L^2$ in $\text{grmod}(B)$. By considering Hilbert series it is clear that $L$ must be a
point module. On restriction to $B^{G,\mu}$, one obtains an isomorphism $\pi(M_q^2) \cong \pi(N)$ in
$\text{qgr}(B^{G,\mu})$ for some $q \in E$. This completes the proof. 
\end{proof}
\begin{rem}\label{rem: restrictfatpoints}
As we saw in the proof of Proposition \ref{prop: onlyfatpoints}, one can consider the fat point modules of multiplicity
2 over $B^{G,\mu}$ as the restriction of the modules over $A^{G,\mu}$ constructed in Proposition \ref{claim: fatpoints}.
\end{rem}

Proposition \ref{prop: onlyfatpoints} enables us to describe the geometric origins of the fat point modules of
multiplicity 2 that we have studied, as the next result demonstrates.
\begin{prop}\label{prop: fatpointsincohE}
For $p \in E$ consider the coherent sheaf of right $M_2(\calO_E)$-modules $k_p^2$, where
$k_p:=\calO_E/\mathcal{I}_p$\index{notation}{k@$k_p$} is a skyscraper sheaf\index{term}{skyscraper sheaf} supported at
$p$. Under a natural functor $\phi: \text{coh}(M_2(\mathcal{O}_E)) \rightarrow \text{coh}(\mathcal{E})$, the irreducible
objects in $\text{coh}(\mathcal{E})$ are of the form $\phi(k_p^2)$. Furthermore, $\phi(k_p^2) \cong \phi(k_q^2)$ if and
only if $q \in [p]$.
\end{prop}
\begin{proof}
By \cite[Corollary 6.11]{artin2000semiprime} there is an equivalence of categories $\varphi_1:
\text{coh}(M_2(\mathcal{O}_E)) \rightarrow \text{qgr}(M_2(B))$. We label the equivalence from Corollary \ref{cor:
astfequivcat} by $\varphi_3: \text{qgr}(B^{G,\mu}) \rightarrow \text{coh}(\mathcal{E})$. Finally, let $\varphi_2$ denote
the functor from $\text{qgr}(M_2(B))$ to $\text{qgr}(B^{G,\mu})$ obtained by the restriction of modules. Consider the
diagram below:
\begin{equation}\label{eq: cohqgrdiag}
\centering
\begin{tikzpicture}
\draw [->] (1.25,0) -- (1.85,0); 
\node[below] at (1.55,0) {\tiny{$\sim$}};
\node[above] at (1.55,0) {\footnotesize{$\varphi_1$}};
\draw [dashed,->] (0,-0.25) -- (0,-1.25); 
\node[left] at (0,-0.75) {\footnotesize{$\phi$}};
\draw [->] (3,-0.25) -- (3,-1.25); 
\node[right] at (3,-0.75) {\footnotesize{$\varphi_2$}};
\draw [<-] (0.65,-1.5) -- (2,-1.5);
\node[above] at (1.325,-1.5) {\tiny{$\sim$}}; 
\node[below] at (1.325,-1.5) {\footnotesize{$\varphi_3$}};
\node at (0,0) {$\text{coh}(M_2(\mathcal{O}_E))$};
\node at (0,-1.5) {$\text{coh}(\mathcal{E})$};
\node at (3,0) {$\text{qgr}(M_2(B))$};
\node at (3,-1.5) {$\text{qgr}(B^{G,\mu})$};
\end{tikzpicture}
\end{equation} 
We denote the clockwise composition of functors beginning at $\text{coh}(M_2(\mathcal{O}_E))$ by $\phi$. 

Let us now consider an irreducible object in $\calF \in \text{coh}(\mathcal{E})$. By Proposition \ref{prop:
onlyfatpoints} we know that $\varphi_3(\varphi_2(\pi(M_p^2)))=\calF$ for some $p \in E$. One also has
$\varphi_1(k^2_p)=\pi(M_p^2)$, therefore $\phi(k_p^2)=\calF$. Since $\varphi_1$ and $\varphi_3$ are equivalences, the
isomorphisms among the sheaves $\phi(k_p^2)$  are governed by the restriction functor $\varphi_2$. Corollary \ref{cor:
qgrisos} describes the only such isomorphisms that can occur, and this completes the proof.
\end{proof}

As was mentioned in the Chapter \ref{chap: sklyanin}, if $|\sigma|= \infty$ then $A$ possesses four non-isomorphic fat
point modules of each multiplicity greater than or equal to 2 by \cite{smith1993irreducible}. The categorical
equivalence given by Theorem \ref{thm: noncomserrethm} shows that these are not $B$-modules. Thus the non-existence of
fat point modules over $B^{G,\mu}$ of multiplicity other than 2 by Proposition \ref{prop: onlyfatpoints} does not
preclude the existence of such modules over $A^{G,\mu}$.

To conclude this chapter we include one final lemma, whose significance will become apparent in \S\ref{sec: vancliffql}.
\begin{lemma}\label{lem: orbit4points}
For $p\in E$, $[p]$ contains precisely 4 points.
\end{lemma}
\begin{proof}
Suppose that $[p]$ contains fewer than 4 points. Then we must have $p^g=p^h$ for some
distinct $g,h \in G$. Looking at any pair of the four points in \eqref{eq: Gactonpoints}, one can see that they have the
same entry in precisely two coordinates, while the other two coordinates differ
by a minus sign. If such a pair defined the same point then at least two of the entries would have
to be zero. However, by Lemma \ref{lem: threegensannihilate} any point on $E$ must have at least 3 non-zero coordinates.
\end{proof}

\section{The Koszul dual}\label{par: kdual}\index{term}{Koszul dual}
In this short section we study the Koszul dual of $B^{G,\mu}$. We begin by noting that $B$ is Koszul by \cite[Theorem
3.9]{stafford1994regularity}, which allows us to apply Proposition \ref{prop: koszuldualtwistcommute} to $B$ and
$B^{G,\mu}$. That proposition tells us that there is an isomorphism of $k$-algebras
\begin{equation}\label{eq: thcrdualtwist}
(B^{G,\mu})^! \cong (B^!)^{G,\mu},
\end{equation}
since $\mu^2$ is the trivial 2-cocycle over $G$. 

The relations of $B^!$ are given in \cite[Lemma 2.2]{stafford1994regularity}. Using \eqref{eq: thcrdualtwist} we can
easily write down the relations in the twisted dual:
\begin{gather}
\begin{align*}%\label{eq: kdualrelns}
&[\overline{v_{0}},\overline{v_{1}}]+\frac{1}{\alpha}[\overline{v_{2}},\overline{v_{3}}],\;\;
[\overline{v_{0}},\overline{v_{1}}]_+ + [\overline{v_{2}},\overline{v_{3}}]_+,\;\;
[\overline{v_{0}},\overline{v_{2}}]+\frac{1}{\beta}[\overline{v_{3}},\overline{v_{1}}],\;\;
[\overline{v_{0}},\overline{v_{2}}]_+ + [\overline{v_{3}},\overline{v_{1}}]_+, \\
&[\overline{v_{0}},\overline{v_{3}}]-\frac{1}{\gamma}[\overline{v_{1}},\overline{v_{2}}],\;\;
[\overline{v_{0}},\overline{v_{3}}]_+ - [\overline{v_{1}},\overline{v_{2}}]_+, \;\; \left( \frac{1-\alpha}{1+\gamma}
-1\right)\overline{v_{0}}^2 +\left( \frac{1-\alpha}{1+\gamma} \right)\overline{v_{1}}^2+\overline{v_{3}}^2, \\
&\left( \frac{1+\alpha}{1-\beta} -1\right)\overline{v_{0}}^2 +\left( \frac{1+\alpha}{1-\beta}
\right)\overline{v_{1}}^2-\overline{v_{2}}^2.
\end{align*}
\end{gather}

Since $(B^!)^{G,\mu}$ has the same underlying $k$-vector space structure as $B^!$, $(B^{G,\mu})^!$ must have Hilbert
series $\left(\frac{1+t}{1-t}\right)^2$. Moreover, this implies that the 8 degree 2 elements corresponding to those
exhibited in \cite[Proposition 2.3]{stafford1994regularity} form a basis of $(B^{G,\mu})^!_2$. These elements are
\begin{gather}
\begin{aligned}\label{eq: kdualdeg2basis}
&t_1=\overline{v_{0}}^2,\;\;  t_2=\overline{v_{1}}^2, \;\;  t_3=[\overline{v_{0}}, \overline{v_{1}}], \;\;
t_4=[\overline{v_{0}}, \overline{v_{1}}]_+, \;\; t_5=[\overline{v_{0}}, \overline{v_{2}}], \\ &t_6=[\overline{v_{0}},
\overline{v_{2}}]_+, \;\; t_7=[\overline{v_{0}} ,\overline{v_{3}}], \;\;  t_8=[\overline{v_{0}}, \overline{v_{3}}]_+.
\end{aligned}
\end{gather}

Note that in the twist none of the scalar coefficients have changed from their untwisted counterparts. 

It is shown in \cite[Proposition 2.3(ii)]{stafford1994regularity} that each degree 2 element is normal or central in
$B^!$, and a description is given of how they commute with the generators $\overline{x_{i}}$. We will show that the
elements in \eqref{eq: kdualdeg2basis} are also either normal or central, but with different commutation rules. 

Let us adopt Stafford's notation. For a normal element $y\in (B^{G,\mu})^!$ and $i=0,1,2,3$ there exist $r_i \in
(B^{G,\mu})^!$ such that $\overline{v_{i}}y=yr_i$. This is encoded by
$(\overline{v_{0}},\overline{v_{1}},\overline{v_{2}},\overline{v_{3}})^y=(r_0,r_1,r_2,r_3)$.
\begin{prop}\label{prop: normalsinkdual}
The elements $t_1,t_2,t_4,t_6$ and $t_8$ are central in $(B^{G,\mu})^!$, while $t_3, t_5$ and $t_7$ are normal. For
$j=3,5,7$ we have
$(\overline{v_{0}},\overline{v_{1}},\overline{v_{2}},\overline{v_{3}})^{t_{j}}=(-\overline{v_{0}},-\overline{v_{1}},
-\overline{v_{2}},-\overline{v_{3}})$.
\end{prop}
\begin{proof}
The elements described in \cite[Proposition 2.3(ii)]{stafford1994regularity} are all normal and homogeneous with respect
to the $G$-grading on $B^!$. Lemma \ref{prop: stillregular} then implies that their twisted counterparts, exhibited in
\eqref{eq: kdualdeg2basis}, will be normal in $(B^{G,\mu})^!$. 

To see that the commutation rules are as described in the statement of the result, one simply computes using the
information from $B^!$. For example, the commutation rule for $t_3$ in $B^!$ is 
\begin{equation}\label{eq: t3commutation}
(\overline{x_{0}},\overline{x_{1}},\overline{x_{2}},\overline{x_{3}})^{t_3}=(-\overline{x_{0}},-\overline{x_{1}},
\overline{x_{2}},\overline{x_{3}}).
\end{equation}

Under the $G$-grading, $t_3$ lies in the component $B^!_{g_1}$. By regarding $g_1$ as an element of $kG_{\mu}$ one can
see that it satisfies the commutation rule
\begin{equation}\label{eq: normalkernel1}
(e,g_1,g_2,g_1g_2)^{g_{1}}= (e,g_1,-g_2,-g_1g_2).
\end{equation}
in $kG_{\mu}$. It is clear by combining the commutation rules in \eqref{eq: t3commutation} and \eqref{eq: normalkernel1}
that $t_3$ behaves as indicated in the statement of the result. Similar calculations for the remaining elements
completes the proof.
\end{proof}
\begin{cor}\label{cor: twistkdualPI}
$(B^{G,\mu})^!$ is PI.
\end{cor}
\begin{proof}
The Veronese ring $((B^{G,\mu})^!)^{(2)}$ is commutative (hence PI), which can be seen from the commutation rules in
Proposition \ref{prop: normalsinkdual}. Since $(B^{G,\mu})^!$ is a finite module over this Veronese subring, one can
then apply \cite[Corollary 13.4.9(i)]{mcconnell2001noncommutative} to reach the desired conclusion.
\end{proof}
As noted in \cite[\S 2.4]{stafford1994regularity} for $B$ and $B^!$, the algebras $B^{G,\mu}$ and $(B^{G,\mu})^!$ are
very different. The latter has a large centre, while the former has a trivial centre by Proposition \ref{prop:
centrethcrtwist}.

%%%%%%%%%%%%%%%%%%%%%%%%%%%%%%%%%%%%%%%%%%%%%%%%%%%%%%%%%%%%%%%%%%%%%

\chapter{Other twists}\label{chap: othertwists} 
\chaptermark{Other twists}
The examples of cocycle twists that we have studied thus far have all been either twists of the 4-dimensional Sklyanin
algebra or a factor ring of it. The goal of this section is to illustrate that there are other algebras that admit
interesting cocycle twists. 

The examples in \S\ref{subsec: staffordalgs} will concern the algebras defined by Stafford in
\cite{stafford1994regularity}. Such algebras are related to the ring $B$, a twist of which was studied in Chapter
\ref{chap: thcrtwist}. The algebras that we will study in \S\ref{subsec: rogzhangalgebras} are those that were first
introduced by Rogalski and Zhang in \cite{rogalski2012regular}. They are examples of AS-regular algebras of dimension 4
with three degree 1 generators. In \S\ref{sec: vancliffql} we will study cocycle twists of the algebras defined in
\cite{vancliff1994quadratic}. Such algebras depend on a nonsingular quadric $Q$ and a line $L$ in \proj{k}{3}, together
with an automorphism $\sigma$ of $Q \cup L$. The algebras twisted in \S\ref{subsec: gradedskewclifford} generalise
graded Clifford algebras, while our final examples in \S\ref{subsec: homenvelopalg} are twists of the universal
enveloping algebra of $\mathfrak{sl}_2(\C)$ and its homogenisation.

Many of the examples in this thesis are twists of AS-regular algebras of dimension 4, for which no classification exists
but many examples are known. Several families of such algebras have appeared in this thesis, for example 4-dimensional
Sklyanin algebras and their cocycle twists, while there are further examples given by double Ore extensions
\cite{zhang2008double}, algebras with an additional $\Z^2$-grading \cite{lu2007regular} (like Rogalski and Zhang's
algebras in \cite{rogalski2012regular} but with two degree 1 generators) and the generalised Laurent polynomial rings
studied in \cite{cassidy2006generalized}. 

One may wonder if there are any interesting twists of AS-regular algebras of smaller dimension. In dimensions 2 and 3
there exist classifications of such algebras (see \cite{artin1987graded} and
\cite{artin1990some,artin1991modules,stephenson1996artin} respectively), which are preserved under cocycle twists by
Corollary \ref{cor: asreg}. For such algebras we have been unable to find any cocycle twists that are not Zhang twists
of their \N-graded structure (in the manner of Proposition \ref{prop: recoverztwist} or the example given in Remark
\ref{rem: otheractions}). Certainly for AS-regular algebras of dimension 2 there are no other cocycle twists -- there
are only two families of such algebras up to isomorphism, with neither possessing enough $\N$-graded algebra
automorphisms to accommodate a twist that is not a Zhang twist of their $\N$-grading. 

It would be interesting to know if in the dimension 3 case any of the families in the classification are related by
cocycle twists. If that were the case then this behaviour would be similar to that which we will uncover in
\S\ref{subsec: rogzhangalgebras} for Rogalski and Zhang's algebras.

\section{Twists of Stafford's algebras}\label{subsec: staffordalgs}
\sectionmark{Stafford's algebras}
In this section we will study twists of the algebras introduced by Stafford in \cite{stafford1994regularity}. In that
paper it is assumed that $\text{char}(k)\neq 2$, thus we make the same assumption for the duration of \S\ref{subsec:
staffordalgs}.

We will first explain how Stafford's algebras were discovered, before showing that certain twists of them possess many
good properties. We then briefly discuss how the fat point modules of multiplicity 2 over $A^{G,\mu}$ described in
Proposition \ref{claim: fatpoints} are fat point modules over the twists of Stafford's algebras as well. Following this,
we describe the point scheme of one family of twists and conjecture what the point scheme is for the remaining algebras.

Let us begin by stating the question which motivated the work in \cite{stafford1994regularity}. Recall that $B$ is an
$\N$-graded factor ring of some 4-dimensional Sklyanin algebra $A$, thus there is a natural surjection $A
\twoheadrightarrow B$.
\begin{que}[{\cite[\S 0.2]{stafford1994regularity}}]\label{que: staffordasregexist}
Are there any other AS-regular algebras of dimension 4 which possess an $\N$-graded surjection onto $B$, or does this
property characterise the 4-dimensional Sklyanin algebra $A$? 
\end{que}

The defining relations of $B$ are given by the 8-dimensional space of quadratic elements 
\begin{equation}\label{eq: 8defeqns}
R_B=\text{span}_k(f_1,f_2,f_3,f_4,f_5,f_6,\Omega_1,\Omega_2),
\end{equation}
where the $f_i$ and $\Omega_i$ are given explicitly in \eqref{eq: 4sklyaninrelns} and \eqref{eq: 4sklyanincentre}
respectively. Any algebra answering Question \ref{que: staffordasregexist} would, for Hilbert series reasons, have
relations given by a 6-dimensional subspace of $R_B$. Stafford proved the existence of subspaces other than that spanned
by the $f_i$ with the required property. Thus the surjection $A \twoheadrightarrow B$ is not a uniquely defining
characteristic of the 4-dimensional Sklyanin algebra $A$ among AS-regular algebras of dimension 4. 

As we did in \S\ref{subsec: 4sklyanintwistandptscheme}, we will assume that the scalars $\alpha,\beta,\gamma \in k$
satisfy \eqref{eq: 4sklyanincoeffcond}. Let $V=\text{span}_k(x_0,x_1,x_2,x_3)$. The algebras defined in
\cite{stafford1994regularity} are presented by
\begin{equation}\label{eq: staffordinftyrelns}
S_{\infty}\index{notation}{s@$S_{\infty}$}:=T(V)/(R_{\infty})\text{, where }
R_{\infty}=\text{span}_k(f_1,f_2,f_3,f_4,\Omega_1,\Omega_2),
\end{equation}
and
\begin{equation}\label{eq: staffordlincombrelns}
S_{d,i}\index{notation}{s@$S_{d,i}$}:=T(V)/(R_{d,i})\text{, where } R_{d,i}=\text{span}_k(d_1\Omega_1+d_2\Omega_2,f_j:\;
1 \leq j \leq 6,\; j \neq i), 
\end{equation}
for some $1 \leq i \leq 6$ and $d=(d_1,d_2)\in \proj{k}{1}$. 

There are also two further variants of $S_{\infty}$, obtained by replacing the relations $f_i$ and $f_{i+1}$ in $A$ with
$\Omega_1$ and $\Omega_2$ for either $i=1$ or $i=3$; we denote these algebras by $S_{\infty,1,2}$ and $S_{\infty,3,4}$
respectively. In this notation one could therefore write $S_{\infty}=S_{\infty,5,6}$. Lemma \ref{lem: cyclicisostafford}
will allow us to dispense with this cumbersome notation.

With some restrictions on $d$, the algebras defined in \eqref{eq: staffordinftyrelns} and \eqref{eq:
staffordlincombrelns} are shown by Stafford in \cite[Theorem 0.4]{stafford1994regularity} to be noetherian domains which
are AS-regular of dimension 4. The method used to show this is to prove that their Koszul duals have certain properties
which have good consequences for the algebras themselves. 

We will twist the algebras in \eqref{eq: staffordinftyrelns} and \eqref{eq: staffordlincombrelns} in a manner that is
consistent with the twists of $A$ and $B$ that we have already studied. We will let $S=T(V)/(R_S)$ denote one of
Stafford's algebras if we do not wish to specify which family it belongs to. 

Consider the familiar action of $G=(C_2)^2 = \langle g_1, g_2 \rangle$ on $V$ from \eqref{eq: sklyaninactionIuse}. Under
the isomorphism $G \cong G^{\vee}$ defined by \eqref{eq: chartable}, this corresponds to the $G$-grading on generators
\begin{equation}\label{eq: sklyaninfullgrading}
x_0 \in S_{e},\;\; x_1 \in S_{g_{1}},\;\; x_2 \in S_{g_{2}},\;\; x_3 \in S_{g_{1}g_{2}}. 
\end{equation}

One can then use the 2-cocycle $\mu$ defined in \eqref{eq: mucocycledefn} to twist Stafford's algebras. Since we already
know the behaviour of the relations of $B$ under this twist, it is a simple matter to write down the new relations.
Recall that the elements $f_i^{\mu}$ and $\Theta_i$ are described explicitly in \eqref{eq: twistrelns} and \eqref{eq:
4sklyanintwistcentre} respectively.
\begin{lemma}\label{lem: stafftwists}
Twisting the algebras defined in \eqref{eq: staffordinftyrelns} and \eqref{eq: staffordlincombrelns} using the data
described above produces the following algebras. Firstly,
$S_{\infty}^{G,\mu}:=T(V)/(R_{\infty}^{\mu})$\index{notation}{s@$S_{\infty}^{G,\mu}$}, whose relations are given by the
ideal generated by 
\begin{equation}\label{eq: sinftymu}
R_{\infty}^{\mu} = \text{span}_k(f_1^{\mu},f_2^{\mu},f_3^{\mu},f_4^{\mu},\Theta_1,\Theta_2). 
\end{equation}
The other family of algebras is given by $S_{d,i}^{G,\mu}:=T(V)/(R_{d,i}^{\mu})$\index{notation}{s@$S_{d,i}^{G,\mu}$},
whose relations are given by
\begin{equation}\label{eq: sdimu}
R_{d,i}^{\mu} = \text{span}_k(d_1\Theta_1+d_2\Theta_2,f_j^{\mu}:\; 1 \leq j \leq 6,\; j \neq i), 
\end{equation}
for some $1 \leq i \leq 6$ and $d=(d_1,d_2)\in \proj{k}{1}$. 
\end{lemma}
\begin{proof}
Since we are using the same twisting data as in Chapters \ref{chap: sklyanin} and \ref{chap: thcrtwist}, we know the
behaviour under twisting of each of the eight relations in \eqref{eq: 8defeqns}. Explicitly, see equations \eqref{eq:
twistrelns} and \eqref{eq: 4sklyanintwistcentre}. An application of Lemma \ref{lem: defrelns} is needed to see that the
ideal of relations in each twist is generated by the elements we have indicated.
\end{proof}
\begin{rem}\label{rem: otherstaffalgs}
As mentioned previously, there are two further variants of $S_{\infty}$ which can be twisted in the same manner to form
the algebras $S_{\infty,1,2}^{G,\mu}$ and $S_{\infty,3,4}^{G,\mu}$. 
\end{rem}

Let us show immediately that, under some conditions on the parameters, such algebras have many good properties.
\begin{thm}\label{thm: stafftwistprops}
Assume that $\alpha,\beta$ and $\gamma$ satisfy \eqref{eq: 4sklyanincoeffcond}. Let $S=S_{d,i}$ for some $d=(d_1,d_2)\in
\proj{k}{1}$ and $1 \leq i \leq 2$ or let $S=S_{\infty}$. If $S=S_{d,1}$, assume that $d \neq (1,0),(1,-1-\beta \gamma)$
and if $S=S_{d,2}$ assume that $d \neq (1,\beta-1),(1,-1-\gamma)$. Then $S^{G,\mu}$ has the following properties:
\begin{itemize}
 \item[(i)] it is universally noetherian domain;
 \item[(ii)] it is an AS-regular algebra of dimension 4;
 \item[(iii)] it has Hilbert series $1/(1-t)^4$;
 \item[(iv)] it is Koszul;
 \item[(v)] it is Auslander regular and satisfies the Cohen-Macaulay property.
\end{itemize}
\end{thm}
\begin{proof}
Firstly, note that $S$ surjects onto $B$, with the kernel being generated by a regular sequence of normal elements. This
is a consequence of Proposition 2.10, Lemma 2.14 and then Theorem 1.3(v) from \cite{stafford1994regularity} (see the
remarks at the beginning of the proof of the latter result). By Theorem \ref{thm: thcrtwistprops}(i), $B$ is universally
noetherian. Applying \cite[Proposition 4.9(1)]{artin1999generic} twice shows that $S$ is also universally noetherian,
and using Corollary \ref{cor: uninoeth} then gives a proof of part of (i).

Let us now move on to properties (ii)-(iv). The algebra $S$ has these properties by \cite[Theorem
0.4]{stafford1994regularity}. That (ii)-(iv) are true for $S^{G,\mu}$ then follows from Lemma \ref{lem: hilbseries},
Corollary \ref{cor: asreg} and Proposition \ref{prop: koszul}. One can then apply \cite[Theorem 3.9]{artin1991modules}
to conclude that $S^{G,\mu}$ is also a domain, completing the proof of (i). To show that $S^{G,\mu}$ has the properties
in (v) we can apply \cite[Corollary 6.7]{levasseur1992some} to $S$, followed by Proposition \ref{prop: cohenmac}.
\end{proof}

Although Theorem \ref{thm: stafftwistprops} is only stated for $S_{\infty}^{G,\mu}$ and $S_{d,i}^{G,\mu}$ for $1 \leq i
\leq 2$, it remains valid for $S_{\infty,1,2}^{G,\mu}$, $S_{\infty,3,4}^{G,\mu}$ and $S_{d,i}^{G,\mu}$ for $3 \leq i
\leq 6$. To see this, note that cyclically permuting the generators $x_1,x_2$ and $x_3$ of $S_{\infty,i,i+1}$ or
$S_{d,i}$ defines an isomorphism with an algebra in the family $S_{\infty,i+2,i+3}$ or $S_{d,i+2}$ respectively (see the
remarks after \cite[Theorem 0.4]{stafford1994regularity}). The indices here are considered mod 6. We can use the same
ideas to find similar relations among the cocycle twists, as the following lemma shows.
\begin{lemma}\label{lem: cyclicisostafford}
Cyclically permuting the generators $v_1,v_2$ and $v_3$ gives the following isomorphisms:
\begin{itemize}
\item[(i)] $S_{\infty,5,6}^{G,\mu}(\alpha, \beta, \gamma) \cong S_{\infty,1,2}^{G,\mu}(\gamma, \alpha, \beta) \cong
S_{\infty,3,4}^{G,\mu}(\beta, \gamma, \alpha)$;
\item[(ii)] $S_{d,1}^{G,\mu}(\alpha, \beta, \gamma) \cong S_{d',3}^{G,\mu}(\gamma, \alpha, \beta) \cong
S_{d'',5}^{G,\mu}(\beta, \gamma, \alpha)$ where $d'=(\left(\frac{1+\gamma}{1-\alpha}\right)d_1,d_2)$ and
$d''=(\left(\frac{1-\beta}{1+\alpha}\right)d_1,d_2)$;
\item[(iii)] $S_{d,2}^{G,\mu}(\alpha, \beta, \gamma) \cong S_{d',4}^{G,\mu}(\gamma, \alpha, \beta) \cong
S_{d'',6}^{G,\mu}(\beta, \gamma, \alpha)$ where $d'$ and $d''$ are as in (ii).
\end{itemize}
\end{lemma}
\begin{proof}
We give a proof for (i) and remark that the remaining cases are similar. Since we will need to distinguish between
different parameters, we introduce the notation $f_{i,(\alpha,\beta,\gamma)}^{\mu}$ to indicate the parameters
associated to a particular relation.

The ideal of relations in $S_{\infty,5,6}^{G,\mu}(\alpha, \beta, \gamma)$ is given in \eqref{eq: sinftymu}. Now permute
the generators using the permutation $(123)$ on indices, then apply the automorphism $\varphi_{1}$ which sends $v_1
\mapsto -i v_1$ and $v_3 \mapsto i v_3$, where $i^2 =-1$. The following equations indicate the relations obtained:
\begin{gather}
\begin{aligned}\label{eq: permute1relns1}
\begin{array}{ccccccl}
f_{1,(\alpha,\beta,\gamma)}^{\mu}    &\overset{(123)}{\rightsquigarrow}  &[v_0,v_2]-\alpha[v_3,v_1]    
&\overset{\varphi_{1}}{\rightsquigarrow} &[v_0,v_2]-\alpha[v_3,v_1] &=  &f_{3,(\gamma,\alpha,\beta)}^{\mu},\\
f_{3,(\alpha,\beta,\gamma)}^{\mu}     &\rightsquigarrow  &[v_0,v_3]-\beta[v_1,v_2]      &\rightsquigarrow  
&i[v_0,v_3]+i\beta[v_1,v_2] &= & i f_{5,(\gamma,\alpha,\beta)}^{\mu},\\
f_{2,(\alpha,\beta,\gamma)}^{\mu}     &\rightsquigarrow  &[v_0,v_2]_{+} -[v_3,v_1]_{+}  &\rightsquigarrow 
&[v_0,v_2]_{+} -[v_3,v_1]_{+} &= & f_{4,(\gamma,\alpha,\beta)}^{\mu},\\
f_{4,(\alpha,\beta,\gamma)}^{\mu}    &\rightsquigarrow  &[v_0,v_3]_{+} -[v_1,v_2]_{+}  &\rightsquigarrow
&i[v_0,v_3]_{+}+i[v_1,v_2]_{+} &= & i f_{6,(\gamma,\alpha,\beta)}^{\mu},\\
\Theta_{1,(\alpha,\beta,\gamma)}     &\rightsquigarrow & -v_0^2+v_2^2+v_3^2-v_1^2      &\rightsquigarrow 
&-v_0^2+v_2^2-v_3^2+v_1^2 &= &\Theta_{1,(\gamma,\alpha,\beta)}.
\end{array}
\end{aligned}
\end{gather}
Under the same composition of maps, the relation $\Theta_{2,(\alpha,\beta,\gamma)}$ is sent to 
\begin{equation}\label{eq: permute1relns2}
 v_2^2-\left(\frac{1+\alpha}{1-\beta} \right)v_3^2+\left(\frac{1-\alpha}{1+\gamma} \right)v_1^2 = 
\left(\frac{1-\alpha}{1+\gamma} \right) \Theta_{2,(\gamma,\alpha,\beta)},
\end{equation}
where we have used \eqref{eq: 4sklyanincoeffcond1} to obtain the equality. The relations obtained in \eqref{eq:
permute1relns1} and \eqref{eq: permute1relns2} are those in $S_{\infty,1,2}^{G,\mu}(\gamma, \alpha, \beta)$, which
proves the first part of (i). 

We proceed in a similar manner to prove the second part of (i). We first apply the permutation $(132)$ to the
generators, then apply the automorphism $\varphi_{2}$ which sends $v_2 \mapsto - i v_2$ and $v_3 \mapsto i v_3$, where
$i^2 =-1$ once again. The relations in \eqref{eq: sinftymu} are transformed under this composition as indicated below:
\begin{gather}
\begin{aligned}\label{eq: permute2relns1}
\begin{array}{ccccccl}
f_{1,(\alpha,\beta,\gamma)}^{\mu}    &\overset{(132)}{\rightsquigarrow}  &[v_0,v_3]-\alpha[v_1,v_2]    
&\overset{\varphi_{2}}{\rightsquigarrow} & i[v_0,v_3]+ i\alpha[v_1,v_2] &=  &if_{5,(\beta,\gamma,\alpha)}^{\mu},\\
f_{3,(\alpha,\beta,\gamma)}^{\mu}     &\rightsquigarrow  &[v_0,v_1]-\beta[v_2,v_3]      &\rightsquigarrow  
&[v_0,v_1]-\beta[v_2,v_3] &= & f_{1,(\beta,\gamma,\alpha)}^{\mu},\\
f_{2,(\alpha,\beta,\gamma)}^{\mu}     &\rightsquigarrow  &[v_0,v_3]_{+} -[v_1,v_2]_{+}  &\rightsquigarrow 
&i[v_0,v_3]_{+} +i[v_1,v_2]_{+} &= & if_{6,(\beta,\gamma,\alpha)}^{\mu},\\
f_{4,(\alpha,\beta,\gamma)}^{\mu}    &\rightsquigarrow  &[v_0,v_1]_{+} -[v_2,v_3]_{+}  &\rightsquigarrow
&[v_0,v_1]_{+}-[v_2,v_3]_{+} &= & f_{2,(\gamma,\alpha,\beta)}^{\mu},\\
\Theta_{1,(\alpha,\beta,\gamma)}     &\rightsquigarrow & -v_0^2+v_3^2+v_1^2-v_2^2      &\rightsquigarrow 
&-v_0^2-v_3^2+v_1^2+v_2^2 &= &\Theta_{1,(\beta,\gamma,\alpha)}.
\end{array}
\end{aligned}
\end{gather}
Under the same maps the relation $\Theta_{2,(\alpha,\beta,\gamma)}$ is sent to 
\begin{equation}\label{eq: permute2relns2}
 -v_3^2+\left(\frac{1+\alpha}{1-\beta} \right)v_1^2+\left(\frac{1-\alpha}{1+\gamma} \right)v_2^2 = 
\left(\frac{1+\alpha}{1-\beta} \right) \Theta_{2,(\beta,\gamma,\alpha)},
\end{equation}
where we have used \eqref{eq: 4sklyanincoeffcond1} to obtain the equality. The relations obtained in \eqref{eq:
permute2relns1} and \eqref{eq: permute2relns2} are those in $S_{\infty,3,4}^{G,\mu}(\beta,\gamma, \alpha)$, which
completes the proof.
\end{proof}

By Lemma \ref{lem: cyclicisostafford} one obtains the following corollary.
\begin{cor}\label{cor: staff3456props}
Under conditions on $d$ determined by the isomorphisms in Lemma \ref{lem: cyclicisostafford}, the properties listed in
Theorem \ref{thm: stafftwistprops} also hold for the algebras $S_{d,i}^{G,\mu}$ for $3 \leq i \leq 6$. The same is true
for $S_{\infty,1,2}^{G,\mu}$ and $S_{\infty,3,4}^{G,\mu}$ also.
\end{cor}

In light of the isomorphisms given in Lemma \ref{lem: cyclicisostafford}, we will assume now that $S$ is subject to the
hypotheses of Theorem \ref{thm: stafftwistprops}.

We chose the $G$-grading on $S$ given in \eqref{eq: sklyaninfullgrading} in order that it be compatible with the
$G$-grading on $B$. A consequence of this is that there is a graded surjection $S^{G,\mu} \twoheadrightarrow B^{G,\mu}$.
As for the Sklyanin algebra, we have three families of algebras which surject onto the same geometric ring; instead of
this being a twisted homogeneous coordinate ring it is now $B^{G,\mu}$, which was shown in Theorem \ref{thm:
geomdescthcrtwist} to be a twisted ring. The situation is illustrated by Figure \ref{fig: sklyanintwists}.
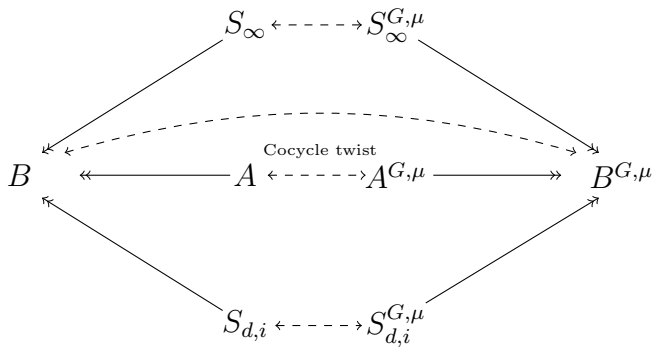
\begin{figure}
\centering
\begin{tikzpicture}
\node at (-1,0) {$A$};
\node at (1,0) {$A^{G,\mu}$};
\node at (-1,2) {$S_{\infty}$};
\node at (1,2) {$S_{\infty}^{G,\mu}$};
\node at (-1,-2) {$S_{d,i}$};
\node at (1,-2) {$S_{d,i}^{G,\mu}$};
\draw [<->,dashed] (-0.65,2) -- (0.55,2);
\draw [<->,dashed] (-0.7,0) -- (0.6,0);
\draw [<->,dashed] (-0.6,-2) -- (0.55,-2);
\node at (0,0.32) {\tiny{Cocycle twist}};
\node at (-4,0) {$B$};
\node at (4,0) {$B^{G,\mu}$};
%\node at (6.75,0) {$=\frac{k\{v_0,v_1,v_2,v_3\}}{(f_i^{\mu};_{\mbox{\tiny i=1,\ldots,6}}, \Theta_1,\Theta_2)}$};
\draw [->>] (1.3,1.8) -- (3.7,0.3);
\draw [->>] (1.5,0) -- (3.2,0);
\draw [->>] (1.4,-1.7) -- (3.7,-0.3);
\draw [->>] (-1.3,1.8) -- (-3.7,0.3);
\draw [->>] (-1.2,0) -- (-3.2,0);
\draw [->>] (-1.3,-1.8) -- (-3.7,-0.3);
\draw [<->,dashed] (-3.4,0.3) .. controls (-1,1) and (1,1) .. (3.4,0.3); 
\end{tikzpicture}
\caption{The Sklyanin algebra and related algebras/twists}
\label{fig: sklyanintwists}
\end{figure}
\begin{rem}\label{rem: staffordpermisos}
In \S\ref{subsec: permuteaction} we studied cocycle twists of $A$ obtained by using the 24 possible actions of $G$ by
the regular representation on the standard generators. All such actions also respect the relations in $S_{\infty}$ and
$S_{d,i}$, thus one could study the associated cocycle twists. We remark that it was shown in the proof of Proposition
\ref{lem: 24to1} that $\text{Aut}_{\text{grp}}(G)$ stabilises $\mu$, hence in each case there are four isomorphism
classes of cocycle twists. As in \S\ref{subsec: permuteaction}, these are governed by the generator that belongs to the
identity component of the induced $G$-grading.
\end{rem}

We will now consider fat point modules over the algebras we have just constructed. The situation encapsulated by Figure
\ref{fig: sklyanintwists} implies that the point schemes of the algebras $A$, $S_{\infty}$ and $S_{d,i}$ all contain the
elliptic curve $E$. Note that the full point schemes of the latter two algebras are given by \cite[Lemma
2.16]{stafford1994regularity}. In contrast, the surjections from $S_{\infty}^{G,\mu}$ and $S_{d,i}^{G,\mu}$ onto
$B^{G,\mu}$ do not give us any information regarding point modules when $|\sigma|=\infty$. This is because the point
scheme of $B^{G,\mu}$ was proved to be empty under that hypothesis in Proposition \ref{prop: bgmunoptmodules}. 

However, the surjections onto $B^{G,\mu}$ do have the following consequence. The fat point modules of multiplicity 2
over $B^{G,\mu}$ described in Proposition \ref{prop: onlyfatpoints} are fat point modules over the algebras
$S_{\infty}^{G,\mu}$ and $S_{d,i}^{G,\mu}$. Recall that these modules can be obtained by restricting $M_2(S)$-modules of
the form $M_p^2$ for $p \in E$ to $S^{G,\mu}$.

Let us now restrict our attention to point modules and describe the point scheme of algebras in the family
$S_{\infty}^{G,\mu}$. As we did for $A^{G,\mu}$, we will study the graph of the point scheme under the associated
automorphism. Unlike in that case, however, we cannot give an explicit description of all of the points, merely their
existence. Computer calculations suggest that there are 20 distinct points of multiplicity 1, as was the case for
$A^{G,\mu}$. 

We will prove the following proposition through a series of lemmas.
\begin{prop}\label{prop: staffordptschemes}
Assume that $\text{char}(k) =0$. For generic parameters $\alpha, \beta$ and $\gamma$ the point scheme\index{term}{point
scheme!calculations of} $\Gamma''$ of $S_{\infty}^{G,\mu}$ consists of 20 points up to multiplicity. 
\end{prop}
The hypotheses of Proposition \ref{prop: staffordptschemes} are needed due to the use of computer calculations in the
proof.

The first lemma allows us to focus on the multilinearisations of the defining relations of $S_{\infty}^{G,\mu}$. One
should recall the definition of the points $e_j \in \proj{k}{3}$ from \eqref{eq: 4sklyanineipts}.
\begin{lemma}\label{lem: inftyptscheme1}
Let $\Gamma_2$ denote the vanishing locus of the multilinearisations\index{term}{multilinearisations} of the quadratic
relations of $S_{\infty}^{G,\mu}$. Then $\Gamma_2$ is the graph of $\Gamma''$ under the associated automorphism $\tau$.
Moreover, $\tau$ has order 2.
\end{lemma}
\begin{proof}
By Theorem \ref{thm: stafftwistprops}, $S_{\infty}^{G,\mu}$ satisfies the hypotheses of Theorem \ref{thm:
pointschemenice}. Applying that theorem implies the first part of the lemma. Consider the multilinearisations of the
quadratic relations of $S_{\infty}^{G,\mu}$:
\begin{gather}
\begin{aligned}\label{eqn: multilinsquasiprojsinfty1}
&m_1:= v_{01}v_{12}- v_{11}v_{02} -\alpha v_{21}v_{32}+\alpha v_{31}v_{22},\\ 
&m_2:=v_{01}v_{12}+v_{11}v_{02}-v_{21}v_{32}-v_{31}v_{22}, \\
&m_3:=v_{01}v_{22}-v_{21}v_{02}+\beta v_{11}v_{32}-\beta v_{31}v_{12},\\ 
&m_4:=v_{01}v_{22}+v_{21}v_{02} -v_{31}v_{12}-v_{11}v_{32}, \\
&m_5:=-v_{01}v_{02} + v_{11}v_{12}+v_{21}v_{22}- v_{31}v_{32}, \\ 
&m_6:=v_{11}v_{12}+\left(\frac{1+ \alpha}{1-\beta}
\right)v_{21}v_{22}-\left(\frac{1-\alpha}{1+\gamma}\right)v_{31}v_{32}. 
\end{aligned}
\end{gather}
It is clear that the multilinearisations in \eqref{eqn: multilinsquasiprojsinfty1} are invariant under the map $v_{i1}
\leftrightarrow v_{i2}$, and moreover that $(e_0,e_3)$ is a solution to them. As in the proof of Lemma \ref{lem:
phiaut}, one can conclude that $\tau$ has order 2.
\end{proof}

\begin{lemma}\label{lem: inftyptscheme2}
Consider the quasiprojective subscheme of $\Gamma_2$ in which $v_{01}\neq 0$ and $v_{02}=0$. This subscheme contains
only one point, namely $(e_0,e_3)$.
\end{lemma}
\begin{proof}
In this quasiprojective subscheme we may assume that $v_{01}=1$. The multilinearisations in \eqref{eqn:
multilinsquasiprojsinfty1} then become
\begin{gather}
\begin{align*}%\label{eqn: multilinsquasiprojsinfty}
&m_1':=v_{12}-\alpha v_{21}v_{32}+\alpha v_{31}v_{22},\;\; m_2':=v_{12}-v_{21}v_{32}-v_{31}v_{22}, \\
&m_3':=v_{22}+\beta v_{11}v_{32}-\beta v_{31}v_{12},\;\; m_4':=v_{22}-v_{31}v_{12}-v_{11}v_{32}, \\
&m_5':=v_{11}v_{12}+v_{21}v_{22}- v_{31}v_{32}, \;\; m_6':=v_{11}v_{12}+\left(\frac{1+ \alpha}{1-\beta}
\right)v_{21}v_{22}-\left(\frac{1-\alpha}{1+\gamma}\right)v_{31}v_{32}. 
\end{align*}
\end{gather}
By equating $m_1'$ and $m_2'$ one obtains the equation
\begin{equation}\label{eq: sinftyptscheme1}
(\alpha-1)v_{21}v_{32}=(\alpha +1)v_{31}v_{22},
\end{equation}
while doing the same for $m_3'$ and $m_4'$ gives
\begin{equation}\label{eq: sinftyptscheme2}
(\beta-1)v_{31}v_{12}=(\beta +1)v_{11}v_{32}.
\end{equation}
In light of \eqref{eq: sinftyptscheme1} we split the analysis into three cases; either $v_{21}=0$, $v_{32}=0$ or
$v_{21},v_{32}\neq 0$. 

\textit{Case 1:} If $v_{21}=0$, then \eqref{eq: sinftyptscheme1} implies that either $v_{31}=0$ or $v_{22}=0$. 

\textit{Case 1(a):} If $v_{31}=0$ then \eqref{eq: sinftyptscheme2} implies that either $v_{11}=0$ or $v_{32}=0$. In
either situation one can use $m_1',m_2',m_3'$ and $m_4'$ to see that $v_{12}=v_{22}=0$, which results in a
contradication when $v_{32}=0$. If $v_{11}=0$ then the only solution that we obtain is $(e_0,e_3)$, which we have
already seen lies in $\Gamma_2$. 

\textit{Case 1(b):} If $v_{22}=0$ then one must have $v_{12}=0$ by $m_1'$. Using \eqref{eq: sinftyptscheme2} this
implies that either $v_{11}=0$ or $v_{32}=0$, as in the previous case. One can proceed in a similar manner to conclude
that the only solution in this case is $(e_0,e_3)$ once again.

\textit{Case 2:} Now assume that $v_{32}=0$. Then \eqref{eq: sinftyptscheme2} implies that either $v_{31}=0$ or
$v_{12}=0$. If the former is true then one can use $m_1'$ and $m_3'$ to show that $v_{12}=v_{22}=0$, which is absurd. In
the latter situation one can use $m_3'$ to conclude that $v_{22}=0$, then use $m_1'$ to show that $v_{12}=0$. This gives
a contradiction once again.

\textit{Case 3:} Finally, suppose that $v_{21},v_{32}\neq 0$. In this case one can see from \eqref{eq: sinftyptscheme1}
that $v_{31},v_{22}\neq 0$ must hold. By \eqref{eq: sinftyptscheme2} one must therefore have either $v_{11}=v_{12}=0$ or
$v_{11},v_{12} \neq 0$. If the former is true then $m_3'$ implies that $v_{22}=0$, which is a contradiction. Suppose
therefore that $v_{11},v_{12} \neq 0$. The Macaulay2 code given by Code \ref{code: lineschemeainfty2} in Appendix
\ref{subsec: staffordcalcs1} shows that there are no solutions to the multilinearisations in this case. Note that we can
always scale the second copy of $\proj{k}{3}$ so that $v_{11}v_{21}v_{31}v_{12}v_{22}v_{32}=1$.
\end{proof}

\begin{lemma}\label{lem: inftyptscheme3}
Consider the closed subscheme of $\Gamma_2$ in which $v_{01}=v_{02}=0$. This subscheme contains only 2 points, namely
$(e_2,e_1)$ and $(e_1,e_2)$.
\end{lemma}
\begin{proof}
In this subscheme the multilinearisations from \eqref{eqn: multilinsquasiprojsinfty1} become
\begin{align*}
&m_1'':=-\alpha v_{21}v_{32}+\alpha v_{31}v_{22}, \quad\quad\>  m_2'':=-v_{21}v_{32}-v_{31}v_{22}, \\ 
&m_3'':=\beta v_{11}v_{32}-\beta v_{31}v_{12}, \quad\quad\quad  m_4'':=-v_{31}v_{12}-v_{11}v_{32}, \\ 
&m_5'':=v_{11}v_{12}+v_{21}v_{22}- v_{31}v_{32}, \;\; m_6'':=v_{11}v_{12}+\left(\frac{1+ \alpha}{1-\beta}
\right)v_{21}v_{22}-\left(\frac{1-\alpha}{1+\gamma}\right)v_{31}v_{32}.  
\end{align*}
Using the equations $m_1'',\ldots,m_4''$ and our assumptions on $(\alpha,\beta,\gamma)$, we must have
\begin{equation*}%\label{eq: multilinsclosedsubsinfty1}
v_{21}v_{32}=v_{31}v_{22}=v_{31}v_{12}=v_{11} v_{32}=0.
\end{equation*}

If $v_{32}\neq 0$ then $v_{11}=v_{21}=0$ which is absurd. Similarly, if $v_{31}\neq 0$ we obtain a contradiction. Thus
we can assume that $v_{31}=v_{32}=0$. Since $\frac{1+ \alpha}{1-\beta} \neq 1$ (this would imply $\gamma=0$ or
$\beta=-1$), we must have $v_{11}v_{12}=v_{21}v_{22}=0$. The only admissible solutions to these equations are the points
$(e_2,e_1)$ and $(e_1,e_2)$, the second of which we could have deduced by symmetry from the first, or vice versa. 
\end{proof}

We are now in a position to prove that the point scheme of $S_{\infty}^{G,\mu}$ consists of 20 points up to
multiplicity.
\begin{proofof}{Proposition \ref{prop: staffordptschemes}}
By Lemma \ref{lem: inftyptscheme1} we know that $\Gamma_2$ is the graph of $\Gamma''$ under the automorphism $\tau$. We
will first study $\Gamma_2$, which is covered by the following four subschemes of $\proj{k}{3} \times \proj{k}{3}$:
\begin{equation*}
U_1:\; v_{01},v_{02} \neq 0,\;\; U_2:\; v_{01} \neq 0, v_{02}=0,\;\; U_3:\; v_{01}=0,v_{02} \neq 0,\;\; U_4:\;
v_{01}=v_{02}=0.
\end{equation*}

Note that the Macaulay2 code given by Code \ref{code: lineschemeainfty1} in Appendix \ref{subsec: staffordcalcs1} shows
that $U_1 \cap \Gamma_2$ contains 16 points up to multiplicity. By Lemma \ref{lem: inftyptscheme2} we know that $U_2\cap
\Gamma_2$ contains only a single point of the form $(e_0,e_3)$. But by Lemma \ref{lem: inftyptscheme1} the
multilinearisations defining $\Gamma_2$ are symmetric. Thus Lemma \ref{lem: inftyptscheme2} implies that $U_3\cap
\Gamma_2$ also contains only a single point, namely $(e_3,e_0)$. Finally, by Lemma \ref{lem: inftyptscheme3} we know
that $U_4\cap \Gamma_2$ contains only the two points $(e_2,e_1)$ and $(e_1,e_2)$.

We have shown that $\Gamma_2$ is a 0-dimensional scheme, therefore $\Gamma''= \pi_1(\Gamma_2)$ must also be
0-dimensional. To complete the proof we note that Proposition \ref{prop: genericpointscheme} implies that the
multiplicities of the points in $\Gamma''$ must add up to 20. In particular, this fact in conjunction with Code
\ref{code: lineschemeainfty1} implies that the four points we have exhibited explicitly must each have multiplicity 1.
\end{proofof}
\begin{conj}
Let $k$ be an algebraically closed field with $\text{char}(k) \neq 2$. Assume that $\alpha,\beta$ and $\gamma$ satisfy
\eqref{eq: 4sklyanincoeffcond} and $d \neq (1,\beta-1),(1,-1-\gamma)$. Then the point schemes of $S_{\infty}^{G,\mu}$
and $S_{d,i}(\alpha,\beta,\gamma)^{G,\mu}$ consist of 20 points counting multiplicity.
\end{conj}
As evidence for this conjecture in relation to $S_{d,i}(\alpha,\beta,\gamma)^{G,\mu}$, we note that it can be verified
by computer for some specific values of the parameters. 

To end this section we will compute the dimension of the line scheme of $S_{\infty}^{G,\mu}$.
\begin{prop}\label{prop: staffordlinescheme}\index{term}{line scheme!calculations}
Assume that $\text{char}(k)=0$. For generic parameters $\alpha,\beta$ and $\gamma$ the line scheme of
$S_{\infty}^{G,\mu}$ is a 1-dimensional projective scheme.
\end{prop}
\begin{proof}
As in Proposition \ref{prop: linescheme1dim}, our hypotheses reflect the fact that the proof relies on computer
calculations. Let a general relation in $S_{\infty}^{G,\mu}$ be written in the form
\begin{equation*}%\label{eq: generaleqnSinfty}
t_1f_1^{\mu}+t_2 f_2^{\mu}+t_3 f_3^{\mu}+t_4 f_4^{\mu}+ t_5 \Theta_1 + t_6 \Theta_2,
\end{equation*}
for some $t_i \in k$. Mimicking the method used to prove Proposition \ref{prop: linescheme1dim}, we form the following
matrix:
\begin{equation*}%\label{eq: sinftylinescheme}
\begin{pmatrix}
       -t_5 & t_1+t_2 & t_3+t_4 & 0 \\
       t_2-t_1 & t_5+t_6 & 0 & \beta t_3-t_4 \\
       t_4-t_3 & 0 & t_5+t_6 \left(\frac{1+\alpha}{1-\beta}\right) & -\alpha t_1-t_2 \\
       0 & -\beta t_3-t_4 & \alpha t_1-t_2 & -t_5-t_6 \left(\frac{1-\alpha}{1+\gamma}\right) 
       \end{pmatrix}.
\end{equation*}

The line scheme of $S_{\infty}^{G,\mu}$ is the closed subscheme of points $(t_1,\ldots,t_6) \in \proj{k}{5}$ such that
this matrix has rank less than or equal to 2. The Macaulay2 code given by Code \ref{code: lineschemeainfty} in Appendix
\ref{subsec: staffordcalcs1} shows that, when regarded as defining the affine cone of the line scheme in
$\mathbb{A}_k^6$, the ideal generated by the $3 \times 3$ minors of this matrix has codimension 4. Equivalently, the
affine scheme it defines has dimension 2. Since its affine cone has dimension 2, the line scheme must have dimension 1
when considered as a projective scheme in $\proj{k}{5}$.
\end{proof}

As we discussed in \S\ref{subsec: irrinqgr}, AS-regular algebras of dimension 4 with a 0-dimensional point scheme and a
1-dimensional line are considered important examples of regular algebras. When $\text{char}(k)= 0$ and for generic
parameters $\alpha, \beta$ and $\gamma$, Propositions \ref{prop: staffordptschemes} and \ref{prop: staffordlinescheme}
imply that $S_{\infty}^{G,\mu}$ is such an algebra.

\section{Twists in Rogalski and Zhang's classification}\label{subsec: rogzhangalgebras}
The algebras that we will now study were classified by
Rogalski and Zhang in their paper \cite{rogalski2012regular}. We will need to work over an algebraically closed field of
characteristic 0 for the duration of this section.

Rogalski and Zhang's algebras are AS-regular domains of dimension 4
satisfying two extra conditions; they are generated by three degree 1 elements and admit a \emph{proper
$\Z^{2}$-grading}\index{term}{proper $\Z^{2}$-grading}. Properness of such a grading, $A= \bigoplus_{n,m \in \Z}
A_{m,n}$ say,
means that $A_{0,1}\neq 0$ and $A_{1,0}\neq 0$. 

As the authors of \cite{rogalski2012regular} note in their introduction, AS-regular algebras of dimension 4 can have
either two, three or four
generators. Their aim was to develop examples in the least studied of these cases; the 4-dimensional Sklyanin
algebras and Stafford's algebras from \cite{stafford1994regularity} are of course examples of the four generator case
(not to mention their cocycle twists), whilst examples with two generators were studied in \cite{lu2007regular}. 

Rogalski and Zhang's main results are summarised below.
\begin{thm}[{\cite[Theorems 0.1 and 0.2]{rogalski2012regular}}]\label{thm: rogzhangmain}
Let $A$ be an AS-regular domain of dimension 4 which is generated by three degree 1 elements and properly
$\Z^{2}$-graded. Then either $A$ is a normal extension of an AS-regular algebra of dimension 3, or up to isomorphism it
falls into one of eight 1 or 2 parameter families, $\mathcal{A}-\mathcal{H}$. Moreover, any such algebra is strongly
noetherian, Auslander regular and Cohen-Macaulay.
\end{thm}

The final three properties in Theorem \ref{thm: rogzhangmain} hold under weaker conditions as \cite[Corollary
0.3]{rogalski2012regular} shows: one no longer needs to assume that the algebra is a domain or make an assumption on the
number of generators. 

There is another family of such algebras, namely $\underline{\mathcal{F}}$. However, by the isomorphism in \cite[Example
3.11]{rogalski2012regular} we can ignore such algebras and work with the family $\mathcal{F}$ instead. We will not
concern ourselves with the algebras that are normal extensions.

In order to apply our cocycle twist construction we require graded algebra automorphisms, where in this case graded
refers to the connected graded structure rather than the additional $\Z^{2}$-grading. Section 5 of Rogalski and Zhang's
paper is concerned with precisely this topic. The key result is the following, where generic means avoiding some
finite set of parameters given in the statement of \cite[Lemma 5.1]{rogalski2012regular}:
\begin{thm}[{\cite[Theorem 5.2(a)]{rogalski2012regular}}]\label{thm: rogzhangauts}
Consider a generic AS-regular algebra $A$ in one of the families $\mathcal{A} - \mathcal{H}$. The graded automorphism
group of $A$ is isomorphic either to $k^{\times}\times k^{\times}$ or to $k^{\times}\times
k^{\times} \times C_2$. The first case occurs for the families $\mathcal{A}(b,q)$ with $q \neq -1$, $\mathcal{D}(h,b)$
with $h
\neq b^4$, $\mathcal{F}$, $\underline{\mathcal{F}}$ and $\mathcal{H}$. The second case occurs if $A$ belongs to one of
the families $\mathcal{A}(b,-1)$, $\mathcal{B}$, $\mathcal{C}$, $\mathcal{D}(h,b)$ with $h=b^4$, $\mathcal{E}$ or
$\mathcal{G}$.
\end{thm}

We now show that any cocycle twist of an algebra in one of these families must be isomorphic to another algebra in the
classification. Before stating the result, let us fix some notation for the remainder of the section. The algebra $A$
will be generated by the three degree 1 elements $x_1,x_2$ and $x_3$, where $x_1, x_2 \in A_{1,0}$ and $x_3 \in
A_{0,1}$. We will follow Rogalski and Zhang in referring to the extra automorphism of order 2 as the
\emph{quasi-trivial} automorphism. This automorphism interchanges $x_1$ and $x_2$ whilst fixing $x_3$.
\begin{lemma}\label{lem: preserveclassrogzhang}
Let $A$ be a generic algebra in one of the families $\mathcal{A} - \mathcal{H}$. Any cocycle twist $A^{G,\mu}$ must also
belong to one of these families. Moreover, if $A$ possesses the quasi-trivial automorphism then so must $A^{G,\mu}$.
\end{lemma}
\begin{proof}
First note that all of the automorphisms described in Theorem \ref{thm: rogzhangauts} preserve the $\Z^{2}$-grading on
the algebra. Therefore by Lemma \ref{lem: autpresgrad} any cocycle twist must also possess a proper $\Z^{2}$-grading. By
Corollary \ref{cor: asreg} and Lemma \ref{lem: hilbseries} respectively, such a twist will be AS-regular of dimension 4
and have the same Hilbert series as $A$. Moreover, Lemma \ref{lemma: finitelygenerated} and Remark \ref{rem:
montfingenremark} imply that $A^{G,\mu}$ must have three degree 1 generators. The proof of the first part of the lemma
is completed by applying \cite[Theorem 3.9]{artin1991modules}, which implies that such a twist is also a domain. 

Suppose now that $A$ admits the quasi-trivial automorphism $\phi$. It suffices to show that this preserves any induced
$G$-grading, since in that case for all homogeneous elements $x \in A_g$, $y \in A_h$ one has
\begin{equation*}%\label{eq: autgradingpreserveextra}
\phi(x \ast_{\mu} y)=\mu(g,h)\phi(xy)=\mu(g,h)\phi(x)\phi(y)=\phi(x) \ast_{\mu} \phi(y),
\end{equation*}
as $\phi(x) \in A_g$ and $\phi(y) \in A_h$. Observe that the following elements must be homogeneous with respect to any
induced $G$-grading, since any automorphism acts on them diagonally:
\begin{equation}\label{eq: rogzhangcob}
w_1=x_1+x_2,\;\; w_2=x_1-x_2,\;\; w_3=x_3. 
\end{equation}

In particular, $\phi$ acts on them diagonally. These elements generate $A$ and therefore $\phi$ must preserve any
induced $G$-grading.
\end{proof}

The automorphisms corresponding to $k^{\times}\times k^{\times}$ come from scaling components of the $\Z^{2}$-grading.
The additional presence of the quasi-trivial automorphism implies the existence of cocycle twists relating algebras in
different families, as we now show. Recall that Lemma \ref{lem: defrelns} concerns generators of ideals under cocycle
twists; we use this lemma implicitly in the proof.
\begin{thm}\label{thm: rogzhangmymain}
Let $G=(C_2)^2 = \langle g_1,g_2 \rangle$ and let $\mu$ denote the 2-cocycle of $G$ defined in \eqref{eq:
mucocycledefn}. Then there are $k$-algebra isomorphisms
\begin{equation*}%\label{eq: rogzhangmymainisos}
\mathcal{A}(1,-1)^{G,\mu}\cong \mathcal{D}(1,1),\;\; \mathcal{B}(1)^{G,\mu} \cong \mathcal{C}(1), \;\;
\mathcal{E}(1,\gamma)^{G,\mu}\cong \mathcal{E}(1,-\gamma), \;\; \mathcal{G}(1,\gamma)^{G,\mu} \cong
\mathcal{G}(1,\overline{\gamma}).
\end{equation*}
\end{thm}
\begin{proof}
Once again, we will use the isomorphism $G \cong G^{\vee}$ given by \eqref{eq: chartable}. Let us begin by defining the
action of $G$ which we will use for each of the cocycle twists we perform. Note that all of the algebras in the
statement of the result admit the quasi-trivial automorphism. Therefore we can let $g_1$ act via the quasi-trivial
automorphism and $g_2$ act by multiplying $x_3$ by -1 and fixing the other two generators. 

Since the standard generators are not diagonal with respect to this action, we will instead use the generators
$w_1,w_2,w_3$ described above in \eqref{eq: rogzhangcob}. Denoting the algebra we wish to twist by $A$, the induced
$G$-grading on the new generators is given by
\begin{equation*}%\label{eq: rogzhanggrading}
w_1 \in A_e,\;\; w_2 \in A_{g_{2}},\;\; w_3 \in A_{g_{1}}. 
\end{equation*}

The defining relations of any algebra in one of the eight families belong to different components of the $\Z^2$-grading.
Observe that the algebras $\mathcal{A}(1,-1)$, $\mathcal{B}(1)$, $\mathcal{C}(1)$ and $\mathcal{D}(1,1)$ share three
relations, only being distinguished from each other by their relations in the $(2,1)$-component. Writing the shared
relations in terms of the diagonal basis we show that they are left invariant under the twist -- the first two are
quadratic relations:
\begin{align*}%\label{eq: rogzhangAreln1}
0 &= w_1^2 - w_2^2 = \frac{w_1 \ast_{\mu} w_1}{\mu(e,e)} - \frac{w_2 \ast_{\mu} w_2}{\mu(g_2,g_2)} = w_1 \ast_{\mu} w_1
-
w_2 \ast_{\mu} w_2 = v_1^2 - v_2^2,\\
0 &= w_3 w_1 - w_1 w_3 = \frac{w_3 \ast_{\mu} w_1}{\mu(g_1,e)} - \frac{w_1 \ast_{\mu} w_3}{\mu(e,g_1)} = w_3 \ast_{\mu}
w_1 - w_1 \ast_{\mu} w_3 = v_3 v_1 - v_1 v_3,
\end{align*}
while the third relation is cubic:
\begin{align*}%\label{eq: rogzhangAreln1}
0 = w_3^2w_2 - w_2 w_3^2 &= \frac{w_3 \ast_{\mu} w_3 \ast_{\mu} w_2}{\mu(g_1,g_1)\mu(e,g_2)} - \frac{w_2 \ast_{\mu} w_3
\ast_{\mu} w_3}{\mu(g_2,g_1)\mu(g_1g_2,g_1)} \\
&= w_3 \ast_{\mu} w_3 \ast_{\mu} w_2 - w_2 \ast_{\mu} w_3 \ast_{\mu} w_3 \\ &= v_3^2 v_2 - v_2 v_3^2.
\end{align*}

Thus, to verify the first two isomorphisms in the statement of the result it suffices to consider the behaviour under
the twist of the only relation they do not share. We first twist this relation in the algebra $\mathcal{A}(1,-1)$,
having once again written it in terms of the new generators beforehand:
\begin{align*}%\label{eq: rogzhangAreln1}
0 &= [w_3,[w_1,w_2]_+] \\ 
&= w_3 w_1 w_2 + w_3 w_2 w_1 - w_1 w_2 w_3 - w_2 w_1 w_3 \\
&= \frac{w_3 \ast_{\mu} w_1 \ast_{\mu} w_2}{\mu(g_1,e)\mu(g_1,g_2)} + \frac{w_3 \ast_{\mu} w_2 \ast_{\mu}
w_1}{\mu(g_1,g_2)\mu(g_1g_2,e)} - \frac{w_1 \ast_{\mu} w_2 \ast_{\mu} w_3}{\mu(e,g_2)\mu(g_2,g_1)} - \frac{w_2
\ast_{\mu} w_1 \ast_{\mu} w_3}{\mu(g_2,e)\mu(g_2,g_1)}  \\
&= -w_3 \ast_{\mu} w_1 \ast_{\mu} w_2 - w_3 \ast_{\mu} w_2 \ast_{\mu} w_1 - w_1 \ast_{\mu} w_2 \ast_{\mu} w_3 - w_2
\ast_{\mu} w_1 \ast_{\mu} w_3 \\
&= -[v_3,[v_1,v_2]_+]_+.
\end{align*}
This relation is the same as that in $\mathcal{D}(1,1)$ under the new generators, which proves the first isomorphism. 

Let us now move on to $\mathcal{B}(1)$. Twisting the non-shared relation we see that
\begin{align*}%\label{eq: rogzhangBreln4}
0 &= [w_3,[w_2,w_1]]_+ \\ 
&= w_3 w_2 w_1 - w_3 w_1 w_2 + w_2 w_1 w_3 - w_1 w_2 w_3 \\
&= \frac{w_3 \ast_{\mu} w_2 \ast_{\mu} w_1}{\mu(g_1,g_2)\mu(g_1g_2,e)} - \frac{w_3 \ast_{\mu} w_1 \ast_{\mu}
w_2}{\mu(g_1,e)\mu(g_1,g_2)} + \frac{w_2 \ast_{\mu} w_1 \ast_{\mu} w_3}{\mu(g_2,e)\mu(g_2,g_1)} - \frac{w_1 \ast_{\mu}
w_2 \ast_{\mu} w_3}{\mu(e,g_2)\mu(g_2,g_1)} \\
&= -w_3 \ast_{\mu} w_2 \ast_{\mu} w_1 + w_3 \ast_{\mu} w_1 \ast_{\mu} w_2 + w_2 \ast_{\mu} w_1 \ast_{\mu} w_3 - w_1
\ast_{\mu} w_2 \ast_{\mu} w_3 \\
&= [v_3,[v_1,v_2]].
\end{align*}
This relation is shared by $\mathcal{C}(1)$ under the new generating set, which proves the second isomorphism.

We now move on to the remaining two isomorphisms. The algebras in the relevant families share three relations, two of
which we have already shown are preserved under the cocycle twist. This is also true for the third relation, which as
yet we have not encountered:
\begin{align*}%\label{eq: rogzhangEtwistreln3}
0 = w_3^2w_2 +w_2w_3^2 &= \frac{w_3 \ast_{\mu} w_3 \ast_{\mu} w_2}{\mu(g_1,g_1)\mu(e,g_2)} + \frac{w_2 \ast_{\mu} w_3
\ast_{\mu} w_3}{\mu(g_2,g_1)\mu(g_1g_2,g_1)} \\&= w_3 \ast_{\mu} w_3 \ast_{\mu} w_2 + w_2 \ast_{\mu} w_3 \ast_{\mu} w_3
\\ &= v_3^2 v_2 +v_2 v_3^2.
\end{align*}

Once again, it suffices to see what happens to the non-shared relation. In $\mathcal{E}(1,\gamma)$, where $\gamma = \pm
i$, one has
\begin{align*}%\label{eq: rogzhangEtwistreln4}
0 &= w_3w_2w_1 - w_1w_3w_2 +\gamma w_1w_2w_3 - \gamma w_2w_1w_3 \\
&= \frac{w_3 \ast_{\mu} w_2 \ast_{\mu} w_1}{\mu(g_1,g_2)\mu(g_1g_2,e)} - \frac{w_1 \ast_{\mu} w_3 \ast_{\mu}
w_2}{\mu(e,g_1)\mu(g_1,g_2)} + \gamma \frac{w_1 \ast_{\mu} w_2 \ast_{\mu} w_3}{\mu(e,g_2)\mu(g_2,g_1)} - \gamma
\frac{w_2 \ast_{\mu} w_1 \ast_{\mu} w_3}{\mu(g_2,e)\mu(g_2,g_1)} \\
&= -w_3 \ast_{\mu} w_2 \ast_{\mu} w_1 + w_1 \ast_{\mu} w_3 \ast_{\mu} w_2 + \gamma w_1 \ast_{\mu} w_2 \ast_{\mu} w_3 -
\gamma w_2 \ast_{\mu} w_1 \ast_{\mu} w_3 \\
&= -v_3v_2v_1 + v_1v_3v_2 +\gamma v_1v_2v_3 - \gamma v_2v_1v_3.
\end{align*}
This is the final relation in $\mathcal{E}(1,-\gamma)$ under the new generators, which proves the penultimate
isomorphism.

We now twist the final relation of $\mathcal{G}(1,\gamma)$, where $\gamma=\frac{1 + i}{2}$ and so
$\overline{\gamma}=\frac{1}{2 \gamma}$:
\begin{align*}%\label{eq: rogzhangGrelncob}
0 &= w_3 w_1 w_2 +w_3 w_2 w_1 +i w_1w_2w_3 + i w_2w_1w_3 \\
&= \frac{w_3 \ast_{\mu} w_1 \ast_{\mu} w_2}{\mu(g_1,e)\mu(g_1,g_2)} + \frac{w_3 \ast_{\mu} w_2 \ast_{\mu}
w_1}{\mu(g_1,g_2)\mu(g_1g_2,e)} +i\frac{w_1 \ast_{\mu} w_2 \ast_{\mu} w_3}{\mu(e,g_2)\mu(g_2,g_1)} + i \frac{w_2
\ast_{\mu} w_1 \ast_{\mu} w_3}{\mu(g_2,e)\mu(g_2,g_1)} \\
&= -w_3 \ast_{\mu} w_1 \ast_{\mu} w_2 - w_3 \ast_{\mu} w_2 \ast_{\mu} w_1 + i w_1 \ast_{\mu} w_2 \ast_{\mu}
w_3 + i w_2 \ast_{\mu} w_1 \ast_{\mu} w_3 \\
&= -v_3  v_1  v_2 - v_3  v_2  v_1 + i v_1  v_2  v_3 + i v_2  v_1  v_3.
\end{align*}
This is precisely the final relation of $\mathcal{G}(1,\overline{\gamma})$ under the new generators, which proves the
last isomorphism in the statement of the theorem.
\end{proof}

Combined with the fact that $\mathcal{A}(b,-1)$, $\mathcal{B}(b)$, $\mathcal{C}(b)$, $\mathcal{D}(b,b^4)$,
$\mathcal{E}(b,\gamma)$ and $\mathcal{G}(b,\gamma)$ are Zhang twists of the respective algebras in the statement of
Theorem \ref{thm: rogzhangmymain} for any parameter $b \in k^{\times}$ \cite[\S 3]{rogalski2012regular}, this result
gives a partial description of such algebras up to cocycle twisting. 

\section{Twisting an algebra of Vancliff}\label{sec: vancliffql}
\sectionmark{Vancliff's algebras}
In this section we will investigate cocycle twists of algebras studied in \cite{vancliff1994quadratic}, which are
defined for algebraically closed fields of characteristic not equal to 2. Their properties are strongly controlled by
some associated geometry; corresponding to each algebra there is nonsingular quadric $Q$ and a line $L$ in
$\proj{k}{3}$, along with an automorphism $\sigma \in \text{Aut}(Q \cup L)$. The algebras split into two families
depending upon whether $\sigma$ preserves or interchanges the rulings on $Q$. 

We will focus on the algebras for which the automorphism preserves the two rulings. While there exist graded
automorphisms in the other case, those which are diagonal with respect to the relations given in \cite[Lemma
1.3(b)]{vancliff1994quadratic} produce cocycle twists which are Zhang twists of the \N-grading. This will not be true
for our examples, since the point scheme of the twists is only 1-dimensional, as proved in Proposition \ref{prop:
vancliffptscheme}.

The algebras we shall twist will be denoted by $R(\alpha,\beta,\lambda)$\index{notation}{r@$R(\alpha,\beta,\lambda)$},
where $\alpha, \beta, \lambda \in k^{\times}$ are scalars satisfying $\lambda \neq \alpha \beta$. As we did in Chapter
\ref{chap: sklyanin} for Sklyanin algebras, we will omit the parameters if no ambiguity will arise. The defining
relations of $R(\alpha,\beta,\lambda)$ are given in \cite[Lemma 1.3(a)]{vancliff1994quadratic}:
\begin{gather}
\begin{aligned}\label{eq: vancliffqlrelns}
&x_2 x_1 =\alpha x_1 x_2,\;\; x_3 x_1 = \lambda x_1 x_3,\;\; x_4 x_1 = \alpha\lambda x_1 x_4, \;\; x_4x_3=\alpha x_3
x_4,\\ &x_4 x_2 = \lambda x_2 x_4,\;\; x_3x_2 - \beta x_2 x_3=(\alpha \beta - \lambda)x_1 x_4. 
\end{aligned}
\end{gather}

Example 1.5 from \cite{vancliff1994quadratic} shows that taking $\alpha = q, \beta = 1$ and $\lambda = q^{-1}$ for some
$q \in k^{\times}$ such that $q^2 \neq 1$, one obtains the coordinate ring of quantum $2 \times 2$ matrices, denoted by
$\calO_q(M_2(k))$. 

The geometric data associated to $R$ lies inside $\mathbb{P}(R_1)=\proj{k}{3}$, namely the quadric $Q=V(x_1x_4+x_2x_3)$
and the line $L=V(x_1,x_4)$. The automorphism $\sigma$ is described in the proof of Lemma 1.3(a) op. cit.:
\begin{equation}\label{eq: vancliffqlaut}
{\sigma}|_{Q}= \begin{pmatrix}
                      \alpha \lambda & 0 & 0 & 0 \\
		      0 & \lambda & 0 & 0 \\
		      0 & 0 & \alpha & 0 \\
		      0 & 0 & 0 & 1
                     \end{pmatrix},\;\;\; {\sigma}|_{L}=\begin{pmatrix}
                      \beta & 0 \\
		      0 & 1
                     \end{pmatrix}.
\end{equation}

To find $\N$-graded algebra automorphisms for which the given generators form a diagonal basis, one only needs to
consider the relation $x_3x_2 - \beta x_2 x_3=(\alpha \beta - \lambda)x_1 x_4$; the other relations are fixed by
\emph{any} automorphism that acts diagonally on the generators. Suppose that an automorphism acts by $x_i \mapsto
\lambda_i x_i$ for some $\lambda_i \in k^{\times}$. From the relevant relation we see that $\lambda_2 \lambda_3=
\lambda_1 \lambda_4$ must hold. It is clear that there are many automorphisms that would satisfy this condition, but we
will focus on the following action of $G = (C_2)^2 = \langle g_1,g_2 \rangle$ on $R$:
\begin{gather}
\begin{aligned}\label{eq: vancliffqlact}
&g_1:\; x_1 \mapsto x_1,\;\; x_2 \mapsto x_2,\;\; x_3 \mapsto -x_3,\;\; x_4 \mapsto -x_4, \\
&g_2:\; x_1 \mapsto x_1,\;\; x_2 \mapsto -x_2,\;\; x_3 \mapsto x_3,\;\; x_4 \mapsto -x_4.
\end{aligned}
\end{gather}

Since $|G|=4$, our assumption on the characteristic of $k$ means that $\text{char}(k) \nmid |G|$ always hold. Under the
isomorphism $G \cong G^{\vee}$ given by \eqref{eq: chartable}, the action in \eqref{eq: vancliffqlact} induces the
following $G$-grading on the generators of $R$:
\begin{equation}\label{eq: vancliffqlgrad}
x_1 \in R_{e},\;\; x_2 \in R_{g_{1}},\;\; x_3 \in R_{g_{2}},\;\; x_4 \in R_{g_{1}g_{2}}.
\end{equation}

We now give the relations of the cocycle twist of $R$ that we will study. 
\begin{lemma}\label{lem: vancliffrelns}
Consider the $G$-grading on $R(\alpha,\beta,\gamma)$ induced by the action of $G$ given in \eqref{eq: vancliffqlact},
and let $\mu$ be the 2-cocycle defined in \eqref{eq: mucocycledefn}. Then the associated cocycle twist
$R(\alpha,\beta,\gamma)^{G,\mu}$\index{notation}{r@$R(\alpha,\beta,\lambda)^{G,\mu}$} has defining relations
\begin{gather}
\begin{aligned}\label{eq: vancliffqltwistrelns}
&v_2 v_1 =\alpha v_1 v_2,\;\; v_3 v_1 = \lambda v_1 v_3,\;\; v_4 v_1 = \alpha\lambda v_1 v_4, \;\; v_4v_3=-\alpha v_3
v_4,\\ 
&v_4 v_2 = -\lambda v_2 v_4,\;\; v_3v_2 + \beta v_2 v_3=(\alpha \beta - \lambda)v_1 v_4.  
\end{aligned}
\end{gather}
\end{lemma}
\begin{proof}
Note that since $x_1 \in R_{e}$, the three commutation relations from \eqref{eq: vancliffqlrelns} involving $x_1$ are
left invariant under a twist by $\mu$. Computations for the remaining three relations when twisted by this 2-cocycle are
deferred to Appendix \ref{subsec: calcvancliff}. Lemma \ref{lem: defrelns} shows that the ideal of relations in the
twist is generated by the six relations in \eqref{eq: vancliffqltwistrelns}.
\end{proof}

We now highlight some of the properties that the cocycle twist possesses.
\begin{thm}\label{thm: vancliffasregular}
Suppose that $\alpha, \beta, \lambda \in k^{\times}$ and $\lambda \neq \alpha \beta$. Then
$R(\alpha,\beta,\lambda)^{G,\mu}$ has the following properties:
\begin{itemize}
 \item[(i)] it is an iterated Ore extension over $k$;
 \item[(ii)] it is generated in degree 1 with Hilbert series $1/(1-t)^4$;
 \item[(iii)] it is noetherian;
 \item[(iv)] it is Cohen-Macaulay and Auslander regular of global dimension 4; 
 \item[(v)] it is AS-regular.
\end{itemize}
\end{thm}
\begin{proof}
By adjoining the generators in the order $v_1,v_2,v_4,v_3$ to $k$, it is clear that the twist is an iterated Ore
extension (in which the endomorphisms are in fact automorphisms). This has two consequences; firstly, it implies that
$R(\alpha,\beta,\lambda)^{G,\mu}$ has Hilbert series $1/(1-t)^4$ and is generated in degree 1, while an application of
\cite[Theorem 2.6]{goodearl2004introduction} shows that it is noetherian. 

By Corollary 1.9 and Proposition 2.3 from \cite{vancliff1994quadratic}, $R(\alpha,\beta,\lambda)$ is Cohen-Macaulay and
Auslander regular of global dimension 4. These properties are preserved under twisting by Propositions \ref{prop: gldim}
and \ref{prop: cohenmac}. Part (v) can then be proved by applying \cite[Theorem 6.3]{levasseur1992some} to
$R(\alpha,\beta,\lambda)^{G,\mu}$.
\end{proof}

We will drop the parameters from our notation in future since there will be no ambiguity, thus $R^{G,\mu}$ is our object
of study.

Our focus now turns to geometry, beginning with a description of the point scheme of $R^{G,\mu}$ at the level of its
closed points. This will show that such algebras do not belong to either of the families defined in
\cite{vancliff1994quadratic}, nor are Zhang twists of the $\N$-grading of such an algebra. 
\begin{lemma}\label{lem: pointsparameterisevcliff}
Let $\Gamma_2$ be the projective scheme determined by the multilinearisations\index{term}{multilinearisations} of the
defining relations of $R^{G,\mu}$ given in \eqref{eq: vancliffqltwistrelns}. Then such multilinearisations can be
expressed in terms of matrices as $M \cdot \underline{v} = 0$, where
\begin{equation}\label{eq: vancliffqlmultilin}
M = \begin{pmatrix}
       v_{21} & - \alpha v_{11} & 0 & 0 \\
       v_{31} & 0 & -\lambda v_{11} & 0 \\
       v_{41} & 0 & 0 & -\alpha \lambda v_{11} \\
       0 & 0 & v_{41} & \alpha v_{31} \\
       0 & v_{41} & 0 & \lambda v_{21} \\
       0 & v_{31} & \beta v_{21} & -(\alpha \beta - \lambda)v_{11}
      \end{pmatrix}\;\text{ and }\; \underline{v}= \begin{pmatrix}
                      v_{12} \\ v_{22} \\ v_{32} \\ v_{42}  
                     \end{pmatrix}.
\end{equation}
Furthermore, $\Gamma_2$ is the graph of the point scheme $\Gamma$ of $R^{G,\mu}$ under the associated automorphism
$\phi$. Points in $\Gamma$ are precisely those points in \proj{k}{3} for which the matrix $M$ has rank 3 when evaluated
at them.
\end{lemma}
\begin{proof}
It is a routine calculation to verify that the matrices in \eqref{eq: vancliffqlmultilin} and the equation $M \cdot
\underline{v} = 0$ give an alternative description of the multilinearisations of the relations in \eqref{eq:
vancliffqltwistrelns}. This formulation will prove useful to us when proving Proposition \ref{prop: vancliffptscheme}.

To prove the second part of the lemma it would suffice to apply Theorem \ref{thm: pointschemenice}. This is possible by
the conclusion of Theorem \ref{thm: vancliffasregular}, thus $\Gamma_2$ is the graph of $\Gamma$ under the associated
automorphism $\phi$.

Let us now address the final part of the lemma by supposing that $p \in \Gamma$. As $\Gamma_2$ is the graph of an
automorphism, $p^{\phi}$ is the unique point $q$ for which $(p,q) \in \Gamma_2$. In terms of the matrix equation $M
\cdot \underline{v} = 0$, this means precisely that the rank of $M$ when evaluated at $p$ must be 3.
\end{proof}

We are now in a position to describe the closed points of $\Gamma$. Rather than first finding the points of $\Gamma_2$
as we have done on previous occasions (in Lemma \ref{lem: ptschemecontains} for example), we will find the points of
$\Gamma$ directly.
\begin{prop}\label{prop: vancliffptscheme}
The point scheme $\Gamma$ of $R^{G,\mu}$ consists of the union of the following lines in \proj{k}{3}:\index{term}{point
scheme!calculations of}
\begin{gather}
\begin{aligned}\label{eq: vancliffqlptscheme}
&L_{12}:\; v_1=v_2=0,\;\; L_{13}:\; v_1=v_3=0,\;\; L=L_{14}:\; v_1=v_4=0, \\ 
&L_{24}:\; v_2=v_4=0, \;\; L_{34}:\; v_3=v_4=0.
\end{aligned}
\end{gather}
Moreover, the associated automorphism $\phi$ is defined on each of these lines by
\begin{gather}
\begin{aligned}\label{eq: vancliffqlptaut}
&{\phi}|_{L_{12}}=\begin{pmatrix} -\alpha & 0 \\ 0 & 1 \end{pmatrix},\;\;
{\phi}|_{L_{13}}=\begin{pmatrix} -\lambda & 0 \\ 0 & 1 \end{pmatrix},\;\;
{\phi}|_{L_{14}}=\begin{pmatrix} -\beta & 0 \\ 0 & 1 \end{pmatrix}, \\
&{\phi}|_{L_{24}}=\begin{pmatrix} \lambda & 0 \\ 0 & 1 \end{pmatrix}, \;
{\phi}|_{L_{34}}=\begin{pmatrix} \alpha & 0 \\ 0 & 1 \end{pmatrix},
\end{aligned}
\end{gather}
where these lines have been identified with \proj{k}{1} in the obvious manner.
\end{prop}
\begin{proof}
Consider the matrix $M$ in \eqref{eq: vancliffqlmultilin}, where we now suppress the second subscript of the variables
therein. By Lemma \ref{lem: pointsparameterisevcliff} we are looking for points $p \in \proj{k}{3}$ at which the matrix
$M$ has rank 3. In particular, the $4 \times 4$ minors of $M$ must vanish. On computing these minors, one obtains the
following equations -- note that there are only 14 equations since one of the minors vanishes at all points of
\proj{k}{3}.
\begin{gather}
\begin{aligned}\label{eq: vancliffqlminors}
&m_1 := 2 \alpha^2 \lambda v_1^2 v_3 v_4,\;\;\;\;\;\, m_2:=2 \alpha \lambda^2 v_1^2 v_2 v_4,\;\;\,m_3:=2 \alpha \lambda
v_1 v_2 v_4^2, \\ 
&m_4:=-2 \alpha \beta \lambda v_1 v_2^2 v_4,\;\; m_5:=2 \alpha \lambda v_1 v_3 v_4^2, \;\;\;\; m_6:=2 \alpha \lambda v_1
v_3^2 v_4, \\
&m_7:=\alpha \lambda v_1^2((\lambda-\alpha \beta)v_1v_4+ (\lambda+\alpha \beta)v_2v_3), \\ 
&m_8:=\alpha v_1v_3((\lambda-\alpha \beta)v_1v_4-(\lambda+\alpha \beta)v_2 v_3), \\
&m_9:=\lambda v_1 v_2((\lambda-\alpha \beta)v_1v_4 - (\lambda+\alpha \beta)v_2v_3), \\ 
&m_{10}:=\alpha v_1v_4((\lambda-\alpha \beta)v_1v_4+(\lambda- \alpha \beta)v_2 v_3), \\
&m_{11}:=v_2 v_4((\lambda+\alpha \beta)v_2v_3-(\lambda-\alpha \beta)v_1v_4), \\ 
&m_{12}:=\lambda v_1v_4((\lambda-\alpha \beta)v_2v_3-(\lambda-\alpha \beta)v_1v_4), \\
&m_{13}:=v_3v_4((\lambda+\alpha \beta)v_2v_3-(\lambda-\alpha \beta)v_1 v_4), \\ 
&m_{14}:=v_4^2((\lambda+\alpha \beta)v_2 v_3-(\lambda-\alpha \beta)v_1 v_4).
\end{aligned}
\end{gather}

We will now analyse \eqref{eq: vancliffqlminors}. Equations $m_1$ through $m_6$ vanish if either $v_1=0$, $v_4=0$ or
$v_2=v_3=0$. It is clear by looking at $m_7$ or $m_{14}$ that the latter case can be subsumed into either of the former
two cases. If $v_1=0$ then $m_{14}$ implies that one of the other generators must also vanish, in which case all of the
equations vanish. When $v_4=0$ one can use $m_7$ to reach the same conclusion. 

The points obtained from this argument are those on the lines in \eqref{eq: vancliffqlptscheme}. It remains to see that
the matrix $M$ has rank exactly 3 when evaluated at any of such point; this is clear from \eqref{eq: vancliffqlminors}.
We can then determine the behaviour of $\phi$ on each line in $\Gamma$ by solving the matrix equation $M \cdot
\underline{v} = 0$ when $M$ has been evaluated at a general point on the line. This analysis reveals that $\phi$ behaves
as described in \eqref{eq: vancliffqlptaut}; in particular, it preserves each of the lines in \eqref{eq:
vancliffqlptscheme}.
\end{proof}
In future we will write $L_{14}$ for the line $L$ to fit in with the notation used in \eqref{eq: vancliffqlptscheme} for
the other lines in the point scheme of $R^{G,\mu}$.

\begin{cor}\label{cor: vancliffnotztwist}
Suppose that $\alpha, \beta, \lambda \in k^{\times}$ and $\lambda \neq \alpha \beta$. The twist
$R(\alpha,\beta,\lambda)^{G,\mu}$ is not isomorphic as an $\N$-graded algebra to an algebra from one of the two families
studied in \cite{vancliff1994quadratic}. Furthermore, it is not a Zhang twist of the $\N$-grading of such an algebra
either. 
\end{cor}
\begin{proof}
Both graded isomorphisms and Zhang twists preserve the point scheme of an algebra (the latter via Theorem \ref{thm:
ztwistgmodequiv}). However, the algebras in the two families in \cite{vancliff1994quadratic} have a point scheme
containing a quadric by Proposition 2.1 op. cit., therefore in particular are 2-dimensional. By Proposition \ref{prop:
vancliffptscheme} the point scheme of $R(\alpha,\beta,\lambda)^{G,\mu}$ is 1-dimensional, from which the result follows.
\end{proof}

It is shown in \cite[Corollary 3.4]{vancliff1994quadratic} that for some normal degree 2 element $\Omega$ there is an
isomorphism of $k$-algebras
\begin{equation*}%\label{eq: vancliffqlthcriso}
B(Q,\mathcal{M},\varsigma) \cong R/(\Omega).
\end{equation*}
The invertible sheaf $\mathcal{M}=j^{\ast}\calO_{\proj{k}{3}}(1)$ is associated to the embedding $j:Q \hookrightarrow
\proj{k}{3}$. We will write $S:= B(Q,\mathcal{M},\varsigma)$ in future. In Vancliff's notation $\varsigma=\sigma|_{Q}$,
however we wish to distinguish this automorphism from that associated to the point scheme of $R$. 

One can say more about the normal element governing this factor ring: Lemma 1.11(b) from \cite{vancliff1994quadratic}
implies that $\Omega$ is the unique (up to scalar multiple) element which annihilates those point modules lying on $Q$,
while not annihilating all of those on $L_{14}$. The explicit form of $\Omega$ is not used in Vancliff's paper, however
the following calculation implies that $\Omega=\alpha x_1 x_4 + x_2 x_3$ up to scalar: evaluating this element at a
point $p=(p_1,p_2,p_3,p_4) \in Q$ gives
\begin{equation}\label{eq: vancliffqlnormal}
(\alpha x_1 x_4 + x_2 x_3)(p,p^{\sigma}) = \alpha x_1(p) x_4(p^{\sigma}) + x_2(p) x_3(p^{\sigma})=\alpha^2 \lambda
(p_1p_4+p_2 p_3)=0,
\end{equation}
upon using \eqref{eq: vancliffqlaut}. The only points on $L_{14}$ which are annihilated by $\Omega$ are $e_2$ and $e_3$,
which are the only points in the intersection $Q \cap L_{14}$. 

Notice that $\Omega$ is homogeneous with respect to the $G$-grading defined in \eqref{eq: vancliffqlact} and therefore
by Lemma \ref{lem: defrelns} one has 
\begin{equation}\label{eq: vancliffqltwistiso}
S^{G,\mu}= (R/(\Omega))^{G,\mu} \cong R^{G,\mu}/(\Theta),
\end{equation}
where $\Theta:=\alpha v_1 v_4 - v_2 v_3$, the element corresponding to $\Omega$ under the twist. 

We now prove a lemma concerning point modules over $S^{G,\mu}$. 
\begin{lemma}\label{lem: vanclifffacptmods}
The point scheme $\Gamma'$ of $S^{G,\mu}$ consists of the four lines $L_{12}, L_{13},L_{24}$ and $L_{34}$.
\end{lemma}
\begin{proof}
Recall from Proposition \ref{prop: vancliffptscheme} that the point scheme $\Gamma$ of $R^{G,\mu}$ consists of the four
lines in the statement of the lemma together with $L_{14}$. Since $S^{G,\mu}$ is a factor ring, proving the lemma
amounts to showing that $\Theta$ vanishes at points on $L_{12}, L_{13},L_{24}$ and $L_{34}$, but not at points of
$L_{14}$ (other than points of intersection with the other lines). 

Evaluating $\Theta$ at a point $p=(p_1,p_2,p_3,p_4) \in \Gamma$ gives
\begin{equation*}%\label{eq: vancliffqlnormal1}
\Theta(p,p^{\phi}) = \alpha v_1(p) v_4(p^{\phi}) - v_2(p) v_3(p^{\phi}).
\end{equation*}
One can now use \eqref{eq: vancliffqlptaut}, which describes the automorphism $\phi$, to see that $\Theta$ vanishes on
the four lines given in the statement of the lemma but not on $L_{14} \setminus \{e_2,e_3\}$.
\end{proof}

While both $S^{G,\mu}$ and $B^{G,\mu}$ are cocycle twists of twisted homogeneous coordinate rings, the former has a
1-dimensional family of point modules by Lemma \ref{lem: vanclifffacptmods}, while the latter has none. We know that
there are fat point modules of multiplicity 2 over $B^{G,\mu}$, therefore it is natural to ask if $S^{G,\mu}$ possesses
such modules too, and moreover if they arise in the same way. We will show that this is indeed the case, and in so doing
explain the existence of the point modules over $S^{G,\mu}$. To allow us to use results from \S\ref{sec: modules} we
will work with $R$ and $R^{G,\mu}$ rather than the factor rings $S$ and $S^{G,\mu}$ respectively.

The following result is suggestive because the locus it describes is precisely the point scheme of $S^{G,\mu}$.
\begin{lemma}\label{lem: ramlocus}
The locus of points on $Q$ that belong to $G$-orbits containing fewer than 4 points is $L_{12} \cup L_{13} \cup L_{24}
\cup L_{34}$. 
\end{lemma}
\begin{proof}
Firstly, note that we are in the situation described by Hypotheses \ref{hyp: genhypforfatpts}; $R$ is generated in
degree 1, $G$ is the Klein-four group and $\mu$ is the 2-cocycle defined in \eqref{eq: mucocycledefn}. Moreover, the
action of $G$ on $R_1$ affords the regular representation by \eqref{eq: vancliffqlgrad}, while $R$ satisfies the
hypotheses of Theorem \ref{thm: pointschemenice} by Corollary 1.9 and Proposition 2.3 from \cite{vancliff1994quadratic}.
Consequently, the action of $G$ on the point scheme $Q \cup L_{14}$ of $R$ is as given in \eqref{eq: Gactonpoints}.

We can now observe that, as in the proof of Lemma \ref{lem: orbit4points}, a point $p \in Q$ with three non-zero
coordinates lies in a $G$-orbit containing 4 points. We are led to look for points on $Q$ for which at least two
coordinates are zero, which is precisely the union of lines given in the statement of the lemma. Note that orbits on
this union of lines all contain 2 points apart from the singleton orbits $[e_0],[e_1],[e_2]$ and $[e_3]$, which are
fixed by the action of $G$.
\end{proof}

We will refer to orbits containing fewer than 4 points as \emph{degenerate}\index{term}{degenerate orbit}, whereas those
containing 4 points will be described as \emph{nondegenerate}. From the defining equation of $Q$ one can see that the
nondegenerate orbits contain precisely those points for which at least three of their coordinates are non-zero.

In the next result we once again use Notation \ref{not: tilde} to avoid confusion; point modules over $R$ will be
denoted by $M_p$, while those over $R^{G,\mu}$ will be written $\widetilde{M}_p$.
\begin{prop}\label{prop: modulesramify}
Let $M_p$ be a point module over $R$. The right $M_2(R)$-module $M_p^2$ becomes an $R^{G,\mu}$-module upon restriction
for which:
\begin{itemize}
\item[(i)] if $p$ belongs to a nondegenerate orbit then $M_p^2$ is a fat point module of multiplicity 2 over
$R^{G,\mu}$, where $M_p^2 \cong M_q^2$ if and only if $q \in [p]$; 
\item[(ii)] if $p$ belongs to a degenerate orbit containing 2 points then $M_p^2 \cong \widetilde{M}_{p'} \oplus
\widetilde{M}_{q'}$, where $[p'] = \{p',q'\}$. Furthermore, if $p \notin L_{14}$ then $[p] = [p']$;
\item[(iii)] if $p=e_j$ then $M_p^2 \cong \widetilde{M}_p^2$.
\end{itemize}
\end{prop}
\begin{proof}
As we saw in Lemma \ref{lem: ramlocus}, the action of $G$ on $Q \cup L_{14}$ is given in \eqref{eq: Gactonpoints}. The
automorphism $\sigma$ which is associated to the point scheme was described in \eqref{eq: vancliffqlaut}. By comparing
the two actions on points one can see that they commute. Thus the action of $\sigma$ on $Q \cup L_{14}$ preserves
(non)degeneracy of orbits. 

Suppose now that a point $p$ lies in a nondegenerate orbit of $Q \cup L_{14}$. By definition of such an orbit, three
coordinates of $p$ are non-zero and so we may apply Proposition \ref{prop: fatpoints} to construct a fat point module
$M_p^2$ associated to it. If there were an isomorphism between two such modules, $M_p^2 \cong M_q^2$ say, then Corollary
\ref{cor: fatpointisoclasses} implies that $q=p^g$, which completes the proof of (i).

For the remaining cases we will have to get our hands dirtier, therefore we recall from Lemma \ref{lem: matrixgens} that
the generators of $R^{G,\mu}$ as a subring of $M_2(R)$ are given by the matrices
\begin{equation}\label{eq: matrixembedding23}
v_1 = \begin{pmatrix} x_1 & 0 \\ 0 & x_1 \end{pmatrix},\; v_2 = \begin{pmatrix} x_2 & 0 \\ 0 & -x_2 \end{pmatrix},\; v_3
= \begin{pmatrix} 0 & x_3 \\ x_3 & 0 \end{pmatrix}, \;\;
v_4 = \begin{pmatrix} 0 & -x_4 \\ x_4 & 0 \end{pmatrix}.
\end{equation}
Recall also that for a point module $M_p$ the action of the generators of $R$ on $(M_p)_n$ is governed by coordinates of
$\sigma^n(p)$. 

We now begin to prove (ii). Let $p \in Q$ belong to a degenerate orbit of size 2, thus $[p]= \{p,q\}$. The point module
$M_p$ is governed by a right ideal $I_p$ in $R$ of the form $(x_i,x_j,\lambda x_k-x_l)$ for some $\lambda \in
k^{\times}$ and distinct $i,j,k,l \in \{1,2,3,4\}$. We will deal with each line in the locus described in Lemma
\ref{lem: ramlocus} separately.

First suppose that $p=(0,0,\omega,1) \in L_{12}$ for some $\omega \in k^{\times}$. Note that $\sigma^n(p)=(0,0,\alpha^n
\omega,1)$ for all $n \in \N$. The generators $v_1$ and $v_2$ annihilate the module, and for scalars $\eta, \zeta \in k$
one has
\begin{equation*}%\label{eq: zetaeta1}
(\eta m_n, \zeta m_n)  \cdot v_3 = (\alpha^n \omega \zeta m_{n+1},\alpha^n \omega \eta m_{n+1}),\;\; 
(\eta m_n, \zeta m_n)  \cdot v_4 = (\zeta m_{n+1}, -\eta m_{n+1}).
\end{equation*}

Suppose that both $v_3$ and $v_4$ send this element to the same 1-dimensional subspace. One is forced to have either
$\eta=0$ or $\zeta=0$. The submodules generated by $(m_0,m_0)$ and $(m_0,-m_0)$ are therefore point modules over
$R^{G,\mu}$, which we will denote by $\widetilde{M}_{p'}$ and $\widetilde{M}_{q'}$ for some $p',q' \in \Gamma$
respectively. They span $M_p^2$ and do not intersect each other. To discover which points $p'$ and $q'$ are, observe
that $v_3+(-1)^{n}\alpha^n \omega v_4$ annihilates $(\widetilde{M}_{p'})_n$ and $v_3-(-1)^{n}\alpha^n \omega v_4$
annihilates $(\widetilde{M}_{q'})_n$ for all $n \in \N$. From this one can see that there is a decomposition of right
$R^{G,\mu}$-modules
\begin{equation*}%\label
M_{p}^2 \cong (m_0,0)R^{G,\mu} \oplus (0,m_0)R^{G,\mu} \cong \widetilde{M}_{p^{g_{2}}} \oplus \widetilde{M}_{p}. 
\end{equation*}

Now suppose that $p=(\omega,1,0,0) \in L_{34}$ for some $\omega \in k^{\times}$, in which case one has
$\sigma^n(p)=(\alpha^n \omega,1,0,0)$ for all $n \in \N$. The generators $v_3$ and $v_4$ annihilate the module, and for
scalars $\eta, \zeta \in k$ one has
\begin{equation*}%\label{eq: zetaeta1}
(\eta m_n, \zeta m_n)  \cdot v_1 = (\alpha^n \omega \eta m_{n+1},\alpha^n \omega \zeta m_{n+1}),\;\; 
(\eta m_n, \zeta m_n)  \cdot v_2 = (\eta m_{n+1}, -\zeta m_{n+1}).
\end{equation*}

Once again, the disjoint submodules generated by $(m_0,0)$ and $(0,m_0)$ are point modules over $R^{G,\mu}$. We will
denote these modules by $\widetilde{M}_{p'}$ and $\widetilde{M}_{q'}$ respectively, as we did for points on $L_{12}$.
Observe that $v_1-\alpha^n \omega v_2$ annihilates $(\widetilde{M}_{p'})_n$ and $v_1+\alpha^n \omega v_2$ annihilates
$(\widetilde{M}_{q'})_n$ for all $n \in \N$. Thus there is a decomposition of right $R^{G,\mu}$-modules
\begin{equation*}%\label
M_{p}^2 \cong (m_0,0)R^{G,\mu} \oplus (0,m_0)R^{G,\mu} \cong \widetilde{M}_{p} \oplus \widetilde{M}_{p^{g_{2}}}. 
\end{equation*}

In the remaining two cases there is a slightly different direct sum decomposition. Suppose that $p=(0,\omega,0,1) \in
L_{13}$ for some $\omega \in k^{\times}$. One has $\sigma^n(p)=(0,\lambda^n \omega,0,1)$ for all $n \in \N$. The
generators $v_1$ and $v_3$ annihilate the module, and for scalars $\eta, \zeta \in k$ one has
\begin{equation*}%\label{eq: zetaeta1}
(\eta m_n, \zeta m_n)  \cdot v_2 = (\lambda^n \omega \eta m_{n+1}, - \lambda^n \omega \zeta m_{n+1}),\;\; 
(\eta m_n, \zeta m_n)  \cdot v_4 = (\zeta m_{n+1}, -\eta m_{n+1}).
\end{equation*}

In this case one must have $\eta^2=\zeta^2$ to ensure that both generators map $(\eta m_n, \zeta m_n)$ into the same
1-dimensional vector space. One can see that the submodules generated by $(m_0,m_0)$ and $(m_0,-m_0)$ are disjoint point
modules over $R^{G,\mu}$, which we will denote by $\widetilde{M}_{p'}$ and $\widetilde{M}_{q'}$ respectively, as above.
Observe that $v_2-(-1)^n \lambda^n \omega v_4$ annihilates $(\widetilde{M}_{p'})_n$ and $v_2+(-1)^n\lambda^n \omega v_4$
annihilates $(\widetilde{M}_{q'})_n$ for all $n \in \N$. Thus there is a decomposition of right $R^{G,\mu}$-modules
\begin{equation*}%\label
M_{p}^2 \cong (m_0,m_0)R^{G,\mu} \oplus (m_0,-m_0)R^{G,\mu} \cong \widetilde{M}_{p} \oplus \widetilde{M}_{p^{g_{1}}}. 
\end{equation*}

Finally, suppose that $p=(\omega,0,1,0) \in L_{24}$ for some $\omega \in k^{\times}$. One has $\sigma^n(p)=(\lambda^n
\omega,0,1,0)$ for all $n \in \N$. The generators $v_2$ and $v_4$ annihilate the module, and for scalars $\eta, \zeta
\in k$ one has
\begin{equation*}%\label{eq: zetaeta1}
(\eta m_n, \zeta m_n)  \cdot v_1 = (\lambda^n \omega \eta m_{n+1}, \lambda^n \omega \zeta m_{n+1}),\;\; 
(\eta m_n, \zeta m_n)  \cdot v_3 = (\zeta m_{n+1}, \eta m_{n+1}).
\end{equation*}

As for the line $L_{13}$, the submodules generated by $(m_0,m_0)$ and $(m_0,-m_0)$ are disjoint point modules over
$R^{G,\mu}$, which we denote by $\widetilde{M}_{p'}$ and $\widetilde{M}_{q'}$ once again. Observe that $v_1- \lambda^n
\omega  v_3$ annihilates $(\widetilde{M}_{p'})_n$ and $v_1+\lambda^n \omega  v_3$ annihilates $(\widetilde{M}_{q'})_n$
for all $n \in \N$. Thus there is a decomposition of right $R^{G,\mu}$-modules
\begin{equation*}%\label
M_{p}^2 \cong (m_0,m_0)R^{G,\mu} \oplus (m_0,-m_0)R^{G,\mu} \cong \widetilde{M}_{p} \oplus \widetilde{M}_{p^{g_{1}}}. 
\end{equation*}

To complete the proof of part (ii) we must consider the points belonging to degenerate orbits of order 2 for which $p
\in L_{14}$. Suppose that $p = (0,\omega,1,0)$ for some $\omega \in k^{\times}$ (thus $p \in L_{14} \setminus Q$) and
consider the restriction of the right $M_2(R)$-module $M_p^2$ to $R^{G,\mu}$. One has
$\sigma^n(p)=(0,\beta^n\omega,1,0)$ for all $n \in \N$. The generators $v_1$ and $v_4$ annihilate the module, and for
scalars $\eta, \zeta \in k$ one has
\begin{equation*}%\label{eq: zetaeta1}
(\eta m_n, \zeta m_n)  \cdot v_2 = (\beta^n \omega \eta m_{n+1}, -\beta^n \omega \zeta m_{n+1}),\;\; 
(\eta m_n, \zeta m_n)  \cdot v_3 = (\zeta m_{n+1}, \eta m_{n+1}).
\end{equation*}

The submodules generated by $(m_0, i m_0)$ and $(m_0,-i m_0)$ can be seen to be disjoint point modules over $R^{G,\mu}$.
Observe that $v_2+ (-\beta)^n i \omega v_3$ and $v_2 - (-\beta)^n i \omega v_3$ annihilate the degree $n$ pieces of
these modules respectively. It follows that there is a decomposition of right $R^{G,\mu}$-modules
\begin{equation*}%\label
M_{p}^2 \cong (m_0, i m_0)R^{G,\mu} \oplus (m_0,-i m_0)R^{G,\mu} \cong \widetilde{M}_{(0,- i \omega,1,0)} \oplus
\widetilde{M}_{(0,i \omega,1,0)}. 
\end{equation*}

We must now address (iii) and the four points $p=e_j$ for $j=0,1,2,3$. Since in that case three of the generators of $R$
annihilate $M_p$, the $R^{G,\mu}$-submodules of $M_p^2$ generated by $(m_0,0)$ and $(0,m_0)$ are disjoint, span $M_p^2$
and are both isomorphic to $\widetilde{M}_p$.
\end{proof}

Let us now use Proposition \ref{prop: modulesramify} to consider the restriction of modules from $R^{G,\mu}$ to
$S^{G,\mu}$. Recall that the point scheme of $S$ is the quadric $Q$, thus the modules over $R^{G,\mu}$ considered in the
proposition that correspond to points on $Q$ can be restricted to $S^{G,\mu}$. In particular, note that Proposition
\ref{prop: modulesramify}(i) implies that $S^{G,\mu}$ has a family of fat point modules of multiplicity 2 parameterised
by an open subset of $Q^G$. 

One can compare the behaviour described in Proposition \ref{prop: modulesramify} with the work in
\cite{lebruyn1995central}, where Le Bruyn studies fat point modules over algebras that are finite over their centre. As
stated in the introduction to that paper, the study of fat point modules of a certain multiplicity is equivalent to
studying the ramification locus of a push-forward sheaf. In our situation, if one regards points in degenerate orbits as
being ramified --- which coincides with the meaning of the term for curves, see \cite[Lemma 8.15]{ueno2003algebraic} ---
then the behaviour of fat point modules is once again related to ramification.

The behaviour we have encountered in \S\ref{subsubsec: geomdescrthcrtwist} and \S\ref{sec: vancliffql} with regard to
fat point modules over certain cocycle twists suggests the following questions. 
\begin{ques}\label{ques: thcrtwist}
Let $T=B(X,\calL,\sigma)$ be a twisted homogeneous coordinate ring. Suppose that a finite abelian group $G$ acts by
graded automorphisms on $T$, and let $T^{G,\mu}$ be a cocycle twist of the induced $G$-grading.
\begin{itemize}
\item[(i)] Can $T^{G,\mu}$ be described geometrically?
\item[(ii)] Can one construct fat point modules over $T^{G,\mu}$ via restriction of modules from $TG_{\mu}$, even when
the twisted group algebra is not a matrix ring as in our examples?
\item[(iii)] If there is a positive answer to (ii), is the decomposition of such a $T^{G,\mu}$-module into 1-critical
modules --- in the sense of a critical composition series --- determined geometrically?
\end{itemize}
\end{ques}

\section{Twisting a graded skew Clifford algebra}\label{subsec: gradedskewclifford}
In this section we study graded skew Clifford algebras, which were introduced by Cassidy and Vancliff in
\cite{cassidy2010generlizations}. Such algebras are --- as the name suggests --- generalisations of graded Clifford
algebras. In \cite{nafari2011classifying} it was shown that almost all AS-regular algebras of dimension 3 generated by
three degree 1 generators are either graded skew Clifford algebras or Ore extensions of such algebras (which are
themselves AS-regular of dimension 2). 

There is a recent corrigendum to \cite{cassidy2010generlizations}, namely \cite{cassidy2013corrigendum}, which contains
modifications of several definitions and results. Where appropriate we will reference the corrigendum; for a description
of what remains valid in the original paper, see the introduction to \cite{cassidy2013corrigendum}.

We will assume throughout this section that $\text{char}(k)\neq 2$. Let $\mu \in M_n(k)$ be a \emph{multiplicatively
skew-symmetric matrix} with $\mu_{ii}=1$. This means that $\mu_{ij}=\mu_{ji}^{-1}$ for all $i,j=1,\ldots,n$. Such a
matrix defines a skew polynomial ring $S$ on the generators $z_1,\ldots,z_n$ in a natural manner, with defining
relations $z_jz_i=\mu_{ij}z_iz_j$. A matrix $M \in M_n(k)$ is \emph{$\mu$-symmetric}\index{term}{m@$\mu$-symmetric
matrix} if $M_{ij}=\mu_{ij}M_{ji}$ for all $i,j=1,\ldots,n$.
\begin{defn}[{\cite[Definition 1.12]{cassidy2010generlizations}}]\label{defn: gradedskewclifford}
Let $\mu$ be as above and $M_1,\ldots,M_n$ be $\mu$-symmetric matrices. The \emph{graded skew Clifford
algebra}\index{term}{graded skew Clifford algebra} $A=A(\mu,M_1,\ldots,M_n)$\index{notation}{a@$A(\mu,M_1,\ldots,M_n)$}
associated to this data is a graded $k$-algebra with degree 1 generators $x_1,\ldots,x_n$ and degree 2 generators
$y_1,\ldots,y_n$. The following conditions are satisfied:
\begin{itemize}
\item[(i)] the relations are of the form $x_ix_j+\mu_{ij}x_jx_i=\sum_{k=1}^n (M_k)_{ij} y_k$ for all $i,j=1,\ldots,n$;
\item[(ii)] there exists a \emph{normalising sequence}\index{term}{normalising sequence} $\{r_1,\ldots,r_n\}$ that spans
$ky_1+\ldots+ky_n$ (see \cite[Definition 1.9(a)]{cassidy2010generlizations}).
\end{itemize}
\end{defn}

Although it may appear that this algebra is not generated in degree 1, Lemma 1.13 op. cit. gives equivalent conditions
such that $y_i \in (A_1)^2$ for all $i=1,\ldots,n$. One such condition is that the matrices $M_1,\ldots,M_n$ are
linearly independent.

The main result regarding such algebras is Theorem 4.2 op. cit., which is correct as stated under the changes in the
relevant definitions given by \cite[Definition 2]{cassidy2013corrigendum}. It relates to a condition on the quadric
system determined by $q_k=zM_k z^T$ for $k=1,\ldots,n$, where $z=(z_1,\ldots,z_n)$. Suppose that this system is
normalising, thus $Sq_k=q_kS$ for all $k=1,\ldots,n$, and satisfies one of the equivalent conditions of \cite[Corollary
11]{cassidy2013corrigendum}. Then the main theorem states that $A(\mu,M_1,\ldots,M_n)$ is an AS-regular domain of global
dimension $n$.

Let us consider a concrete example which is AS-regular of dimension 4. The example we will study has already appeared in
Theorem \ref{thm: new} and is the algebra obtained by factoring the free algebra $k\{x_1,x_2,x_3,x_4\}$ by the ideal
generated by the relations
\begin{gather}
\begin{aligned}\label{eq: cliffordeg}
&x_4x_1-ix_1x_4,\;\; x_3^2-x_1^2,\;\; x_3x_1-x_1x_3+x_2^2,\;\; x_3x_2-ix_2x_3,\;\; x_4^2-x_2^2,\\ 
&x_4x_2-x_2x_4+\gamma x_1^2.
\end{aligned} 
\end{gather}
This algebra --- which we denote by $A(\mu,M_1,\ldots,M_4)$ --- is discussed in \cite[Example
5.1]{cassidy2010generlizations}. The corresponding multiplicatively skew-symmetric matrix is
\begin{equation}\label{eq: cliffordskew}
\mu=\begin{pmatrix} 1 & i & -1 & i \\  -i & 1 & i & -1 \\ -1 & -i & 1 & i \\ -i & -1 & -i & 1 \end{pmatrix},
\end{equation}
with associated normalising quadric system
\begin{equation}\label{eq: cliffordnormal}
q_1=z_1z_2,\;\; q_2=z_3z_4,\;\; q_3=z_1^2+z_3^2+\gamma z_2 z_4,\;\; q_4=z_2^2+z_4^2+z_1 z_3. 
\end{equation}

One can calculate using $q_k=zM_k z^T$ that the corresponding $\mu$-symmetric matrices are
\begin{gather}
\begin{aligned}\label{eq: cliffordmatrices}
&M_1=\begin{pmatrix} 0 & \frac{1}{2} & 0 & 0 \\  -\frac{i}{2} & 0 & 0 & 0 \\  0 & 0 & 0 & 0 \\ 0 & 0 & 0 & 0
\end{pmatrix},\;\;  
M_2=\begin{pmatrix} 0 & 0 & 0 & 0 \\  0 & 0 & 0 & 0 \\  0 & 0 & 0 & \frac{1}{2} \\  0 & 0 & -\frac{i}{2} & 0 
\end{pmatrix} \;\;, \\  
&M_3=\begin{pmatrix} 1 & 0 & 0 & 0 \\  0 & 0 & 0 & \frac{\gamma}{2} \\ 0 & 0 & 1 & 0 \\ 0 & -\frac{\gamma}{2} & 0 & 0
\end{pmatrix},\;\;  
M_4=\begin{pmatrix} 0 & 0 & \frac{1}{2} & 0 \\ 0 & 1 & 0 & 0 \\ -\frac{1}{2} & 0 & 0 & 0 \\  0 & 0 & 0 & 1
\end{pmatrix}.     
\end{aligned} 
\end{gather}

These matrices are easily seen to be linearly independent, which implies that the degree 2 generators can be written in
terms of the degree 1 generators.

One equivalent condition in \cite[Corollary 11]{cassidy2013corrigendum} is that factoring out the elements in \eqref{eq:
cliffordnormal} from $S$ produces a finite-dimensional algebra. The relations of $S$ imply that to prove this it
suffices to show that sufficiently high powers of the generators vanish; any monomial can be rearranged into
lexicographic order and will contain such a power if of high enough degree. One can verify that this is true in our case
via equations \eqref{eq: 5thpowervanish1} in Appendix \ref{subsec: calcskewcliff}. By \cite[Theorem
4.2]{cassidy2010generlizations} the algebra $A(\mu,M_1,\ldots,M_4)$ is therefore an AS-regular domain of global
dimension 4. 

We will now apply a cocycle twist to $A(\mu,M_1,\ldots,M_4)$. Consider the graded action of $G=(C_2)^2=\langle
g_1,g_2\rangle$, defined on the generators by
\begin{equation}\label{eq: gencliffordaction}
g_1: x_1 \mapsto x_1,\; x_2 \mapsto -x_2,\; x_3 \mapsto x_3,\; x_4 \mapsto -x_4, \;\;\; x_i^{g_{2}}=-x_i^{g_{1}}\text{
for }i=1,2,3,4. 
\end{equation}

Our assumption on the characteristic of $k$ means that $\text{char}(k) \nmid |G|$ will always hold. As usual, we use the
isomorphism $G\cong G^{\vee}$ given by \eqref{eq: chartable} and the 2-cocycle $\mu$ defined in \eqref{eq:
mucocycledefn}. To avoid ambiguity we relabel this cocycle by $\tau$ for the duration of this section.
\begin{lemma}\label{lem: relnsskewclifford}
The cocycle twist $A(\mu,M_1,\ldots,M_4)^{G,\tau}$ can be presented as the factor of the free $k$-algebra
$k\{v_1,v_2,v_3,v_4\}$ by the ideal generated by the relations
\begin{gather}
\begin{aligned}\label{eq: gencliffordnewrelns}
&v_4v_1+iv_1v_4,\;\; v_3^2-v_1^2,\;\; v_3v_1-v_1v_3+v_2^2,\;\; v_3v_2+iv_2v_3,\;\; v_4^2-v_2^2,\\ 
&v_4v_2-v_2v_4+\gamma v_1^2.
\end{aligned}
\end{gather}
\end{lemma}
\begin{proof}
The algebra $A(\mu,M_1,\ldots,M_4)$ is generated in degree 1, thus by the action of $G$ defined in \eqref{eq:
gencliffordaction} and Remark \ref{rem: montfingenremark} one can conclude that $A(\mu,M_1,\ldots,M_4)^{G,\tau}$ is also
generated in degree 1. The computations needed to obtain the relations in \eqref{eq: gencliffordnewrelns} from those in
\eqref{eq: cliffordeg} are given in \eqref{eq: relnskewcliff} in Appendix \ref{subsec: calcskewcliff}. Lemma \ref{lem:
defrelns} confirms that these relations are in fact the defining relations in the twist.
\end{proof}

In fact, the cocycle twist $A(\mu,M_1,\ldots,M_4)^{G,\tau}$ can also be described as a Zhang twist of the $\N$-grading
on $A(\mu,M_1,\ldots,M_4)$, as the following proposition shows. One might expect this behaviour since $g_1g_2$ acts by a
scalar, and using such an action was the key idea in the proof of Proposition \ref{prop: recoverztwist}. 
\begin{prop}\label{prop: zcocycletwist}
Let $\phi$ denote the algebra automorphism by which the element $g_1 \in G$ acts. Then there is an isomorphism of
$k$-algebras
\begin{equation*}%\label{eq: cliffordiso}
A(\mu,M_1,\ldots,M_4)^{G,\tau} \cong A(\mu,M_1,\ldots,M_4)^{\N,\phi}, 
\end{equation*}
where the algebra on the right is the Zhang twist of the $\N$-grading on $A(\mu,M_1,\ldots,M_4)$ by $\phi$.
\end{prop}
\begin{proof}
Consider the relations of $A(\mu,M_1,\ldots,M_4)$ in \eqref{eq: cliffordeg}. Their behaviour under a Zhang twist by
$\phi$ is given below:
\begin{gather}
\begin{aligned}\label{eq: ztwistskewcliffrelns}
&x_4x_1-ix_1x_4 = x_4 \ast_{\phi} x_1 + ix_1 \ast_{\phi} x_4,\;\;\;\;\;\; x_3^2-x_1^2 =  x_3 \ast_{\phi} x_3  - x_1
\ast_{\phi} x_1,\\
&x_3x_2-ix_2x_3 =  -x_3 \ast_{\phi} x_2 - i x_2 \ast_{\phi} x_3,\;\;\; x_4^2-x_2^2 =  -x_4 \ast_{\phi} x_4 + x_2
\ast_{\phi} x_2,\\ 
&x_3x_1-x_1x_3+x_2^2  = x_3 \ast_{\phi} x_1 - x_1 \ast_{\phi} x_3 - x_2 \ast_{\phi} x_2,\\ 
&x_4x_2-x_2x_4+\gamma x_1^2 =  - x_4 \ast_{\phi} x_2 + x_2 \ast_{\phi} x_4 +\gamma x_1 \ast_{\phi} x_1.
\end{aligned} 
\end{gather}

To obtain the relations in $A(\mu,M_1,\ldots,M_4)^{G,\tau}$ from those in \eqref{eq: ztwistskewcliffrelns}, one needs to
rescale the generators by $x_2 \mapsto -x_2$ and $x_3 \mapsto -x_3$.  
\end{proof}

We now consider the question of whether $A(\mu,M_1,\ldots,M_n)^{G,\tau}$ is another graded skew Clifford algebra. At the
end of \cite[Example 5.1]{cassidy2010generlizations} it is explained that some Zhang twists of their example correspond
to other possible skew-symmetric matrices for the same normalising sequence. Since our example is a Zhang twist of this
form by Proposition \ref{prop: zcocycletwist}, it seems likely that it would be associated to the same normalising
quadric system. The relations in \eqref{eq: gencliffordnewrelns} suggest that we would have
\begin{equation}\label{eq: multskwesymmetric}
\mu_{13}=-1,\;\; \mu_{14}=i,\;\; \mu_{24}=-1,\;\; \mu_{23}=-i, 
\end{equation}
in the associated multiplicatively skew-symmetric matrix if this were the case. We prove that the twist is indeed
another graded skew Clifford algebra with the same normalising sequence, but a different skew-symmetric matrix.
\begin{prop}\label{prop: cliffordtwist}
For $\mu$ and $M_1,\ldots,M_n$ as in \eqref{eq: cliffordskew} and \eqref{eq: cliffordmatrices},
\begin{equation*}%\label{eq: gradedcliffordiso}
A(\mu,M_1,\ldots,M_n)^{G,\tau}\cong A(\mu',M_1,\ldots,M_n), 
\end{equation*}
as $k$-algebras, where
\begin{equation}\label{eq: musystem}
\mu'=\begin{pmatrix} 1 & -i & -1 & i \\  i & 1 & -i & -1 \\ -1 & i & 1 & -i \\ -i & -1 & i & 1 \end{pmatrix}.  
\end{equation}
Moreover, $A(\mu,M_1,\ldots,M_n)^{G,\tau}$ is an AS-regular domain of global dimension 4. 
\end{prop}
\begin{proof}
From \eqref{eq: multskwesymmetric} we already have several entries of a possible multiplicatively skew-symmetric matrix.
The requirement that \eqref{eq: cliffordnormal} is still normalising allows us to calculate the rest of the entries as
demonstrated in equations \eqref{eq: normalisingsystem} through \eqref{eq: normalisingsystem4} in Appendix \ref{subsec:
calcskewcliff}. One obtains the matrix $\mu'$ given above.

We will show that $A(\mu',M_1,\ldots,M_n)$ satisfies the relations in \eqref{eq: gencliffordnewrelns}. Using Definition
\ref{defn: gradedskewclifford}(i) and the data from \eqref{eq: cliffordmatrices} and \eqref{eq: musystem}, the relations
in $A(\mu',M_1,\ldots,M_n)$ are
\begin{gather}
\begin{align*}%\label{eq: cliffordrelnstwist}
&2x_1^2=y_3,\;\; 2x_2^2=y_4,\;\; 2x_3^2=y_3,\;\; 2x_4^2=y_4,\;\; x_1x_2-i x_2 x_1 = \frac{1}{2} y_1,\;\; x_1 x_3-x_3 x_1
= \frac{1}{2} y_4,\\
&x_1x_4-ix_4x_1=0,\;\; x_2x_3-i x_3 x_2 = 0,\;\; x_2 x_4-x_4 x_2 = \frac{\gamma}{2} y_3,\;\;
x_3x_4-ix_4x_3=\frac{1}{2}y_2.
\end{align*}
\end{gather}

One obtains the six relations from \eqref{eq: gencliffordnewrelns} and some extra relations telling us how to write the
remaining degree 2 generators in terms of the degree 1 generators. Thus one has a surjective map
$A(\mu,M_1,\ldots,M_n)^{G,\tau} \rightarrow A(\mu',M_1,\ldots,M_n)$. Notice that the associated quadric system of
$A(\mu',M_1,\ldots,M_n)$ is the same as for $A(\mu,M_1,\ldots,M_n)$, thus \cite[Theorem 4.2]{cassidy2010generlizations}
implies that it must be an AS-regular domain of dimension 4. Since it is generated in degree 1 it must have the same
Hilbert series as the twist, which implies that the surjection above is an isomorphism, completing the proof.  
\end{proof}
\begin{rem}\label{rem: twistotherproofcliff}
We could also have proved that the twist is AS-regular of dimension 4 using results from Chapter \ref{chap:
cocycletwists}. 
\end{rem}

Recall from the proof of Theorem \ref{thm: new} that $A(\mu,M_1,\ldots,M_n)$ can be considered as part of a 1-parameter
family of algebras. This family can be denoted by $A(\gamma)$ in light of the final relation in \eqref{eq: cliffordeg}.
Cassidy and Vancliff state in \cite[Example 5.1]{cassidy2010generlizations} that the algebra $A(\gamma)$ has a
0-dimensional point scheme and a 1-dimensional line scheme for all $\gamma \in k^{\times}$. By Proposition \ref{prop:
zcocycletwist} the cocycle twist $A(\mu,M_1,\ldots,M_n)^{G,\tau}$ is also a Zhang twist of the $\N$-grading on
$A(\mu,M_1,\ldots,M_n)$. It therefore follows from Theorem \ref{thm: ztwistgmodequiv} that
$A(\mu,M_1,\ldots,M_n)^{G,\tau}$ also has a 0-dimensional point scheme and a 1-dimensional line scheme.

\section{Twisting a universal enveloping algebra}\label{subsec: homenvelopalg}
In this section we investigate cocycle twists of a universal enveloping algebra and its homogenisation. The relevant Lie
algebra is $\mathfrak{sl}_2(k)$, where $k$ is some algebraically closed field of characteristic not equal to 2. Later in
this section we will assume that $k=\C$. After studying these twists we will discuss them in relation to the papers
\cite{lebruyn1993homogenized} and \cite{lebruyn1993on}. The first of these papers contrasts the Lie-theoretic geometry
of $\mathfrak{sl}_2(\C)$ with the
geometry encoded in $\text{qgr}(\mathfrak{sl}_2(\C))$, while the latter generalises these ideas to the homogenised
enveloping algebra of any finite-dimensional complex Lie algebra.

The standard presentation of $U(\mathfrak{sl}_2(k))$\index{notation}{u@$U(\mathfrak{sl}_2(k))$} is given by $k$-algebra
generators $e,f,h$ subject to the relations
\begin{equation*}%\label{eq: homogenizedsl2relns}
ef-fe=[e,f]=h,\;\; he-eh=[h,e]=2e,\;\; hf-fh=[h,f]=-2f,
\end{equation*}
where $[-,-]$ denotes the Lie bracket on $\mathfrak{sl}_2(k)$. 

As usual, to twist this algebra we need to first find some algebra automorphisms. Observe that if such an automorphism
acts diagonally on the given generators and is defined by
\begin{equation*}%\label{eq: homogenisedsl2aut}
e \mapsto \lambda_1 e,\;\; f \mapsto \lambda_2 f,\;\; h \mapsto \lambda_3 h,
\end{equation*}
then we must have $\lambda_3=1$ and $\lambda_1 \lambda_2=1$. In addition to such scalar automorphisms, the maps defined
by
\begin{equation*}%\label{eq: homogenisedsl2autnondg}
e \mapsto f,\;\; f \mapsto e,\;\; h \mapsto \pm h,
\end{equation*}
also define automorphisms, although the given generators do not form a diagonal basis. 

We will consider the action of $G=(C_2)^2=\langle g_1,g_2\rangle$ given by
\begin{equation}\label{eq: homogenisedsl2nondiagact}
g_1:\; e \mapsto f,\;\; f \mapsto e,\;\; h \mapsto -h, \;\; g_2:\; e \mapsto -e,\;\; f \mapsto -f,\;\; h \mapsto h.
\end{equation}

The change of generators $E=e+f,F=e-f,H=h$ produces a diagonal basis, and the relations of $U(\mathfrak{sl}_2(k))$ in
terms of this new basis are
\begin{equation}\label{eq: homogenisedsl2newbasis}
EF-FE+2H,\;\; HE-EH-2F,\;\; HF-FH-2E.
\end{equation}

Under the action of $G$ defined in \eqref{eq: homogenisedsl2nondiagact} and the isomorphism $G \cong G^{\vee}$ given by
\eqref{eq: chartable}, the induced $G$-grading on the new generators of $U(\mathfrak{sl}_2(k))$ is
\begin{equation}\label{eq: homogenisedsl2nondiaggrad}
E \in U(\mathfrak{sl}_2(k))_{g_{1}},\;\; F \in U(\mathfrak{sl}_2(k))_{g_{1}g_{2}},\;\; H \in
U(\mathfrak{sl}_2(k))_{g_{2}}.
\end{equation}

We can now give the relations of the cocycle twist that we will study henceforth.
\begin{lemma}
Let $\mu$ be the 2-cocycle defined in \eqref{eq: mucocycledefn}. Then the cocycle twist
$U(\mathfrak{sl}_2(k))^{G,\mu}$\index{notation}{u@$U(\mathfrak{sl}_2(k))^{G,\mu}$} obtained using the $G$-grading in
\eqref{eq: homogenisedsl2nondiaggrad} has the following three defining relations:
\begin{equation}\label{eq: homogenisedsl2nondiagrelns}
E \ast_{\mu} F+F \ast_{\mu} E-2H,\;\; H \ast_{\mu} E+E \ast_{\mu} H-2F,\;\; H \ast_{\mu} F+F \ast_{\mu} H-2E.
\end{equation}
\end{lemma}
\begin{proof}
These relations can be computed from those in \eqref{eq: homogenisedsl2newbasis}, with the computations being given in
Appendix \ref{subsec: calcenvelop}. Lemma \ref{lem: defrelns} implies that the ideal of relations in the twist is
generated by precisely these three relations.
\end{proof}

The homogenisation of $U(\mathfrak{sl}_2(k))$, denoted
$U_h(\mathfrak{sl}_2(k))$\index{notation}{u@$U_h(\mathfrak{sl}_2(k))$}, is the $\N$-graded algebra obtained
by homogenising the defining relations of the enveloping algebra by a central generator, $t$ say. One immediate
consequence of this is that 
\begin{equation}\label{eq: dehomogenise}
U(\mathfrak{sl}_2(k))\cong U_h(\mathfrak{sl}_2(k))/(t-1), 
\end{equation}
while factoring out the ideal generated by $t$ produces a polynomial ring on three generators by the PBW Theorem
\cite[Theorem 6.8]{krause2000growth}.

The action of $G$ defined in \eqref{eq: homogenisedsl2nondiagact} can be extended to $U_h(\mathfrak{sl}_2(k))$ by
letting $G$ act trivially on $t$. One can then twist the induced $G$-grading by the 2-cocycle $\mu$ to obtain
$U_{h}(\mathfrak{sl}_2(k))^{G,\mu}$\index{notation}{u@$U_h(\mathfrak{sl}_2(k))^{G,\mu}$}, whose defining relations are
comprised of the homogenisations of those in \eqref{eq: homogenisedsl2nondiagrelns}:
\begin{equation*}
E \ast_{\mu} F+F \ast_{\mu} E-2H \ast_{\mu} t,\;\; H \ast_{\mu} E+E \ast_{\mu} H-2F \ast_{\mu} t,\;\; H \ast_{\mu} F+F
\ast_{\mu} H-2E \ast_{\mu} t,
\end{equation*}
together with the three additional relations
\begin{equation*}
t \ast_{\mu} E - E \ast_{\mu} t,\;\;\; t \ast_{\mu} F - F \ast_{\mu} t,\;\text{ and }\; t \ast_{\mu} H - H \ast_{\mu} t.
\end{equation*}

The following result is a twisted version of \eqref{eq: dehomogenise}.
\begin{lemma}\label{lem: isoassgr}
There is an isomorphism of $k$-algebras
\begin{equation*}%\label{eq: homogenisefrac}
U(\mathfrak{sl}_2(k))^{G,\mu}\cong U_h(\mathfrak{sl}_2(k))^{G,\mu}/(t-1). 
\end{equation*}
\end{lemma}
\begin{proof}
The result follows from the fact that $t-1$ is fixed by the action of $G$. 
\end{proof}

Let us now assume that $k=\C$ and prove that $U_h(\mathfrak{sl}_2(\C))^{G,\mu}$ has several good properties.
\begin{prop}\label{prop: homogeniseprops}
The cocycle twist $U_h(\mathfrak{sl}_2(\C))^{G,\mu}$ is noetherian, Auslander regular of global dimension 4 and
Cohen-Macaulay. Furthermore, it has Hilbert seres $1/(1-t)^4$ and is generated in degree 1.
\end{prop}
\begin{proof}
As noted at the beginning of \cite[\S 2]{lebruyn1993homogenized}, $U_h(\mathfrak{sl}_2(\C))$ has all of the properties
mentioned in the statement of the proposition. The result follows by Remark \ref{rem: montfingenremark} and an
application of Lemma \ref{lem: hilbseries}, Corollary \ref{cor: uninoeth} and Proposition \ref{prop: cohenmac}.
\end{proof}

Now let us turn to the papers cited at the beginning of the section. It is remarked at the top of pg. 728 in
\cite{lebruyn1993homogenized} that there is a dichotomy in the problem of finding linear modules over
$U_h(\mathfrak{sl}_2(\C))$. In fact, this occurs in the more general situation of a finite dimensional complex Lie
algebra $\mathfrak{g}$, its universal enveloping algebra $U(\mathfrak{g})$ and homogenisation $U_{h}(\mathfrak{g})$. Le
Bruyn and Van den Bergh use the following fact in the proof of \cite[Theorem 2.2]{lebruyn1993on}: the homogenising
generator must either act faithfully on a linear module over $U_{h}(\mathfrak{g})$ or annihilate it.

The $d$-linear modules over $U_{h}(\mathfrak{g})$ that are not annihilated by the homogenising
generator come from 1-dimensional representations of Lie subalgebras of $\mathfrak{g}$ of codimension $d$, which are
induced up to $U(\mathfrak{g})$ and then homogenised. Corollary 3.4 and Theorem 3.5 from \cite{lebruyn1993on} describe
the case when $d=0$ in more detail. The structure of the point scheme of $U_{h}(\mathfrak{g})$ depends on whether
$\mathfrak{g} = [\mathfrak{g},\mathfrak{g}]$ or not: if this is true then the point scheme contains an embedded
component; otherwise, the point scheme is reduced.

We will now consider point modules over $U_h(\mathfrak{sl}_2(\C))^{G,\mu}$, for which there is also a dichotomy. By
definition such modules are 1-critical, therefore by \cite[Lemma 2.10]{levasseur1993modules} the generator $t$ either
acts faithfully on a point module or annihilates it. Our final result describes the point modules over
$U_h(\mathfrak{sl}_2(\C))^{G,\mu}$. 
\begin{lemma}
The point modules of $U_h(\mathfrak{sl}_2(\C))^{G,\mu}$ correspond to the three lines $t=E=0$, $t=F=0$ and $t=H=0$, as
well as the additional points
\begin{equation}\label{eq: homsl2extrapts}
(1,1,1,1),\;\; (1,1,-1,-1),\;\; (1,-1,1,-1)\;\text{ and }\; (1,-1,-1,1). 
\end{equation} 
\end{lemma}
\begin{proof}
By Proposition \ref{prop: homogeniseprops}, $U_h(\mathfrak{sl}_2(\C))^{G,\mu}$ satisfies the hypotheses of 
Theorem \ref{thm: pointschemenice}. That theorem implies that the graph of the point scheme of
$U_h(\mathfrak{sl}_2(\C))^{G,\mu}$ under its associated automorphism arises as the scheme determined by the
multilinearisations\index{term}{multilinearisations} of the defining relations. 

The matrix formulation of such multilinearisations (as in Lemma \ref{lem: pointsparameterisevcliff}) can be used to see
that the point scheme consists of points $(t,E,F,H) \in \proj{k}{3}$  for which the matrix
\begin{equation}\label{eq: homsl2matrix}
\left(\begin{array}{cccc}
-2H & F & E & 0 \\
-2F & H & 0 & E \\
-2E & 0 & H & F \\
E & -t & 0 & 0 \\
F & -t & 0 & 0 \\
H & 0 & 0 & -t 
\end{array}\right) 
\end{equation}
has rank precisely 3.

Consider those point modules which are annihilated by $t$, which enables us to set $t= 0$ in \eqref{eq: homsl2matrix}.
Such point modules restrict to the factor ring $U_h(\mathfrak{sl}_2(\C))^{G,\mu}/(t)$. This ring is isomorphic as a
$\C$-algebra to 
 \begin{equation*}%\label{eq: }
\C_{-1}[E,F,H] := \C\{E,F,H\}/(EF+FE,EH+HE,FH+HF). 
\end{equation*} 
It is well-known that any point module over this algebra is annihilated by one of the generators, with any point on one
of the lines in the statement of the result giving rise to a point module.
 
By the dichotomy on point modules it remains to consider point modules on which $t$ acts faithfully. Thus we may assume
that $t=1$ into \eqref{eq: homsl2matrix}. Under this condition, the only points at which the matrix in \eqref{eq:
homsl2matrix} has rank 3 are precisely those stated in \eqref{eq: homsl2extrapts}. 
\end{proof}

% APPENDICES
\appendix

\chapter{Calculations}\label{app: calc}
In this appendix we collect together some of the calculations that were omitted earlier in the thesis.

\section{Calculations for Chapter \ref{chap: thcrtwist}}\label{sec: calcthcr}

\subsection{Additional proof of Proposition \ref{prop: bgmunoptmodules}}\label{app: calcptmodnonproof}
We now give the more computational proof that the point scheme of $B^{G,\mu}$ is empty when $|\sigma| = \infty$.

The point modules over $B^{G,\mu}$ are precisely those over $A^{G,\mu}$ that are annihilated by both of the central
elements $\Theta_1$ and $\Theta_2$. We give the necessary computations to show that the $\Theta_1$ does not vanish at
any of the 20 points in the point scheme of $A^{G,\mu}$. 

By Lemma \ref{lem: annGinvariant}, if an element that is fixed by $G$ annihilates a point module then it annihilates the
other point modules in the same $G$-orbit. Thus it suffices to show that $\Theta_1$ does not vanish when evaluated at
one point from each $G$-orbit. 

In Lemma \ref{lem: gammaorbits} we described the $G$-orbits of the point scheme of $A^{G,\mu}$; there are four singleton
orbits consisting of a point of the form $e_j$ and four more orbits of order 4. Furthermore, Lemma \ref{lem:
ptschemecontains} describes the behaviour of these points under the automorphism $\phi$ that is associated to the point
scheme. It is clear that $\Theta_1(e_j,e_j)\neq 0$ for $i=0,1,2,3$, therefore we can immediately dismiss the singleton
orbits. 

Now let $p=(1, i, -i ,-1)$, which will represent the remaining orbit of points fixed under $\phi$. Computing
$\Theta_1(p,p)$, we find that
\begin{equation*}%\label{eq: 4sklyaninptmodeg}
\Theta_1(p,p) = -v_0(1)v_0(1)+v_1(i)v_1(i)+v_2(-i)v_2(-i)-v_3(-1)v_3(-1) =-4.
\end{equation*}

Now let $p = \left(1,-(\beta \gamma)^{-\frac{1}{2}},\gamma^{-\frac{1}{2}},\beta^{-\frac{1}{2}}\right)$, therefore
$p^{\phi}=\left(1,-(\beta \gamma)^{-\frac{1}{2}},-\gamma^{-\frac{1}{2}},-\beta^{-\frac{1}{2}}\right)$. Evaluating
$\Theta_1$ at $(p,p^{\phi})$ gives
\begin{gather}
\begin{aligned}\label{eq: 4sklyaninptmodeg1}
\Theta_1(p,p^{\phi}) =\; &-v_0(1)v_0(1)+v_1(-(\beta \gamma)^{-\frac{1}{2}})v_1(-(\beta
\gamma)^{-\frac{1}{2}})+v_2(\gamma^{-\frac{1}{2}})v_2(-\gamma^{-\frac{1}{2}})\\
&-v_3(\beta^{-\frac{1}{2}})v_3(-\beta^{-\frac{1}{2}})\\
=\; &-1 + \frac{1}{\beta \gamma}-\frac{1}{\gamma}+\frac{1}{\beta}.
\end{aligned}
\end{gather}
If \eqref{eq: 4sklyaninptmodeg1} were equal to zero then it could be rearranged to the form $(1-\beta)(\gamma+1)=0$.
Both solutions correspond to choices of parameters that do not satisfy \eqref{eq: 4sklyanincoeffcond}, which is not
permitted.   

Now let $p = \left(1,i \gamma^{-\frac{1}{2}},(\alpha \gamma)^{-\frac{1}{2}}, i \alpha^{-\frac{1}{2}}\right)$, thus
$p^{\phi} =\left(1,-i \gamma^{-\frac{1}{2}},(\alpha \gamma)^{-\frac{1}{2}}, -i \alpha^{-\frac{1}{2}}\right)$. One has
\begin{gather}
\begin{aligned}\label{eq: 4sklyaninptmodeg2}
\Theta_1(p,p^{\phi} ) =\; &-v_0(1)v_0(1)+v_1(i \gamma^{-\frac{1}{2}})v_1(-i \gamma^{-\frac{1}{2}})+v_2((\alpha
\gamma)^{-\frac{1}{2}})v_2((\alpha \gamma)^{-\frac{1}{2}})\\
&-v_3(i \alpha^{-\frac{1}{2}})v_3(-i \alpha^{-\frac{1}{2}})\\
=\; &-1 + \frac{1}{\gamma}+\frac{1}{\alpha\gamma}-\frac{1}{\alpha}.
\end{aligned}
\end{gather}
As for the previous orbit, if \eqref{eq: 4sklyaninptmodeg2} were equal to zero then it could be rearranged to the form
$(\alpha+1)(1-\gamma)=0$. Again, both solutions correspond to choices of parameters that we have excluded. 

Finally, let $p = \left(1,\beta^{-\frac{1}{2}}, i \alpha^{-\frac{1}{2}},i(\alpha \beta)^{-\frac{1}{2}}\right)$, thus
$p^{\phi} =\left(1,-\beta^{-\frac{1}{2}}, -i \alpha^{-\frac{1}{2}},i(\alpha \beta)^{-\frac{1}{2}}\right)$. Once again
\begin{gather}
\begin{aligned}\label{eq: 4sklyaninptmodeg3}
\Theta_1(p,p^{\phi}) =\; &-v_0(1)v_0(1)+v_1(\beta^{-\frac{1}{2}})v_1(-\beta^{-\frac{1}{2}})+v_2(i
\alpha^{-\frac{1}{2}})v_2(-i \alpha^{-\frac{1}{2}})\\
&-v_3(i(\alpha \beta)^{-\frac{1}{2}})v_3(i(\alpha \beta)^{-\frac{1}{2}})\\
=\; &-1 - \frac{1}{\beta}+\frac{1}{\alpha}+\frac{1}{\alpha \beta}.
\end{aligned}
\end{gather}
If \eqref{eq: 4sklyaninptmodeg3} were equal to zero then it could be rearranged to the form $(1-\alpha)(\beta+1)=0$,
both of whose solutions lead to forbidden parameter triples. 

\section{Calculations for Chapter \ref{chap: othertwists}}

\subsection{Relations in \S\ref{sec: vancliffql}}\label{subsec: calcvancliff}
The relations of $R(\alpha,\beta,\lambda)^{G,\mu}$ that are not preserved under twisting are calculated as follows:
\begin{align*}%\label{eq: vancliffcomputerrelns}
0 &=x_4x_3 - \alpha x_3 x_4 = \frac{x_4 \ast_{\mu} x_3}{\mu(g_1g_2,g_2)} - \alpha \frac{x_3 \ast_{\mu}
x_4}{\mu(g_2,g_1g_2)} = -(x_4 \ast_{\mu} x_3 + \alpha x_3 \ast_{\mu} x_4). \\ 
0 &=x_4 x_2 - \lambda x_2 x_4 = \frac{x_4 \ast_{\mu} x_2}{\mu(g_1g_2,g_1)} - \lambda \frac{x_2 \ast_{\mu}
x_4}{\mu(g_1,g_1g_2)} = x_4 \ast_{\mu} x_2+ \lambda x_2 \ast_{\mu} x_4. \\
0 &=x_3x_2 - \beta x_2 x_3 - (\alpha \beta - \lambda)x_1 x_4 \\
&= \frac{x_3 \ast_{\mu} x_2}{\mu(g_2,g_1)} - \beta \frac{x_2 \ast_{\mu} x_3}{\mu(g_1,g_2)} - (\alpha \beta - \lambda)
\frac{x_1 \ast_{\mu} x_4}{\mu(e,g_1g_2)} \\
&= x_3 \ast_{\mu} x_2 + \beta x_2 \ast_{\mu} x_3  - (\alpha \beta - \lambda) x_1 \ast_{\mu} x_4. 
\end{align*} 

\subsection{Calculations for \S\ref{subsec: gradedskewclifford}}\label{subsec: calcskewcliff}
The relations of $A(\mu,M_1,\ldots,M_n)^{G,\tau}$ are computed as follows:
\begin{gather}
\begin{aligned}\label{eq: relnskewcliff}
0 &= x_4x_1-ix_1x_4= \frac{x_4 \ast_{\tau} x_1}{\tau(g_2,g_1)}-i \frac{x_1 \ast_{\tau} x_4}{\tau(g_1,g_2)} = x_4
\ast_{\tau} x_1 + i x_1 \ast_{\tau} x_4.  \\
0 &= x_3^2-x_1^2 = \frac{x_3 \ast_{\tau} x_3}{\tau(g_1,g_1)}- \frac{x_1 \ast_{\tau} x_1}{\tau(g_1,g_1)} = x_3
\ast_{\tau} x_3 - x_1 \ast_{\tau} x_1.  \\
0 &= x_3x_1-x_1x_3+x_2^2 = \frac{x_3 \ast_{\tau} x_1}{\tau(g_1,g_1)}- \frac{x_1 \ast_{\tau} x_3}{\tau(g_1,g_1)}
+\frac{x_2 \ast_{\tau} x_2}{\tau(g_2,g_2)}  \\
&= x_3 \ast_{\tau} x_1 - x_1 \ast_{\tau} x_3 + x_2 \ast_{\tau} x_2.  \\
0 &= x_3x_2-ix_2x_3 = \frac{x_3 \ast_{\tau} x_2}{\tau(g_1,g_2)}- i \frac{x_2 \ast_{\tau} x_3}{\tau(g_2,g_1)} = - x_3
\ast_{\tau} x_2 - i x_2 \ast_{\tau} x_3. \\
0 &= x_4^2-x_2^2 = \frac{x_4 \ast_{\tau} x_4}{\tau(g_2,g_2)}- \frac{x_2 \ast_{\tau} x_2}{\tau(g_2,g_2)} = x_4
\ast_{\tau} x_4 - x_2 \ast_{\tau} x_2. \\
0 &= x_4x_2-x_2x_4+\gamma x_1^2 = \frac{x_4 \ast_{\tau} x_2}{\tau(g_2,g_2)}- \frac{x_2 \ast_{\tau} x_4}{\tau(g_2,g_2)} +
\gamma \frac{x_1 \ast_{\tau} x_1}{\tau(g_1,g_1)} \\
&= x_4 \ast_{\tau} x_2 - x_2 \ast_{\tau} x_4 + \gamma x_1 \ast_{\tau} x_1.
\end{aligned}
\end{gather}

We now show that $S/(q_1,q_2,q_3,q_4)$ is finite-dimensional. As noted in \S\ref{subsec: gradedskewclifford}, it
suffices to show that sufficiently high powers of the generators vanish. The following calculations verify this:
\begin{gather}
\begin{aligned}\label{eq: 5thpowervanish1}
z_1^5 &=z_1(-z_3^2-\gamma z_2 z_4)z_1^2=-z_1z_3^2 z_1^2=(z_2^2+z_4^2)z_3z_1^2=0, \\
z_2^5 &=z_2(-z_4^2- z_1 z_3)z_2^2=-z_2z_4^2 z_2^2=\gamma^{-1}(z_1^2+z_3^2)z_4z_2^2=0, \\
z_3^5 &=z_3(-z_1^2-\gamma z_2 z_4)z_3^2=-z_3z_1^2 z_3^2=z_3 z_1(z_2^2+z_4^2)z_3=0, \\
z_4^5 &=z_4(-z_2^2-z_1 z_3)z_4^2=-z_4z_2^2 z_4^2=\gamma^{-1} z_4 z_2(z_1^2+z_3^2)z_4=0.
\end{aligned}
\end{gather}

Our final calculation in this section shows that the quadric system given in \eqref{eq: cliffordnormal} is still
normalising for the multiplicatively skew-symmetric matrix in \eqref{eq: musystem}. Note that $q_1$ and $q_2$ are
monomials and therefore are clearly normal in $S$, therefore we only need to check that $q_3$ and $q_4$ are normal. At
first we only have the partial information given in \eqref{eq: multskwesymmetric}, from which we will deduce the
remaining entries.
\begin{gather}
\begin{aligned}\label{eq: normalisingsystem}
z_1 q_3 = z_1^3 +z_1 z_3^2 + \gamma z_1 z_2 z_4 &= (z_1^2+ \mu_{31}^2 z_3^2+ \mu_{21}\mu_{41} \gamma z_2 z_4) z_1 \\
&= (z_1^2+ z_3^2- \mu_{21} i \gamma z_2 z_4) z_1.
\end{aligned}
\end{gather}
This suggests that we must take $\mu_{21}=i=-\mu_{12}$ so that $z_1 q_3= q_3 z_1$. Continuing in the same manner we have
\begin{gather}
\begin{aligned}\label{eq: normalisingsystem1} 
z_2 q_3 = z_2 z_1^2 +z_2 z_3^2 + \gamma z_2^2 z_4 &= (\mu_{12}^2 z_1^2+ \mu_{32}^2 z_3^2+ \mu_{42} \gamma z_2 z_4) z_2 =
-q_3 z_2, \\
z_3 q_3 = z_3 z_1^2 + z_3^3 + \gamma z_3 z_2 z_4 &= (\mu_{13}^2 z_1^2+ z_3^2+ \mu_{23}\mu_{43} \gamma z_2 z_4) z_3 \\
&= (z_1^2+ z_3^2-i \mu_{43} \gamma z_2 z_4) z_3.
\end{aligned}
\end{gather}
This suggests that we must take $\mu_{43}=i=-\mu_{34}$ so that $z_3 q_3= q_3 z_3$, completing the skew-symmetric matrix
$\mu'$. The final calculation for $q_3$ is
\begin{gather}
\begin{aligned}\label{eq: normalisingsystem3} 
z_4 q_3 &= z_4 z_1^2 +z_4 z_3^2 + \gamma z_4 z_2 z_4 = (\mu_{14}^2 z_1^2+ \mu_{34}^2 z_3^2+ \mu_{24} \gamma z_2 z_4) z_4
= -q_3 z_4.
\end{aligned}
\end{gather}

Checking that $q_4$ is normal is now straightforward:
\begin{gather}
\begin{aligned}\label{eq: normalisingsystem4}
z_1 q_4 &= z_1 z_2^2+z_1 z_4^2+z_1^2 z_3 = (\mu_{21}^2 z_2^2+\mu_{41}^2 z_4^2+ \mu_{31} z_1 z_3) z_1 = -q_4 z_1,\\
z_2 q_4 &= z_2^3+z_2 z_4^2+z_2 z_1 z_3 = (z_2^2+\mu_{42}^2 z_4^2+ \mu_{12}\mu_{32} z_1 z_3) z_2 = q_4 z_2,\\
z_3 q_4 &= z_3 z_2^2+z_3 z_4^2+z_3 z_1 z_3 = (\mu_{23}^2 z_2^2+\mu_{43}^2 z_4^2+ \mu_{13} z_1 z_3) z_3 = -q_4 z_3,\\
z_4 q_4 &= z_4 z_2^2+ z_4^3+z_4 z_1 z_3 = (\mu_{24}^2 z_2^2+ z_4^2+ \mu_{14}\mu_{34} z_1 z_3) z_4 = q_4 z_4,\\
\end{aligned}
\end{gather}

\subsection{Relations in \S\ref{subsec: homenvelopalg}}\label{subsec: calcenvelop}
The relations of $U(\mathfrak{sl}_2)^{G,\mu}$ are computed below.
\begin{align*}\label{eq: relnenvelop}
0 &= EF-FE+2H = \frac{E\ast_{\mu}F}{\mu(g_1,g_1g_2)}-\frac{F\ast_{\mu}E}{\mu(g_1g_2,g_1)} +2H = -E\ast_{\mu}F-
F\ast_{\mu}E +2H. \\
0 &= HE-EH-2F = \frac{H\ast_{\mu}E}{\mu(g_2,g_1)}-\frac{E\ast_{\mu}H}{\mu(g_1,g_2)} -2F = H\ast_{\mu}E + E\ast_{\mu}H -2
F.  \\
0 &= HF-FH-2E = \frac{H\ast_{\mu}F}{\mu(g_2,g_1g_2)}-\frac{F\ast_{\mu}H}{\mu(g_1g_2,g_2)} -2E = H\ast_{\mu}F +
F\ast_{\mu}H - 2E. 
\end{align*}

\chapter{Computer code}\label{app: comp}
In this appendix we give the computer code used in several proofs earlier in the thesis. The two programs that we used
in relation to the work in this thesis are Macaulay2 and Affine, a brief description of which follow.

Macaulay2 is a commutative algebra/algebraic geometry program that we use to prove results regarding the point and line
schemes of some cocycle twists. More information on this software can be obtained from \cite{M2} or
\cite{eisenbud2002computations}.

Affine is a package for the computer algebra system Maxima \cite{maxima}. It can be used to perform Groebner basis-like
computations with noncommutative algebras. More specifically, Affine implements the diamond lemma, an important result
that originated in graph theory. It was first studied in the context of noncommutative algebras in
\cite{bergman1978diamond}. We utilised Affine to study the behaviour of several examples of cocycle twists.

\section{Code for \S\ref{subsec: 4sklyanintwistandptscheme}}\label{sec: sklyanincalcs}
The following Macaulay2 code is needed in Proposition \ref{prop: linescheme1dim} to calculate the dimension of the line
scheme of $A(\alpha,\beta,\gamma)^{G,\mu}$.
\begin{code}\label{code: lineschemeagmu}
\begin{flalign*}%\label{eq: lineschemesklytwist}
&\mathtt{F = QQ[a,b,c]/(a+b+c+a*b*c);}&\\ 
&\mathtt{k = frac F;}&\\ 
&\mathtt{S=k[t1,t2,t3,t4,t5,t6];}&\\ 
&\mathtt{I=minors(3, matrix\{\{0,t1+t2,t3+t4,t5+t6\},}&\\ 
&\mathtt{\{t2-t1,0,c*t5+t6,b*t3-t4\},\{t4-t3,t6-c*t5,0,-a*t1-t2\},}&\\ 
&\mathtt{\{t6-t5,-b*t3-t4,a*t1-t2,0\}\});}&\\
&\mathtt{R=S/I;}&\\
&\mathtt{codim\;I}&\\
&\mathtt{Output:\; 4}&
\end{flalign*}
\end{code} 

\section{Code for \S\ref{subsec: staffordalgs}}\label{subsec: staffordcalcs1}
The following two pieces of Macaulay2 code are used in the proof of Proposition \ref{prop: staffordptschemes}. The first
piece of code is used to determine the intersection of the scheme $\Gamma_2$ associated to
$S_{\infty}^{G,\mu}(\alpha,\beta,\gamma)$ with the affine subscheme of $\proj{k}{3} \times \proj{k}{3}$ in which
$v_{01},v_{02} \neq 0$.
\begin{code}\label{code: lineschemeainfty1}
\begin{flalign*}%\label{eq: staffordcalcs1}
&\mathtt{F = QQ[a,b,c]/(a+b+c+a*b*c);}&\\
&\mathtt{k = frac F;}&\\
&\mathtt{S = k[v11,v12,v21,v22,v31,v32];}&\\
&\mathtt{I = ideal(v12-v11-a*v21*v32+a*v31*v22,}&\\
&\mathtt{v11+v12-v21*v32-v31*v22,}&\\ 
&\mathtt{v22-v21+b*v11*v32-b*v31*v12, v22+v21-v11*v32}&\\
&\mathtt{-v31*v12, -1+v11*v12+v21*v22-v31*v32,}&\\
&\mathtt{v11*v12+((1+a)/(1-b))*v21*v22-((1-a)/(1+c))*v31*v32);}&\\
&\mathtt{R = S/I;}&\\
&\mathtt{basis\; R} &\\
&\mathtt{Output:\; 16-dimensional\;basis}&
\end{flalign*}
\end{code}

The second piece of code shows that there are no points lying in the intersection of $\Gamma_2$ with the locus where
$v_{02}=0$ and the remaining coordinates are non-zero. We can therefore assume that $v_{01}=1$ and
$v_{11}v_{12}v_{21}v_{22}v_{31}v_{32}=1$ by rescaling the two copies of $\proj{k}{3}$.
\begin{code}\label{code: lineschemeainfty2}
\begin{flalign*}%\label{eq: staffordcalcs2}
&\mathtt{F = QQ[a,b,c]/(a+b+c+a*b*c);}&\\
&\mathtt{k = frac F;}&\\
&\mathtt{S = k[v11,v12,v21,v22,v31,v32];}&\\
&\mathtt{I = ideal(v12-a*v21*v32+a*v31*v22, v12-v21*v32-v31*v22,}&\\
&\mathtt{v21+b*v11*v32-b*v31*v12, v22-v11*v32-v31*v12,}&\\
&\mathtt{v11*v12+v21*v22-v31*v32, v11*v12*v21*v22*v31*v32-1,}&\\
&\mathtt{v11*v12+((1+a)/(1-b))*v21*v22-((1-a)/(1+c))*v31*v32);}&\\
&\mathtt{R = S/I;}&\\
&\mathtt{dim\; R}&\\
&\mathtt{Output:\; -infinity}&
\end{flalign*}
\end{code} 

Finally, we give the Macaulay2 code used in the proof of Proposition \ref{prop: staffordlinescheme}, which is used to
calculate the dimension of the line scheme of $S_{\infty}^{G,\mu}$.
\begin{code}\label{code: lineschemeainfty}
\begin{flalign*}%\label{eq: staffordcalcs3}
&\mathtt{F = QQ[a,b,c]/(a+b+c+a*b*c);}&\\
&\mathtt{k = frac F;}&\\
&\mathtt{S = k[t1,t2,t3,t4,t5,t6];}&\\
&\mathtt{I = minors(3, matrix\{\{-t5,t1+t2,t3+t4,0\},}&\\
&\mathtt{\{t2-t1,t5+t6,0,b*t3-t4\},}&\\
&\mathtt{\{t4-t3,0,t5+t6*((1+a)/(1-b)),-a*t1-t2\},}&\\
&\mathtt{\{0,-b*t3-t4,a*t1-t2,-t5-t6*((1-a)/(1+c))\}\});}&\\
&\mathtt{codim\; I}&\\
&\mathtt{Output:\; 4}&
\end{flalign*} 
\end{code}

\chapter{Properties preserved under twisting}\label{app: preserve}
In this appendix we give a table summarising the properties that we show are preserved under cocycle twisting in
\S\ref{sec: preservation}. The table includes a reference for each result and any relevant hypotheses.
\\

\begin{tabular}{| c | c | c |}
 \hline \begin{large}\textbf{Property}\end{large} & \begin{large}\textbf{Reference}\end{large} &
\begin{large}\textbf{Hypotheses}\end{large}  \\
 
 \hline Hilbert series & Lemma \ref{lem: hilbseries} & Hyp. \ref{hyp: gradedcase} \\
 
 \hline Strongly noetherian & Corollary \ref{cor: uninoeth} &  \\ 
 
 \cline{1-2} GK dimension & Proposition \ref{prop: gkdim} & Hyp. \ref{hyp: generalcase} \\
 
 \cline{1-2} Global dimension & Proposition \ref{prop: gldim} &  \\
 
 \hline AS-Gorenstein & Proposition \ref{prop: asgor} & Hyp. \ref{hyp: gradedcase} \\
 
 \hline AS-regular & Corollary \ref{cor: asreg} & Hyp. \ref{hyp: gradedcase}, c.g.\ algebra \\
 
 \hline Koszul & Proposition \ref{prop: koszul} & Hyp. \ref{hyp: gradedcase}, quadratic algebra \\
 
 \hline Cohen-Macaulay &  &  \\
 
  Auslander-Gorenstein & Proposition \ref{prop: cohenmac} & Hyp. \ref{hyp: gradedcase}, noetherian algebra \\
  
  Auslander regular &  &  \\
 \hline
  
\end{tabular}

% BIBLIOGRAPHY
\bibliographystyle{amsplain}   
\bibliography{allrefs.bib}

\def\cprime{$'$} \def\cprime{$'$} \def\cprime{$'$} \def\cprime{$'$}
  \def\cprime{$'$} \def\cprime{$'$} \def\cprime{$'$} \def\cprime{$'$}
  \def\cprime{$'$} \def\cprime{$'$}
\providecommand{\bysame}{\leavevmode\hbox to3em{\hrulefill}\thinspace}
\providecommand{\MR}{\relax\ifhmode\unskip\space\fi MR }
\providecommand{\MRhref}[2]{%
  \href{http://www.ams.org/mathscinet-getitem?mr=#1}{#2}
}
\providecommand{\href}[2]{#2}
\begin{thebibliography}{10}

\bibitem{ardakov2007primeness}
K.~Ardakov and K.~A. Brown, \emph{Primeness, semiprimeness and localisation in
  {I}wasawa algebras}, Trans. Amer. Math. Soc. \textbf{359} (2007), no.~4,
  1499--1515. \MR{2272136 (2007j:16031)}

\bibitem{artin1990geometry}
M.~Artin, \emph{Geometry of quantum planes}, Azumaya algebras, actions, and
  modules ({B}loomington, {IN}, 1990), Contemp. Math., vol. 124, Amer. Math.
  Soc., Providence, RI, 1992, pp.~1--15. \MR{1144023 (93b:14004)}

\bibitem{artin1987graded}
M.~Artin and W.~F. Schelter, \emph{Graded algebras of global dimension {$3$}},
  Adv. in Math. \textbf{66} (1987), no.~2, 171--216. \MR{917738 (88k:16003)}

\bibitem{artin1999generic}
M.~Artin, L.~W. Small, and J.~J. Zhang, \emph{Generic flatness for strongly
  {N}oetherian algebras}, J. Algebra \textbf{221} (1999), no.~2, 579--610.
  \MR{1728399 (2001a:16006)}

\bibitem{artin1995noncommutative}
M.~Artin and J.~T. Stafford, \emph{Noncommutative graded domains with quadratic
  growth}, Invent. Math. \textbf{122} (1995), no.~2, 231--276. \MR{1358976
  (96g:16027)}

\bibitem{artin2000semiprime}
\bysame, \emph{Semiprime graded algebras of dimension two}, J. Algebra
  \textbf{227} (2000), no.~1, 68--123. \MR{1754226 (2002j:16032)}

\bibitem{artin1990some}
M.~Artin, J.~Tate, and M.~Van~den Bergh, \emph{Some algebras associated to
  automorphisms of elliptic curves}, The {G}rothendieck {F}estschrift, {V}ol.\
  {I}, Progr. Math., vol.~86, Birkh\"auser Boston, Boston, MA, 1990,
  pp.~33--85. \MR{1086882 (92e:14002)}

\bibitem{artin1991modules}
\bysame, \emph{Modules over regular algebras of dimension {$3$}}, Invent. Math.
  \textbf{106} (1991), no.~2, 335--388. \MR{1128218 (93e:16055)}

\bibitem{artin1990twisted}
M.~Artin and M.~Van~den Bergh, \emph{Twisted homogeneous coordinate rings}, J.
  Algebra \textbf{133} (1990), no.~2, 249--271. \MR{1067406 (91k:14003)}

\bibitem{artin1994noncommutative}
M.~Artin and J.~J. Zhang, \emph{Noncommutative projective schemes}, Adv. Math.
  \textbf{109} (1994), no.~2, 228--287. \MR{1304753 (96a:14004)}

\bibitem{artin2001abstract}
\bysame, \emph{Abstract {H}ilbert schemes}, Algebr. Represent. Theory
  \textbf{4} (2001), no.~4, 305--394. \MR{1863391 (2002h:16046)}

\bibitem{bazlov2012cocycle}
Y.~Bazlov and A.~Berenstein, \emph{Cocycle twists and extensions of braided
  doubles}, arXiv preprint, arXiv:1211.5279 (2012).

\bibitem{koszul1996beilinson}
A.~Beilinson, V.~Ginzburg, and W.~Soergel, \emph{Koszul duality patterns in
  representation theory}, J. Amer. Math. Soc. \textbf{9} (1996), no.~2,
  473--527. \MR{1322847 (96k:17010)}

\bibitem{bergman1978diamond}
G.~M. Bergman, \emph{The diamond lemma for ring theory}, Adv. in Math.
  \textbf{29} (1978), no.~2, 178--218. \MR{506890 (81b:16001)}

\bibitem{brown1985cohomology}
K.~A. Brown and T.~Levasseur, \emph{Cohomology of bimodules over enveloping
  algebras}, Math. Z. \textbf{189} (1985), no.~3, 393--413. \MR{783564
  (86m:17011)}

\bibitem{cassidy2006generalized}
T.~Cassidy, P.~Goetz, and B.~Shelton, \emph{Generalized {L}aurent polynomial
  rings as quantum projective 3-spaces}, J. Algebra \textbf{303} (2006), no.~1,
  358--372. \MR{2253666 (2008b:16042)}

\bibitem{cassidy2013corrigendum}
T.~Cassidy and M.~Vancliff, \emph{Corrigendum to `{G}eneralizations of graded
  {C}lifford algebras and of complete intersections'}, preprint,
  \url{http://www.uta.edu/math/vancliff/R/scliff-cor.pdf}.

\bibitem{cassidy2010generlizations}
\bysame, \emph{Generalizations of graded {C}lifford algebras and of complete
  intersections}, J. Lond. Math. Soc. (2) \textbf{81} (2010), no.~1, 91--112.
  \MR{2580455 (2011b:16101)}

\bibitem{eisenbud2002computations}
D.~Eisenbud, D.~R. Grayson, M.~Stillman, and B.~Sturmfels (eds.),
  \emph{Computations in algebraic geometry with {M}acaulay 2}, Algorithms and
  Computation in Mathematics, vol.~8, Springer-Verlag, Berlin, 2002.
  \MR{1949544 (2004b:14002)}

\bibitem{froberg1975determination}
R.~Fr{\"o}berg, \emph{Determination of a class of {P}oincar\'e series}, Math.
  Scand. \textbf{37} (1975), no.~1, 29--39. \MR{0404254 (53 \#8057)}

\bibitem{goetz2003noncommutative}
P.~D. Goetz, \emph{The noncommutative algebraic geometry of quantum projective
  spaces}, Ph.D. Thesis, University of Oregon, 2003,
  \url{http://users.humboldt.edu/pgoetz/Papers/GOETZthesis.pdf}.

\bibitem{goodearl2000graded}
K.~R. Goodearl and J.~T. Stafford, \emph{The graded version of {G}oldie's
  theorem}, Algebra and its applications ({A}thens, {OH}, 1999), Contemp.
  Math., vol. 259, Amer. Math. Soc., Providence, RI, 2000, pp.~237--240.
  \MR{1780524 (2001g:16087)}

\bibitem{goodearl2004introduction}
K.~R. Goodearl and R.~B. Warfield, Jr., \emph{An introduction to noncommutative
  {N}oetherian rings}, second ed., London Mathematical Society Student Texts,
  vol.~61, Cambridge University Press, Cambridge, 2004. \MR{2080008
  (2005b:16001)}

\bibitem{M2}
D.~R. Grayson and M.~E. Stillman, \emph{Macaulay2, a software system for
  research in algebraic geometry}, \url{http://www.math.uiuc.edu/Macaulay2/}.

\bibitem{hartshorne1977algebraic}
R.~Hartshorne, \emph{Algebraic geometry}, Springer-Verlag, New York, 1977,
  Graduate Texts in Mathematics, No. 52. \MR{0463157 (57 \#3116)}

\bibitem{isaacs1976character}
I.~M. Isaacs, \emph{Character theory of finite groups}, Academic Press
  [Harcourt Brace Jovanovich Publishers], New York, 1976, Pure and Applied
  Mathematics, No. 69. \MR{0460423 (57 \#417)}

\bibitem{karpilovsky1987schur}
G.~Karpilovsky, \emph{The {S}chur multiplier}, London Mathematical Society
  Monographs. New Series, vol.~2, The Clarendon Press Oxford University Press,
  New York, 1987. \MR{1200015 (93j:20002)}

\bibitem{keeler2000criteria}
D.~S. Keeler, \emph{Criteria for $\sigma$-ampleness}, J. Amer. Math. Soc.
  \textbf{13} (2000), no.~3, 517--532 (electronic). \MR{1758752 (2001d:14003)}

\bibitem{krahmernotes}
U.~Kr{\"a}hmer, \emph{Notes on {K}oszul algebras},
  \url{http://www.maths.gla.ac.uk/~ukraehmer/connected.pdf}, 2011, Accessed:
  10/12/13.

\bibitem{krause2000growth}
G.~R. Krause and T.~H. Lenagan, \emph{Growth of algebras and
  {G}elfand-{K}irillov dimension}, revised ed., Graduate Studies in
  Mathematics, vol.~22, American Mathematical Society, Providence, RI, 2000.
  \MR{1721834 (2000j:16035)}

\bibitem{lebruyn1995central}
L.~Le~Bruyn, \emph{Central singularities of quantum spaces}, J. Algebra
  \textbf{177} (1995), no.~1, 142--153. \MR{1356364 (96k:16051)}

\bibitem{lebruyn1993homogenized}
L.~Le~Bruyn and S.~P. Smith, \emph{Homogenized ${\mathfrak s}{\mathfrak
  l}(2)$}, Proc. Amer. Math. Soc. \textbf{118} (1993), no.~3, 725--730.
  \MR{1136235 (93i:16056)}

\bibitem{lebruyn1993on}
L.~Le~Bruyn and M.~Van~den Bergh, \emph{On quantum spaces of {L}ie algebras},
  Proc. Amer. Math. Soc. \textbf{119} (1993), no.~2, 407--414. \MR{1149975
  (93k:17021)}

\bibitem{levasseur1992some}
T.~Levasseur, \emph{Some properties of noncommutative regular graded rings},
  Glasgow Math. J. \textbf{34} (1992), no.~3, 277--300. \MR{1181768
  (93k:16045)}

\bibitem{levasseur1993modules}
T.~Levasseur and S.~P. Smith, \emph{Modules over the {$4$}-dimensional
  {S}klyanin algebra}, Bull. Soc. Math. France \textbf{121} (1993), no.~1,
  35--90. \MR{1207244 (94f:16054)}

\bibitem{liu2002algebraic}
Q.~Liu, \emph{Algebraic geometry and arithmetic curves}, Oxford Graduate Texts
  in Mathematics, vol.~6, Oxford University Press, Oxford, 2002, Translated
  from the French by Reinie Ern{\'e}, Oxford Science Publications. \MR{1917232
  (2003g:14001)}

\bibitem{lorenz1988on}
M.~Lorenz, \emph{On {G}el\cprime fand-{K}irillov dimension and related topics},
  J. Algebra \textbf{118} (1988), no.~2, 423--437. \MR{969682 (89m:16004)}

\bibitem{lu2007regular}
D.-M. Lu, J.~H. Palmieri, Q.-S. Wu, and J.~J. Zhang, \emph{Regular algebras of
  dimension 4 and their {$A_{\infty}$}-{E}xt-algebras}, Duke Math. J.
  \textbf{137} (2007), no.~3, 537--584. \MR{2309153 (2008d:16022)}

\bibitem{maxima}
Maxima, \emph{Maxima, a computer algebra system},
  \url{http://maxima.sourceforge.net/}.

\bibitem{mcconnell2001noncommutative}
J.~C. McConnell and J.~C. Robson, \emph{Noncommutative {N}oetherian rings},
  Pure and Applied Mathematics (New York), John Wiley \& Sons Ltd., Chichester,
  1987, With the cooperation of L. W. Small, A Wiley-Interscience Publication.
  \MR{934572 (89j:16023)}

\bibitem{montgomery1980fixd}
S.~Montgomery, \emph{Fixed rings of finite automorphism groups of associative
  rings}, Lecture Notes in Mathematics, vol. 818, Springer, Berlin, 1980.
  \MR{590245 (81j:16041)}

\bibitem{montgomery1993hopf}
\bysame, \emph{Hopf algebras and their actions on rings}, CBMS Regional
  Conference Series in Mathematics, vol.~82, Published for the Conference Board
  of the Mathematical Sciences, Washington, DC, 1993. \MR{1243637 (94i:16019)}

\bibitem{montgomery2005algebra}
\bysame, \emph{Algebra properties invariant under twisting}, Hopf algebras in
  noncommutative geometry and physics, Lecture Notes in Pure and Appl. Math.,
  vol. 239, Dekker, New York, 2005, pp.~229--243. \MR{2106932 (2005h:16067)}

\bibitem{nafari2011classifying}
M.~Nafari, M.~Vancliff, and Jun Zhang, \emph{Classifying quadratic quantum
  {$\mathbb{P}^2$}s by using graded skew {C}lifford algebras}, J. Algebra
  \textbf{346} (2011), 152--164. \MR{2842075 (2012h:16054)}

\bibitem{nastasescu1983strongly}
C.~N{\u{a}}st{\u{a}}sescu, \emph{Strongly graded rings of finite groups}, Comm.
  Algebra \textbf{11} (1983), no.~10, 1033--1071. \MR{700723 (84k:16004)}

\bibitem{nastasescu1982graded}
C.~N{\u{a}}st{\u{a}}sescu and F.~van Oystaeyen, \emph{Graded ring theory},
  North-Holland Mathematical Library, vol.~28, North-Holland Publishing Co.,
  Amsterdam, 1982. \MR{676974 (84i:16002)}

\bibitem{odesskii2002elliptic}
A.~V. Odesski{\u\i}, \emph{Elliptic algebras}, Uspekhi Mat. Nauk \textbf{57}
  (2002), no.~6(348), 87--122. \MR{1991863 (2004f:16044)}

\bibitem{rogalski2006proj}
Z.~Reichstein, D.~Rogalski, and J.~J. Zhang, \emph{Projectively simple rings},
  Adv. Math. \textbf{203} (2006), no.~2, 365--407. \MR{2227726 (2007i:16071)}

\bibitem{rogalski2004generic}
D.~Rogalski, \emph{Generic noncommutative surfaces}, Adv. Math. \textbf{184}
  (2004), no.~2, 289--341. \MR{2054018 (2005e:16047)}

\bibitem{rogalski2008canonical}
D.~Rogalski and J.~J. Zhang, \emph{Canonical maps to twisted rings}, Math. Z.
  \textbf{259} (2008), no.~2, 433--455. \MR{2390090 (2010b:16056)}

\bibitem{rogalski2012regular}
\bysame, \emph{Regular algebras of dimension 4 with 3 generators}, New trends
  in noncommutative algebra, Contemp. Math., vol. 562, Amer. Math. Soc.,
  Providence, RI, 2012, pp.~221--241. \MR{2905562}

\bibitem{rotman2008introduction}
J.~J. Rotman, \emph{An introduction to homological algebra}, Pure and Applied
  Mathematics, vol.~85, Academic Press Inc. [Harcourt Brace Jovanovich
  Publishers], New York, 1979. \MR{538169 (80k:18001)}

\bibitem{shelton2001koszul}
B.~Shelton and C.~Tingey, \emph{On {K}oszul algebras and a new construction of
  {A}rtin-{S}chelter regular algebras}, J. Algebra \textbf{241} (2001), no.~2,
  789--798. \MR{1843325 (2002e:16045)}

\bibitem{shelton1999embedding}
B.~Shelton and M.~Vancliff, \emph{Embedding a quantum rank three quadric in a
  quantum {$\mathbb{P}^3$}}, Comm. Algebra \textbf{27} (1999), no.~6,
  2877--2904. \MR{1687269 (2000m:16044)}

\bibitem{shelton1999some}
\bysame, \emph{Some quantum {$\mathbb{P}^3$}'s with one point}, Comm. Algebra
  \textbf{27} (1999), no.~3, 1429--1443. \MR{1669119 (2000c:16059)}

\bibitem{shelton2002schemes}
\bysame, \emph{Schemes of line modules {I}}, J. London Math. Soc. (2)
  \textbf{65} (2002), no.~3, 575--590. \MR{1895734 (2003i:16044)}

\bibitem{smith1989can}
S.~P. Smith, \emph{Can the {W}eyl algebra be a fixed ring?}, Proc. Amer. Math.
  Soc. \textbf{107} (1989), no.~3, 587--589. \MR{962247 (90b:16010)}

\bibitem{smith1992the}
\bysame, \emph{The 4-dimensional {S}klyanin algebra at points of finite order},
  preprint (1992).

\bibitem{smith1994four}
\bysame, \emph{The four-dimensional {S}klyanin algebras}, Proceedings of
  {C}onference on {A}lgebraic {G}eometry and {R}ing {T}heory in honor of
  {M}ichael {A}rtin, {P}art {I} ({A}ntwerp, 1992), vol.~8, 1994, pp.~65--80.
  \MR{1273836 (95i:16016)}

\bibitem{smith1992regularity}
S.~P. Smith and J.~T. Stafford, \emph{Regularity of the four-dimensional
  {S}klyanin algebra}, Compositio Math. \textbf{83} (1992), no.~3, 259--289.
  \MR{1175941 (93h:16037)}

\bibitem{smith1993irreducible}
S.~P. Smith and J.~M. Staniszkis, \emph{Irreducible representations of the
  {$4$}-dimensional {S}klyanin algebra at points of infinite order}, J. Algebra
  \textbf{160} (1993), no.~1, 57--86. \MR{1237078 (95c:16027)}

\bibitem{smith1994center}
S.~P. Smith and J.~Tate, \emph{The center of the {$3$}-dimensional and
  {$4$}-dimensional {S}klyanin algebras}, Proceedings of {C}onference on
  {A}lgebraic {G}eometry and {R}ing {T}heory in honor of {M}ichael {A}rtin,
  {P}art {I} ({A}ntwerp, 1992), vol.~8, 1994, pp.~19--63. \MR{1273835
  (95i:16031)}

\bibitem{stafford1994regularity}
J.~T. Stafford, \emph{Regularity of algebras related to the {S}klyanin
  algebra}, Trans. Amer. Math. Soc. \textbf{341} (1994), no.~2, 895--916.
  \MR{1148046 (94d:16040)}

\bibitem{staniszkis1996linear}
J.~M. Staniszkis, \emph{Linear modules over {S}klyanin algebras}, J. London
  Math. Soc. (2) \textbf{53} (1996), no.~3, 464--478. \MR{1396711 (97m:16079)}

\bibitem{stephenson1996artin}
D.~R. Stephenson, \emph{Artin-{S}chelter regular algebras of global dimension
  three}, J. Algebra \textbf{183} (1996), no.~1, 55--73. \MR{1397387
  (97h:16053)}

\bibitem{stephenson2007constructing}
D.~R. Stephenson and M.~Vancliff, \emph{Constructing {C}lifford quantum
  {$\mathbb{P}^3$}'s with finitely many points}, J. Algebra \textbf{312}
  (2007), no.~1, 86--110. \MR{2320448 (2009d:16042)}

\bibitem{tate1996homological}
J.~Tate and M.~Van~den Bergh, \emph{Homological properties of {S}klyanin
  algebras}, Invent. Math. \textbf{124} (1996), no.~1-3, 619--647. \MR{1369430
  (98c:16057)}

\bibitem{teo1996homological}
K.-M. Teo, \emph{Homological properties of {S}klyanin algebras}, Comm. Algebra
  \textbf{24} (1996), no.~9, 3027--3035. \MR{1396871 (97m:16081)}

\bibitem{ueno2003algebraic}
K.~Ueno, \emph{Algebraic geometry 3}, Translations of Mathematical Monographs,
  vol. 218, American Mathematical Society, Providence, RI, 2003, Further study
  of schemes, translated from the 1998 Japanese original by Goro Kato, Iwanami
  Series in Modern Mathematics. \MR{1989998 (2004g:14001)}

\bibitem{van1988example}
M.~Van~den Bergh, \emph{An example with 20 points}, Notes (1988).

\bibitem{van1996translation}
\bysame, \emph{A translation principle for the four-dimensional {S}klyanin
  algebras}, J. Algebra \textbf{184} (1996), no.~2, 435--490. \MR{1409223
  (98a:16047)}

\bibitem{vancliff1994quadratic}
M.~Vancliff, \emph{Quadratic algebras associated with the union of a quadric
  and a line in {$\mathbb P^3$}}, J. Algebra \textbf{165} (1994), no.~1,
  63--90. \MR{1272579 (95f:16032)}

\bibitem{vancliff1997embedding}
M.~Vancliff and K.~Van~Rompay, \emph{Embedding a quantum nonsingular quadric in
  a quantum {${\mathbb P}^3$}}, J. Algebra \textbf{195} (1997), no.~1, 93--129.
  \MR{1468885 (99a:16040)}

\bibitem{vancliff1998some}
M.~Vancliff, K.~Van~Rompay, and L.~Willaert, \emph{Some quantum {${\mathbb
  P}^3$}'s with finitely many points}, Comm. Algebra \textbf{26} (1998), no.~4,
  1193--1208. \MR{1612220 (99c:16045)}

\bibitem{yi1995injective}
Z.~Yi, \emph{Injective homogeneity and the {A}uslander-{G}orenstein property},
  Glasgow Math. J. \textbf{37} (1995), no.~2, 191--204. \MR{1333738}

\bibitem{zhang1998twisted}
J.~J. Zhang, \emph{Twisted graded algebras and equivalences of graded
  categories}, Proc. London Math. Soc. (3) \textbf{72} (1996), no.~2, 281--311.
  \MR{1367080 (96k:16078)}

\bibitem{zhang2008double}
James~J. Zhang and Jun Zhang, \emph{Double {O}re extensions}, J. Pure Appl.
  Algebra \textbf{212} (2008), no.~12, 2668--2690. \MR{2452318 (2010h:16066)}

\end{thebibliography}

\end{document}